\documentclass[a4paper]{article}

\usepackage[english]{babel}
\usepackage[utf8]{inputenc}
\usepackage[T1]{fontenc}

\usepackage{mathtools}
\usepackage{amsthm}
\usepackage{amsfonts}
\usepackage{amssymb}
\usepackage{stmaryrd}
\usepackage{relsize}
\usepackage{accents}

\usepackage[margin=3cm]{geometry}
\usepackage[pdftex,pdfborder={0 0 0},linktoc=all]{hyperref}
\usepackage{pdfpages}

\theoremstyle{plain}
	\newtheorem{thm}{Theorem}[section]
	\newtheorem{cor}[thm]{Corollary}
	\newtheorem{lem}[thm]{Lemma}
	\newtheorem{prop}[thm]{Proposition}

\theoremstyle{definition}
	\newtheorem{dfn}[thm]{Definition}

\theoremstyle{remark}
	\newtheorem{rem}[thm]{Remark}
	\newtheorem{ex}[thm]{Example}

\numberwithin{equation}{section}
\counterwithin{figure}{section}

\newcommand{\B}{\mathbb{B}}
\newcommand{\C}{\mathbb{C}}
\newcommand{\G}{\mathbb{G}}
\newcommand{\N}{\mathbb{N}}
\renewcommand{\P}{\mathbb{P}}

\newcommand{\R}{\mathbb{R}}
\renewcommand{\S}{\mathbb{S}}

\newcommand{\cA}{\mathcal{A}}
\newcommand{\cB}{\mathcal{B}}
\newcommand{\cC}{\mathcal{C}}
\newcommand{\cD}{\mathcal{D}}
\newcommand{\cE}{\mathcal{E}}
\newcommand{\cF}{\mathcal{F}}
\newcommand{\cG}{\mathcal{G}}
\newcommand{\cH}{\mathcal{H}}
\newcommand{\cI}{\mathcal{I}}
\newcommand{\cJ}{\mathcal{J}}
\newcommand{\cK}{\mathcal{K}}
\newcommand{\cL}{\mathcal{L}}
\newcommand{\cM}{\mathcal{M}}
\newcommand{\cN}{\mathcal{N}}
\newcommand{\cO}{\mathcal{O}}
\newcommand{\cP}{\mathcal{P}}
\newcommand{\cQ}{\mathcal{Q}}
\newcommand{\cR}{\mathcal{R}}
\newcommand{\cS}{\mathcal{S}}
\newcommand{\cT}{\mathcal{T}}
\newcommand{\cV}{\mathcal{V}}
\newcommand{\cW}{\mathcal{W}}

\newcommand{\fC}{\mathfrak{C}}
\newcommand{\fP}{\mathfrak{P}}
\newcommand{\fp}{\mathfrak{p}}

\newcommand{\odiag}{\obullet{\diag}_{\brackets{A},\eta}}
\newcommand{\oK}{\obullet{K}}
\newcommand{\oux}{\obullet{\underline{x}}}
\newcommand{\ouy}{\obullet{\underline{y}}}

\newcommand{\ug}{\underline{g}}
\newcommand{\uh}{\underline{h}}
\newcommand{\ul}{\underline{\ell}}
\newcommand{\um}{\underline{m}}
\newcommand{\un}{\underline{n}}
\newcommand{\uP}{\underline{P}}
\newcommand{\up}{\underline{p}}
\newcommand{\uq}{\underline{q}}
\newcommand{\ur}{\underline{r}}
\newcommand{\us}{\underline{s}}
\newcommand{\ut}{\underline{t}}
\newcommand{\ux}{\underline{x}}
\newcommand{\uy}{\underline{y}}
\newcommand{\uz}{\underline{z}}

\newcommand{\ualpha}{\underline{\alpha}}
\newcommand{\ubeta}{\underline{\beta}}
\newcommand{\ugamma}{\underline{\gamma}}
\newcommand{\uLambda}{\underline{\Lambda}}
\newcommand{\uphi}{\underline{\phi}}
\newcommand{\uvarphi}{\underline{\varphi}}
\newcommand{\uSigma}{\underline{\Sigma}}
\newcommand{\utau}{\underline{\tau}}

\renewcommand{\epsilon}{\varepsilon}
\renewcommand{\geq}{\geqslant}
\renewcommand{\leq}{\leqslant}
\renewcommand{\tilde}{\widetilde}

\newcommand{\acts}{\curvearrowright}
\newcommand{\cov}[2]{\Cov\parentheses*{#1, #2}}
\newcommand{\cCM}{\cC^{0,\infty}(\cM \times \Omega)}
\newcommand{\cwedge}{\curlywedge 2}
\newcommand{\diag}{\mathlarger{\mathlarger{\boxslash}}}
\newcommand{\dx}{\dmesure\!}
\newcommand{\esp}[2][]{\mathbb{E}_{#1}\squarebrackets*{#2}}
\newcommand{\espcond}[3][]{\mathbb{E}_{#1}\squarebrackets*{#2 \mvert #3}}
\newcommand{\gauss}[1]{\mathcal{N}\parentheses*{0,#1}}
\newcommand{\gr}[2]{\Gr\parentheses*{#1,#2}}
\newcommand{\jac}[1]{\Jac\parentheses*{#1}}
\newcommand{\mvert}{\mathrel{}\middle|\mathrel{}}
\newcommand{\obullet}[1]{\accentset{\bullet}{#1}}
\newcommand{\one}{\mathbf{1}}
\newcommand{\overgroup}[1]{\overbracket[0.4pt]{#1}}
\newcommand{\simplex}{\mathlarger{\mathlarger{\blacktriangle}}}
\newcommand{\trans}[1]{\,\prescript{\text{t}}{}{\! #1}}
\newcommand{\var}[1]{\Var\parentheses*{#1}}

\DeclareMathOperator{\card}{Card}
\DeclareMathOperator{\conv}{Conv}
\DeclareMathOperator{\Cov}{Cov}
\DeclareMathOperator{\dist}{dist}
\DeclareMathOperator{\dmesure}{d}
\DeclareMathOperator{\ev}{ev}
\DeclareMathOperator{\Gr}{Gr}
\DeclareMathOperator{\Id}{Id}
\DeclareMathOperator{\Jac}{Jac}
\DeclareMathOperator{\Span}{Span}
\DeclareMathOperator{\sym}{Sym}
\DeclareMathOperator{\Tr}{Tr}
\DeclareMathOperator{\Var}{Var}
\DeclareMathOperator{\Vol}{Vol}

\DeclarePairedDelimiter{\brackets}{\{}{\}}

\DeclarePairedDelimiter{\floor}{\lfloor}{\rfloor}
\DeclarePairedDelimiter{\norm}{\lvert}{\rvert}
\DeclarePairedDelimiter{\Norm}{\lVert}{\rVert}
\DeclarePairedDelimiter{\parentheses}{(}{)}
\DeclarePairedDelimiterX{\prsc}[2]{\langle}{\rangle}{#1\,, #2}
\DeclarePairedDelimiter{\squarebrackets}{[}{]}
\DeclarePairedDelimiterX{\ssquarebrackets}[2]{\llbracket}{\rrbracket}{#1,#2}

\newcounter{hypotheses}

\newcommand{\hypitem}[1]{\item[#1.] \refstepcounter{hypotheses}}

\newcommand{\hypDC}[2]{(\hyperref[hyp: decay of correlations]{\textbf{DC}($#1,#2$)})}
\newcommand{\hypDCL}[2]{(\hyperref[hyp: decay of correlations limit]{\textbf{DCL}($#1,#2$)})}
\newcommand{\hypND}[1]{(\hyperref[hyp: non-degeneracy]{\textbf{ND}($#1$)})}
\newcommand{\hypNND}[1]{(\hyperref[hyp: nabla-non-degeneracy]{$\nabla$-\textbf{ND}($#1$)})}
\newcommand{\hypReg}[1]{(\hyperref[hyp: regularity]{\textbf{Reg}($#1$)})}
\newcommand{\hypScL}[1]{(\hyperref[hyp: scaling limit]{\textbf{ScL}($#1$)})}

\author{Michele Ancona\,\thanks{Université Côte d’Azur, CNRS, Laboratoire J.A. Dieudonné; \url{michele.ancona@unice.fr}.} \and Louis Gass\,\thanks{University of Luxembourg; \url{louis.gass@uni.lu}.} \and Thomas Letendre\,\thanks{Université Paris-Saclay, CNRS, Laboratoire de Mathématiques d’Orsay; \url{letendre@math.cnrs.fr}.} \and Michele Stecconi\,\thanks{University of Luxembourg; \url{michele.stecconi@uni.lu}\\ This work was supported by the Luxembourg National Research Fund (Grant number: 021/16236290/HDSA).}}
\date{\today}
\title{Zeros and critical points of Gaussian fields: cumulants asymptotics and limit theorems}

\begin{document}

\maketitle

\begin{abstract}
Let $f:\R^d \to \R^k$ be a smooth centered stationary Gaussian field and $\cB \subset \R^d$ be a bounded Borel set. In this paper, we determine the asymptotics as $R \to \infty$ of all the cumulants of the $(d-k)$-dimensional volume of $f^{-1}(0) \cap R\cB$. When $k=1$, we obtain similar asymptotics for the number of critical points of $f$ in $R\cB$. Our main hypotheses are some regularity and non-degeneracy of the field, as well as mild integrability conditions on the first derivatives of its covariance kernel. As corollaries of these cumulants estimates, we deduce a strong Law of Large Numbers and a Central Limit Theorem for the nodal volume (resp.~the number of critical points) of a regular and non-degenerate enough field whose covariance decays fast enough at infinity. Our results hold more generally for a one-parameter family $(f_R)$ of Gaussian fields admitting a stationary local scaling limit as $R \to \infty$, for example Kostlan polynomials in the large degree limit. They also hold for the random measures of integration over the vanishing locus of $f_R$ as $R \to +\infty$.
\end{abstract}



%%%%%%%%%%%%%%%%%%%%%%%%%%%%%%%%%%%%%%%%%%%%%%%%%%%%%%%%%%%%%%%%%%%%%%%%%%%%%%%%%%%%%%%%%%%%%%%%%%%%%%%
%%%%%%%%%%%%%%%%%%%%%%%%%%%%%%%%%%%%%%%%%%%%%%%%%%%%%%%%%%%%%%%%%%%%%%%%%%%%%%%%%%%%%%%%%%%%%%%%%%%%%%%

\tableofcontents

%%%%%%%%%%%%%%%%%%%%%%%%%%%%%%%%%%%%%%%%%%%%%%%%%%%%%%%%%%%%%%%%%%%%%%%%%%%%%%%%%%%%%%%%%%%%%%%%%%%%%%%
%%%%%%%%%%%%%%%%%%%%%%%%%%%%%%%%%%%%%%%%%%%%%%%%%%%%%%%%%%%%%%%%%%%%%%%%%%%%%%%%%%%%%%%%%%%%%%%%%%%%%%%

\section{Introduction}
\label{sec: Introduction}

%%%%%%%%%%%%%%%%%%%%%%%%%%%%%%%%%%%%%%%%%%%%%%%%%%%%%%%%%%%%%%%%%%%%%%%%%%%%%%%%%%%%%%%%%%%%%%%%%%%%%%%

\subsection{Zero set of a stationary Gaussian field in a growing set}
\label{subsec: zero set of a stationary Gaussian field in a growing set}

Let us begin by describing our results in a model case. The general framework we consider will be introduced in Section~\ref{subsec: cumulants asymptotics} below. Let $d \in \N^*$ and $k \in \brackets{1,\dots,d}$. Let $f:\R^d \to \R^k$ be a centered stationary $\cC^2$ Gaussian field. Under mild non-degeneracy assumptions, its zero set $Z=f^{-1}(0)$ is almost-surely a codimension-$k$ submanifold in $\R^d$. The Euclidean metric of $\R^d$ then induces a volume measure on $Z$, which is its $(d-k)$-dimensional Hausdorff measure. Let $\B$ denote the unit ball in $\R^d$. For all $R>0$, we let $R\B = \brackets{Rx\mid x \in \B}$. It is a classical problem in probability to study the asymptotic distribution as $R \to +\infty$ of the random variables $\Vol_{d-k}\parentheses*{Z \cap R\B}$, where the volume is computed with respect to the $(d-k)$-dimensional Hausdorff measure. More generally, we can consider the same question where the unit ball $\B$ is replaced by any bounded Borel set~$\cB$.

Letting $\esp{X}$ denote the expectation of a random variable $X$, the classical Kac--Rice formula, cf.~\cite{AT2007,AW2009,Kac1943,Ric1944}, implies that:
\begin{equation}
\label{eq: volume expectation}
\forall R >0, \qquad \esp{\frac{\Vol_{d-k}\parentheses*{Z \cap R\cB}}{R^d}} = \gamma_1(f) \Vol_d(\cB),
\end{equation}
where the volume of $\cB$ is computed with respect to the Lebesgue measure on $\R^d$, and $\gamma_1(f)$ is a constant depending only on the distribution of $f$. Variance asymptotics and a Central Limit Theorem (CLT) were obtained for the length of random curves ($d=2$ and $k=1$) in~\cite{KL2001}. Similar results for the volume of random hypersurfaces ($k=1$) were obtained in~\cite{KV2018}, see also~\cite{DEL2021}. Proving that $\Vol_{d-k}\parentheses*{Z \cap \cB}$ admits a finite $p$-th moment is already a delicate problem, that received a lot of attention, in several contexts, over the past years, see for instance~\cite{AAD+2023,BCW2019a,BMM2024,MV1993}. In~\cite{AL2025,GS2024}, the authors gave simple sufficient conditions on the field $f$ ensuring that $\Vol_{d-k}\parentheses*{Z \cap R\cB}$ admits a finite $p$-th moment for any bounded Borel set $\cB$ and $R>0$. Our first result establishes the large $R$ asymptotics of the $p$-th central moment of $\Vol_{d-k}\parentheses*{Z \cap R\cB}$, see Theorem~\ref{thm: moments asymptotics volume}.

Let us introduce some notation in order to state our main hypotheses. Given a multi-index $\alpha=(\alpha_1,\dots,\alpha_d) \in \N^d$, we denote by $\norm{\alpha}=\alpha_1+\dots+\alpha_d$ and by $\partial^\alpha$ the partial derivative $\partial_1^{\alpha_1}\dots\partial_d^{\alpha_d}$. If $\alpha,\beta \in \N^d$ then $(\alpha,\beta) \in \N^{2d}$ and we simply let $\partial^{\alpha,\beta}=\partial^{(\alpha,\beta)}$. For all $x,y \in \R^d$, we denote by $r(x,y)=\cov{f(x)}{f(y)}$ the covariance operator of $f(x)$ and $f(y)$, that is, the linear map from~$\R^k$ to itself whose matrix is equal to $\esp{f(x) \trans{f(y)}}$. If $f$ is~$\cC^q$ for some $q \in \N$, then its covariance kernel $r$ admits continuous partial derivatives $\partial^{\alpha,\beta}r$ on $\R^d \times \R^d$ for all $\alpha,\beta \in \N^d$ such that $\norm{\alpha}\leq q$ and $\norm{\beta}\leq q$. Moreover, since $f$ is stationary, we have $r^{\alpha,\beta}(x,y)=r^{\alpha,\beta}(0,y-x)$ for all $x,y \in \R^d$.

Let $q \in \N$ and $s \in [1,+\infty]$. Assuming that $f$ is at least of class $\cC^q$, we introduce the following hypotheses. We will comment on these after the statements of our main theorems, see Section~\ref{subsec: discussion of the main hypotheses}.

\begin{description}
\hypitem{Non-degeneracy (ND$(q)$)}
\label{hyp: non-degeneracy}
The field $f$ is \emph{$q$-non-degenerate}, that is: for any $l\in \brackets{1,\dots,q+1}$, any $m_1,\dots,m_l \in \N^*$ such that $m_1 + \dots + m_l =q+1$, and any family of pairwise distinct points $x_1,\dots,x_l \in \R^d$, the centered Gaussian vector $\parentheses*{\partial^\alpha f(x_i)}_{1 \leq i \leq l, \norm{\alpha} < m_i}$ is non-degenerate.

\hypitem{Decay of correlations in the limit (DCL$(q,s)$)}
\label{hyp: decay of correlations limit}
There exist $\omega>0$ and an even bounded function $g:\R^d \to [0,+\infty)$ such that:
\begin{itemize}
\item we have $g(x)\xrightarrow[\Norm{x}\to +\infty]{}0$;
\item the function $g_\omega:x \mapsto \sup_{\Norm{y}\leq \omega} g(x+y)$ belongs to $L^s(\R^d)$;
\item for all $z \in \R^d$, we have $\max_{\norm{\alpha} \leq q, \norm{\beta}\leq q} \ \Norm*{\strut \partial^{\alpha,\beta} r(0,z)}\leq g(z)$.
\end{itemize}
\end{description}

In the following, we denote by $\gauss{\lambda}$ the real-valued centered Gaussian distribution of variance $\lambda\geq 0$. Recalling that the expected volume of $Z$ in a growing set is given by Equation~\eqref{eq: volume expectation}, our first result is the following.

\begin{thm}[Central moments asymptotics for volumes]
\label{thm: moments asymptotics volume}
Let $p \geq 2$ be an integer and $Z$ be the zero set of a centered stationary Gaussian field $f:\R^d \to \R^k$ of class $\cC^{2p-1}$. We assume that~\hypND{2p-1} and~\hypDCL{2p-1}{2} hold. Then, there exists $\gamma_2(f) \geq 0$, depending only on the distribution of $f$, such that, for any bounded Borel subset $\cB \subset \R^d$, we have:
\begin{equation*}
\esp{\parentheses*{\frac{\Vol_{d-k}\parentheses*{Z \cap R\cB} - R^d \gamma_1(f)\Vol_d(\cB)}{R^\frac{d}{2}}}^p} \xrightarrow[R \to +\infty]{} \parentheses*{\gamma_2(f)\strut\Vol_d(\cB)}^\frac{p}{2}\ \esp{\strut\gauss{1}^p}.
\end{equation*}
\end{thm}

\begin{rem}
\label{rem: moments asymptotics volume}
Under the hypotheses of Theorem~\ref{thm: moments asymptotics volume}, we have $\gamma_1(f) >0$. The constant $\gamma_2(f)$ is non-negative as a limit variance, and we expect it to be positive for a generic field. Conditions ensuring that $\gamma_2(f) >0$ are stated in~\cite{Gas2025}. Under our assumptions, if $f=(f_1,\dots,f_k)$, where $f_1,\dots,f_k$ are independent real-valued stationary isotropic fields on $\R^d$, then $\gamma_2(f)>0$, see~\cite[Thm.~1.7]{Gas2025}. Recall also that $\esp{\gauss{1}^p}=0$ if and only if $p$ is odd.
\end{rem}

Let us now focus on the important case of gradient fields. Note that it is not directly covered by Theorem~\ref{thm: moments asymptotics volume} since, if $f:\R^d \to \R^d$ is the gradient of some Gaussian field $h:\R^d \to \R$, then $f$ does not satisfy~\hypND{q} for any $q \geq 1$. Indeed, $\parentheses*{\partial_jf(0)}_{1\leq j \leq d}=\parentheses*{\partial_i\partial_jh(0)}_{1 \leq i,j\leq d}$ is supported on the set of symmetric matrices, hence is degenerate as a Gaussian vector in $\R^{d\times d}$. In order to deal with gradient fields, we introduce a variation on our non-degeneracy hypothesis \hypND{q} adapted to these fields.

Let $q \in \N$. We assume that $f=\nabla h$, where $h:\R^d \to \R$ is a centered stationary Gaussian field at least of class $\cC^{q+1}$. The following condition depends only on $f$ but is more conveniently stated in terms of $h$.

\begin{description}
\hypitem{$\nabla$-non-degeneracy ($\nabla$-ND$(q)$)}
\label{hyp: nabla-non-degeneracy}
The field $h$ such that $\nabla h=f$ satisfies the following: for any $l\in \brackets{1,\dots,q+1}$, any $m_1,\dots,m_l \in \N^*$ such that $m_1 + \dots + m_l =q+1$, and any pairwise distinct~points $x_1,\dots,x_l \in \R^d$, the Gaussian vector $\parentheses*{\partial^\alpha h(x_i)}_{1 \leq i \leq l, 0<\norm{\alpha} \leq m_i}$ is non-degenerate.
\end{description}

This alternative non-degeneracy hypothesis allows us to prove the analogue of Theorem~\ref{thm: moments asymptotics volume} for the number of critical points of $h$. Note that, in this case, $Z$ is a locally finite set, and its $0$-dimensional Hausdorff measure is its counting measure.

\begin{thm}[Central moments asymptotics for critical points]
\label{thm: moments asymptotics critical points}
Let $p \geq 2$ be an integer and $h:\R^d \to \R$ be a $\cC^{2p}$ centered stationary Gaussian field. Let $f =\nabla h$ and $Z=f^{-1}(0)$. We assume that \hypNND{2p-1} and~\hypDCL{2p-1}{2} hold. Then, there exists $\gamma_2(f) > 0$, depending only on the distribution of $f$, such that, for any bounded Borel subset $\cB \subset \R^d$, we have:
\begin{equation*}
\esp{\parentheses*{\frac{\card\parentheses*{Z \cap R\cB} - R^d \gamma_1(f)\Vol_d(\cB)}{R^\frac{d}{2}}}^p} \xrightarrow[R \to +\infty]{} \parentheses*{\gamma_2(f)\strut\Vol_d(\cB)}^\frac{p}{2}\esp{\strut\gauss{1}^p},
\end{equation*}
where $\gamma_1(f)$ is the constant appearing in Equation~\eqref{eq: volume expectation}.
\end{thm}

\begin{rem}
\label{rem: moments asymptotics critical points}
Under the hypotheses of Theorem~\ref{thm: moments asymptotics critical points}, we have $\gamma_1(f) >0$. Note that, in this case, we always have $\gamma_2(f) >0$.
\end{rem}

In the case where the field $f$ is smooth and non-degenerate at any order, we recover a CLT as $R \to +\infty$, by the method of moments. If $f$ is the gradient of an isotropic field, a similar result was proved by Nicolaescu~\cite[Thm.~2.2]{Nic2017}, under slightly different hypotheses.

\begin{thm}[Central Limit Theorem, stationary case]
\label{thm: central limit theorem stationary case}
Let $f:\R^d \to \R^k$ be a $\cC^\infty$ centered stationary Gaussian field such that \hypDCL{q}{2} holds for all $q \in \N$. We assume that, either \hypND{q} holds for all $q \in \N$, or $f$ is the gradient of some $h:\R^d \to \R$ and \hypNND{q} holds for all $q \in \N$. Let $Z=f^{-1}(0)$. Then, for any bounded Borel subset $\cB \subset \R^d$, we have the following convergence in distribution:
\begin{equation*}
\frac{\Vol_{d-k}\parentheses*{Z \cap R\cB} - R^d \gamma_1(f)\Vol_d(\cB)}{R^\frac{d}{2}} \xrightarrow[R \to +\infty]{\text{law}} \gauss{\gamma_2(f)\strut\Vol_d(\cB)}.
\end{equation*}
\end{thm}

The Borel--Cantelli Lemma and a monotonicity argument also yield a Law of Large Numbers as $R \to +\infty$ for the volume of $f^{-1}(0)$ in nice enough growing sets. Recall that a set $\cB \subset \R^d$ is said to be star-shaped with respect to $0$ if, for all $x \in \cB$ and $t \in [0,1]$, we have $tx \in \cB$.

\begin{thm}[Law of Large Numbers, stationary case]
\label{thm: law of large numbers stationary case}
Let $f:\R^d \to \R^k$ be a $\cC^3$ centered stationary Gaussian field such that \hypDCL{3}{2} holds. We assume that, either \hypND{3} holds, or $f$ is the gradient of some field $h:\R^d \to \R$ and \hypNND{3} holds. Let $Z=f^{-1}(0)$. Then, for any bounded Borel subset $\cB \subset \R^d$ which is star-shaped with respect to $0$, we have the following almost-sure convergence:
\begin{equation*}
\frac{\Vol_{d-k}\parentheses*{Z \cap R\cB}}{R^d} \xrightarrow[R \to +\infty]{\text{a.s.}} \gamma_1(f)\Vol_d(\cB).
\end{equation*}
\end{thm}

%%%%%%%%%%%%%%%%%%%%%%%%%%%%%%%%%%%%%%%%%%%%%%%%%%%%%%%%%%%%%%%%%%%%%%%%%%%%%%%%%%%%%%%%%%%%%%%%%%%%%%%

\subsection{Cumulants asymptotics}
\label{subsec: cumulants asymptotics}

The results from Section~\ref{subsec: zero set of a stationary Gaussian field in a growing set} are special cases of results obtained in a general framework that we now introduce. As above, let $d \in \N^*$ denote the dimension of the ambient space and $k \in \brackets{1,\dots,d}$ be the almost-sure codimension of the random submanifolds we are interested in. Let $f:\R^d \to \R^k$ be a centered stationary Gaussian field. Let $U \subset \R^d$ be a convex open set and $\Omega \subset U$ be a non-empty open subset at positive distance from $\R^d \setminus U$. That is, $\inf \brackets*{\strut \Norm{y-x} \mvert x \in \Omega, y \notin U} >0$, where $\Norm{\cdot}$ denotes the Euclidean norm on $\R^d$.

Let $\cR \subset (0,+\infty)$ be an unbounded set. For all $R \in \cR$, let $f_R$ be a $\cC^2$ centered Gaussian field from $RU = \brackets{Rx \mid x \in U}\subset \R^d$ to~$\R^k$. We let $r_R$ denote its covariance kernel, which is defined of $RU \times RU$. Under mild assumptions, the zero set $Z_R=f_R^{-1}(0)$ is almost-surely a submanifold of codimension $k$ in~$RU$. As such, integrating over $Z_R$ with respect to the $(d-k)$-dimensional Hausdorff measure $\dx \cH^{d-k}$ defines a Radon measure on $RU$ (i.e., a Borel measure which is finite on compact subsets), that we also denote by $Z_R$. For all $\phi \in L^1(RU)$, we denote by
\begin{equation}
\label{eq: def linear statistics ZR}
\prsc{Z_R}{\phi} = \int_{Z_R} \phi(x) \dx \cH^{d-k}(x),
\end{equation}
which is an almost-surely well-defined random variable, see Section~\ref{subsec: Bulinskaya Lemma and Kac--Rice formulas for factorial moments} for details. Quantities of the type $\prsc{Z_R}{\phi}$ are called the \emph{linear statistics} of $Z_R$.

In this paper, we are interested in the asymptotic behavior of the random measure $Z_R$ as $R \to +\infty$, under the assumption that $f_R$ converges towards $f$, in a sense to be made precise below. We study this asymptotic behavior through the asymptotic distribution of some linear statistics. In order to obtain converging objects, we have to consider a scaled version of $Z_R$, denoted by $\nu_R$, which is the random Radon measure on $U$ defined by:
\begin{equation}
\label{eq: def linear statistics nuR}
\forall \phi \in L^1(U), \qquad \prsc{\nu_R}{\phi} = \frac{1}{R^d}\prsc*{Z_R}{\phi\parentheses*{\frac{\cdot}{R}}}.
\end{equation}
Alternatively, $\nu_R$ can be described as the measure of integration over the submanifold $\frac{1}{R}Z_R \subset U$, up to a scaling factor $\frac{1}{R^k}$ since we consider $(d-k)$-dimensional submanifolds. Note that we normalized $\nu_R$ so that $\esp{\prsc{\nu_R}{\phi}}$ converges as $R \to +\infty$. For example, if $\cB \subset U$ is a bounded Borel set and $\one_\cB$ is its indicator function then, by~\eqref{eq: def linear statistics ZR} and~\eqref{eq: def linear statistics nuR}, we have:
\begin{equation}
\label{eq: relation nuR volume}
\prsc{\nu_R}{\one_\cB}= \frac{\Vol_{d-k}\parentheses*{Z_R \cap R\cB}}{R^d}.
\end{equation}
If moreover $f_R$ is the restriction $f_{\vert RU}$ of $f$ to $RU$ for all $R \in \cR$, then the linear statistics in~\eqref{eq: relation nuR volume} are random variables of the form studied in Section~\ref{subsec: zero set of a stationary Gaussian field in a growing set}, and their expectation is given by~\eqref{eq: volume expectation}.

Let $q \in N$ and $s \in [1,+\infty]$. We already introduced Hypotheses~\hypND{q}, \hypNND{q} and \hypDCL{q}{s} in Section~\ref{subsec: zero set of a stationary Gaussian field in a growing set}. In the setting of the present section, we introduce an additional set of hypotheses. We will comment on these in Section~\ref{subsec: discussion of the main hypotheses} below.

\begin{description}
\hypitem{Regularity (Reg$(q)$)}
\label{hyp: regularity}
The centered Gaussian fields $f$ and $(f_R)_{R \in \cR}$ all are of class $\cC^q$.
\end{description}

Assuming that \hypReg{q} holds, we can state our other hypotheses as follows.

\begin{description}
\hypitem{Local scaling limit (ScL$(q)$)}
\label{hyp: scaling limit}
For all non-empty compact set $\Gamma \subset \R^d$ the following holds:
\begin{equation*}
\sup_{\norm{\alpha} \leq q, \norm{\beta}\leq q}\ \sup_{x \in \Omega}\ \sup_{w,z \in \Gamma}\ \Norm*{\strut \partial^{\alpha,\beta} r_R(Rx+w,Rx+z) - \partial^{\alpha,\beta}r(w,z)} \xrightarrow[R \to +\infty]{}0.
\end{equation*}

\hypitem{Uniform decay of correlations (DC$(q,s)$)}
\label{hyp: decay of correlations}
There exist $\omega>0$ and $g:\R^d \to [0,+\infty)$ such that:
\begin{itemize}
\item the function $g$ is even, bounded, and $g(x)\xrightarrow[\Norm{x}\to +\infty]{}0$;
\item the function $g_\omega:x \mapsto \sup_{\Norm{y}\leq \omega} g(x+y)$ belongs to $L^s(\R^d)$;
\item for all $R \in \cR$, for all $x,y \in RU$, we have $\max_{\norm{\alpha} \leq q, \norm{\beta}\leq q} \ \Norm*{\strut \partial^{\alpha,\beta} r_R(x,y)}\leq g(y-x)$.
\end{itemize}
\end{description}

Given $R \in \cR$ and $\phi_1,\dots,\phi_p \in L^1(U)$, we denote by $\kappa(\nu_R)(\phi_1,\dots,\phi_p)$ the \emph{mixed cumulant} of the family of random variables $\parentheses*{\prsc{\nu_R}{\phi_i}}_{1 \leq i \leq p}$, if it is well-defined. The precise definition of this mixed cumulant being a bit technical, we postpone it until Section~\ref{subsec: cumulants}. For now, let us say that $\kappa(\nu_R)(\phi_1,\dots,\phi_p)$ is defined as a universal polynomial in the mixed moments $\parentheses*{\esp{\prod_{a \in A}\prsc{\nu_R}{\phi_a}}}_{A \subset \brackets{1,\dots,p}}$. For example, $\kappa(\nu_R)(\phi_1) = \esp{\prsc{\nu_R}{\phi_1}}$, and $\kappa(\nu_R)(\phi_1,\phi_2)$ is the covariance of $\prsc{\nu_R}{\phi_1}$ and $\prsc{\nu_R}{\phi_2}$. Conversely, the mixed moments can be expressed as universal polynomials in the mixed cumulants, so that studying the cumulants of a family of random variables is equivalent to studying its moments or its central moments. In particular, the classical method of moments can be translated into the method of cumulants, see Proposition~\ref{prop: method of cumulants}. The main interest of cumulants is that they are designed to detect Gaussianity: a family of random variables is jointly Gaussian if and only if its mixed cumulants of order at least $3$ vanish, while the mixed cumulants of order $1$ and $2$ describe its mean and covariance. The asymptotic behavior of the cumulants of a sequence of random variables can thus be understood as a form of quantitative Central Limit Theorem. In this spirit, the first main result of this paper is the following.

\begin{thm}[Cumulants asymptotics for zero sets]
\label{thm: cumulants asymptotics for zero sets}
Let $f: \R^d \to \R^k$ be a centered stationary Gaussian field. For all $R \in \cR$, let $f_R:RU \to \R^k$ be a centered Gaussian field and $\nu_R$ be the measure defined by~\eqref{eq: def linear statistics nuR}. Let $p \in \N^*$, we assume that Hypotheses~\hypReg{\max(2,2p-1)}, \hypND{2p-1} and \hypScL{2p-1} hold. Then, for all $\phi_1,\dots,\phi_p \in L^1(\Omega) \cap L^\infty(\Omega)$ we have the following.
\begin{enumerate}
\item \label{item: cumulants asymptotics infty} If \hypDC{2p-1}{\infty} holds, then there exists $R_p \in \cR$, not depending on $\parentheses*{\phi_1,\dots,\phi_p}$, such that for all $R \in [R_p,+\infty) \cap \cR$ the mixed cumulant $\kappa(\nu_R)(\phi_1,\dots,\phi_p)$ is well-defined and finite.

\item \label{item: cumulants asymptotics 2} If $p \geq 3$ and \hypDC{2p-1}{2} holds, then $\kappa(\nu_R)(\phi_1,\dots,\phi_p) = o\parentheses*{R^{-\frac{pd}{2}}}$ as $R \to +\infty$.

\item \label{item: cumulants asymptotics p} If $p=1$, or if $p \geq 2$ and \hypDC{2p-1}{\frac{p}{p-1}} holds, then there exists $\gamma_p(f) \in \R$ such that:
\begin{equation*}
R^{d(p-1)}\kappa(\nu_R)(\phi_1,\dots,\phi_p) \xrightarrow[R \to +\infty]{} \gamma_p(f) \int_\Omega \prod_{i=1}^p \phi_i(x) \dx x.
\end{equation*}
Moreover, $\gamma_p(f)$ depends only on the distribution of $f$, it is explicit, and $\gamma_1(f)>0$.
\end{enumerate}
\end{thm}

In the case where the ambient dimension is $d=1$, Theorem~\ref{thm: cumulants asymptotics for zero sets} was proved in~\cite{Gas2023b} by refining a result of~\cite{AL2021}. More generally, there is a huge literature concerning zeros of Gaussian fields in dimension $1$, and we refer to the introductions of~\cite{AL2021,Gas2023b} for a discussion of related results in this setting. As in the model case, Theorem~\ref{thm: cumulants asymptotics for zero sets} does not cover the important case where $f$ and the fields $(f_R)_{R \in \cR}$ are gradient fields. Our second main result deals precisely with this case.

\begin{thm}[Cumulants asymptotics for critical points]
\label{thm: cumulants asymptotics for critical points}
Let us consider a centered stationary Gaussian field $h: \R^d \to \R$ and $f= \nabla h$. For all $R \in \cR$, let $h_R:RU \to \R$ be a centered Gaussian field, $f_R=\nabla h_R$ and $\nu_R$ be the measure on $U$ defined by~\eqref{eq: def linear statistics nuR}. Let $p \in \N^*$, we assume that Hypotheses~\hypReg{\max(2,2p-1)}, \hypNND{2p-1} and \hypScL{2p-1} hold. Then, for all $\phi_1,\dots,\phi_p \in L^1(\Omega) \cap L^\infty(\Omega)$ we have the following.
\begin{enumerate}
\item \label{item: cumulants asymptotics infty crit} If \hypDC{2p-1}{\infty} holds, then there exists $R_p \in \cR$, not depending on $\parentheses*{\phi_1,\dots,\phi_p}$, such that for all $R \in [R_p,+\infty) \cap \cR$ the cumulant $\kappa(\nu_R)(\phi_1,\dots,\phi_p)$ is well-defined and finite.

\item \label{item: cumulants asymptotics 2 crit} If $p \geq 3$ and \hypDC{2p-1}{2} holds, then $\kappa(\nu_R)(\phi_1,\dots,\phi_p) = o\parentheses*{R^{-\frac{pd}{2}}}$ as $R \to +\infty$.

\item \label{item: cumulants asymptotics p crit} If $p=1$, or if $p \geq 2$ and \hypDC{2p-1}{\frac{p}{p-1}} holds, then there exists $\gamma_p(f) \in \R$ such that:
\begin{equation*}
R^{d(p-1)}\kappa(\nu_R)(\phi_1,\dots,\phi_p) \xrightarrow[R \to +\infty]{} \gamma_p(f) \int_\Omega \prod_{i=1}^p \phi_i(x) \dx x.
\end{equation*}
Moreover, $\gamma_p(f)$ depends only on the distribution of $f$, it is explicit, $\gamma_1(f)>0$ and $\gamma_2(f)>0$.
\end{enumerate}
\end{thm}

\begin{rem}
\label{rem: test function class}
Theorems~\ref{thm: cumulants asymptotics for zero sets} and~\ref{thm: cumulants asymptotics for critical points} describe the asymptotic behavior of $\nu_R$ restricted to the subset $\Omega \subset U$, since we consider test-functions in $L^1(\Omega) \cap L^\infty(\Omega)$. Note that this class of test-functions contains the indicator functions of bounded Borel subsets of $\Omega$, so that the associated linear statistics completely characterize the restriction of $\nu_R$ to $\Omega$.
\end{rem}

%%%%%%%%%%%%%%%%%%%%%%%%%%%%%%%%%%%%%%%%%%%%%%%%%%%%%%%%%%%%%%%%%%%%%%%%%%%%%%%%%%%%%%%%%%%%%%%%%%%%%%%

\subsection{Limit theorems and other corollaries}
\label{subsec: limit theorems and other corollaries}

The cumulant asymptotics in Theorems~\ref{thm: cumulants asymptotics for zero sets} and~\ref{thm: cumulants asymptotics for critical points} imply that the random measures $(\nu_R)_{R \in \cR}$ satisfy a Law of Large Numbers and a Central Limit Theorem as $R \to +\infty$. They also translate into corollaries that can be understood as asymptotic equidistribution results for the random submanifolds $Z_R$. This section gathers these various corollaries.

First, our cumulant asymptotics translates into the following moment asymptotics.

\begin{cor}[Moment asymptotics]
\label{cor: moment asymptotics}
In the same setting as Theorem~\ref{thm: cumulants asymptotics for zero sets} (resp.~Theorem~\ref{thm: cumulants asymptotics for critical points}), let $p \in \N^*$ and assume that \hypReg{\max(2,2p-1)}, \hypND{2p-1} (resp.~\hypNND{2p-1}) and \hypScL{2p-1} hold. Then, for all $\phi_1,\dots, \phi_p \in L^1(\Omega) \cap L^\infty(\Omega)$, we have the following.
\begin{enumerate}
\item \label{item: moment asymptotics infty} If \hypDC{2p-1}{\infty} holds, then there exists $R_p \in \cR$, not depending on $\parentheses*{\phi_1,\dots,\phi_p}$, such that for all $R \in [R_p,+\infty) \cap \cR$ the mixed moment $\displaystyle\esp{\prod_{i=1}^p \norm*{\prsc{\nu_R}{\phi_i}\strut}}$ is finite.

\item \label{item: moment asymptotics 2} If \hypDC{2p-1}{2} holds, then we have:
\begin{equation}
\label{eq: moment asymptotics}
\esp{\prod_{i=1}^p \prsc{\nu_R}{\phi_i}} \xrightarrow[R \to +\infty]{} \gamma_1(f)^p \prod_{i=1}^p \int_\Omega \phi_i(x) \dx x
\end{equation}
and
\begin{equation}
\label{eq: central moment asymptotics}
R^\frac{pd}{2}\esp{\prod_{i=1}^p \parentheses*{\prsc{\nu_R}{\phi_i}-\esp{\prsc{\nu_R}{\phi_i}\strut}}} \xrightarrow[R \to +\infty]{} \gamma_2(f)^\frac{p}{2} \hspace{-1.5ex} \sum_{\substack{(\brackets{a_j,b_j})_{1\leq j \leq \floor{\frac{p}{2}}} \\ \bigsqcup_{j=1}^{\floor{\frac{p}{2}}} \brackets{a_j,b_j} = \brackets{1,\dots,p}}}\! \prod_{j=1}^{\floor{\frac{p}{2}}} \int_{\Omega} \phi_{a_j}(x)\phi_{b_j}(x) \dx x,
\end{equation}
where $\gamma_1(f)$ and $\gamma_2(f)$ are the constants appearing in Theorem~\ref{thm: cumulants asymptotics for zero sets}.\ref{item: cumulants asymptotics p} (resp.~Theorem~\ref{thm: cumulants asymptotics for critical points}.\ref{item: cumulants asymptotics p crit}), and the sum in~\eqref{eq: central moment asymptotics} is indexed by the partitions in pairs of $\brackets{1,\dots,p}$. In particular, if $p$ is odd, the right-hand side of~\eqref{eq: central moment asymptotics} vanishes.
\end{enumerate}
\end{cor}

Then, Markov's inequality yields the following concentration result.

\begin{cor}[Concentration]
\label{cor: concentration}
In the setting of Theorem~\ref{thm: cumulants asymptotics for zero sets} (resp.~Theorem~\ref{thm: cumulants asymptotics for critical points}), let $\epsilon >0$ and $\phi \in L^1(\Omega) \cap L^\infty(\Omega)$. For all $p \in \N^*$, if \hypReg{4p-1}, \hypND{4p-1} (resp.~\hypNND{4p-1}), \hypScL{4p-1} and~\hypDC{4p-1}{2} hold, then we have:
\begin{equation*}
\P\parentheses*{\norm*{\prsc{\nu_R}{\phi}-\gamma_1(f) \int_{\Omega}\phi(x) \dx x}\geq\epsilon} = O\parentheses*{R^{-pd}},
\end{equation*}
where $\gamma_1(f)$ is the constant appearing in Theorem~\ref{thm: cumulants asymptotics for zero sets}.\ref{item: cumulants asymptotics p} (resp.~Theorem~\ref{thm: cumulants asymptotics for critical points}.\ref{item: cumulants asymptotics p crit}).
\end{cor}

Recall that, for all $R \in \cR$, we denoted by $Z_R$ the zero set of $f_R$ and that $\nu_R$ is, up to a normalization factor, the measure of integration over the homothetic copy $\frac{1}{R}Z_R \subset U$. A particular case of Corollary~\ref{cor: concentration} yields the following equidistribution result.

\begin{cor}[Hole probability]
\label{cor: hole probability}
In the setting of Theorem~\ref{thm: cumulants asymptotics for zero sets} (resp.~Theorem~\ref{thm: cumulants asymptotics for critical points}), let $\cO$ be a non-empty open subset of $\Omega$. For all $p \in \N^*$, if \hypReg{4p-1}, \hypND{4p-1} (resp.~\hypNND{4p-1}), \hypScL{4p-1} and~\hypDC{4p-1}{2} hold, then we have:
\begin{equation*}
\P\parentheses*{\frac{1}{R}Z_R \cap \cO = \emptyset\strut} = O\parentheses*{R^{-pd}}.
\end{equation*}
\end{cor}

For all $R \in \cR$, the random measure $\nu_R$ is almost surely Radon, that is, finite on compact subsets. Equivalently, it defines a positive continuous linear form on the space $\cC^0_c(U)$ of continuous functions with compact support in $U$, endowed with its standard topology. This point of view allows us to state a functional Law of Large Numbers for $(\nu_R)_{R \in \cR}$, which is an almost-sure equidistribution result.

\begin{thm}[Law of Large Numbers]
\label{thm: law of large numbers}
In the same setting as Theorem~\ref{thm: cumulants asymptotics for zero sets} (resp.~Theorem~\ref{thm: cumulants asymptotics for critical points}), let $p \in \N^*$ and assume that \hypReg{4p-1}, \hypND{4p-1} (resp.~\hypNND{4p-1}), \hypScL{4p-1} and~\hypDC{4p-1}{2} hold. Let $(R_n)_{n \in \N^*}$ be a sequence of numbers in $\cR$ such that $\sum_{n \geq 1} R_n^{-pd} <\!+\infty$. Then, for all $\phi \in L^1(\Omega) \cap L^\infty(\Omega)$, we have the almost-sure convergence:
\begin{equation}
\label{eq: thm law of large numbers}
\prsc{\nu_{R_n}}{\phi} \xrightarrow[n \to +\infty]{\text{a.s.}} \gamma_1(f) \int_\Omega \phi(x) \dx x,
\end{equation}
where $\gamma_1(f)$ appears in Theorem~\ref{thm: cumulants asymptotics for zero sets}.\ref{item: cumulants asymptotics p} (resp.~Theorem~\ref{thm: cumulants asymptotics for critical points}.\ref{item: cumulants asymptotics p crit}) and $\dx x$ is the Lebesgue measure on~$\Omega$.

Moreover, almost-surely, $\nu_{R_n} \xrightarrow[n \to +\infty]{\text{weak}-*}\gamma_1(f) \dx x$ as Radon measures. That is, almost-surely, the convergence in~\eqref{eq: thm law of large numbers} holds simultaneously for all $\phi \in \cC^0_c(\Omega)$.
\end{thm}

Finally, we deduce a functional CLT from our cumulant asymptotics. In the following, we denote by $\cD(\Omega)$ the space of test-function on $\Omega$, that is, the space of smooth compactly supported functions on $\Omega$. We also denote by $\cD'(\Omega)$ the space of generalized functions on $\Omega$, which is the topological dual of $\cD(\Omega)$, endowed with its usual topology.

\begin{thm}[Central Limit Theorem]
\label{thm: central limit theorem}
In the setting of Theorem~\ref{thm: cumulants asymptotics for zero sets} (resp.~Theorem~\ref{thm: cumulants asymptotics for critical points}), we assume that \hypReg{q}, \hypND{q} (resp.~\hypNND{q}), \hypScL{q} and~\hypDC{q}{2} hold for all $q \in \N$. Then, for all $\phi_1,\dots,\phi_n \in L^1(\Omega) \cap L^\infty(\Omega)$, the random vector $R^\frac{d}{2}\parentheses*{\prsc*{\nu_R}{\phi_i}-\esp{\strut\prsc*{\nu_R}{\phi_i}}}_{1 \leq i \leq n}$ converges in distribution as $R \to +\infty$ to the centered Gaussian vector in $\R^n$ with covariance matrix
\begin{equation*}
\gamma_2(f)\, \begin{pmatrix}\displaystyle\int_\Omega\phi_i(x)\phi_j(x) \dx x \end{pmatrix}_{1 \leq i,j \leq n},
\end{equation*}
where $\gamma_2(f)$ is the constant appearing in Theorem~\ref{thm: cumulants asymptotics for zero sets}.\ref{item: cumulants asymptotics p} (resp.~Theorem~\ref{thm: cumulants asymptotics for critical points}.\ref{item: cumulants asymptotics p crit}).

Moreover, we have the following convergence in distribution in the space $\cD'(\Omega)$:
\begin{equation*}
R^\frac{d}{2}\parentheses*{\nu_R -\strut \esp{\nu_R}} \xrightarrow[R \to +\infty]{\text{law}} \sqrt{\gamma_2(f)} W,
\end{equation*}
where $W$ is the standard Gaussian White Noise on $\Omega$ and, for all $R \in \cR$, the generalized function $\esp{\nu_R}$ is defined by $\prsc*{\esp{\nu_R}}{\phi} = \esp{\prsc{\nu_R}{\phi}}$ for all $\phi \in \cD(\Omega)$.
\end{thm}

%%%%%%%%%%%%%%%%%%%%%%%%%%%%%%%%%%%%%%%%%%%%%%%%%%%%%%%%%%%%%%%%%%%%%%%%%%%%%%%%%%%%%%%%%%%%%%%%%%%%%%%

\subsection{Discussion of the main hypotheses and examples}
\label{subsec: discussion of the main hypotheses}

Let us take some time to discuss the hypotheses of our main theorems. First, note that our assumptions satisfy some monotonicity with respect to the parameters. Given $p,q \in \N$ such that $p \leq q$, if \hypReg{q} (resp.~\hypND{q}, resp.~\hypNND{q}, resp.~\hypScL{q}) holds then so does \hypReg{p} (resp.~\hypND{p}, resp.~\hypNND{p}, resp.~\hypScL{p}). Similarly, for all $s \in [1,+\infty]$, if \hypDC{q}{s} (resp.~\hypDCL{q}{s}) holds then so does \hypDC{p}{s} (resp.~\hypDCL{p}{s}), with the same $\omega$ and~$g$. These last two hypotheses are also monotonous with respect to their second parameter. Let $q \in \N$ and $s,t \in [1,+\infty]$ be such that $s \leq t$. If \hypDC{q}{s} (resp.~\hypDCL{q}{s}) holds, then the function $g_\omega$ appearing in its statement belongs to $L^s(\R^d) \cap L^\infty(\R^d) \subset L^t(\R^d)$. Hence, \hypDC{q}{t} (resp.~\hypDCL{q}{t}) also holds, with the same parameter $\omega$ and the same function $g$.

Assumption \hypReg{q} is self-explanatory. In this paper, we consider fields $f$ and $(f_R)_{R \in \cR}$ at least of class~$\cC^2$. This is part of a classical set of hypotheses ensuring that the vanishing loci of these fields are almost-surely submanifolds of the right dimension. Higher regularity assumptions are used to control the degeneracy of centered Gaussian vectors of the type $\parentheses*{f(x_1),\dots,f(x_p)}$, as the distinct points $x_1,\dots,x_p \in \R^d$ collide with one another. Our approach relies on such a control, already to prove the finiteness of the moments or cumulants we study.

Hypothesis \hypDCL{q}{s} is a quantitative decorrelation assumption for the stationary field $f$ and its partial derivatives of order at most $q$. Note that this condition holds if and only if it holds for some parameter $\omega>0$ and the bounded even function $g:z  \mapsto \max_{\norm{\alpha} \leq q, \norm{\beta}\leq q} \ \Norm*{\strut \partial^{\alpha,\beta} r(0,z)}$. Let us also stress the fact that it is not enough to assume that $g \in L^s(\R^d)$. Indeed, even in the simpler setting of Section~\ref{subsec: zero set of a stationary Gaussian field in a growing set}, we need to control a fuzzy version on the correlation kernel $r$. Thus, we have to assume that $g_\omega \in L^s(\R^d)$, which is strictly stronger than $g \in L^s(\R^d)$ in general.

The meaning of \hypND{q} may be less obvious, so let us give some examples in the case of a scalar field, i.e., with $k=1$. Recalling that $f:\R^d \to \R$ is a stationary Gaussian field, it satisfies \hypND{0} if and only if $\var{f(0)}>0$. It satisfies \hypND{1} if and only if $\parentheses*{f(0),\nabla_0 f}$ is non-degenerate in~$\R^{d+1}$ and, for all $z \in \R^d \setminus \brackets{0}$, the Gaussian vector $\parentheses*{f(0),f(z)}$ is non-degenerate in $\R^2$. If $f$ satisfies \hypND{q} then, for any distinct $x_0,\dots,x_q \in \R^d$, the Gaussian vector $\parentheses*{f(x_0),\dots,f(x_q)}$ is non-degenerate, and so is $\parentheses*{\partial^\alpha f(0)}_{\norm{\alpha}\leq q}$. More generally, \hypND{q} entails the non-degeneracy of Gaussian vectors formed by evaluating partial derivatives of $f$ on $l$ pairwise distinct geometric points, up to orders that depend on the multiplicities $(m_i)_{1 \leq i \leq l}$ of these points.

The prototypical example of a field satisfying the previous assumptions is the Bargmann--Fock field $f_\text{BF}:\R^d \to \R$, whose distribution is described by its correlation kernel $r_\text{BF}:(x,y) \mapsto e^{-\frac{\Norm{y-x}^2}{2}}$. This field is smooth because $r_\text{BF}$ is smooth. Its correlations have Gaussian decay, so that it satisfies \hypDCL{q}{s} for any $q \in \N$ and $s \in [1,+\infty]$. It is also true, but not obvious, that $f_\text{BF}$ satisfies \hypND{q} for all $q \in \N$. Recall that the distribution of a centered stationary Gaussian field $f:\R^d \to \R$ with covariance kernel $r$ is characterized by its spectral measure, which is the finite measure on~$\R^d$ whose Fourier transform equals $z \mapsto r(0,z)$. Adapting the argument of~\cite[Lem.~A.2]{AL2021} to fields on $\R^d$ shows that a sufficient condition for $f$ to satisfy \hypND{q} for all $q \in \N$ is that: the support of its spectral measure is not contained in the vanishing locus of a non-zero analytic function. This condition is satisfied by $f_\text{BF}$, whose spectral measure is the standard Gaussian on~$\R^d$, and more generally by any field whose spectral measure contains an open ball in its support. If $f_1,\dots,f_k$ are independent real-valued Gaussian fields that satisfy \hypND{q}, then $f=(f_1,\dots,f_k)$ also satisfies \hypND{q}. For example, $k$ independent copies of $f_\text{BF}$ define a smooth field $f:\R^d \to \R^k$ that satisfies \hypND{q} and \hypDCL{q}{s} for all $q \in \N$ and $s \in [1,+\infty]$.

The previous discussion shows that fields satisfying \hypND{q} for all $q \in \N$ are fairly generic. Yet, there are interesting examples that do not have this property. For example, the Berry field $f_\text{Ber}:\R^d \to \R$, which is best described by saying that its spectral measure is the uniform measure on the unit Euclidean sphere, does not satisfy \hypND{2} since $\sum_{i=1}^d \partial_i^2 f_\text{Ber} = f_\text{Ber}$ almost-surely.

Assumption \hypNND{q} is a version of \hypND{q} adapted to gradient fields, for which \hypND{1} never holds because of the Schwarz Theorem. It means that the only relations between partial derivatives of $f$ are those coming from the fact that $f=\nabla h$, for some stationary Gaussian field $h:\R^d \to \R$. Given $p \in \N$, Hypothesis \hypNND{2p-1} in Theorems~\ref{thm: moments asymptotics critical points} and~\ref{thm: cumulants asymptotics for critical points}, is implied by the fact that $h$ satisfies \hypND{4p-1}, but not by the fact that $h$ satisfies \hypND{q} for some $q <4p-1$. In particular, the gradient field $\nabla f_\text{BF}$ satisfies the hypotheses of Theorem~\ref{thm: moments asymptotics critical points} for all $p \in \N^*$.

In the general framework of Section~\ref{subsec: cumulants asymptotics}, the fact that $U$ is convex is a technical assumption used in some interpolation procedure, see Section~\ref{sec: polynomial interpolation and Gaussian fields}. It could probably be removed, at the price of some technicalities. The assumption that $\Omega$ is at positive distance from $\R^d \setminus U$ is however important. It ensures that, for any compact subset $\Gamma \subset \R^d$, we have $R\Omega + \Gamma \subset RU$ for all $R \in \cR$ large enough. This is a necessary condition for Hypothesis~\hypScL{q} to make sense.

Hypothesis \hypScL{q} means that, for all $x \in \Omega$, the family of Gaussian fields $\parentheses*{f_R(Rx+\cdot)\strut}_{R \in \cR}$ converges in distribution as $R \to +\infty$ towards $f$, locally uniformly in the $\cC^q$-sense. Moreover, this convergence holds uniformly with respect to $x \in \Omega$. This type of convergence naturally occurs by considering the local scaling limits, in nice charts, of some classical families of Gaussian fields on manifolds, see~\cite[Sect.~2.2]{Sod2016}. Examples of this kind include the Cut-Off (or Band-Limited) ensembles of Random Waves~\cite{Gas2023,GW2016a,GW2017,Let2016,Nic2015,SW2019,Zel2009} and the complex Fubini--Study model~\cite{Anc2021,AL2021a,FLL2015,GW2011,GW2016,LP2019}, of which the Kostlan polynomials discussed in Section~\ref{subsec: Kostlan polynomials} are a special case. The Monochromatic Random Waves~\cite{CH2020,KKW2013,MPRW2016,NS2009,Wig2010,Zel2009} also give rise to a local scaling limit of the type~\hypScL{q}, but their scaling limit is the Berry field, which is degenerate in the sense that it does not satisfy \hypND{2}.

The meaning of \hypDC{q}{s} is the same as that of \hypDCL{q}{s}, but we control the correlations of the fields $f_R$ uniformly with respect to $R$, that is, by the same function $g$ for all $R \in \cR$. Note that, if \hypScL{q} and \hypDC{q}{s} hold, then \hypDCL{q}{s} holds, with the same $\omega$ and~$g$. Indeed, let $x \in \Omega$. For all $z \in \R^d$ and $\alpha,\beta \in \N^d$ such that $\norm{\alpha}\leq q$ and $\norm{\beta}\leq q$, we have $\partial^{\alpha,\beta}r_R(Rx,Rx+z) \xrightarrow[R \to +\infty]{}\partial^{\alpha,\beta}r(0,z)$, by \hypScL{q}. By \hypDC{q}{s}, for all $R \in \cR$ large enough, we have $\Norm{\partial^{\alpha,\beta}r_R(Rx,Rx+z)}\leq g(z)$, hence $\Norm{\partial^{\alpha,\beta}r(0,z)}\leq g(z)$.

The model case of Section~\ref{subsec: zero set of a stationary Gaussian field in a growing set} fits in our general framework by considering $\cR=(0,+\infty)$, $U=\R^d=\Omega$ and $f_R=f$ for all $R \in \cR$. In this case, $Z_R=Z=f^{-1}(0)$ for all $R >0$, and the random measures $\nu_R$ are related to volumes by~\eqref{eq: relation nuR volume}. Moreover, we have $r_R =r$ for all $R >0$. Hence, \hypScL{q} is tautologically satisfied for any $q \in \N$, and \hypDC{q}{s} is equivalent to \hypDCL{q}{s} for all $q \in \N$ and $s \in [1,+\infty]$. Thus, Theorems~\ref{thm: moments asymptotics volume} and \ref{thm: cumulants asymptotics for critical points} are special cases of Corollary~\ref{cor: moment asymptotics}.\ref{item: moment asymptotics 2}, and Theorem~\ref{thm: central limit theorem stationary case} is a special case of Theorems~\ref{thm: central limit theorem}. Translating Theorem~\ref{thm: law of large numbers} in this model case yields a statement which is strictly weaker than Theorem~\ref{thm: law of large numbers stationary case}. The proof of Theorem~\ref{thm: law of large numbers stationary case} requires an additional monotonicity argument, see Section~\ref{subsec: proof of thm law of large numbers}.

In order to prove the almost-sure convergence in Theorem~\ref{thm: law of large numbers}, we have to extract a subsequence $\parentheses*{\nu_{R_n}}_{n \in \N^*}$, where $R_n$ goes to infinity fast enough. In view of Theorem~\ref{thm: law of large numbers stationary case}, one might wonder whether this almost-sure convergence holds for the whole sequence $(\nu_R)_{R>0}$ as $R\to + \infty$ or not. Actually, it is not even clear that such a question makes sense in our general setting. Indeed, we cannot rule out the possibility that there exists a random positive sequence $(t_n)_{n \in \N}$ going to infinity such that, say with positive probability, $\nu_{t_n}$ is ill-defined for all $n \in \N$.

%%%%%%%%%%%%%%%%%%%%%%%%%%%%%%%%%%%%%%%%%%%%%%%%%%%%%%%%%%%%%%%%%%%%%%%%%%%%%%%%%%%%%%%%%%%%%%%%%%%%%%%

\subsection{Kostlan polynomials}
\label{subsec: Kostlan polynomials}

The techniques developed in this paper allow to deal with some models of Gaussian fields on manifolds. In this section, we state results analogous to those of Sections~\ref{subsec: cumulants asymptotics} and~\ref{subsec: limit theorems and other corollaries} for such a model: Kostlan polynomials on the sphere, in the large degree limit.

For any $\alpha = \parentheses*{\alpha_0,\dots,\alpha_d} \in \N^{d+1}$, we let $\norm{\alpha} = \alpha_0 + \dots + \alpha_d$ and $\binom{n}{\alpha} = \frac{n!}{\alpha_0! \dots \alpha_d!}$ if $\norm{\alpha}=n$. We also denote by $X=(X_0,\dots,X_d)$ and $X^\alpha = X_0^{\alpha_0}\dots X_d^{\alpha_d}$. A \emph{Kostlan polynomial}~\cite{Kos1993} of degree $n \in \N^*$ is a random homogeneous polynomial in $(d+1)$ variables, of the form:
\begin{equation}
\label{eq: def Kostlan polynomial}
\sum_{\norm{\alpha}=n} a_\alpha \sqrt{\binom{n}{\alpha}} X^\alpha,
\end{equation}
where $\parentheses*{a_\alpha}_{\norm{\alpha}=n}$ are independent real standard Gaussian variables. Let $\prsc{\cdot}{\cdot}$ denote the standard inner product on $\R^{d+1}$, and $\S^d \subset \R^{d+1}$ be the unit sphere endowed with its Euclidean metric. The Kostlan polynomial~\eqref{eq: def Kostlan polynomial} defines a smooth centered Gaussian field on~$\S^d$. Its covariance kernel is $(x,y) \mapsto \prsc{x}{y}^n = \cos\parentheses*{\strut\dist(x,y)}^n$, where $\dist(x,y)$ stands for the geodesic distance between $x$ and $y$ on the sphere. In particular, this field is stationary, in the sense that its distribution is invariant under the action of the orthogonal group $O_{d+1}(\R)$ on~$\S^d$.

Let $n \in \N^*$, we denote by $P_n:\S^d \to \R^k$ the centered Gaussian field whose components are $k$ independent Kostlan polynomials of degree $n$. Kostlan~\cite[Thm.~5.1]{Kos1993} proved that $P_n^{-1}(0)$ is almost-surely a submanifold of codimension $k$ in $\S^d$ and that $\esp{\Vol_{d-k}\parentheses*{P_n^{-1}(0)}} = n^\frac{k}{2} \Vol_{d-k}\parentheses*{\S^{d-k}}$. Let $\nu^\text{Kos}_n$ denote the random Radon measure on $\S^d$ whose linear statistics are defined by:
\begin{equation}
\label{eq: def nu Kostlan}
\forall \phi \in L^1(\S^d), \qquad \prsc*{\nu^\text{Kos}_n}{\phi} = n^{-\frac{k}{2}} \int_{P_n^{-1}(0)} \phi(x) \dx \cH^{d-k}(x).
\end{equation}
In order to be consistent with the notation of Section~\ref{subsec: cumulants asymptotics}, the measure $\nu^\text{Kos}_n$ should rather be denoted by $\nu^\text{Kos}_{\sqrt{n}}$. But then, the notation would be inconsistent with the whole literature on the subject. Here, we chose to stay consistent with the existing literature, and use the nicer-looking $\nu^\text{Kos}_n$. The expectation and asymptotic variance of the linear statistics~\eqref{eq: def nu Kostlan} were studied in~\cite{AADL2018,AADL2022,Dal2015,LS2019a,Let2019,LP2019}. In particular, it is known that:
\begin{equation}
\label{eq: Kostlan expectation}
\forall \phi \in L^1(\mathbb{S}^d), \qquad \esp{\prsc*{\nu^\text{Kos}_n}{\phi}} = \gamma_1 \int_{\mathbb{S}^d}\phi(x) \dx \cH^d(x),
\end{equation}
where $\gamma_1 = \frac{\Vol_{d-k}\parentheses*{\S^{d-k}}}{\Vol_d\parentheses*{\S^d}}$, and that there exists $\gamma_2>0$ such that, for all $\phi_1,\phi_2 \in L^1(\mathbb{S}^d)$, we have:
\begin{equation}
\label{eq: Kostlan variance}
n^\frac{d}{2}\esp{\prod_{i=1}^2 \parentheses*{\prsc*{\nu^\text{Kos}_n}{\phi_i}-\gamma_1\int_{\mathbb{S}^d}\phi_i(x) \dx \cH^d(x)}} \xrightarrow[n \to +\infty]{} \gamma_2 \int_{\mathbb{S}^d}\phi_1(x)\phi_2(x) \dx \cH^d(x),
\end{equation}
cf.~\cite[Thm.~1.6 and 1.8]{LP2019}.

Kostlan polynomials being either even or odd, their values on antipodal points are always fully correlated. In order to avoid long-range correlations, one usually works in the real projective space, and then deduces results on $\S^d$ from the projective ones. In the context of the following discussion, it is equivalent to consider only pairs of points $(x,y) \in \S^d \times \S^d$ at distance at most $\frac{\pi}{2}$, that is, such that $\prsc{x}{y}\geq 0$.

It is well-known that Kostlan polynomials admit the Bargmann--Fock field $f_\text{BF}$ as their local scaling limit. More precisely, let us consider any point on the sphere. In the exponential chart around this point, a Kostlan polynomial of degree~$n$ reads as a Gaussian field on the ball $\frac{\pi}{2}\B$ of center $0$ and radius $\frac{\pi}{2}$ in $\R^d$. Rescaling by~$\frac{1}{\sqrt{n}}$, we obtain a centered Gaussian field on $\sqrt{n}\frac{\pi}{2}\B$ whose covariance kernel is:
\begin{equation}
\label{eq: local kernel Kostlan}
r_n:(x,y) \longmapsto \parentheses*{\cos\parentheses*{\frac{\Norm{x}}{\sqrt{n}}}\cos\parentheses*{\frac{\Norm{y}}{\sqrt{n}}}+\sin\parentheses*{\frac{\Norm{x}}{\sqrt{n}}}\sin\parentheses*{\frac{\Norm{y}}{\sqrt{n}}}\prsc*{\frac{x}{\Norm{x}}}{\frac{y}{\Norm{y}}}}^n.
\end{equation}
From~\eqref{eq: local kernel Kostlan}, one can check that $r_n(x,y) \xrightarrow[n \to +\infty]{}e^{-\frac{\Norm{y-x}^2}{2}} = r_\text{BF}(x,y)$ point-wise. Actually, this convergence holds locally uniformly in the $\cC^\infty$-sense, which is reminiscent of~\hypScL{q}. It is also well-known that Kostlan polynomials decorrelate exponentially fast at scale $\frac{1}{\sqrt{n}}$. By stationarity, it is enough to observe that~\eqref{eq: local kernel Kostlan} implies that $\norm*{r_n(0,x)}\leq e^{-\frac{\Norm{x}^2}{\pi}}$ for any $n \in \N^*$ and $x \in \sqrt{n}\frac{\pi}{2}\B$. Similar estimates hold for all the partial derivatives of $r_n$, which is reminiscent of \hypDC{q}{s}.

Using the previous two facts, our methods yield the following result, which was conjectured in~\cite[Conj.~1.16]{AL2021a}.
\begin{thm}[Moments asymptotics for Kostlan polynomials]
\label{thm: moments Kostlan}
Let $d \in \N^*$ and $k \in \brackets{1,\dots,d}$. For all $n \in \N^*$, let $P_n:\S^d \to \R^k$ be the centered Gaussian field whose components are $k$ independent Kostlan polynomials and let $\nu^\text{Kos}_n$ be defined by~\eqref{eq: def nu Kostlan}. Let $p \geq 3$ be an integer. Then, for all $\phi_1,\dots,\phi_p \in L^1(\S^d)$ we have the following.
\begin{equation*}
n^\frac{pd}{4}\esp{\prod_{i=1}^p \parentheses*{\prsc*{\nu^\text{Kos}_n}{\phi_i}-\gamma_1\int_{\mathbb{S}^d}\phi_i \dx \cH^d}} \xrightarrow[n \to +\infty]{}\gamma_2(f)^\frac{p}{2} \sum_{\substack{(\brackets{a_j,b_j})_{1\leq j \leq \floor{\frac{p}{2}}} \\ \bigsqcup_{j=1}^{\floor{\frac{p}{2}}} \brackets{a_j,b_j} = \brackets{1,\dots,p}}}\! \prod_{j=1}^{\floor{\frac{p}{2}}} \int_{\S^d} \phi_{a_j}\phi_{b_j} \dx \cH^d,
\end{equation*}
where $\gamma_1$ and $\gamma_2$ are the constants appearing in~\eqref{eq: Kostlan expectation} and~\eqref{eq: Kostlan variance} respectively, and the sum is indexed by the partitions in pairs of $\brackets{1,\dots,p}$. In particular, for all $\phi \in L^1(\S^d)$, we have:
\begin{equation*}
n^\frac{pd}{4}\esp{\parentheses*{\prsc*{\nu^\text{Kos}_n}{\phi}-\gamma_1\int_{\mathbb{S}^d}\phi \dx \cH^d}^p} \xrightarrow[n \to +\infty]{} \parentheses*{\gamma_2\int_{\S^d} \phi^2 \dx \cH^d}^\frac{p}{2} \esp{\gauss{1}^p}.
\end{equation*}
\end{thm}

As in the Euclidean case, this result implies a functional CLT. In particular, we recover the CLT for volumes established in~\cite{AADL2021,AADL2022,Dal2015}, by the method of moments.

\begin{thm}[Central Limit Theorem for Kostlan polynomials]
\label{thm: central limit theorem Kostlan}
In the setting of Theorem~\ref{thm: moments Kostlan}, for all $\phi_1,\dots,\phi_n \in L^1(\S^d)$, the random vector $n^\frac{d}{4}\parentheses*{\prsc*{\nu^\text{Kos}_n}{\phi_i}-\gamma_1\int_{\S^d}\phi_i \dx \cH^d}_{1 \leq i \leq n}$ converges in distribution as $n \to +\infty$ to the centered Gaussian with covariance matrix $\gamma_2 \begin{pmatrix}\int_{\S^d}\phi_i\phi_j \dx \cH \end{pmatrix}_{1 \leq i,j \leq n}$. Moreover, we have $n^\frac{d}{4}\parentheses*{\nu^\text{Kos}_n -\strut \esp{\nu^\text{Kos}_n}} \xrightarrow[R \to +\infty]{\text{law}} \sqrt{\gamma_2} W$ in $\cD'(\S^d)$, where $W$ is the standard Gaussian White Noise on $\S^d$.
\end{thm}

The functional Law of Large Numbers for Kostlan polynomials was already established in dimension $d \neq 2$, see~\cite{AL2021a,LP2019}. Theorem~\ref{thm: moments Kostlan} implies that it also holds in dimension $d=2$, so that we obtain the following result.

\begin{thm}[Law of Large Numbers for Kostlan polynomials]
In the setting of Theorem~\ref{thm: moments Kostlan}, for all $\phi \in L^1(\S^d)$ we have $\prsc{\nu^\text{Kos}_n}{\phi} \xrightarrow[n \to +\infty]{\text{a.s.}} \gamma_1 \int_{\S^d} \phi \dx \cH^d$. Moreover, almost-surely, we have $\nu^\text{Kos}_n \xrightarrow[n \to +\infty]{\text{weak}-*}\gamma_1 \dx \cH^d$ as Radon measures.
\end{thm}

%%%%%%%%%%%%%%%%%%%%%%%%%%%%%%%%%%%%%%%%%%%%%%%%%%%%%%%%%%%%%%%%%%%%%%%%%%%%%%%%%%%%%%%%%%%%%%%%%%%%%%%

\subsection{Strategy of proof}
\label{subsec: strategy of proof}

In this section, we sketch the main steps of the proof of the cumulants asymptotics in Theorem~\ref{thm: cumulants asymptotics for zero sets}. In order to keep things as simple as possible, we consider random hypersurfaces ($d \geq 2$ and $k=1$) in $U=\R^d =\Omega$ with $\cR=(0,+\infty)$. We also focus on the case $p=2$, with $\phi_1$ and $\phi_2 \in \cC^0_c(\R^d)$. When considering cumulants of order $p \geq 3$, the main additional difficulty is that we have to deal with some intricate combinatorics. Part of it comes from the combinatorics of cumulants. Another part comes from the fact that we consider the integral of some function on $(\R^d)^p$, and we need to keep track of how its arguments $x_1,\dots,x_p$ are clustered together in $\R^d$. Apart from these combinatorics, understanding the case $p=2$ is enough to understand our strategy of proof.

Let us consider the setting of Theorem~\ref{thm: cumulants asymptotics for zero sets} in this case. We want to compute the large $R$ asymptotics of $\kappa_2(\nu_R)(\phi_1,\phi_2)$, which is the covariance of the random variables $\prsc{\nu_R}{\phi_1}$ and $\prsc{\nu_R}{\phi_2}$, defined by~\eqref{eq: def linear statistics nuR}. As usual in this business, the first step of the proof is to apply the Kac--Rice formula for moments and derive an integral expression for $\kappa_2(\nu_R)(\phi_1,\phi_2)$. We obtain that
\begin{equation}
\label{eq: Kac-Rice strategy}
\kappa_2(\nu_R)(\phi_1,\phi_2) = \frac{1}{R^{2d}} \int_{\R^d \times \R^d} \phi_1\parentheses*{\frac{x}{R}}\phi_2\parentheses*{\frac{y}{R}} \cF_2\parentheses*{f_R,x,y} \dx x \dx y.
\end{equation}
Here, the density $\cF_2(f_R,\cdot)$ is defined by $\cF_2(f_R,x,y) = \rho_2(f_R,x,y)-\rho_1(f_R,x)\rho_1(f_R,y)$ for all $x,y \in \R^d$, where $\rho_1(f_R,\cdot)$ and $\rho_2(f_R,\cdot)$ are the first two Kac--Rice densities associated with $f_R$. Under our assumptions, at least for $R$ large enough, $\rho_1(f_R,\cdot)$ is continuous on $\R^d$, and $\rho_2(f_R,\cdot)$ is continuous outside of the diagonal $\diag = \brackets{(x,x) \mid x \in \R^d} \subset \R^d \times \R^d$ but admits a pole along $\diag$.

\paragraph*{Model case.} The next step is to understand the asymptotics of~\eqref{eq: Kac-Rice strategy} in the model case where $f_R =f$ for all $R >0$. By stationarity of $f$ and a change of variable, we can rewrite~\eqref{eq: Kac-Rice strategy} as:
\begin{equation*}
R^d \kappa_2(\nu_R)(\phi_1,\phi_2) = \int_{\R^d \times \R^d} \phi_1\parentheses*{x}\phi_2\parentheses*{x+\frac{z}{R}}\cF_2(f,0,z) \dx x \dx z.
\end{equation*}
We want to prove that this quantity converges to $\gamma_2(f) \int_{\R^d}\phi_1(x)\phi_2(x)\dx x$ as $R \to +\infty$, where $\gamma_2(f) = \int_{\R^d} \cF_2(f,0,z)\dx z$. The difficulty here is to prove that $\cF_2(f,0,\cdot)$ is integrable on $\R^d$. First, this function is locally integrable on $\R^d$. Indeed, the stationarity of $f$ implies that $\rho_1(f,\cdot)$ equals some constant $\gamma_1(f)>0$, while it is known that $\rho_2(f,\cdot)$ is locally integrable on $\R^d$, see~\cite{AL2025,GS2024}.

It remains to prove the integrability of $\cF_2(f,0,\cdot)$ at infinity. In order to do this, we factor $\cF_2$ through a space of symmetric matrices. For all $(x,y) \notin \diag$, the value of $\cF_2(f,x,y)$ only depends on the variance matrix $\Sigma(f,x,y)$ of $\parentheses*{f(x),\nabla_xf,f(y),\nabla_yf}$, which is positive thanks to \hypND{3}. Thus, there exists a function $F_2$ on the open cone of positive symmetric matrices of size $2(d+1)$ such that $\cF_2(f,x,y)=F_2\parentheses*{\Sigma(f,x,y)}$ for all $(x,y) \notin \diag$. A matrix in the domain of $F_2$ can be written by blocks as $\Sigma=\parentheses*{\begin{smallmatrix}\Sigma_x & \Sigma_{xy} \\ \trans{\Sigma_{xy}} & \Sigma_y\end{smallmatrix}}$, where the blocks are square of size $d+1$. We prove that, if $\Sigma_{xy}=0$, then $F_2(\Sigma)=0$ and some part of the differential of $F_2$ at $\Sigma$ also vanishes. Since $F_2$ is actually smooth, a Taylor expansion shows that $\norm*{F_2(\Sigma)}$ is controlled by $\Norm{\Sigma_{xy}}^2$, uniformly on compact sets. Thus, $\norm{\cF_2(f,x,y)}$ is controlled by the square norm of the off-diagonal blocks of $\Sigma(f,x,y)$. The corresponding coefficients are partial derivatives of $r$ evaluated on $(x,y)$, hence they are controlled by $g(y-x)$. Finally, $\norm*{\cF_2(f,0,z)}$ is controlled by $g(z)^2$, say uniformly on the domain $\Norm{z}\geq 1$. Since we assumed~\hypDC{3}{2}, this proves the integrability of $\cF_2(f,0,\cdot)$ at infinity.

\paragraph*{General case.} Now that we understand the model case, we have to explain why replacing $f_R$ by the limit field $f$ in~\eqref{eq: Kac-Rice strategy} only contributes an error term going to $0$ as $R \to +\infty$. For this, we split the integral in two parts: the integral over $\diag_1 = \brackets*{(x,y) \in \R^d \times \R^d\mvert \Norm{y-x}\leq 1}$ and the integral over its complement.
\begin{itemize}

\item On $(\R^d \times \R^d) \setminus \diag_1$, we prove that the covariance matrix $\Sigma(f_R,x,y)$ converges to $\Sigma\parentheses*{f,x,y}$, uniformly with respect to $(x,y) \notin \diag_1$. Then, the regularity of $F_2$ and the same estimates as above allow us to prove that replacing $f_R$ by $f$ in the integral over $(\R^d \times \R^d) \setminus \diag_1$ only contributes an error term.

\item On $\diag_1$, the situation is trickier, mostly because the pole of $\cF_2$ along the diagonal plays havoc with our estimates. In order to deal with the singularity of $\cF_2(f,\cdot)$ along the diagonal, we use the Kergin interpolant of $f$ at $x$ and $y$, which is a multivariate polynomial that interpolates both $f$ and $\nabla f$ at $x$ and~$y$. For any $(x,y) \notin \diag$, this Kergin interpolant defines a centered Gaussian vector in the space $\R_3[X]$ of polynomials of degree at most $3$ in $d$ variables, whose variance operator is denoted by $\Sigma'(f,x,y)$. We can recover $\parentheses*{f(x),\nabla_xf,f(y),\nabla_yf}$ from the Kergin interpolant of $f$ and some geometric quantities depending only on $(x,y)$. We leverage this fact to write $\cF_2(f,x,y)$ as $\Upsilon_2(x,y)F_2'\parentheses*{\Sigma'(f,x,y)}$, where the function $\Upsilon_2:\R^d \times \R^d\to \R$ is locally integrable, continuous outside of $\diag$ and singular along $\diag$, and $F_2'$ is a smooth map on the set of positive symmetric operator on $\R_3[X]$. Formally, $F_2'$ also depends on $(x,y)$ through some geometric parameters but, since we are only sketching the proof, let us sweep this under the rug. The function $\Upsilon_2$ is the universal part of the Kac--Rice density on~$\diag_1$: it does not depend on the field, and it fully accounts for the pole of $\cF_2(f,\cdot)$ along~$\diag$. On the other hand, the distribution of the Kergin interpolant of $f$ at $x$ and $y$ does not degenerate as $x$ and $y$ approach one another. Thus, we can deal with $F_2'\parentheses*{\Sigma'(f_R,x,y)}$ in the same way we dealt with $F_2\parentheses*{\Sigma(f_R,x,y)}$ previously. We prove that it converges to $F_2'\parentheses*{\Sigma'(f,x,y)}$ uniformly with respect to $(x,y)\in\diag_1$, and this is enough to prove that replacing $f_R$ by $f$ in~\eqref{eq: Kac-Rice strategy} only contributes an error term. This completes our sketch of the proof that $R^d \kappa_2(\nu_R)(\phi_1,\phi_2) \xrightarrow[R \to +\infty]{} \gamma_2(f) \int_{\R^d}\phi_1(x)\phi_2(x)\dx x$.
\end{itemize}

%%%%%%%%%%%%%%%%%%%%%%%%%%%%%%%%%%%%%%%%%%%%%%%%%%%%%%%%%%%%%%%%%%%%%%%%%%%%%%%%%%%%%%%%%%%%%%%%%%%%%%%

\subsection{Organization of the paper}
\label{subsec: organisation}

Section~\ref{sec: partitions cumulants and diagonals} introduces combinatorial tools that are used throughout the paper. In particular, we recall the definition of the cumulants of a family of random variables in that section. Section~\ref{sec: polynomial interpolation and Gaussian fields} is concerned with Kergin interpolation, and more precisely with describing the distribution of the Kergin interpolant of a Gaussian field. We discuss various forms of the Kac--Rice formulas in Section~\ref{sec: Kac-Rice formulas revisited}, and explain how to write the corresponding densities as functions on a space of symmetric operators. In Section~\ref{sec: refined Hadamard lemma}, we prove a key estimate for the Kac--Rice densities for cumulants. Section~\ref{sec: edge-connected graphs} introduces a family of finite graphs that allows us to rewrite the estimate of Section~\ref{sec: refined Hadamard lemma} more nicely. In Section~\ref{sec: HBL inequalities}, we discuss Hölder--Brascamp--Lieb inequalities, and the estimates they yield for the Kac--Rice densities. Section~\ref{sec: proof of thm cumulants asymptotics zero sets} is dedicated to the proof of Theorem~\ref{thm: cumulants asymptotics for zero sets}. In Section~\ref{sec: the case of critical points}, we explain briefly how to adapt the content of the previous sections to deal with the case of gradient fields and prove Theorem~\ref{thm: cumulants asymptotics for critical points}. Finally, Section~\ref{sec: proofs limit theorems and other corollaries} contains the proofs of Theorem~\ref{thm: law of large numbers stationary case} and of the limit theorems stated in Section~\ref{subsec: limit theorems and other corollaries}.

%%%%%%%%%%%%%%%%%%%%%%%%%%%%%%%%%%%%%%%%%%%%%%%%%%%%%%%%%%%%%%%%%%%%%%%%%%%%%%%%%%%%%%%%%%%%%%%%%%%%%%%
%%%%%%%%%%%%%%%%%%%%%%%%%%%%%%%%%%%%%%%%%%%%%%%%%%%%%%%%%%%%%%%%%%%%%%%%%%%%%%%%%%%%%%%%%%%%%%%%%%%%%%%

\section{Partitions, cumulants and diagonals}
\label{sec: partitions cumulants and diagonals}

The purpose of this section is to introduce notation, mostly related with partitions, that will appear in many places in the paper. We also recall some basic properties of the objects we introduce. Section~\ref{subsec: partitions of a finite set} is concerned with partitions of a finite set and the partial order on the set of these partitions. In Section~\ref{subsec: cumulants}, we recall the definition and some properties of the cumulants of a family of random variables. Finally, Section~\ref{subsec: product spaces and diagonals} introduces notation for strata of the diagonal in a product space and Section~\ref{subsec: thick diagonals} studies a thick version of these strata.

%%%%%%%%%%%%%%%%%%%%%%%%%%%%%%%%%%%%%%%%%%%%%%%%%%%%%%%%%%%%%%%%%%%%%%%%%%%%%%%%%%%%%%%%%%%%%%%%%%%%%%%

\subsection{Partitions of a finite set}
\label{subsec: partitions of a finite set}

In this section we first introduce notation related with finite sets and their partitions. Then, we recall some useful facts about the canonical partial order on a set of partitions.

We denote by $\N$ (resp.~$\N^*$) the set of non-negative (resp.~positive) integers. Given $m$ and $n \in \N$ such that $m \leq n$, we denote by $\ssquarebrackets{m}{n}=[m,n]\cap \N$ the associated integer interval.

Let $A$ be a non-empty finite set, we denote either by $\card(A)$ or by $\norm{A}$ its cardinality. We also introduce the following notation for the set of unordered pairs of elements of $A$.

\begin{dfn}[Set of pairs]
\label{def: A wedge 2}
We denote by $A^{\cwedge} = \brackets*{\strut \brackets{a,b}\subset A \mvert a \neq b}$.
\end{dfn}

This set naturally indexes coefficients above the diagonal in a symmetric matrix whose rows and columns are indexed by $A$. It is also related to the set of edges of a graph with vertex set $A$.

\begin{rem}
\label{rem: product as pairs}
Let $B$ and $C$ be non-empty disjoint subsets of $A$. The map $(b,c) \mapsto \brackets{b,c}$ is a canonical injection of $B \times C$ into $A^{\cwedge}$. Using this injection, we often consider $B \times C$ as a subset of $A^{\cwedge}$ in the following.
\end{rem}

Next, we introduce the set of partitions of $A$ and its partial order.

\begin{dfn}[Partition]
\label{def: partition}
A \emph{partition} of $A$ is a family $\cI=\brackets{I_1,\dots,I_m}$ of non-empty disjoint subsets such that $\bigsqcup_{i=1}^m I_i = A$. We denote by $\cP_A$ (resp.~$\cP_p$) the set of partitions of $A$ (resp.~$\ssquarebrackets{1}{p}$).
\end{dfn}

\begin{dfn}[Blocks and induced partition]
\label{def: blocks and induced partition}
Let $\cI$ be a partition of $A$, let $a \in A$ and $B \subset A$.
\begin{itemize}
\item An element $I \in \cI$ is called a \emph{block} of $\cI$.

\item We denote by $[a]_{\cI}$ the only block of $\cI$ that contains $a$.

\item We denote by $\cI_B = \brackets*{I \cap B \mvert I \in \cI} \setminus \brackets*{\emptyset}$ the partition of $B$ \emph{induced} by $\cI$.

\end{itemize}
\end{dfn}

\begin{dfn}[Partial order on partitions]
\label{def: partial order on PA}
Let $\cI$ and $\cJ \in \cP_A$, we denote $\cJ \leq \cI$ if for all $J \in \cJ$ there exists $I \in \cI$ such that $J \subset I$. In this case, we say that $\cJ$ is \emph{finer} than $\cI$, or that $\cJ$ \emph{refines}~$\cI$, or that $\cI$ is \emph{coarser} than $\cJ$. We denote by $\cJ < \cI$ the fact that $\cJ \leq \cI$ and $\cJ \neq \cI$.
\end{dfn}

It is a classical fact that $\parentheses*{\cP_A,\leq}$ is a partially ordered set, and even a lattice (see~\cite[Sect.~2.2]{PT2011}). In particular it admits a maximum, equal to~$\brackets{A}$, and a minimum, equal to~$\brackets*{\strut \brackets{a} \mvert a \in A}$. Moreover, any two partitions of $A$ have a coarsest common refinement, called their meet.

\begin{dfn}[Meet]
\label{def: meet}
Let $\cI$ and $\cJ \in \cP_A$, we denote by $\cI \wedge \cJ$ the coarsest element of $\cP_A$ that refines both $\cI$ and $\cJ$. The partition $\cI \wedge \cJ$ is called the \emph{meet} of $\cI$ and $\cJ$ and is described by:
\begin{equation*}
\cI \wedge \cJ = \brackets*{I \cap J \mvert I \in \cI, J \in \cJ} \setminus \brackets{\emptyset} = \bigsqcup_{J \in \cJ} \cI_J = \bigsqcup_{I \in \cI} \cJ_I.
\end{equation*}
More generally, let $\cQ = \brackets*{\cI_1,\dots,\cI_m} \subset \cP_A$, we denote the coarsest partition refining all the elements of $\cQ$, by $\bigwedge \cQ = \bigwedge_{i=1}^m \cI_i = \brackets*{\strut \bigcap_{i=1}^m I_i \mvert \forall i \in \ssquarebrackets{1}{m}, I_i \in \cI_i} \setminus \brackets{\emptyset}$. Accordingly, we denote $\bigwedge \cP_A = \brackets*{\strut \brackets{a} \mvert a \in A}$ for the minimum of $\cP_A$.
\end{dfn}

We will need to consider the set obtained by forming the disjoint union of two copies of $A$. The notation we use is the following.

\begin{dfn}[Doubled sets and partitions]
\label{def: doubled sets and partitions}
For all $B \subset A$, we denote by $2B =\brackets{0,1} \times B$. Given $\cI \in \cP_A$, we denote by $2\cI = \brackets*{2I \mvert I \in \cI}$, which is a partition of $2A$.
\end{dfn}

Note that, if $\cJ \in \cP_{2A}$ is such that $[(0,a)]_\cJ = [(1,a)]_\cJ$ for all $a \in A$, then there exists a unique $\cI \in \cP_A$ such that $\cJ = 2\cI$.

We conclude this section by expliciting two bijections. They will be useful in computations to re-index sums indexed by sets of partitions.

\begin{lem}[Two bijections]
\label{lem: bijections partitions}
Let $\cI \in \cP_A$, then the following two maps are bijections:
\begin{enumerate}
\item \label{item: bijection coarser} the map $\cJ \mapsto \brackets*{\cI_J \mvert J \in \cJ}$ from $\brackets*{\cJ \in \cP_A \mvert \cI \leq \cJ}$ to $\cP_\cI$;

\item \label{item: bijection finer} the map $\cJ \mapsto \parentheses*{\cJ_I}_{I \in \cI}$ from $\brackets*{\cJ \in \cP_A \mvert \cJ \leq \cI}$ to $\prod_{I \in \cI} \cP_I$.
\end{enumerate}
\end{lem}

\begin{proof}
Let $\cJ \in \cP_A$ be coarser than $\cI$. Since $\cI \leq \cJ$, we have $\cI_J = \brackets*{I \in \cI \mvert I \subset J} \subset~\cI$ for all $J \in \cJ$, see Definition~\ref{def: blocks and induced partition}. Moreover $\cI = \bigsqcup_{J \in \cJ} \cI_J$, so that the map in Item~\ref{item: bijection coarser} is well-defined. Given $K \in \cK \in \cP_\cI$, we have $K \subset \cI \in \cP_A$ so that $\bigsqcup_{I \in K} I \subset A$ is a union of blocks of $\cI$. Hence,
\begin{equation*}
\cK \in \cP_\cI \longmapsto \brackets*{\bigsqcup_{I \in K} I \mvert K \in \cK} \in \brackets*{\cJ \in \cP_A \mvert \cI \leq \cJ}
\end{equation*}
is a well-defined map. One can check that it is the converse map of the one in Item~\ref{item: bijection coarser}, which proves that both maps are bijections.

If $\cJ \in \cP_A$, then $\cJ_I \in \cP_I$ for all $I \in \cI$, by definition. Hence, the map from Item~\ref{item: bijection finer} is well-defined. The converse map is $\parentheses*{\cK_I}_{I \in \cI} \mapsto \bigsqcup_{I \in \cI} \cK_I$, which proves that it is bijective.
\end{proof}

Let $\cI$ and $\cJ \in \cP_A$ be such that $\cI \leq \cJ$. For all $I \in \cI$, there exists a unique $J \in \cJ$ such that $I \subset J$. The partition $\cJ$ corresponds to some $\cK \in \cP_\cI$ through the bijection of Lemma~\ref{lem: bijections partitions}.\ref{item: bijection coarser}. The block $J$ then corresponds to $[I]_{\cK}$, in the sense that $J = \bigsqcup_{K \in [I]_\cK}K$. This shows that the following is consistent with Definition~\ref{def: blocks and induced partition} and Lemma~\ref{lem: bijections partitions}.\ref{item: bijection coarser}.

\begin{dfn}[Block with respect to a coarser partition]
\label{def: block coarser partition}
Let $\cI$ and $\cJ \in \cP_A$ be such that $\cI \leq \cJ$. For all $I \in \cI$, we denote be $[I]_\cJ \in \cJ$ the only block of $\cJ$ that contains $I$.
\end{dfn}

%%%%%%%%%%%%%%%%%%%%%%%%%%%%%%%%%%%%%%%%%%%%%%%%%%%%%%%%%%%%%%%%%%%%%%%%%%%%%%%%%%%%%%%%%%%%%%%%%%%%%%%

\subsection{Cumulants}
\label{subsec: cumulants}

In this section, we recall how cumulants of random variables are defined from their moments, as well as some of their properties.

Let $A$ be a non-empty finite set and $\emptyset \neq B \subset A$, we denote by $\mu_\cJ = (-1)^{\norm{\cJ}-1}\parentheses*{\norm{\cJ}-1}!$ for all $\cJ \in \cP_B$. Let us consider two sequences $\parentheses*{m_B}_{\emptyset \neq B \subset A}$ and $\parentheses*{\kappa_B}_{\emptyset \neq B \subset A}$ of numbers, indexed by non-empty subsets of $A$. The following result relies on the theory of Möbius functions. We will not go into the details of this theory, and we refer the interested reader to~\cite[Sect.~2.5]{PT2011}.

\begin{prop}[Möbius inversion]
\label{prop: Mobius inversion}
The following two statements are equivalent:
\begin{enumerate}
\item \label{item: Mobius cumulant} for all non-empty $B \subset A$, we have $\kappa_B = \displaystyle\sum_{\cJ \in \cP_B} \mu_\cJ \prod_{J \in \cJ} m_J$;

\item \label{item: Mobius moment} for all non-empty $B \subset A$, we have $m_B = \displaystyle\sum_{\cJ \in \cP_B} \prod_{J \in \cJ} \kappa_J$.
\end{enumerate}
\end{prop}

\begin{proof}
For all non-empty $B \subset A$ and $\cI \in \cP_B$ we introduce the auxiliary notation $\tilde{m}_\cI = \prod_{I \in \cI} m_I$ and $\tilde{\kappa}_\cI = \prod_{I \in \cI} \kappa_I$. Let us assume that Statement~\ref{item: Mobius cumulant} holds. For all $\cI \in \cP_B$, we have:
\begin{equation*}
\tilde{\kappa}_\cI = \prod_{I \in \cI} \sum_{\cJ_I \in \cP_I}\parentheses*{ \mu_{\cJ_I} \prod_{J \in \cJ_I} m_J} = \sum_{\cJ \leq \cI} \parentheses*{\prod_{I \in \cI} \mu_{\cJ_I}} \parentheses*{\prod_{J \in \cJ} m_J}= \sum_{\cJ \leq \cI} \parentheses*{\prod_{I \in \cI} \mu_{\cJ_I}} \tilde{m}_\cJ,
\end{equation*}
where we used Lemma~\ref{lem: bijections partitions}.\ref{item: bijection finer} to get the second equality. Thus, Statement~\ref{item: Mobius cumulant} implies
\begin{equation}
\label{eq: Mobius cumulant partition}
\forall\ \emptyset \neq B \subset A,\ \forall \cI \in \cP_B, \qquad \tilde{\kappa}_\cI = \sum_{\cJ \leq \cI} \parentheses*{\prod_{I \in \cI} \mu_{\cJ_I}} \tilde{m}_\cJ.
\end{equation}
Actually, \eqref{eq: Mobius cumulant partition} is equivalent to Statement~\ref{item: Mobius cumulant} since $\kappa_B = \tilde{\kappa}_{\brackets*{B}}=\sum_{\cJ \leq \brackets{B}} \mu_{\cJ} \prod_{J \in \cJ} m_J$. Similarly, let us assume that Statement~\ref{item: Mobius moment} holds. For all $\cI \in \cP_B$, we have:
\begin{equation*}
\tilde{m}_\cI = \prod_{I \in \cI} \parentheses*{\sum_{\cJ_I \in \cP_I} \prod_{J \in \cJ_I} \kappa_J} = \sum_{\cJ \leq \cI} \parentheses*{\prod_{J \in \cJ} \kappa_J}= \sum_{\cJ \leq \cI} \tilde{\kappa}_\cJ,
\end{equation*}
where we used Lemma~\ref{lem: bijections partitions}.\ref{item: bijection finer} to get the second equality. Hence, Statement~\ref{item: Mobius moment} is equivalent to:
\begin{equation}
\label{eq: Mobius moment partition}
\forall\ \emptyset \neq B \subset A,\ \forall \cI \in \cP_B, \qquad \tilde{m}_\cI = \sum_{\cJ \leq \cI} \tilde{\kappa}_\cJ,
\end{equation}
the converse implication following from $m_B = \tilde{m}_{\brackets*{B}}=\sum_{\cJ \leq \brackets{B}} \prod_{J \in \cJ} \kappa_J$.

To conclude the proof, we need to check that \eqref{eq: Mobius cumulant partition} and~\eqref{eq: Mobius moment partition} are equivalent. This is done in~\cite[pp.~19--20]{PT2011}, where it is stated as the equivalence between~(2.5.25) and~(2.5.26), using the coefficients defined by~(2.5.22).
\end{proof}

\begin{lem}[Cancellation property]
\label{lem: cancellation property}
Let $\parentheses*{m_B}_{\emptyset \neq B \subset A}$ be a sequence of numbers. Let us assume that there exists $C \subset A$ such that $\emptyset \neq C \neq A$; and $m_B = m_{B \cap C}m_{B \setminus C}$ for all $B \subset A$, with the convention that $m_\emptyset =1$. Then, we have $\sum_{\cJ \in \cP_A} \mu_\cJ \prod_{J \in \cJ} m_J=0$.
\end{lem}

\begin{proof}
Denoting by $\kappa_B = \sum_{\cJ \in \cP_B} \mu_\cJ \prod_{J \in \cJ} m_J$ for all $B \subset A$, the sequences $\parentheses*{m_B}_{\emptyset \neq B \subset A}$ and $\parentheses*{\kappa_B}_{\emptyset \neq B \subset A}$ satisfy the equivalent conditions in Proposition~\ref{prop: Mobius inversion}. We want to prove that $\kappa_A=0$.

The proof is a computation involving Möbius inversion, for which we refer to the proof of~\cite[Prop.~4.1]{Spe1983}. Note that, in~\cite{Spe1983}, the author considers the moments and cumulants of a family of random variables. However, the computation we are interested in only uses: the fact that moments and cumulants satisfy the polynomial relations in Proposition~\ref{prop: Mobius inversion}, and a multiplicativity property for moments of independent random variables similar to our hypothesis. Thus, it extends verbatim to our setting.
\end{proof}

We can now focus on the case where the sequences $\parentheses*{m_B}_{\emptyset \neq B \subset A}$ and $\parentheses*{\kappa_B}_{\emptyset \neq B \subset A}$ are respectively the moments and cumulants of a family of random variables. We start by defining these quantities, see~\cite[Sect.~3.2]{PT2011} for example.

\begin{dfn}[Moments and cumulants]
\label{def: moments and cumulants}
Let $A$ be a set of cardinality $p \in \N^*$. Let $(Y_a)_{a \in A}$ be a family of $L^p$ random variables, we denote their \emph{joint moment} and \emph{joint cumulant} respectively~by:
\begin{align*}
m\parentheses*{\strut (Y_a)_{a \in A}} &= \esp{\prod_{a \in A}Y_a} & &\text{and} & \kappa\parentheses*{\strut (Y_a)_{a \in A}} &= \sum_{\cJ \in \cP_A}\mu_\cJ \prod_{J \in \cJ}\esp{\prod_{j \in J} Y_j}.
\end{align*}
If $Y$ is an $L^p$ random variable, we denote its \emph{$p$-th moment} and \emph{$p$-th cumulant} respectively by:
\begin{align*}
&m_p(Y) = m\parentheses{\underbrace{Y,\dots,Y}_{p \ \text{times}}} = \esp{Y^p} & &\text{and} & &\kappa_p(Y) = \kappa\parentheses{\underbrace{Y,\dots,Y}_{p \ \text{times}}}=\sum_{\cJ \in \cP_p} \mu_\cJ \prod_{J \in \cJ} \esp{Y^{\norm{J}}}.
\end{align*}
\end{dfn}

\begin{ex}[Cumulants of Gaussian variables]
\label{ex: Gaussian cumulants}
Let $Y \sim \gauss{\lambda}$ be a centered Gaussian variable in $\R$ of variance $\lambda \geq 0$. By Definition~\ref{def: moments and cumulants}, we have $\kappa_1(Y)=0$ and $\kappa_2(Y)=\lambda$. A generating function argument shows that $\kappa_p(Y)=0$ for all $p \geq 3$, see~\cite[Sect.~3.1]{PT2011} for example.
\end{ex}

Let $(Y_a)_{a \in A}$ be a family of random variables at least of class $L^p$, where $p=\card(A) \in \N^*$. Setting $m_B = m\parentheses*{\strut (Y_b)_{b \in B}}$ and $\kappa_B = \kappa\parentheses*{\strut (Y_b)_{b \in B}}$ for all non-empty $B \subset A$, we get two sequences of the form studied above. We thus obtain the following.

\begin{cor}
\label{cor: moments in terms of cumulants}
For all non-empty $B \subset A$, we have $m\parentheses*{\strut (Y_b)_{b \in B}} = \sum_{\cJ \in \cP_B} \prod_{J \in \cJ} \kappa\parentheses*{\strut (Y_j)_{j \in J}}$. In particular, if $Y$ is an $L^p$ random variable then $m_p(Y) = \sum_{\cJ \in \cP_p} \prod_{J \in \cJ}\kappa_{\norm{J}}(Y)$.
\end{cor}

\begin{proof}
The sequences $(m_B)_{\emptyset \neq B \subset A}$ and $(\kappa_B)_{\emptyset \neq B \subset A}$ were defined so that the first condition in Proposition~\ref{prop: Mobius inversion} is satisfied. Hence so is the second one.
\end{proof}

We conclude this section by recalling the method of cumulants, which is a classical tool to prove Central Limit Theorems.

\begin{prop}[Method of cumulants]
\label{prop: method of cumulants}
Let $\cR \subset (0,+\infty)$ be an unbounded set, $(Y_R)_{R \in \cR}$ be a family of real random variables and $\lambda \geq 0$. For all $p \in \N^*$, we assume that $Y_R$ is $L^p$ for all $R$ large enough (possibly depending on $p$) and that
\begin{equation*}
\kappa_p(Y_R) \xrightarrow[R \to +\infty]{}\kappa_p\parentheses*{\strut \gauss{\lambda}}=\begin{cases}\lambda, & \text{if} \ p =2;\\ 0, & \text{otherwise.} \end{cases}
\end{equation*}
Then $Y_R \xrightarrow[R \to +\infty]{}\gauss{\lambda}$ in distribution.
\end{prop}

\begin{proof}
Let $p \in \N^*$. By Corollary~\ref{cor: moments in terms of cumulants}, the $p$-th moment of a random variable is expressed as a universal polynomial in its cumulants of order at most $p$. Thus, our hypothesis implies that $m_p(Y_R) \xrightarrow[R \to +\infty]{} m_p\parentheses*{\strut \gauss{\lambda}}$, where the left-hand side is well-defined for $R$ large enough.

The conclusion follows by the standard method of moments~\cite[Ex.~30.1 and Thm.~30.2]{Bil1995}. Note that, in~\cite{Bil1995}, the author assumes that $Y_R$ has a finite $p$-th moment for all $R \in \cR$ and $p \in \N^*$. However, his proof only uses the fact that $Y_R$ has a finite $p$-th moment for $R$ large enough, possibly depending on $p$. Also, he only deals with the case $\cR=\N^*$, but this implies the general case.
\end{proof}

%%%%%%%%%%%%%%%%%%%%%%%%%%%%%%%%%%%%%%%%%%%%%%%%%%%%%%%%%%%%%%%%%%%%%%%%%%%%%%%%%%%%%%%%%%%%%%%%%%%%%%%

\subsection{Product spaces and diagonals}
\label{subsec: product spaces and diagonals}

In this section, we introduce notation concerning diagonals in product spaces, that will be used throughout the paper. We also recall the relation between the power measure and the factorial power measure of a submanifold in $\R^d$, see~Lemma~\ref{lem: power vs factorial power} below.

\begin{dfn}[Set-indexed products]
\label{def: set-indexed products}
Let $A$ be a non-empty finite set and $Z$ be any set.
\begin{itemize}
\item We denote by $Z^A$ the Cartesian product of $\card(A)$ copies of $Z$ indexed by $A$.

\item A generic element in $Z^A$ is denoted by $\ux = \parentheses*{x_a}_{a \in A}$.

\item For all $\emptyset \neq B \subset A$, we denote by $\ux_B=\parentheses*{x_b}_{b \in B}$ the projection of $\ux\in Z^A$ onto $Z^B$.

\item Recalling that $2B =\brackets{0,1} \times B$, we denote by $\ux_{2B} = \parentheses*{x_b}_{(i,b) \in 2B} \in Z^{2B}$.

\item Let $\parentheses*{\phi_a}_{a \in A}$ be functions from $Z$ to $\R$, we denote by $\phi_A^\otimes:\ux \mapsto \prod_{a \in A}\phi_a(x_a)$ from $Z^A$ to $\R$.
\end{itemize}
\end{dfn}

\begin{rem}
\label{rem: doubled points}
The point $\ux_{2B}$ corresponds to $\parentheses*{\ux_B,\ux_B}$ under the canonical bijection $Z^{2B} \simeq Z^B \times Z^B$. In the following, we identify $Z^{2B}$ with $Z^B \times Z^B$ through this bijection and write $\ux_{2B} = (\ux_B,\ux_B)$.
\end{rem}

Recalling the notation introduced in Section~\ref{subsec: partitions of a finite set}, we define the strata of the diagonal in~$Z^A$.

\begin{dfn}[Diagonals]
\label{def: diagonals}
We denote the \emph{large diagonal} in $Z^A$ by
\begin{equation*}
\diag = \brackets*{\strut (x_a)_{a \in A} \in Z^A \mvert \exists a,b \in A \ \text{such that} \ a \neq b \ \text{and} \ x_a=x_b}.
\end{equation*}
For all $\cI \in \cP_A$, we denote by $\diag_\cI$ the \emph{stratum} of $\diag$ where multiplicities are described by $\cI$:
\begin{equation*}
\diag_\cI = \brackets*{\strut (x_a)_{a \in A} \in Z^A \mvert \forall a,b \in A, \parentheses*{\strut x_a = x_b \iff [a]_\cI = [b]_\cI}}.
\end{equation*}
\end{dfn}

The sets $Z$ and $A$ will always be clear from the context. With this notation, we have:
\begin{align}
\label{eq: splitting product}
Z^A &= \bigsqcup_{\cI \in \cP_A} \diag_\cI & &\text{and} & \diag &= \bigsqcup_{\cI\, > \,\bigwedge\! \cP_A} \diag_\cI
\end{align}

\begin{rem}
\label{rem: diagonal min PA}
When $\cI = \bigwedge \cP_A=\brackets*{\strut \brackets{a} \mvert a \in A}$, the notation $\diag_{(\bigwedge \cP_A)}$ is misleading, since it denotes the complement of $\diag$ in $Z^A$. Therefore, we avoid this notation and use $Z^A \setminus \diag$ instead.
\end{rem}

\begin{dfn}[Diagonal inclusions]
\label{def: diagonal inclusions}
For all $\cI \in \cP_A$, we denote by $\iota_\cI: Z^\cI \setminus \diag \to \diag_\cI \subset Z^A$ the bijection defined by $\iota_\cI: \parentheses*{y_I}_{I \in \cI} \mapsto \parentheses*{y_{\squarebrackets{a}_\cI}}_{a \in A}$.
\end{dfn}

Let us now focus on the case where $Z$ is a submanifold of dimension $n$ in $\R^d$. The Euclidean metric on $\R^d$ restricts into a Riemannian metric on $Z$, and thus induces a canonical volume measure~$\nu$ on $Z$, which coincides with its $n$-dimensional Hausdorff measure. This gives rise to natural measures on~$Z^A$.

\begin{dfn}[Power measure and factorial power measure]
\label{def: power measure and factorial power measure}
For any non-empty finite set $A$, we denote by $\nu^A$ the product measure $\nu \otimes \dots \otimes \nu$ on $Z^A$, and by $\nu^{[A]}$ its restriction to $Z^A \setminus \diag$.
\end{dfn}

\begin{rem}
\label{rem: power measure and factorial power measure}
In this setting, $Z^A$ (resp.~$Z^A \setminus \diag$) is a submanifold of dimension $n\norm{A}$ in $(\R^d)^A$, and~$\nu^A$ (resp.~$\nu^{[A]}$) is its canonical volume measure induced by the Euclidean metric on $(\R^d)^A$.
\end{rem}

\begin{ex}
\label{ex: power measure and factorial power measure}
If $n=0$ then $Z$ is a discrete closed subset of $\R^d$ and $\nu = \sum_{x \in Z} \delta_x$, where $\delta_x$ stands for the unit Dirac mass at $x$. The product $Z^A$ is also $0$-dimensional and we have:
\begin{align*}
\nu^A &= \sum_{\ux \in Z^A} \delta_{\ux} & &\text{and} & \nu^{[A]} &= \sum_{\ux \in Z^A \setminus \diag} \delta_{\ux}.
\end{align*}
\end{ex}

The power measure $\nu^A$ is the one that appears when we consider a product of linear statistics of~$Z$. Indeed, $\prod_{a \in A} \prsc{\nu}{\phi_a} = \prsc*{\nu^A}{\phi_A^\otimes}$ for any reasonable test-functions $(\phi_a)_{a \in A}$. The factorial power measure $\nu^{[A]}$ is the one that can be dealt with using Kac--Rice formulas, see Section~\ref{subsec: Bulinskaya Lemma and Kac--Rice formulas for factorial moments}. The next result relates these two measures. When $n>0$ the relation is trivial, but if $n=0$ some combinatorics are involved. The following notation will allow us to deal with both cases at once.

\begin{dfn}[Indexing set of partitions]
\label{def: P A n}
Let $A$ be a non-empty finite set and $n \in \N$, we define $\cP_{A,n} = \cP_A$ if $n = 0$ and $\cP_{A,n}=\brackets{\bigwedge \cP_A}$ if $n \geq 1$. If $A = \ssquarebrackets{1}{p}$ with $p \in \N^*$, we let $\cP_{p,n}=\cP_{A,n}$.
\end{dfn}

\begin{lem}[Power vs factorial power]
\label{lem: power vs factorial power}
Let $A$ be a non-empty finite set and $\nu$ be the canonical volume measure of a $n$-dimensional submanifold $Z$ in $\R^d$. Then, we have $\nu^A = \sum_{\cI \in \cP_{A,n}} (\iota_\cI)_* \nu^{[\cI]}$.
\end{lem}

\begin{proof}
If $n=0$, by Equation~\eqref{eq: splitting product}, Definition~\ref{def: diagonal inclusions} and Example~\ref{ex: power measure and factorial power measure}, we have:
\begin{equation*}
\nu^A = \sum_{\cI \in \cP_A} \sum_{\ux \in \diag_\cI} \delta_{\ux} = \sum_{\cI \in \cP_A} \sum_{\uy \in Z^{\cI} \setminus \diag} \delta_{\iota_\cI(\uy)} = \sum_{\cI \in \cP_A} (\iota_\cI)_*\parentheses*{\sum_{\uy \in Z^{\cI} \setminus \diag} \delta_{\uy}} = \sum_{\cI \in \cP_{A,n}} (\iota_\cI)_* \nu^{[\cI]}.
\end{equation*}

Recalling Definition~\ref{def: A wedge 2}, if $n>0$ the diagonal $\diag = \bigcup_{\brackets{a,b} \in A^{\cwedge}} \brackets*{ \ux \in Z^A \mvert x_a=x_b }$ is a finite union of submanifolds of positive codimension in $Z^A$. Hence, it is negligible for the Riemannian volume $\nu^A$ of~$Z^A$, and $\nu^A = \nu^{[A]}$. Here, $\iota_{(\bigwedge \cP_A)}:Z^{(\bigwedge \cP_A)} \setminus \diag \to Z^A \setminus \diag$ is a diffeomorphism, and an isometry. Indeed, it corresponds to the identity of $Z^A \setminus \diag$ under the canonical identification of $\bigwedge \cP_A = \brackets*{\brackets{a} \mvert a \in A}$ with $A$. Hence,
\begin{equation*}
\nu^A =\nu^{[A]} = \parentheses*{\iota_{(\bigwedge \cP_A)}}_* \parentheses*{\nu^{[\bigwedge \cP_A]}} = \sum_{\cI \in \cP_{A,n}} (\iota_\cI)_* \nu^{[\cI]}.\qedhere
\end{equation*}
\end{proof}

%%%%%%%%%%%%%%%%%%%%%%%%%%%%%%%%%%%%%%%%%%%%%%%%%%%%%%%%%%%%%%%%%%%%%%%%%%%%%%%%%%%%%%%%%%%%%%%%%%%%%%%

\subsection{Thick diagonals}
\label{subsec: thick diagonals}

In this section we define sets which can be thought of as thickenings of the strata of the diagonal. In a related context, sets of this kind were introduced in~\cite[Def.~4.10]{Anc2021} and used in~\cite{AL2021a,AL2021} and~\cite{Gas2023b}. Here, we deviate a bit from these previous works in order to take into account some convexity issues.

Let $\parentheses*{V,\prsc{\cdot}{\cdot}}$ be a Euclidean space, that is a finite-dimensional vector space equipped with a Euclidean inner product. Let $\Norm{\cdot}$ denote the Euclidean norm associated with $\prsc{\cdot}{\cdot}$, we denote by $\dist(U,U')= \inf\brackets*{\strut \Norm{u-u'}\mvert u \in U, u' \in U'}$ the distance between the sets $U$ and $U' \subset V$.

\begin{dfn}[Simplex]
\label{def: simplex}
Let $A$ be a non-empty finite set, we denote the \emph{standard simplex} in~$\R^A$ by $\simplex_A = \brackets*{\ut =(t_a) \in [0,1]^A \mvert \sum_{a \in A} t_a=1}$. If $A = \ssquarebrackets{0}{q}$ for some $q \in \N$, we let $\simplex_q=\simplex_{\ssquarebrackets{0}{q}}$ denote the standard simplex of dimension $q$.
\end{dfn}

\begin{dfn}[Convex hull]
\label{def: convex hull}
Given $\ux=(x_a)_{a \in A} \in V^A$ and $\ut=(t_a)_{a \in A} \in \R^A$, we denote by $\ut \cdot \ux = \sum_{a \in A}t_a x_a \in V$. We denote the \emph{convex hull} of $\brackets*{x_a \mvert a \in A}$~by $\conv(\ux) = \brackets*{\strut \ut \cdot \ux \mvert \ut \in \simplex_A}$.
\end{dfn}

\begin{dfn}[Clustering partitions]
\label{def: Q eta}
Let $\eta \geq 0$, for all $\ux \in V^A$, we denote by
\begin{equation*}
\cQ_\eta(\ux)=\brackets*{\strut \cI \in \cP_A \mvert \forall I, J \in \cI \ \text{such that}\ I \neq J \ \text{we have} \ \dist\parentheses*{\conv(\ux_I)\strut,\conv(\ux_J)}> \eta}.
\end{equation*}
\end{dfn}

\begin{lem}[Stability by meet]
\label{lem: stability by meet}
For all $\cI$ and $\cJ \in \cQ_\eta(\ux)$ we have $\cI \wedge \cJ \in \cQ_\eta(\ux)$. In particular, $\bigwedge \cQ_\eta(\ux) \in \cQ_\eta(\ux)$ and is the finest element of $\cQ_\eta(\ux)$.
\end{lem}

\begin{proof}
Let $\cI$ and $\cJ \in \cQ_\eta(\ux)$ and let $K_1,K_2 \in \cI \wedge \cJ$ be distinct. There exist $I_1,I_2 \in \cI$ and $J_1,J_2 \in \cJ$ such that $K_1 = I_1 \cap J_1$ and $K_2 = I_2 \cap J_2$, see Definition~\ref{def: meet}. Without loss of generality, we can assume that $I_1 \neq I_2$. For $i \in \brackets{1,2}$, we have $\conv(\ux_{K_i}) \subset \conv(\ux_{I_i})$. Hence, since $\cI \in \cQ_\eta(\ux)$,
\begin{equation*}
\dist\parentheses*{\conv(\ux_{K_1}),\conv(\ux_{K_2})} \geq \dist\parentheses*{\conv(\ux_{I_1}),\conv(\ux_{I_2})} > \eta.
\end{equation*}
Thus, $\cI \wedge \cJ \in \cQ_\eta(\ux)$. Then, since $\cQ_\eta(\ux)$ is finite and stable by meet, we have $\bigwedge \cQ_\eta(\ux) \in \cQ_\eta(\ux)$.
\end{proof}

\begin{ex}
\label{ex: Q eta}
If $V=\R$ and $A = \ssquarebrackets{1}{4}$, with $\ux=(x_1,x_2,x_3,x_4) = \parentheses*{-1,0,3,6}$ and $\eta=2$, the elements of $\cQ_\eta(\ux)$ are the following partitions of $\ssquarebrackets{1}{4}$, the finest being the last one:
\begin{align*}
&\brackets*{\strut \brackets{1,2,3,4}}, & &\brackets*{\strut \brackets{1,2,3},\brackets{4}}, & &\brackets*{\strut \brackets{1,2},\brackets{3,4}} & &\text{and} & &\brackets*{\strut \brackets{1,2},\brackets{3},\brackets{4}}.
\end{align*}
\end{ex}

\begin{dfn}[Thick diagonals]
\label{def: diag I eta}
Let $A$ be a non-empty finite set and $\eta \geq 0$. For all $\cI \in \cP_A$ we define $\diag_{\cI,\eta} = \brackets*{\ux \in V^A \mvert \bigwedge \cQ_\eta(\ux)=\cI}$. Note that $V^A = \bigsqcup_{\cI \in \cP_A} \diag_{\cI,\eta}$.
\end{dfn}

\begin{ex}
\label{ex: diag I eta}
Let us give some examples.
\begin{itemize}
\item If $\ux \in \diag_{\bigwedge\!\cP_A, \eta}$, then $\bigwedge \cP_A \in \cQ_\eta(\ux)$ and $\Norm{x_a-x_b}> \eta$ for all $a \neq b \in A$. Conversely, if $\Norm{x_a-x_b}> \eta$ for all $a \neq b \in A$, then $\bigwedge \cP_A \in \cQ_\eta(\ux)$ which implies $\ux \in \diag_{\bigwedge\! \cP_A, \eta}$. Thus,
\begin{equation*}
\diag_{\bigwedge\! \cP_A, \eta} = \brackets*{\ux=(x_a)_{a \in A} \in V^A\mvert \forall a,b \in A \ \text{such that} \ a \neq b, \ \text{we have} \ \Norm{x_a-x_b}>\eta}.
\end{equation*}

\item Let us assume that $\eta=0$. Let $\cI \in \cP_A$ and let $\ux \in \diag_\cI \subset V^A$, see Definition~\ref{def: diagonals}. There exists a unique $\uy=(y_I)_{I \in \cI} \in V^\cI \setminus \diag$ such that $\ux = \iota_\cI(\uy)$, see Definition~\ref{def: diagonal inclusions}. For all $I \in \cI$ we have $\conv(\ux_I)=\brackets{y_I}$. Since $\uy \in V^\cI \setminus \diag$, we have $\cI \in \cQ_0(\ux)$, and no proper refinement of $\cI$ belongs to $\cQ_0(\ux)$. Hence $\cI = \bigwedge \cQ_0(\ux)$, and $\ux \in \diag_{\cI,0}$. Thus $\diag_\cI \subset \diag_{\cI,0}$ for any $\cI \in \cP_A$. By Equation~\eqref{eq: splitting product} and Definition~\ref{def: diag I eta}, this implies that $\diag_\cI = \diag_{\cI,0}$ for all $\cI \in \cP_A$.
\end{itemize}
\end{ex}

\begin{rem}
\label{rem: diag I eta doubled points}
Let $\ux \in V^A$. For all $B \subset A$, we have $\parentheses*{\ux_{2A}}_{2B}=\ux_{2B} = \parentheses*{\ux_B,\ux_B}$, which implies that $\conv\parentheses*{(\ux_{2A})_{2B}}=\conv(\ux_B)$. Hence, given any $\eta \geq 0$, we have $\cQ_\eta(\ux_{2A}) = \brackets{2\cI \mid \cI \in \cQ_\eta(\ux)}$. Thus $\bigwedge \cQ_\eta(\ux_{2A}) = 2 \bigwedge \cQ_\eta(\ux)$, since $\cI \mapsto 2\cI$ is increasing. That is $\ux \in \diag_{\cI,\eta}$ if and only if $\ux_{2A}=(\ux,\ux) \in \diag_{2\cI,\eta}$. In particular, for all $\eta\geq 0$,
\begin{equation*}
\brackets*{(\ux,\ux) \mvert \ux \in V^A} = \bigsqcup_{\cI \in \cP_A} \diag_{2\cI,\eta} \subsetneq V^{2A}.
\end{equation*}
\end{rem}

The following lemma partially describes the behavior of the thick diagonals from Definition~\ref{def: diag I eta} when the scale parameter $\eta$ varies.

\begin{lem}[Monotonicity of thick diagonals]
\label{lem: monotonicity diag I eta}
Let $0\leq \delta \leq \epsilon \leq \eta$, for all $\cI \in \cP_A$, we have:
\begin{align*}
\diag_{\cI,\delta} &\subset \bigsqcup_{\cJ \geq \cI} \diag_{\cJ,\eta} & &\text{and} & \diag_{\cI,\delta}\cap \diag_{\cI,\eta} \subset \diag_{\cI,\epsilon}.
\end{align*}
\end{lem}

\begin{proof}
Let $\ux \in V^A$ and let $\cI \in \cQ_\eta(\ux)$. By Definition~\ref{def: Q eta}, for all $I, J \in \cI$ such that $I \neq J$ we have $\dist\parentheses*{\conv(\ux_I),\conv(\ux_J)}> \eta \geq \delta$. Hence $\cI \in \cQ_\delta(\ux)$. Thus $\cQ_{\eta}(\ux) \subset \cQ_\delta(\ux)$, which implies that $\bigwedge \cQ_\eta(\ux) \geq \bigwedge \cQ_\delta(\ux)$.

Let us fix $\cI \in \cP_A$. Let $\ux \in \diag_{\cI,\delta}$, we denote by $\cK = \bigwedge \cQ_\eta(\ux)$ so that $\ux \in \diag_{\cK,\eta}$. Then we have $\cI = \bigwedge \cQ_\delta(\ux) \leq \bigwedge \cQ_\eta(\ux)=\cK$. Thus, $\ux \in \bigsqcup_{\cJ \geq \cI} \diag_{\cJ,\eta}$, which proves the first inclusion.

If $\ux \in \diag_{\cI,\delta}\cap \diag_{\cI,\eta}$ then $\cI = \bigwedge \cQ_\delta(\ux) \leq \bigwedge \cQ_\epsilon(\ux) \leq \bigwedge \cQ_\eta(\ux)=\cI$. Hence $\cQ_\epsilon(\ux)=\cI$ and $\ux \in \diag_{\cI,\epsilon}$. This proves the second inclusion.
\end{proof}

\begin{lem}[Projection of thick diagonals]
\label{lem: projection diag I eta}
Let $\eta \geq 0$ and let $\cI,\cJ \in \cP_A$ be such that $\cI \leq \cJ$. Then, for all $\ux \in \diag_{\cI,\eta}$, for all $J \in \cJ$, we have $\ux_J \in \diag_{\cI_J,\eta}$.
\end{lem}

\begin{proof}
Let $\ux \in \diag_{\cI,\eta}$ and $J \in \cJ$. Since $\cJ \geq \cI$, we have $\cI_J=\brackets*{I \in \cI \mvert I \subset J}$. Let $I_1,I_2 \in \cI_J \subset \cI$ be such that $I_1 \neq I_2$. Since $\cI = \bigwedge \cQ_\eta(\ux) \in \cQ_\eta(\ux)$, we have $\dist\parentheses*{\conv(\ux_{I_1}),\conv(\ux_{I_2})}> \eta$. Hence, $\cI_J \in \cQ_\eta(\ux_J)$ and $\cI_J \geq \bigwedge \cQ_\eta(\ux_J)$.

We prove that $\cI_J = \bigwedge \cQ_\eta(\ux_J)$ by contradiction. If $\cI_J > \bigwedge \cQ_\eta(\ux_J) \in \cP_J$ we define a partition $\cK =(\cI \setminus \cI_J)\sqcup \bigwedge \cQ_\eta(\ux_J) \in \cP_A$, which is such that $\cK < \cI$. Let $K_1$ and $K_2$ be two distinct blocks of $\cK$. If $K_1,K_2 \in \cI \setminus \cI_J$, then $\dist\parentheses*{\conv(\ux_{K_1}),\conv(\ux_{K_2})}> \eta$ because $\cI \in \cQ_\eta(\ux)$. If $K_1,K_2 \in \bigwedge \cQ_\eta(\ux_J)$, then $\dist\parentheses*{\conv(\ux_{K_1}),\conv(\ux_{K_2})}> \eta$. If $K_1 \in \cI \setminus \cI_J$ and $K_2 \in \bigwedge \cQ_\eta(\ux_J)$, there exists $I \in \cI_J$ such that $K_2 \subset I$. Since $\conv(\ux_{K_2})\subset\conv(\ux_I)$, we have once again $\dist\parentheses*{\conv(\ux_{K_1}),\conv(\ux_{K_2})}> \eta$. Hence $\cK \in \cQ_\eta(\ux)$, which contradicts $\cK < \cI = \bigwedge \cQ_\eta(\ux)$.

Finally, we have proved that $\cI_J = \bigwedge \cQ_\eta(\ux_J)$. That is, $\ux_J \in \diag_{\cI_J,\eta}$.
\end{proof}

\begin{lem}[Closedness of the deepest stratum]
\label{lem: closedness deepest stratum}
For all $\eta \geq 0$, the set $\diag_{\brackets{A},\eta}$ is closed in $V^A$.
\end{lem}

\begin{proof}
Let $\eta \geq 0$ and let $\ux \in V^A$ be the limit of a sequence $\parentheses*{\ux^{(n)}}_{n \in \N}$ of points of $\diag_{\brackets{A},\eta}$. Our goal is to prove that $\brackets{A}$ is the only element of $\cQ_\eta(\ux)$, see Definition~\ref{def: Q eta}, which will imply that $\ux \in \diag_{\brackets{A},\eta}$ and prove the result.

Let $\cI \in \cQ_\eta(\ux)$, for all $I,J \in \cI$ such that $I \neq J$, we have $\dist\parentheses*{\conv(\ux_I),\conv(\ux_J)}>\eta$. Hence there exists $\epsilon>0$ such that $\epsilon \leq \frac{1}{4}\min_{I \neq J} \parentheses*{\dist\parentheses*{\conv(\ux_I),\conv(\ux_J)}-\eta}$. Let $n \in \N$ be large enough that $\Norm{x_a-x_a^{(n)}}\leq \epsilon$ for all $a \in A$. Let $I, J \in \cI$ be such that $I \neq J$, let $\us \in \simplex_I$ and $\ut\in \simplex_J$. We have:
\begin{align*}
\eta \leq \Norm*{\us \cdot \ux_I - \ut \cdot \ux_J} &\leq \Norm*{\us \cdot\parentheses*{\ux_I^{(n)}-\ux_I}}+\Norm*{\us \cdot \ux_I^{(n)} - \ut \cdot \ux_J^{(n)}}+\Norm*{\ut \cdot\parentheses*{\ux_J^{(n)}-\ux_J}}\\
&\leq 2\epsilon +\Norm*{\us \cdot \ux_I^{(n)} - \ut \cdot \ux_J^{(n)}}
\end{align*}
so that $\Norm*{\us \cdot \ux_I^{(n)} - \ut \cdot \ux_J^{(n)}} \geq \eta + 2\epsilon$. We obtain $\dist\parentheses*{\conv(\ux_I^{(n)}),\conv(\ux_J^{(n)})}\geq\eta+2\epsilon$ by taking the infimum over $\us \in \simplex_I$ and $\ut \in \simplex_J$. Hence $\cI \in \cQ_{\eta}(\ux^{(n)})$, and $\cI = \brackets{A}$ since $\ux^{(n)} \in \diag_{\brackets{A},\eta}$.
\end{proof}

We conclude this section by proving that if $\ux \in \diag_{\cI,\eta}$, then the components indexed by elements of a given block $I \in \cI$ cannot be too far from one another.

\begin{lem}[Clusters diameter]
\label{lem: clusters diameter}
Let $\eta \geq 0$ and $\cI \in \cP_A$, for all $\ux \in \diag_{\cI,\eta}$ the following hold.
\begin{itemize}
\item For all $I,J \in \cI$ such that $I \neq J$ we have $\dist\parentheses*{\strut \conv(\ux_I),\conv(\ux_J)}> \eta$.

\item For all $a,b \in A$ such that $[a]_\cI=[b]_\cI$ we have $\Norm{x_a-x_b} \leq \eta \card(A)$.

\end{itemize}
\end{lem}

\begin{proof}
Let $\ux =(x_a)_{a \in A} \in \diag_{\cI,\eta}$. By Lemma~\ref{lem: stability by meet}, we have $\cI = \bigwedge \cQ_\eta(\ux) \in \cQ_\eta(\ux)$, which proves the first point, see Definitions~\ref{def: Q eta} and~\ref{def: diag I eta}.

Let $a$ and $b \in A$ and let us assume that $\Norm{x_a-x_b}>\eta \card(A)\geq 0$. We will prove that $[a]_\cI \neq [b]_\cI$, which will yield the second point by contraposition. Let $\pi:x \mapsto \prsc*{x-x_a}{\frac{x_b-x_a}{\Norm{x_b-x_a}}}$. Geometrically, $\pi$ is defined by projecting orthogonally onto the line through $x_a$ and $x_b$, then identifying this line with $\R$ using the isometry sending $x_a$ to $0$ and $x_b$ to $\Norm{x_b-x_a}$. The interval $\squarebrackets*{\strut 0,\Norm{x_b-x_a}}$ has length $\Norm{x_b-x_a}> \eta \card(A)$ and it contains at most $\card(A)$ points of the form $\pi(x_c)$ with $c \in A$, among which its endpoints. Hence, there exists $t >0=\pi(x_a)$ such that $t+\eta < \Norm{x_b-x_a}=\pi(x_b)$ and $\brackets*{\pi(x_c) \mvert c \in A} \cap [t,t+\eta] = \emptyset$.

Let $J = \brackets*{c \in A \mvert \pi(x_c) <t}$ and $\cJ = \brackets{J, A \setminus J}$. Note that $a \in J$ and $b \in A \setminus J$ so that $\cJ \in \cP_A$. We have $A \setminus J = \brackets*{c \in A \mvert \pi(x_c) \geq t}=\brackets*{c \in A \mvert \pi(x_c) > t+\eta}$, thanks to our choice of $t$. Since $\pi$ is affine, for all $x \in \conv(\ux_J)$ and $y \in \conv(\ux_{A \setminus J})$ we have $\pi(x) <t$ and $\pi(y) >t+\eta$, hence
\begin{equation*}
\eta < \pi(y)-\pi(x) = \prsc*{y-x}{\frac{x_b-x_a}{\Norm{x_b-x_a}}}\leq \Norm{y-x}.
\end{equation*}
By compactness of the convex hulls we have $\dist\parentheses*{\conv(\ux_J)\strut,\conv(\ux_{A \setminus J})} > \eta$, i.e.~$\cJ \in \cQ_\eta(\ux)$. Thus, we have $\cI = \bigwedge \cQ_\eta(\ux) \leq \cJ$. Since $[a]_\cJ =J \neq (A \setminus J)=[b]_\cJ$, we have $[a]_\cI \neq [b]_\cI$.
\end{proof}

%%%%%%%%%%%%%%%%%%%%%%%%%%%%%%%%%%%%%%%%%%%%%%%%%%%%%%%%%%%%%%%%%%%%%%%%%%%%%%%%%%%%%%%%%%%%%%%%%%%%%%%
%%%%%%%%%%%%%%%%%%%%%%%%%%%%%%%%%%%%%%%%%%%%%%%%%%%%%%%%%%%%%%%%%%%%%%%%%%%%%%%%%%%%%%%%%%%%%%%%%%%%%%%

\section{Polynomial interpolation and Gaussian fields}
\label{sec: polynomial interpolation and Gaussian fields}

The goal of this section is to introduce the Kergin interpolant of a Gaussian field and to prove some of its properties. In Section~\ref{subsec: divided differences and Kergin interpolation}, we recall the definition of Kergin interpolation for a deterministic function, which relies on generalized divided differences. In Section~\ref{subsec: divided differences of a Gaussian field}, we relate the covariance structure of a Gaussian field to that of its divided differences, that are themselves Gaussian fields. In Section~\ref{subsec: twisted Kergin interpolant}, we introduce a twisted Kergin interpolant, whose distribution is invariant by diagonal translations if the original field is stationary. Finally, Section~\ref{subsec: uniform estimates for Sigma I} is dedicated to proving uniform estimates on the variance operator of the Kergin interpolant associated with a Gaussian field.

%%%%%%%%%%%%%%%%%%%%%%%%%%%%%%%%%%%%%%%%%%%%%%%%%%%%%%%%%%%%%%%%%%%%%%%%%%%%%%%%%%%%%%%%%%%%%%%%%%%%%%%

\subsection{Divided differences and Kergin interpolation}
\label{subsec: divided differences and Kergin interpolation}

In the following, given two open subsets $U$ and $U'$ in some Banach spaces and $q \in \N$, we denote by $\cC^q(U,U')$ the space of maps of class $\cC^q$ from $U$ to $U'$. Let $U \subset \R^d$ be a convex open set. Given $f \in \cC^q(U,\R^k)$ we denote by $D^q_xf$ its $q$-th differential at $x \in U$, which belongs to the space $\sym^q(\R^d,\R^k)$ of symmetric $q$-linear maps from $(\R^d)^q$ to $\R^k$.

Let $A$ be a non-empty finite set. Recall that $\simplex_A$ stands for the standard simplex in $\R^A$. We denote by $\varsigma_A$ the Lebesgue measure on the hyperplane $\brackets*{(t_a)_{a \in A}\in \R^A\mvert \sum_{a \in A} t_a=1}$, normalized so that $\int_{\ut \in \simplex_A} \dx\varsigma_A(\ut)=\frac{1}{\norm{A}!}$.
\begin{dfn}[Divided differences]
\label{def: divided differences}
Let $A$ be a finite set of cardinality $q+1$. Let $f \in \cC^q(U,\R^k)$ and $\ux\in U^A$, we define the \emph{divided difference} of $f$ at $\ux$ by:
\begin{equation*}
f[\ux] = \int_{\ut \in \simplex_A} D^q_{\ut \cdot \ux} f \ \dx\varsigma_A(\ut) \in \sym^q(\R^d,\R^k).
\end{equation*}
\end{dfn}

\begin{rem}
\label{rem: divided differences symmetric}
The divided difference $f[\ux]$ is the integral over $\conv(\ux) \subset U$ (see~Definition~\ref{def: convex hull}) of~$D^qf$, with respect to the push-forward of $\varsigma_A$ by the standard parametrization $\ut \mapsto \ut \cdot \ux$. Since $\varsigma_A$ is invariant under permutation of the vertices of $\simplex_A$, the quantity $f[\ux]$ is invariant under permutation of the components of $\ux=(x_a)_{a \in A}$.
\end{rem}

\begin{ex}
\label{ex: divided differences}
Let us give two examples when $A=\ssquarebrackets{0}{q}$. The second one explains the terminology ``divided differences''.
\begin{itemize}
\item If $x_0=\dots=x_q=y \in U$ then $f[x_0,\dots,x_q] = \frac{1}{q!}D^q_yf$.
\item If $d=1$ then $\sym^q(\R^d,\R^k)\simeq \R^k$ canonically and, for any $\ux \in U^{q+1}\setminus \diag$, we can compute $f[x_0,\dots,x_q] = \frac{f[x_1,\dots,x_q]-f[x_0,\dots,x_{q-1}]}{x_q-x_0}$ recursively from the values of $\parentheses*{\strut f(x_i)}_{0 \leq i \leq q}$.
\end{itemize}
\end{ex}

We denote by $X=(X_1,\dots,X_d)$ and by $\R_q[X]$ the space of polynomials of degree at most~$q$ in $d$ variables. The \emph{Kergin polynomial} of a $\cC^q$ function $f$ at $\ux \in U^A$ is the only polynomial in $\R_q[X]$ interpolating the divided differences of $f$ on all $(\ux_B)_{\emptyset \neq B \subset A}$, cf.~\cite{Ker1980}. Interpolating component-wise we get the following, where tuples in $\R_q[X]^k$ are understood as polynomial maps from $\R^d$ to $\R^k$.

\begin{thm}[Kergin interpolation]
\label{thm: Kergin interpolation}
Let $f \in \cC^q(U,\R^k)$ and $\ux=(x_0,\dots,x_q) \in U^{q+1}$, there exists a unique $K(f,\ux) \in \R_q[X]^k$ such that $K(f,\ux)[\ux_B]=f[\ux_B]$ for all non-empty $B \subset \ssquarebrackets{0}{q}$. In~particular, if $y \in U$ appears with multiplicity at least $l+1$ in $\ux$ then $D^l_y\parentheses*{\strut K(f,\ux)}=D^l_y f$. Moreover,
\begin{equation}
\label{eq: Kergin explicit}
K(f,\ux) = \sum_{i=0}^q f[x_0,\dots,x_i](X-x_0,\dots,X-x_{i-1}).
\end{equation}
\end{thm}

\begin{proof}
We apply~\cite[Thm.~12.5]{BHS1993} with $m=0$ to each component of $f$, see also~\cite{Mic1980,MM1980}. This yields existence and uniqueness of $K(f,\ux)$ as well as the claimed expression. If $y \in U$ appears with multiplicity at least $l+1$ in $\ux$, let $B \subset \ssquarebrackets{0}{q}$ be such that $\norm{B}=l+1$ and $x_i=y$ for all $i \in B$. Then, by Example~\ref{ex: divided differences}, we have $D^l_y\parentheses*{\strut K(f,\ux)}=l!\parentheses*{\strut K(f,\ux)[\ux_B]}=l!f[\ux_B]=D^l_y f$.
\end{proof}

\begin{ex}
\label{ex: Kergin polynomial}
When $k=1$, we recover known interpolating polynomials as particular cases.
\begin{itemize}
\item If $x_0=\dots=x_q=y \in U$ then $K(f,\ux)$ is the Taylor polynomial of degree $q$ of $f$ at $y$.
\item If $d=1$ then $K(f,\ux)$ is the Hermite interpolating polynomial of $f$ at $x_0,\dots,x_q \in \R$.
\end{itemize}
\end{ex}

\begin{rem}
\label{rem: Kergin symmetric}
The polynomial map $K(f,\ux)$ does not depend on an ordering of $(x_i)_{0 \leq i \leq q}$.  Indeed, it is defined by interpolating a family of conditions which is invariant under permutation of~$\ssquarebrackets{0}{q}$, see Remark~\ref{rem: divided differences symmetric}. This allows to make sense of $K(f,\ux)$ for any $\ux \in U^A$, where $\card(A)=q+1$. Note however that each term of the sum in Equation~\eqref{eq: Kergin explicit} depends on the ordering of $(x_i)_{0 \leq i \leq q}$.
\end{rem}

The previous theorem is part of the standard knowledge concerning Kergin's interpolation. From this point on, we will analyze how the Kergin interpolant behaves, with respect to partitioning the set of interpolating points. The next lemma shows that the family of Kergin interpolants on well-separated sets of points defines a surjective map. Recall that, given $\ux \in (\R^d)^A$, we defined $\cQ_0(\ux) \subset \cP_A$ in Definition~\ref{def: Q eta}.

\begin{lem}[Surjectivity of block-wise interpolation]
\label{lem: surjectivity blockwise Kergin}
Let $A\neq \emptyset$ be finite, let $\ux \in (\R^d)^A$ and $\cI \in \cQ_0(\ux)$. The linear map $P \mapsto \parentheses*{K(P,\ux_I)}_{I \in \cI}$ is surjective from $\R_{\norm{A}-1}[X]^k$ to $\prod_{I \in \cI} \R_{\norm{I}-1}[X]^k$.
\end{lem}

\begin{proof}
Since $\cI \in \cQ_0(\ux)$, the convex hulls $\parentheses*{\conv(\ux_I)}_{I \in \cI}$ are disjoint compact subsets of $\R^d$. Thus, for any $I \in \cI$, there exists a cut-off function $\chi_I \in \cC^\infty(\R^d,\R)$ such that $\chi_I= 1$ on a neighborhood of $\conv(\ux_I)$ and $\chi_I= 0$ on a neighborhood of $\bigsqcup_{J \in \cI \setminus \brackets{I}} \conv(\ux_J)$.

Let $(P_I)_{I \in \cI} \in \prod_{I \in \cI} \R_{\norm{I}-1}[X]^k$, we define $f = \sum_{I \in \cI}\chi_I P_I$ and $P=K(f,\ux) \in \R_{\norm{A}-1}[X]^k$. Let $I \in \cI$, for all non-empty $J \subset I$, we have:
\begin{equation*}
P[\ux_J] = K(f,\ux)[\ux_J]=f[\ux_J] = P_I[\ux_J]
\end{equation*}
because $f=P_I$ on a neighborhood of $\conv(\ux_J)\subset \conv(\ux_I)$, see Definition~\ref{def: divided differences}. The uniqueness part in Theorem~\ref{thm: Kergin interpolation} shows that $K(P,\ux_I)=P_I$ for all $I \in \cI$. Hence the result.
\end{proof}

In the case where each geometrical point appears with multiplicity exactly $2$ in $\ux$, the previous result shows that we can interpolate the $1$-jets of any $\cC^1$ function on $q$ pairwise distinct points in~$\R^d$ by a polynomial of degree at most $2q-1$. In the following, given vector spaces $V$ and $W$, we denote by $\cL(V,W)$ the space of continuous linear maps from $V$ to $W$ and by $\cL(V)=\cL(V,V)$.

\begin{cor}[Surjectivity of $1$-jets evaluation]
\label{cor: surjectivity 1-jets}
Let $A\neq \emptyset$ be finite and let $\ux \in (\R^d)^A \setminus \diag$, the linear map $P \mapsto \parentheses*{\strut \parentheses*{P(x_a),D_{x_a}P}}_{a \in A}$ is surjective from $\R_{2\norm{A}-1}[X]^k$ to $\parentheses*{\R^k \times \cL(\R^d,\R^k)}^A$.
\end{cor}

\begin{proof}
Since $\ux \notin \diag$, we have $\bigwedge \cP_A = \brackets*{\brackets{a}\mvert a \in A} \in \cQ_0(\ux)$, see Definition~\ref{def: Q eta}. Then, by Remark~\ref{rem: diag I eta doubled points}, we have $2\bigwedge \cP_A \in \cQ_0(\ux_{2A})$. Lemma~\ref{lem: surjectivity blockwise Kergin} applied to $2\bigwedge \cP_A = \brackets*{\brackets{0,1} \times \brackets{a}\strut \mvert a \in A}$ and $\ux_{2A}$ then shows that $\varpi_{\ux}:P \mapsto \parentheses*{\strut K(P,x_a,x_a)}_{a \in A}$ is surjective from $\R_{2\norm{A}-1}[X]^k$ to $(\R_1[X]^k)^A$.

For all $y \in \R^d$, the map $P \mapsto \parentheses*{P(y),D_yP}$ is an isomorphism from $\R_1[X]^k$ to $\R^k \times \cL(\R^d,\R^k)$ whose inverse is given by $(z,\Lambda) \mapsto z +\Lambda(X-y)$. Hence $\vartheta_{\ux}:\parentheses*{P_a}_{a \in A} \mapsto \parentheses*{\strut \parentheses*{P_a(x_a),D_{x_a}P_a}}_{a \in A}$ is an isomorphism from $(\R_1[X]^k)^A$ to $\parentheses*{\R^k \times \cL(\R^d,\R^k)}^A$.

Let $P \in \R_{2\norm{A}-1}[X]^k$. By Example~\ref{ex: Kergin polynomial}, we know that $K(P,x_a,x_a)$ is the Taylor polynomial of degree $1$ of $P$ at $x_a$. Hence $\vartheta_{\ux}\circ \varpi_{\ux}(P)=\vartheta_{\ux}\parentheses*{\parentheses*{K(P,x_a,x_a)}_{a \in A}} = \parentheses*{\strut \parentheses*{P(x_a),D_{x_a}P}}_{a \in A}$ and the map we are interested in is $\vartheta_{\ux}\circ \varpi_{\ux}:\R_{2\norm{A}-1}[X]^k\to\parentheses*{\R^k \times \cL(\R^d,\R^k)}^A$. It is surjective as the composition of two surjective maps.
\end{proof}

To conclude this section, we discuss the regularity of divided differences. Recall that $U \subset \R^d$ is convex and open. Given a compact $\Gamma \subset U$, we denote by $\Norm{f}_{\cC^q,\Gamma}=\max \brackets*{\Norm{D^i_z f} \mvert \strut i \in \ssquarebrackets{0}{q}, z \in \Gamma}$ for all $f \in \cC^q(U,\R^k)$, where the norm on~$\sym^i(\R^d,\R^k)$ is the operator norm subordinated to the Euclidean norms on $\R^d$ and~$\R^k$. We consider $\cC^q(U,\R^k)$ endowed with the compact-open topology, that is the metric topology characterized by the fact that $f_n \xrightarrow[n \to +\infty]{}~f$ if and only if $\Norm{f_n-f}_{\cC^q,\Gamma}\xrightarrow[n \to +\infty]{}0$ for all compact subsets $\Gamma \subset U$.

\begin{lem}[Regularity of divided differences]
\label{lem: regularity divided differences}
Let $q \in \N$ and $A$ be a set of cardinality $q+1$. The map $\parentheses*{f,\ux} \mapsto f[\ux]$ is continuous from $\cC^q(U,\R^k) \times U^A$ to $\sym^q(\R^d,\R^k)$.
\end{lem}

\begin{proof}
Without loss of generality, we can assume that $A = \ssquarebrackets{0}{q}$. Let us fix $f \in \cC^q(U,\R^k)$ and $\ux=(x_0,\dots,x_q) \in U^{q+1}$, and prove the continuity of the map we are interested in at the point $(f,\ux)$. There exists $\varrho>0$ such that, for all $i \in \ssquarebrackets{0}{q}$, the closed ball $\B_i$ of center $x_i$ and radius $\varrho$ is contained in~$U$. By Carathéodory's Theorem, $\Gamma = \conv\parentheses*{\bigcup_{i=0}^q \B_i} \subset U$ is a compact set.

Let $\epsilon>0$, there exists a neighborhood $\cO$ of $f$ in $\cC^q(U,\R^k)$ such that, for all $g \in \cO$, we have $\Norm{g-f}_{\cC^q,\Gamma}\leq \epsilon$. Besides, by uniform continuity of $D^q f$ on $\Gamma$, there exists $\delta\in (0,\varrho)$ such that $\Norm*{D^q_zf-D^q_wf}\leq \epsilon$ for all $w,z \in \Gamma$ such that $\Norm{z-w}\leq \delta$.

Let $g \in \cO$ and $y_0,\dots,y_q \in U$ be such that $\Norm{y_i-x_i}\leq \delta$ for all $i \in \ssquarebrackets{0}{q}$. Since $\delta <\varrho$, we have $y_0,\dots,y_q \in \Gamma$. By convexity, for any $\ut\in \simplex_q$ we have $\ut \cdot \uy$ and $ \ut \cdot \ux \in \Gamma$ and $\Norm{\ut \cdot \uy - \ut \cdot \ux} \leq \delta$, so that
\begin{equation*}
\Norm*{f[\uy]-f[\ux]} \leq \int_{\ut \in \simplex_q} \Norm*{D^q_{\ut \cdot \uy} f-D^q_{\ut \cdot \ux} f} \dx \varsigma_q(\ut)\leq \epsilon.
\end{equation*}
On the other hand, since $g \in \cO$ and $\ut \cdot\uy \in \Gamma$ for all $\ut \in \simplex_q$, we have:
\begin{equation*}
\Norm*{g[\uy]-f[\uy]} \leq \int_{\ut \in \simplex_q} \Norm*{D^q_{\ut\cdot\uy}(g-f)} \dx \varsigma_q(\ut) \leq \Norm*{g-f}_{\cC^q,\Gamma} \leq \epsilon.
\end{equation*}
Thus, $\Norm*{g[\uy]-f[\ux]}\leq \Norm*{g[\uy]-f[\uy]}+\Norm*{f[\uy]-f[\ux]} \leq 2\epsilon$ for all $g \in \cO$ and $y_0,\dots,y_q \in U$ such that $\Norm{y_i-x_i}\leq \delta$ for all $i \in \ssquarebrackets{0}{q}$. This proves the continuity at $(f,\ux)$ and concludes the proof.
\end{proof}

%%%%%%%%%%%%%%%%%%%%%%%%%%%%%%%%%%%%%%%%%%%%%%%%%%%%%%%%%%%%%%%%%%%%%%%%%%%%%%%%%%%%%%%%%%%%%%%%%%%%%%%

\subsection{Divided differences of a Gaussian field}
\label{subsec: divided differences of a Gaussian field}

In this section, we recall some elementary facts about Gaussian vectors and fields. Then, we study the divided differences associated with a Gaussian field, which are themselves Gaussian fields.

\begin{dfn}[Covariance operators]
\label{def: covariance operators}
Let $Y_1$ and $Y_2$ be centered Gaussian vectors taking values in Euclidean spaces $\parentheses*{V_1,\prsc{\cdot}{\cdot}}$ and $\parentheses*{V_2,\prsc{\cdot}{\cdot}}$ respectively. The \emph{covariance operator} of $Y_1$ and $Y_2$ is the only linear map $\cov{Y_1}{Y_2}\in \cL(V_2,V_1)$ such that:
\begin{equation}
\label{eq: def covariance}
\forall v_1 \in V_1, \forall v_2 \in V_2, \quad \prsc*{v_1}{\strut \cov{Y_1}{Y_2}v_2} = \esp{\strut \prsc{Y_1}{v_1}\prsc{Y_2}{v_2}}.
\end{equation}
Denoting by $\sym(V_1)$ (resp.~$\sym^+(V_1)$) the space (resp.~set) of self-adjoint (resp.~positive self-adjoint) operators on $V_1$, the \emph{variance operator} of $Y_1$ is then $\var{Y_1}=\cov{Y_1}{Y_1}\in \sym(V_1)$. This operator is non-negative, and we say that $Y_1$ is \emph{non-degenerate} if $\var{Y_1} \in \sym^+(V_1)$.
\end{dfn}

Let us introduce some notation in order to talk about the regularity of Gaussian fields and their covariance kernels.

\begin{dfn}[Multi-indices]
\label{def: multi-indices}
Let $\alpha=(\alpha_1,\dots,\alpha_d) \in \N^d$ be a \emph{multi-index}.
\begin{itemize}
\item We denote by $\norm{\alpha}=\sum_{i=1}^d \alpha_i$ its \emph{length} and by $\alpha!=\prod_{i=1}^d (\alpha_i!)$.
\item If $x=(x_1,\dots,x_d) \in \R^d$ we denote by $x^\alpha=\prod_{i=1}^d x_i^{\alpha_i}$, and accordingly $X^\alpha=\prod_{i=1}^d X_i^{\alpha_i}$.
\item Given $i \in \ssquarebrackets{1}{d}$, we denote by $\partial_i$ the partial derivative with respect to the $i$-th variable in~$\R^d$ and by $\partial^\alpha=\partial_1^{\alpha_1}\dots\partial_d^{\alpha_d}$.
\item Let $\beta \in \N^d$, we have $(\alpha,\beta) \in \N^{2d}$ and we denote by $\partial^{\alpha,\beta}=\partial^{(\alpha,\beta)}$, acting on maps on $\R^d \times \R^d$.
\end{itemize}
\end{dfn}

Let $U \subset \R^d$ be an open subset and $V$ be some Euclidean space, we denote by $\cC^{q,q}(U \times U,V)$ the space of maps $r:U\times U\to V$ such that $\partial^{\alpha,\beta}r:U \times U \to V$ is well-defined and continuous for all $\alpha, \beta \in \N^d$ such that $\norm{\alpha}\leq q$ and $\norm{\beta}\leq q$.

\begin{dfn}[Correlation kernel]
\label{def: correlation kernel}
Let $f:U \to V$ be a centered Gaussian field, where $U \subset \R^d$ is open and $V$ is Euclidean. Its \emph{covariance kernel} is the map $r:U\times U \to \cL(V)$ defined by:
\begin{equation*}
r:(x,y) \longmapsto \cov{f(x)}{f(y)}.
\end{equation*}
\end{dfn}

If $f:U \to V$ is a $\cC^q$ centered Gaussian field then $r \in \cC^{q,q}(U\times U,\cL(V))$, see~\cite[Chap.~1.4.3]{AW2009} for example. In this case, we have $\partial^{\alpha,\beta}r:(x,y) \mapsto \cov{\partial^\alpha f(x)}{\partial^\beta f(y)}$ for all $\alpha,\beta \in \N^d$ of length at most $q$. If $U = \R^d$ and $f$ is stationary then $r(x,y) = r(0,y-x)$ for all $x,y \in \R^d$.

Let us assume the open set $U$ to be convex and let $f$ be a centered Gaussian field. We first prove that the divided differences of~$f$ (see~Definition~\ref{def: divided differences}) define Gaussians fields.

\begin{lem}[Divided differences of a Gaussian field]
\label{lem: divided differences Gaussian field}
Let $A$ be a set of cardinality~$q+1$ and $f \in \cC^q(U,\R^k)$ be a centered Gaussian field. Then $f[\,\cdot\,]:U^A \to \sym^q(\R^d,\R^k)$ is a continuous centered Gaussian field.
\end{lem}

\begin{proof}
By Lemma~\ref{lem: regularity divided differences}, for all $n \in \N^*$ and $\ux^{(1)},\dots,\ux^{(n)} \in U^A$, the map $f \mapsto \parentheses*{f[\ux^{(i)}]}_{1 \leq i \leq n}$ is linear and continuous, hence $\parentheses*{f[\ux^{(i)}]}_{1 \leq i \leq n}$ is a centered Gaussian vector, see~\cite[Lem.~2.2.2]{Bog1998}. This shows that $f[\,\cdot\,]$ is a centered Gaussian field. It is continuous, once again by Lemma~\ref{lem: regularity divided differences}.
\end{proof}

Let $(e_1,\dots,e_d)$ stand for the standard basis of $\R^d$. For all multi-index $\alpha$ of length~$q$, we denote by $e_\alpha=(e_1,\dots,e_1,\dots,e_d,\dots,e_d)$, where $e_i$ appears exactly $\alpha_i$ times. We denote by $e_\alpha^*$ the unique element of $\sym^q(\R^d,\R)$ such that $e^*_\alpha(e_\alpha)=1$ and $e^*_\alpha(e_\beta)=0$ if $\norm{\beta}=q$ and $\beta\neq \alpha$. The family $(e_\alpha^* \otimes e_i)_{\norm{\alpha}=q, 1 \leq i \leq k}$ is a basis of $\sym^q(\R^d,\R^k)\simeq \sym^q(\R^d,\R) \otimes \R^k$ and we endow this space with the only inner product making this basis orthonormal. Note that, for any $\cC^q$ function $f=(f_1,\dots,f_k)$ from $\R^d$ to $\R^k$ and $z \in \R^d$, we have
\begin{equation}
\label{eq: differential in terms of e*alpha}
D^q_z f = \sum_{i=1}^k D^q_zf_i \otimes e_i = \sum_{i=1}^k \parentheses*{\sum_{\norm{\alpha}=q} D^q_zf_i(e_\alpha)e_\alpha^*}\otimes e_i = \sum_{\norm{\alpha}=q, 1 \leq i \leq k} \partial^\alpha f_i(z) e_\alpha^* \otimes e_i.
\end{equation}

\begin{lem}[Covariance of divided differences]
\label{lem: covariance divided differences}
Let $f\in \cC^q(U,\R^k)$ be a centered Gaussian field, let $\ux\in U^{m+1}$ and $\uy\in U^{n+1}$ where $m \leq q$ and $n \leq q$. Then, we~have:
\begin{equation*}
\cov{f[\ux]}{f[\uy]} = \int_{\us \in \simplex_m}\int_{\ut \in \simplex_n} \cov{D^m_{\us\cdot \ux}f}{D^n_{\ut\cdot \uy}f} \dx \us \dx \ut.
\end{equation*}
\end{lem}

\begin{proof}
Let $u \in \sym^m(\R^d,\R^k)$ and $v \in \sym^n(\R^d,\R^k)$, by linearity and Equation~\eqref{eq: def covariance}, we have:
\begin{align*}
\prsc*{u}{\parentheses*{\int_{\us \in \simplex_m}\int_{\ut \in \simplex_n} \cov{D^m_{\us\cdot \ux}f}{D^n_{\ut\cdot \uy}f} \dx \us \dx \ut} v} &= \int_{\us \in \simplex_m}\int_{\ut \in \simplex_n} \prsc*{u}{\cov{D^m_{\us\cdot \ux}f}{D^n_{\ut\cdot \uy}f}v} \dx \us \dx \ut\\
&= \int_{\us \in \simplex_m}\int_{\ut \in \simplex_n} \esp{\prsc{D^m_{\us\cdot \ux}f}{u}\prsc{D^n_{\ut\cdot \uy}f}{v}} \dx \us \dx \ut.
\end{align*}
The integrand being a product of two continuous Gaussian fields, Fubini's Theorem applies and we obtain:
\begin{equation*}
\esp{\parentheses*{\int_{\us \in \simplex_m}\prsc{D^m_{\us\cdot \ux}f}{u} \dx \us} \parentheses*{\int_{\ut \in \simplex_n}\prsc{D^n_{\ut\cdot \uy}f}{v} \dx \ut}}= \esp{\prsc{f[\ux]}{u}\prsc{f[\uy]}{v}}=\prsc*{u}{\cov{f[\ux]}{f[\uy]}v}.
\end{equation*}
These computations holding for any $u$ and $v$, this proves the result.
\end{proof}

\begin{lem}[Distance between covariances of divided differences]
\label{lem: distance between covariances of divided differences}
Let $f$ and $\tilde{f}\in \cC^q(U,\R^k)$ be centered Gaussian fields, we denote respectively by $r$ and $\tilde{r}$ their covariance kernels. Let $m$ and $n \in \ssquarebrackets{0}{q}$, for all $\ux\in U^{m+1}$ and $\uy\in U^{n+1}$ we have the following estimates:
\begin{equation*}
\Norm*{\strut \cov{f[\ux]}{f[\uy]}-\cov{\tilde{f}[\ux]}{\tilde{f}[\uy]}} \leq (q+1)^d \max_{\substack{\norm{\alpha}=m\\ \norm{\beta}=n}} \ \max_{\substack{w \in \conv(\ux)\\ z \in \conv(\uy)}} \ \Norm*{\strut \partial^{\alpha,\beta}r(w,z)-\partial^{\alpha,\beta}\tilde{r}(w,z)},
\end{equation*}
where the norms are the operator norms subordinated to the Euclidean ones.
\end{lem}

\begin{proof}
Recall that we have defined an orthonormal basis $\parentheses{e_\alpha^* \otimes e_i}_{\norm{\alpha}=m, 1 \leq i \leq k}$ of $\sym^m(\R^d,\R^k)$ (resp.~$\parentheses{e_\beta^* \otimes e_j}_{\norm{\beta}=n, 1 \leq j \leq k}$ of $\sym^n(\R^d,\R^k)$). Let
\begin{align*}
u &= \sum_{\norm{\alpha}=m, 1 \leq i \leq k} u_{\alpha,i} e_\alpha^* \otimes e_i & &\text{and} & v &= \sum_{\norm{\beta}=n, 1 \leq j \leq k} v_{\beta,j} e_\beta^* \otimes e_j.
\end{align*}
For all $\alpha$ (resp.~$\beta$) of length $m$ (resp.~$n$) we denote by $u_\alpha=\sum_{i=1}^k u_{\alpha,i}e_i$ (resp.~$v_\beta = \sum_{j=1}^k v_{\beta,j}e_j$).

We first consider the field $f=(f_1,\dots,f_k)$. Let $w,z \in U$, by Equations~\eqref{eq: def covariance} and~\eqref{eq: differential in terms of e*alpha} we have:
\begin{align*}
\prsc*{u}{\strut \cov{D^m_wf}{D^n_zf}v} &= \esp{\strut \prsc*{D^m_wf}{u}\prsc*{D^n_zf}{v}} = \sum_{\norm{\alpha}=m,\norm{\beta}=n} \ \sum_{1 \leq i,j \leq k} u_{\alpha,i}v_{\beta,j} \esp{\partial^\alpha f_i(w)\partial^\beta f_j(z)}\\
&= \sum_{\substack{\norm{\alpha}=m\\\norm{\beta}=n}} \esp{\prsc*{\partial^\alpha f(w)}{u_\alpha}\prsc*{\partial^\beta f(z)}{v_\beta}}= \sum_{\substack{\norm{\alpha}=m\\\norm{\beta}=n}} \prsc*{u_\alpha}{\partial^{\alpha,\beta}r(w,z)v_\beta}.
\end{align*}
A similar formula holds for $\tilde{f}$. Hence, by linearity,
\begin{align*}
\prsc*{u}{\parentheses*{\strut \cov{D^m_wf}{D^n_zf}-\cov{D^m_w \tilde{f}}{D^n_z \tilde{f}}}v} &= \sum_{\norm{\alpha}=m,\norm{\beta}=n} \prsc*{u_\alpha}{\parentheses*{\partial^{\alpha,\beta}(r-\tilde{r})(w,z)}v_\beta}\\
&\leq \sum_{\norm{\alpha}=m,\norm{\beta}=n} \Norm{u_\alpha}\Norm{v_\beta} \Norm*{\partial^{\alpha,\beta}(r-\tilde{r})(w,z)}.
\end{align*}

If $\Norm{u} = 1= \Norm{v}$, we have $\sum_{\norm{\alpha}=m}\Norm{u_\alpha}\leq \sqrt{\card\brackets*{\alpha \in \N^d \mvert \norm{\alpha}=m}} \leq (m+1)^\frac{d}{2}\leq (q+1)^\frac{d}{2}$, and similarly $\sum_{\norm{\beta}=n}\Norm{v_\beta}\leq (q+1)^\frac{d}{2}$. Hence, the last term in the previous computation is bounded by $(q+1)^d \max_{\norm{\alpha}=m,\norm{\beta}=n} \Norm*{\partial^{\alpha,\beta}(r-\tilde{r})(w,z)}$. Thus,
\begin{multline*}
\Norm*{\strut \cov{D^m_wf}{D^n_zf}-\cov{D^m_w \tilde{f}}{D^n_z \tilde{f}}}\\
\begin{aligned}
&= \max_{\Norm{u}=1 =\Norm{v}} \prsc*{u}{\parentheses*{\strut \cov{D^m_w f}{D^n_zf }-\cov{D^m_w \tilde{f}}{D^n_z \tilde{f}}}v}\\
&\leq (q+1)^d \max_{\norm{\alpha}=m,\norm{\beta}=n} \Norm*{\partial^{\alpha,\beta}r(w,z)-\partial^{\alpha,\beta}\tilde{r}(w,z)}.
\end{aligned}
\end{multline*}
Finally, by Lemma~\ref{lem: covariance divided differences}, we obtain:
\begin{multline*}
\Norm*{\strut \cov{f[\ux]}{f[\uy]}-\cov{\tilde{f}[\ux]}{\tilde{f}[\uy]}}\\
\begin{aligned}
&\leq \int_{\us \in \simplex_m}\int_{\ut \in \simplex_n} \Norm*{\strut \cov{D^m_{\us\cdot \ux}f}{D^n_{\ut\cdot \uy}f}-\cov{D^m_{\us\cdot \ux}\tilde{f}}{D^n_{\ut \cdot \ut}\tilde{f}}} \dx \us \dx \ut\\
&\leq (q+1)^d \max_{\norm{\alpha}=m, \norm{\beta}=n} \ \max_{w \in \conv(\ux), z \in \conv(\uy)} \ \Norm*{\strut \partial^{\alpha,\beta}r(w,z)-\partial^{\alpha,\beta} \tilde{r}(w,z)},
\end{aligned}
\end{multline*}
since the simplices $\simplex_m$ and $\simplex_n$ have measure at most $1$.
\end{proof}

%%%%%%%%%%%%%%%%%%%%%%%%%%%%%%%%%%%%%%%%%%%%%%%%%%%%%%%%%%%%%%%%%%%%%%%%%%%%%%%%%%%%%%%%%%%%%%%%%%%%%%%

\subsection{Twisted Kergin interpolant}
\label{subsec: twisted Kergin interpolant}

In this section, we introduce a twisted Kergin interpolant which is well-behaved with respect to the action of $\R^d$ by diagonal translations.

Let $A$ be a non-empty finite set, recall that the \emph{diagonal action} $\R^d \acts (\R^d)^A$ is defined by: $\tau \cdot \ux = \parentheses*{x_a+\tau}_{a \in A}$ for all $\tau \in \R^d$ and $\ux=(x_a)_{a \in A}\in (\R^d)^A$. Besides, given $\tau \in \R^d$ and $f$ any function defined on $U \subset \R^d$, we denote by $\tau\cdot f:y \mapsto f(y+\tau)$ the $\tau$-translate of $f$, defined on $U - \tau$. This defines families of compatible actions $\R^d \acts \cC^q(\R^d,\R^k)$ and $\R^d \acts \R_q[X]^k$ for all $q \in \N$. More generally, if $f$ is any function on a subset of $(\R^d)^A$, we denote by $\tau\cdot f:\ux \mapsto f(\tau\cdot \ux)$.

As in the previous sections, we consider a convex open subset $U \subset \R^d$.

\begin{lem}[Translation equivariance of divided differences and Kergin interpolants]
\label{lem: translation equivariance}
Let $q \in \N$ and $A$ be a set of cardinality $q+1$. For any $f \in \cC^q(U,\R^k)$, $\ux \in U^A$ and $\tau \in \R^d$ we have:
\begin{align*}
&D^q (\tau\cdot f)=\tau \cdot D^q f, & &(\tau\cdot f)[\ux] = f[\tau \cdot \ux] & &\text{and} & &K(\tau\cdot f,\ux) = \tau \cdot K(f,\tau\cdot \ux).
\end{align*}
\end{lem}

\begin{proof}
Let $T_\tau:y \mapsto y+\tau$ from $\R^d$ to itself, so that $\tau\cdot f = f \circ T_\tau$. Applying recursively the chain rule, we obtain $D^l\parentheses*{\tau \cdot f} = D^l\parentheses*{f \circ T_\tau} = (D^l f) \circ T_\tau=\tau\cdot D^l f$ for all $l \in \ssquarebrackets{0}{q}$, where $D^l f$ is seen as a map from $U$ to $\sym^l(\R^d,\R^k)$. For $l=q$ this is the first point.

Let $B \subset A$ be of cardinality $l+1$. Then, using the previous point,
\begin{equation*}
(\tau\cdot f)[\ux_B] = \int_{\ut \in \simplex_B} D^l_{\ut \cdot \ux_B}(\tau\cdot f) \dx \varsigma_B(\ut) = \int_{\ut \in \simplex_B} D^l_{\tau+\ut \cdot \ux_B}f \dx \varsigma_q(\ut)= f[\tau \cdot \ux_B].
\end{equation*}
For $B=A$, this yields the second point.

Let $B \subset A$ be non-empty, using the interpolation properties of the Kergin interpolant, we have:
\begin{equation*}
K(\tau\cdot f,\ux)[\ux_B] = (\tau\cdot f)[\ux_B]=f[\tau\cdot\ux_B]=K(f,\tau\cdot \ux)[\tau \cdot \ux_B]=\parentheses*{\strut \tau \cdot K(f,\tau\cdot\ux)}[\ux_B],
\end{equation*}
where the last equality is obtained by applying the previous point to $K(f,\tau \cdot \ux)$. Thus, $K(\tau\cdot f,\ux)$ and $\tau \cdot K(f,\tau\cdot\ux)$ are two polynomial maps of degree at most $q$ with the same divided differences on all $(\ux_B)_{\emptyset \neq B \subset A}$. The uniqueness part in Theorem~\ref{thm: Kergin interpolation} proves that $K(\tau\cdot f,\ux)=\tau \cdot K(f,\tau\cdot\ux)$.
\end{proof}

Let us now consider a centered Gaussian field $f \in \cC^q(\R^d,\R^k)$ and assume that $f$ is stationary. Let $\tau \in \R^d$, by definition we have $\tau\cdot f=f$ in distribution. Then, the second point in Lemma~\ref{lem: translation equivariance} shows that the map $\tau \cdot \parentheses*{f[\,\cdot\,]}:\ux \mapsto f[\tau \cdot \ux]$ is distributed as $f[\,\cdot\,]$. That is, the distribution of the field $f[\,\cdot\,]$ is invariant under diagonal translations. Note that $f[\,\cdot\,]$ is a priori not stationary, as we only control its behavior under translations by vectors of the kind $(\tau,\dots,\tau) \in (\R^d)^A$.

The Kergin interpolant does not behave as nicely. If the distribution of $K(f,\cdot)$ was invariant by diagonal translations, the stationarity of $f$ and the third point in Lemma~\ref{lem: translation equivariance} would imply that
\begin{equation*}
K(f,\ux)=K(\tau \cdot f,\ux)=\tau\cdot K(f,\tau\cdot \ux)=\tau\cdot K(f,\ux)
\end{equation*}
in distribution, for any $\tau \in \R^d$ and $\ux \in (\R^d)^A$. This does not hold in general. For example, consider the case $d=k=q=1$ and $\ux=(0,0)$ with $\var{f(0),f'(0)} = \Id$. Then $K(f,0,0)=f(0)+f'(0)X$ and $\tau\cdot K(f,0,0)=f(0)+\tau f'(0) + f'(0)X$. But $\var{f(0)+\tau f'(0)} = 1+\tau^2 \neq \var{f(0)}$ for any $\tau \in \R^*$, so that $\tau\cdot K(f,0,0)$ is not distributed as $K(f,0,0)$. More fundamentally, a non-degenerate Gaussian vector in $\R_q[X]$ cannot be invariant under the action of $\tau \in \R^d \setminus \brackets{0}$, because $\tau$ acts non-trivially on $\R_q[X]$ by an endomorphism whose only complex eigenvalue is $1$.

In order to recover a translation-invariant object, we introduce a twisted Kergin interpolant.

\begin{dfn}[Barycenter]
\label{def: barycenter}
Let $A$ be a non-empty finite set and $\ux=(x_a)_{a \in A} \in (\R^d)^A$, we denote by $\flat(\ux) = \frac{1}{\norm{A}}\sum_{a \in A}x_a \in \R^d$ the \emph{(iso-)barycenter} of $\ux$. Let $\oux=\parentheses*{x_a-\flat(\ux)}_{a \in A} \in \R^d$ denote the \emph{centered version} of $\ux$, so that $\flat(\oux)=0$ and $\ux = \flat(\ux) \cdot \oux$.
\end{dfn}

\begin{rem}
\label{rem: barycenter}
In general, $\oux_B:=(x_b-\flat(\ux))_{b \in B}$ is not equal to $\obullet{\overgroup{\ux_B}}:=(x_b-\flat(\ux_B))_{b \in B}$ for $B \subset A$.
\end{rem}

\begin{dfn}[Twisted Kergin interpolant]
\label{def: twisted Kergin interpolant}
Let $q \in \N$ and $A$ be a set of cardinality $q+1$. For any $f \in \cC^q(U,\R^k)$ and $\ux \in U^A$, we define $\oK(f,\ux) = K(\flat(\ux)\cdot f,\oux)=\flat(\ux)\cdot K(f,\ux)$, see~Lemma~\ref{lem: translation equivariance}.
\end{dfn}

\begin{ex}
\label{ex: twisted Kergin polynomial}
Let $y \in \R^d$, then $\flat(y,y)=y$ and $\obullet{(y,y)}=0$. Hence $K(f,y,y)=f(y)+f'(y)(X-y)$ is the degree~$1$ Taylor polynomial of $f$ at $y$, while $\oK(f,y,y) = K(y \cdot f,0)=f(y)+f'(y)X$ is the Taylor polynomial of $y\cdot f:z \mapsto f(z+y)$ at $z=0$.
\end{ex}

\begin{lem}[Translation equivariance of twisted Kergin interpolants]
\label{lem: translation equivariance twisted interpolants}
For any $f \in \cC^q(U,\R^k)$, $\ux \in U^A$ and $\tau \in \R^d$ we have $\oK\parentheses*{\tau \cdot f,\ux} = \oK\parentheses*{f,\tau \cdot \ux}$.
\end{lem}

\begin{proof}
We compute, using the fact that $\flat(\tau\cdot \ux)=\flat(\ux)+\tau$ and $\obullet{\overgroup{\tau \cdot \ux}}=\oux$,
\begin{equation*}
\oK\parentheses*{\tau \cdot f,\ux} =K\parentheses*{\strut \flat(\ux)\cdot (\tau \cdot f),\oux} = K\parentheses*{\strut \parentheses*{\flat(\ux)+\tau}\cdot f,\oux}=K\parentheses*{\strut \flat(\tau\cdot \ux)\cdot f, \obullet{\overgroup{\tau \cdot \ux}}}= \oK\parentheses*{f,\tau \cdot \ux}.\qedhere
\end{equation*}
\end{proof}

Let us now introduce some auxiliary operators that will be useful to relate the divided differences of $f$ to $K(f,\cdot)$ and $\oK(f,\cdot)$.

\begin{dfn}[Divided differences operators]
\label{def: isomorphisms Delta Psi}
Let $\ux=(x_0,\dots,x_q) \in U^{q+1}$, we denote by $\Delta_{\ux}: \cC^q(U,\R^k) \to \prod_{i=0}^q \sym^i(\R^d,\R^k)$ and $\Psi_{\ux}:\prod_{i=0}^q \sym^i(\R^d,\R^k) \to \R_q[X]^k$ the linear maps:
\begin{align*}
\Delta_{\ux}:f & \longmapsto \parentheses*{f[x_0,\dots,x_i]}_{0 \leq i \leq q} & &\text{and} & \Psi_{\ux}:\parentheses*{S_i}_{0 \leq i \leq q} & \longmapsto \sum_{i=0}^q S_i(X-x_0,\dots,X-x_{i-1}).
\end{align*}
\end{dfn}

As a consequence of Equation~\eqref{eq: Kergin explicit}, for all $\ux \in U^{q+1}$ we have $K(\cdot,\ux) = \Psi_{\ux} \circ \Delta_{\ux}$. Moreover, Theorem~\ref{thm: Kergin interpolation} shows that the restriction of $\Delta_{\ux}$ to $\R_q[X]^k$ is an isomorphism, whose inverse is $\Psi_{\ux}$. We can also express $\oK(\cdot,\ux)$ in terms of the previous maps.

\begin{lem}[Factorization of $\oK$]
\label{lem: oK in terms of Delta and Psi}
For all $\ux=(x_0,\dots,x_q) \in U^{q+1}$ we have $\oK(\cdot,\ux)=\Psi_{\oux}\circ \Delta_{\ux}$.
\end{lem}

\begin{proof}
Let $f \in \cC^q(U,\R^k)$, by Definitions~\ref{def: twisted Kergin interpolant} and~\ref{def: isomorphisms Delta Psi} and Equation~\eqref{eq: Kergin explicit} we have:
\begin{equation*}
\oK(f,\ux) = K(\flat(\ux)\cdot f,\oux) = \Psi_{\oux} \circ \Delta_{\oux}\parentheses*{\flat(\ux) \cdot f}.
\end{equation*}
For all $i \in \ssquarebrackets{0}{q}$, we have $(\flat(\ux)\cdot f)[x_0-\flat(\ux),\dots,x_i-\flat(\ux)] = f[x_0,\dots,x_i]$ by the second point in Lemma~\ref{lem: translation equivariance}. This implies that $\Delta_{\oux}(\flat(\ux)\cdot f)=\Delta_{\ux}(f)$. Hence, $\oK(f,\ux) = \Psi_{\oux} \circ \Delta_{\ux}(f)$.
\end{proof}

\begin{lem}[Continuity of $\Psi$]
\label{lem: continuity Psi}
The map $\ux \mapsto \Psi_{\ux}$ is continuous from $(\R^d)^{q+1}$ to the space $\cL\parentheses*{\prod_{i=0}^q \sym^i(\R^d,\R^k),\R_q[X]^k}$, endowed with its unique normed topology.
\end{lem}

\begin{proof}
The spaces under consideration being finite-dimensional, we can use any norms we fancy. On $\prod_{i=0}^q \sym^i(\R^d,\R^k)$, we work with the sup-norm associated with the operator norm on each factor. On $\R_q[X]^k$, we use the norm defined by $\Norm{P}=\max_{\Norm{z}\leq 1} \Norm{P(z)}$ for all $P \in \R_q[X]^k$. Finally, we use the associated operator norm on $\cL\parentheses*{\prod_{i=0}^q \sym^i(\R^d,\R^k),\R_q[X]^k}$.

Let $\varrho >0$ and let $\ux=(x_0,\dots,x_q)$ and $\uy=(y_0,\dots,y_q) \in (\R^d)^{q+1}$ be such that $\Norm{\ux}< \varrho$ and $\Norm{\uy}< \varrho$. Let $i \in \ssquarebrackets{1}{q}$ and $S_i \in \sym^i(\R^d,\R^k)$, for all $z \in \R^d$ such that $\Norm{z}\leq 1$ we~have:
\begin{multline*}
\Norm{S_i(z-y_0,\dots,z-y_{i-1})-S_i(z-x_0,\dots,z-x_{i-1})}\\
\begin{aligned}
\leq & \Norm{S_i(z-y_0,z-y_1,\dots,z-y_{i-1})-S_i(z-x_0,z-y_1,\dots,z-y_{i-1})} + \dots&\\
&+ \Norm{S_i(z-x_0,\dots,z-x_{i-2},z-y_{i-1})-S_i(z-x_0,\dots,z-x_{i-2},z-x_{i-1})}\\
\leq & \Norm{S_i}\sum_{0\leq j <i}\parentheses*{\Norm{y_j-x_j}\prod_{0 \leq k <j}\Norm{z-x_k}\prod_{j< k<i}\Norm{z-y_k}} \leq \Norm{S_i}(1+\varrho)^{i-1}  \sum_{0\leq j<i}\Norm{y_j-x_j}.
\end{aligned}
\end{multline*}
Hence $\Norm*{\strut S_i(X-y_0,\dots,X-y_{i-1})-S_i(X-x_0,\dots,X-x_{i-1})} \leq \Norm{S_i} q(1+\varrho)^q\Norm*{\uy-\ux}$.

Then, for all $S=(S_i)_{0 \leq i \leq q}$ such that $\Norm{S_i}\leq 1$ for all $i \in \ssquarebrackets{0}{q}$, we have:
\begin{align*}
\Norm*{\Psi_{\uy}(S)-\Psi_{\ux}(S)} &\leq \sum_{i=1}^q \Norm*{\strut S_i(X-y_0,\dots,X-y_{i-1})-S_i(X-x_0,\dots,X-x_{i-1})}\\
&\leq q^2(1+\varrho)^q\Norm*{\uy-\ux}.
\end{align*}
Thus $\Norm{\Psi_{\uy}-\Psi_{\ux}}\leq q^2(1+\varrho)^q\Norm*{\uy-\ux}$. This proves the continuity of $\ux \mapsto \Psi_{\ux}$ on the ball of center~$0$ and radius~$\varrho$ in $(\R^d)^{q+1}$. Hence the result, since $\varrho$ is arbitrary.
\end{proof}

\begin{lem}[Regularity of the twisted Kergin interpolant]
\label{lem: regularity Kergin interpolant}
Let $A$ be a set of cardinality $q+1$.
\begin{enumerate}
\item \label{item: Kergin continuity} The map $\oK:\cC^q(U,\R^k) \times U^A \to \R_q[X]^k$ is continuous.
\item \label{item: Kergin subspace} Let $V$ be a finite-dimensional subspace of $\cC^q(U,\R^k)$, then the map $\ux \mapsto \oK(\cdot,\ux)_{\vert V}$ is continuous from $U^A$ to $\cL\parentheses*{V,\R_q[X]^k}$ endowed with its unique normed topology.
\end{enumerate}
\end{lem}

\begin{proof}
Without loss of generality, we can assume that $A = \ssquarebrackets{0}{q}$. By Lemma~\ref{lem: regularity divided differences}, we know that $(f,\ux) \mapsto \Delta_{\ux}(f)=\parentheses*{f[x_0,\dots,x_i]}_{0 \leq i \leq q}$ is continuous from $\cC^q(U,\R^k) \times U^A$ to $\prod_{i=0}^q \sym^i(\R^d,\R^k)$. Since $\ux \mapsto \oux$ is continuous, the map $(f,\ux) \mapsto \parentheses*{\psi_{\oux},\Delta_{\ux}(f)}$ is continuous by Lemma~\ref{lem: continuity Psi}. Then, the map $\parentheses*{\ell,v}\mapsto \ell(v)$ is bilinear continuous from $\cL\parentheses*{\prod_{i=0}^q \sym^i(\R^d,\R^k),\R_q[X]^k} \times \prod_{i=0}^q \sym^i(\R^d,\R^k)$ to $\R_q[X]^k$. Hence, $\oK:(f,\ux) \mapsto \Psi_{\oux}\parentheses*{\Delta_{\ux}(f)}$ is continuous, which proves Item~\ref{item: Kergin continuity}.

Let $V \subset \cC^q(U,\R^k)$ be a finite-dimensional subspace. Let us equip $V$ with an arbitrary Euclidean inner product and let $(v_1,\dots,v_n)$ be an orthonormal basis of $V$. Then, for all $\ux \in U^A$, we have
\begin{equation*}
\oK(\cdot,\ux)_{\vert V} = \sum_{j=1}^n \oK(v_j,\ux) \prsc{v_j}{\cdot},
\end{equation*}
and this quantity depends continuously on $\ux$ by the previous point. This proves Item~\ref{item: Kergin subspace}.
\end{proof}

\begin{cor}[Twisted interpolant of a Gaussian field]
\label{cor: twisted interpolant Gaussian field}
Let $A$ be a non-empty finite set and $\cI \in \cP_A$. Let $f \in \cC^{\norm{A}-1}(U,\R^k)$ be a centered Gaussian field, then $\ux \mapsto \parentheses{\oK(f,\ux_I)}_{I \in \cI}$ is a continuous centered Gaussian field on $U^A$ with values in $\prod_{I \in \cI} \R_{\norm{I}-1}[X]^k$.
\end{cor}

\begin{proof}
The proof is similar to that of Lemma~\ref{lem: divided differences Gaussian field}, replacing the use of Lemma~\ref{lem: regularity divided differences} by that of Lemma~\ref{lem: regularity Kergin interpolant}.\ref{item: Kergin continuity}.
\end{proof}

For all $q \in \N$, we consider the Euclidean inner product on $\R_q[X]$ making the basis $(X^\alpha)_{\norm{\alpha}\leq q}$ orthonormal. This induces an inner product on $\R_q[X]^k$ such that the factors are orthogonal. Similarly, if $\cI$ is a partition of some finite set $A$, we endow $\prod_{I \in \cI} \R_{\norm{I}-1}[X]^k$ with the inner product that coincides with the previous one on each factor and makes these factors orthogonal. We can now define the following.

\begin{dfn}[Variance of twisted Kergin interpolants]
\label{def: Sigma I}
Let $A$ be a non-empty finite set and $\cI \in \cP_A$. Let $f \in \cC^{\norm{A}-1}(U,\R^k)$ be a centered Gaussian field. For all $\ux \in U^A$, we denote by:
\begin{equation*}
\Sigma_{\cI}(f,\ux) = \var{\parentheses*{\oK(f,\ux_I)}_{I \in \cI}} \in \sym\parentheses*{\prod_{I \in \cI} \R_{\norm{I}-1}[X]^k}.
\end{equation*}
For all $I$ and $J \in \cI$, we also denote by
\begin{equation*}
\Sigma_I^J(f,\ux) = \cov{\oK(f,\ux_I)}{\oK(f,\ux_J)}:\R_{\norm{J}-1}[X]^k\to \R_{\norm{I}-1}[X]^k,
\end{equation*}
so that $\Sigma_\cI(f,\ux)$ can be written as a block-operator $\Sigma_\cI(f,\ux)=\parentheses*{\strut \Sigma_I^J(f,\ux)}_{I,J \in \cI}$ on $\prod_{I \in \cI} \R_{\norm{I}-1}[X]^k$.
\end{dfn}

\begin{rem}
\label{rem: Sigma I C0}
It follows from Corollary~\ref{cor: twisted interpolant Gaussian field} that $\ux \mapsto \Sigma_\cI(f,\ux)$ is continuous on $U^A$. Moreover, by Lemma~\ref{lem: translation equivariance twisted interpolants}, if $U=\R^d$ and $f$ is stationary then, for all $\tau \in \R^d$ and $\ux \in (\R^d)^A$,
\begin{equation*}
\Sigma_\cI(f,\tau\cdot \ux) = \var{\parentheses*{\oK(\tau \cdot f,\ux_I)}_{I \in \cI}}= \var{\parentheses*{\oK(f,\ux_I)}_{I \in \cI}} = \Sigma_\cI(f,\ux).
\end{equation*}
\end{rem}

We conclude this section by introducing projection operators that will allow us to relate $\Sigma_\cI(f,\ux)$ to $\Sigma_\cJ(f,\ux)$ with $\cJ \leq \cI$. Recall that we denoted by $\cL(V,W)$ the space of continuous linear maps from $V$ to~$W$. In the following, we also denote by $\cL^\dagger(V,W) = \brackets*{\strut \Lambda \in \cL(V,W) \mvert \Lambda \ \text{is surjective}}$.

\begin{dfn}[Projection operators]
\label{def: projection operators}
Let $A$ be a finite set and $\emptyset \neq B \subset A$, we define
\begin{align*}
\Pi_B^A : \R_{\norm{A}-1}[X]^k \times (\R^d)^A &\longmapsto \qquad \R_{\norm{B}-1}[X]^k\\
(P,\ux) \qquad &\longmapsto \oK\parentheses*{\strut (-\flat(\ux))\cdot P,\ux_B}.
\end{align*}
For all $\cI \in \cP_A$, we also define $\Pi_\cI=\parentheses*{\Pi_I^A}_{I \in \cI}$ from $\R_{\norm{A}-1}[X]^k \times (\R^d)^A $ to $\prod_{I \in \cI} \R_{\norm{I}-1}[X]^k$.
\end{dfn}

\begin{lem}[Properties of $\Pi_B^A$]
\label{lem: prop Pi BA}
The projection operator $\Pi_B^A$ satisfies the following properties.
\begin{enumerate}
\item \label{item: Pi and Kergin} If $U$ is an open convex, $f \in \cC^{\norm{A}-1}(U,\R^k)$ and $\ux \in U^A$, then $\Pi_B^A\parentheses*{\oK(f,\ux),\ux} = \oK(f,\ux_B)$.

\item \label{item: Pi linear surj} For all $\ux \in (\R^d)^A$, the map $\Pi_B^A(\cdot,\ux)$ is linear surjective from $\R_{\norm{A}-1}[X]^k$ to $\R_{\norm{B}-1}[X]^k$.

\item \label{item: Pi continuity} The map $\ux \mapsto \Pi_B^A(\cdot,\ux)$ is continuous from $(\R^d)^A$ to $\cL^\dagger\parentheses*{\R_{\norm{A}-1}[X]^k,\R_{\norm{B}-1}[X]^k}$.

\item \label{item: Pi translation inv} For all $\ux \in (\R^d)^A$ and $\tau \in \R^d$, we have $\Pi_B^A(\cdot,\tau\cdot \ux) = \Pi_B^A(\cdot,\ux)$.
\end{enumerate}
\end{lem}

\begin{proof}
Let $U \subset \R^d$ be open and convex. Let $f \in \cC^{\norm{A}-1}(U,\R^k)$ and $\ux \in U^A$, by Definition~\ref{def: twisted Kergin interpolant} we have $\oK(f,\ux) = \flat(\ux) \cdot K(f,\ux)$. Hence $(-\flat(\ux))\cdot \oK(f,\ux) = K(f,\ux)$ and
\begin{equation*}
\Pi_B^A(\oK(f,\ux),\ux) = \oK\parentheses*{\strut K(f,\ux),\ux_B} = \flat(\ux_B)\cdot K\parentheses*{\strut K(f,\ux),\ux_B}.
\end{equation*}
By Theorem~\ref{thm: Kergin interpolation}, for all non-empty $C \subset B$, we have $K\parentheses*{\strut K(f,\ux),\ux_B}[\ux_C]=K(f,\ux)[\ux_C] =f[\ux_C]$, and the uniqueness of the Kergin interpolant yields $K\parentheses*{\strut K(f,\ux),\ux_B} = K(f,\ux_B)$. This implies that $\Pi_B^A\parentheses*{\strut \oK(f,\ux),\ux_B} = \flat(\ux_B)\cdot K(f,\ux_B)=\oK(f,\ux_B)$ and proves Item~\ref{item: Pi and Kergin}.

Let $\ux \in (\R^d)^A$, we denote by $T_{\flat(\ux)}$ the isomorphism of $\R_{\norm{A}-1}[X]^k$ defined as $T_{\flat(\ux)}:P \mapsto \flat(\ux) \cdot P$. Similarly, we denote by $T_{\flat(\ux_B)}:P \mapsto \flat(\ux_B)\cdot P$ from $\R_{\norm{B}-1}[X]^k$ to itself. For all $P \in \R_{\norm{A}-1}[X]^k$,
\begin{equation}
\label{eq: Pi AB surjective}
\Pi_B^A(P,\ux) = \oK\parentheses*{\strut (-\flat(\ux))\cdot P,\ux_B} = \flat(\ux_B) \cdot K\parentheses*{\strut (-\flat(\ux)\cdot P,\ux_B} = T_{\flat(\ux_B)}\circ K(\cdot,\ux_B)\circ T_{\flat(\ux)}^{-1}(P).
\end{equation}
That is, $\Pi_B^A(\cdot,\ux) = T_{\flat(\ux_B)}\circ K(\cdot,\ux_B)\circ T_{\flat(\ux)}^{-1}$. The linear map $K(\cdot,\ux_B)$ is surjective from $\R_{\norm{A}-1}[X]^k$ to $\R_{\norm{B}-1}[X]^k$. Indeed, its restriction to $\R_{\norm{B}-1}[X]^k$ is the identity. Since $T_{\flat(\ux)}$ and $T_{\flat(\ux_B)}$ are isomorphisms, this proves Item~\ref{item: Pi linear surj}.

Let us now focus on Item~\ref{item: Pi continuity}. Without loss of generality, we can assume that $A =\ssquarebrackets{0}{q}$. Recall that, for any $\uy \in (\R^d)^{q+1}$ we have $\Psi_{\uy}\circ \Delta_{\uy} = \Id$ on $\R_q[X]^k$, see Definition~\ref{def: isomorphisms Delta Psi} and Theorem~\ref{thm: Kergin interpolation}. Let $\tau \in \R^d$, applying the previous equality with $\uy = \tau \cdot 0$ we obtain for all $P \in \R_q[X]^k$:
\begin{equation*}
(-\tau)\cdot P = \Psi_{\tau\cdot 0}\circ \Delta_{\tau \cdot 0}\parentheses*{(-\tau)\cdot P} = \Psi_{\tau \cdot 0}(\Delta_0(P))
\end{equation*}
by the second point in Lemma~\ref{lem: translation equivariance}. Then, for all $\ux \in (\R^d)^{q+1}$ and $P \in \R_q[X]^k$, we have
\begin{equation*}
\Pi_B^A(P,\ux) = \oK\parentheses*{\strut (-\flat(\ux))\cdot P,\ux_B}= \oK\parentheses*{\Psi_{\flat(\ux) \cdot 0}(\Delta_0(P)),\ux_B}
\end{equation*}
by applying the previous relation with $\tau = \flat(\ux)$. This shows that $\Pi_B^A(\cdot,\ux) = \oK(\cdot,\ux_B) \circ \Psi_{\flat(\ux)\cdot 0}\circ \Delta_0$. By Lemma~\ref{lem: continuity Psi} and~\ref{lem: regularity Kergin interpolant}.\ref{item: Kergin subspace}, the map $\ux \mapsto \parentheses*{\oK(\cdot,\ux),\Psi_{\flat(\ux)\cdot 0},\Delta_0}$ is continuous from $(\R^d)^{q+1}$ to
\begin{equation*}
\cL\parentheses*{\strut \R_q[X]^k,\R_{\norm{B}-1}[X]^k}\times \cL\parentheses*{\prod_{i=0}^q \sym^i(\R^d,\R^k),\R_q[X]^k} \times \cL\parentheses*{\R_q[X]^k,\prod_{i=0}^q \sym^i(\R^d,\R^k)}.
\end{equation*}
Since $(\Lambda_1,\Lambda_2,\Lambda_3) \mapsto \Lambda_1 \circ \Lambda_2\circ\Lambda_3$ is continuous on this space, we finally get that $\ux \mapsto \Pi_B^A(\cdot,\ux)$ is continuous on $(\R^d)^A$.

Finally, let $P \in \R^{\norm{A}-1}[X]^k$, $\ux \in (\R^d)^A$ and $\tau \in \R^d$, by Lemma~\ref{lem: translation equivariance twisted interpolants} we have:
\begin{align*}
\Pi_B^A(P,\tau\cdot \ux) &= \oK\parentheses*{\strut (-\tau -\flat(\ux))\cdot P,(\tau\cdot \ux)_B} = \oK\parentheses*{\strut (-\tau) \cdot \parentheses*{(-\flat(\ux))\cdot P},\tau \cdot \ux_B}\\
&= \oK\parentheses*{(-\flat(\ux))\cdot P,\ux_B} =\Pi_B^A(P,\ux).\qedhere
\end{align*}
\end{proof}

\begin{lem}[Surjectivity of $\Pi_\cI$]
\label{lem: surjectivity Pi I}
Let $A$ be a non-empty finite set and $\cI \in \cP_A$. For all $\ux \in (\R^d)^A$ such that $\cI \in \cQ_0(\ux)$ the map $\Pi_\cI(\cdot,\ux):\R_{\norm{A}-1}[X]^k \to \prod_{I \in \cI} \R_{\norm{I}-1}[X]^k$ is surjective.
\end{lem}

\begin{proof}
Using the same notation as in the proof of Lemma~\ref{lem: prop Pi BA}.\ref{item: Pi linear surj}, Equation~\eqref{eq: Pi AB surjective} shows that for all $I \in \cI$ we have $\Pi_I^A(\cdot,\ux) = T_{\flat(\ux_I)}\circ K(\cdot,\ux_I) \circ T_{\flat(\ux)}^{-1}$. Thus $\Pi_\cI(\cdot,\ux)$ is the composition of the following three maps: the map $T_{\flat(\ux)}^{-1}:P \mapsto (-\flat(\ux))\cdot P$, which is an isomorphism of $\R_{\norm{A}-1}[X]^k$; the map $P \mapsto \parentheses*{K(P,\ux_I)}_{I \in \cI}$, which is surjective by Lemma~\ref{lem: surjectivity blockwise Kergin} since $\cI \in \cQ_0(\ux)$; and finally $(P_I)_{I \in \cI} \mapsto \parentheses*{\flat(\ux_I)\cdot P_I}_{I \in \cI}$, which is an isomorphism of $\prod_{I \in \cI} \R_{\norm{I}-1}[X]^k$. Hence the result.
\end{proof}

%%%%%%%%%%%%%%%%%%%%%%%%%%%%%%%%%%%%%%%%%%%%%%%%%%%%%%%%%%%%%%%%%%%%%%%%%%%%%%%%%%%%%%%%%%%%%%%%%%%%%%%

\subsection{Uniform estimates for the variance operator \texorpdfstring{$\Sigma_\cI$}{}}
\label{subsec: uniform estimates for Sigma I}

This section is dedicated to finding conditions ensuring the non-degeneracy of~$\Sigma_\cI(f,\ux)$, see Definition~\ref{def: Sigma I}, uniformly with respect to $\ux \in\diag_{\cI,\eta}$, see Definition~\ref{def: diag I eta}. We start by proving non-degeneracy results for projection operators. Then, we use them to deduce non-degeneracy results for variance operators.

\begin{lem}[Bounds on centered tuples]
\label{lem: bound on oux I}
Let $A$ be a non-empty finite set, $\cI \in \cP_A$ and $\eta \geq 0$. Let $\ux \in \diag_{\cI,\eta}$, for all $I \in \cI$ and $i \in I$, we have $\Norm{x_i-\flat(\ux_I)}\leq \card(A)\eta$, and $\Norm{\obullet{\overgroup{\ux_I}}} \leq \card(A)^2 \eta$.
\end{lem}

\begin{proof}
Let $I \in \cI$ and $i \in I$. By Lemma~\ref{lem: clusters diameter}, for all $j \in I$, we have $\Norm{x_i-x_j}\leq \card(A)\eta$. Then, we have $\Norm*{x_i-\flat(\ux_I)}=\Norm*{x_i-\frac{1}{\norm{I}}\sum_{j \in I}x_j} \leq \frac{1}{\norm{I}}\sum_{j \in I}\Norm{x_i-x_j}\leq \card(A)\eta$, see Definition~\ref{def: barycenter}. Finally, $\Norm{\obullet{\overgroup{\ux_I}}}\leq \sum_{i \in I} \Norm{x_i-\flat(\ux)} \leq \card(A)^2\eta$.
\end{proof}

\begin{cor}[Compactness of the deepest centered stratum]
\label{cor: compactness diag A eta 0}
Let $A$ be a non-empty finite set and $\eta \geq 0$. Then, the set $\odiag := \brackets*{\ux \in \diag_{\brackets{A},\eta} \mvert \flat(\ux)=0}$ is compact.
\end{cor}

\begin{proof}
Since $\flat$ is continuous and $\diag_{\brackets{A},\eta}$ is closed by Lemma~\ref{lem: closedness deepest stratum}, the set $\odiag$ is closed in $(\R^d)^A$ as the intersection of two closed sets. Then, if $\ux \in \diag_{\brackets{A},\eta}$ is such that $\flat(\ux)=0$, Lemma~\ref{lem: bound on oux I} yields $\Norm*{\ux} = \Norm{\oux} \leq \card(A)^2 \eta$. Hence, $\odiag$ is also bounded, therefore compact.
\end{proof}

Given two vector spaces $V$ and $W$, we denoted by $\cL^\dagger(V,W)$ the space of continuous linear surjections from $V$ to $W$ and defined $\Pi_\cI$ in Definition~\ref{def: projection operators}. Recall that a set is said to be \emph{relatively compact} in a topological space if its closure is compact.

\begin{lem}[Uniform non-degeneracy of projectors]
\label{lem: uniform non-degeneracy Pi}
Let $A$ be a non-empty finite set and $\cI \in \cP_A$. Let $0 < \delta \leq \eta$, the set
\begin{equation*}
\brackets*{\Pi_\cI(\cdot,\ux)\mvert \ux \in \diag_{\cI,\delta}\cap \diag_{\brackets{A},\eta}}
\end{equation*}
is relatively compact in $\cL^\dagger\parentheses*{\R_{\norm{A}-1}[X]^k, \prod_{I \in \cI}\R_{\norm{I}-1}[X]^k}$.
\end{lem}

\begin{proof}
We proved in Corollary~\ref{cor: compactness diag A eta 0} that $\odiag$ is compact. Hence, the set $\brackets*{\Pi_\cI(\cdot,\ux) \mvert \ux \in \odiag}$ is compact by Lemma~\ref{lem: prop Pi BA}.\ref{item: Pi continuity}. By Lemma~\ref{lem: prop Pi BA}.\ref{item: Pi translation inv}, the map $\ux \mapsto \Pi_\cI(\cdot,\ux)$ is translation-invariant. Thus, we have
\begin{equation*}
\brackets*{\Pi_\cI(\cdot,\ux)\mvert \ux \in \diag_{\cI,\delta}\cap \diag_{\brackets{A},\eta}} \subset \brackets*{\Pi_\cI(\cdot,\ux) \mvert \ux \in \odiag}
\end{equation*}
and the left-hand side has compact closure in $\cL\parentheses*{\R_{\norm{A}-1}[X]^k, \prod_{I \in \cI}\R_{\norm{I}-1}[X]^k}$. To conclude, we need to prove that its closure is contained in $\cL^\dagger\parentheses*{\R_{\norm{A}-1}[X]^k, \prod_{I \in \cI}\R_{\norm{I}-1}[X]^k}$.

Let $(\ux^{(n)})_{n \in \N}$ be a sequence in $\diag_{\cI,\delta}\cap \diag_{\brackets{A},\eta}$ such that $\Pi_\cI(\cdot,\ux^{(n)})\xrightarrow[n \to +\infty]{}\Pi$. By Lemma~\ref{lem: prop Pi BA}.\ref{item: Pi translation inv}, we can assume that $\flat(\ux^{(n)})=0$ for all $n \in \N$. Then, by compactness of $\odiag$, we can assume that $\ux^{(n)}\xrightarrow[n \to +\infty]{}\ux$. By continuity, we have $\Pi = \Pi_\cI(\cdot,\ux)$, and we need to prove that it is surjective.

Let $I,J \in \cI$ be such that $I \neq J$. By contradiction, if $\dist\parentheses*{\conv(\ux_I),\conv(\ux_J)}<\delta$, there exist $\us \in \simplex_I$ and $\ut \in \simplex_J$ such that $\Norm*{\us \cdot \ux_I - \ut \cdot \ux_J}< \delta$. This would imply that $\Norm*{\us \cdot \ux_I^{(n)} - \ut \cdot \ux_J^{(n)}}< \delta$ for $n$ large enough and contradict $\ux^{(n)} \in \diag_{\cI,\delta}$, see Lemma~\ref{lem: clusters diameter}. Hence, for all $I \neq J$, we have $\dist\parentheses*{\conv(\ux_I),\conv(\ux_J)}\geq \delta>0$. Thus $\cI \in \cQ_0(\ux)$, and $\Pi_\cI(\cdot,\ux)$ is surjective by Lemma~\ref{lem: surjectivity Pi I}.
\end{proof}

\begin{dfn}[Block-diagonal operators]
\label{def: block diagonal operators}
Let $(\Lambda_a)_{a \in A} \in \prod_{a \in A}\cL(V_a,W_a)$ be a finite family of linear maps. We denote the associated block-diagonal operator from $\prod_{a \in A} V_a$ to $\prod_{a \in A}W_a$ by $\bigoplus_{a \in A} \Lambda_a:(v_a)_{a \in A} \mapsto \parentheses*{\Lambda_a(v_a)}_{a \in A}$.
\end{dfn}

\begin{cor}
\label{cor: uniform non-degeneracy Pi}
Let $\cI$ and $\cJ \in \cP_A$ be such that $\cI \leq \cJ$ and let $0 < \delta \leq \eta$. Then, the set
\begin{equation*}
\brackets*{\bigoplus_{J \in \cJ} \Pi_{\cI_J}(\cdot,\ux_J) \mvert \ux \in \diag_{\cI,\delta}\cap \diag_{\cJ,\eta}}
\end{equation*}
is relatively compact in $\displaystyle\cL^\dagger\parentheses*{\prod_{J \in \cJ}\R_{\norm{J}-1}[X]^k,\prod_{J \in \cJ}\prod_{I \in \cI_J}\R_{\norm{I}-1}[X]^k}$.
\end{cor}

\begin{proof}
By Lemma~\ref{lem: uniform non-degeneracy Pi}, for all $J \in \cJ$ there exists a compact set $\Gamma_J$ such that
\begin{equation*}
\brackets*{\Pi_{\cI_J}(\cdot,\ux_J) \mvert \ux_J \in \diag_{\cI_J,\delta}\cap\diag_{\brackets{J},\eta}} \subset \Gamma_J \subset \cL^\dagger\parentheses*{\R_{\norm{J}-1}[X]^k,\prod_{I \in \cI_J}\R_{\norm{I}-1}[X]^k}.
\end{equation*}
If $\parentheses*{\Lambda_J}_{J \in \cJ} \in \prod_{J \in \cJ} \Gamma_J$ then $\bigoplus_{J \in \cJ} \Lambda_J$ is a block-diagonal operator with surjective diagonal blocks, hence it is surjective. By continuity of the map $(\Lambda_J)_{J \in \cJ} \mapsto \bigoplus_{J\in \cJ} \Lambda_J$, the set
\begin{equation}
\label{eq: uniform non-degeneracy Pi}
\brackets*{\bigoplus_{J \in \cJ} \Lambda_J \mvert \forall J \in \cJ, \Lambda_J \in \Gamma_J} \subset \cL^\dagger\parentheses*{\prod_{J \in \cJ}\R_{\norm{J}-1}[X]^k,\prod_{J \in \cJ}\prod_{I \in \cI_J}\R_{\norm{I}-1}[X]^k}
\end{equation}
is compact. By Lemma~\ref{lem: projection diag I eta}, if $\ux \in \diag_{\cI,\delta}\cap \diag_{\cJ,\eta}$ then for all $J \in \cJ$ we have $\ux_J \in \diag_{\cI_J,\delta}\cap \diag_{\brackets{J},\eta}$ and $\Pi_{\cI_J}(\cdot,\ux_J) \in \Gamma_J$. Hence $\brackets*{\bigoplus_{J \in \cJ} \Pi_{\cI_J}(\cdot,\ux_J) \mvert \ux \in \diag_{\cI,\delta}\cap \diag_{\cJ,\eta}}$ is contained in the compact set on the left-hand side of Equation~\eqref{eq: uniform non-degeneracy Pi}, which concludes the proof.
\end{proof}

Let us now consider the blocks of the variance operators introduced in Definition~\ref{def: Sigma I}.

\begin{lem}[Distance between covariances of twisted interpolants]
\label{lem: distance between covariances of twisted interpolants}
Let $q \in \N$ and $\eta \geq 0$, there exists $C\geq 0$ with the following property: for any open convex $U \subset \R^d$, for  any centered Gaussian fields $f$ and $\tilde{f}\in \cC^q(U,\R^k)$, for any set $A$ of cardinality $q+1$ and $\cI \in \cP_A$, for any $I,J \in \cI$ and $\ux \in \diag_{\cI,\eta}$, we have:
\begin{equation*}
\Norm*{\Sigma_I^J(f,\ux)-\Sigma_I^J(\tilde{f},\ux)} \leq C \max_{\substack{\norm{\alpha}\leq q\\ \norm{\beta}\leq q}} \ \max_{\substack{w \in \conv(\ux_I)\\ z \in \conv(\ux_J)}} \ \Norm*{\strut \partial^{\alpha,\beta}r(w,z)-\partial^{\alpha,\beta}\tilde{r}(w,z)},
\end{equation*}
where $r$ and $\tilde{r}$ stand for the covariance kernels of $f$ and $\tilde{f}$ respectively.
\end{lem}

\begin{proof}
Let $f$ and $\tilde{f} \in \cC^q(U,\R^k)$ be centered Gaussian fields. Let $m,n \in \ssquarebrackets{0}{q}$ and let us consider $x_0,\dots,x_m,y_0,\dots,y_n \in U$. By Definition~\ref{def: isomorphisms Delta Psi}, we have
\begin{equation*}
\cov{\Delta_{\ux}(f)}{\Delta_{\uy}(f)} = \begin{pmatrix}
\cov{f[x_0,\dots,x_i]}{f[y_0,\dots,y_j]}
\end{pmatrix}_{0 \leq i \leq m, 0 \leq j \leq n}
\end{equation*}
as a block-operator, and similarly for $\tilde{f}$. Using Lemma~\ref{lem: distance between covariances of divided differences}, we obtain the following bound for the operator norm:
\begin{multline*}
\Norm*{\cov{\Delta_{\ux}(f)}{\Delta_{\uy}(f)}-\cov{\Delta_{\ux}(\tilde{f})}{\Delta_{\uy}(\tilde{f})}}\\
\begin{aligned}
&\leq \sum_{0 \leq i \leq m, 0\leq j \leq n} \Norm*{\cov{f[x_0,\dots,x_i]}{f[y_0,\dots,y_j]}-\cov{\tilde{f}[x_0,\dots,x_i]}{\tilde{f}[y_0,\dots,y_j]}}\\
&\leq (q+1)^{d+2} \max_{\substack{\norm{\alpha}\leq q\\ \norm{\beta}\leq q}}\ \max_{\substack{w \in \conv(\ux)\\ z \in \conv(\uy)}}\ \Norm*{\partial^{\alpha,\beta}r(w,z)-\partial^{\alpha,\beta}\tilde{r}(w,z)}.
\end{aligned}
\end{multline*}

By Lemma~\ref{lem: oK in terms of Delta and Psi}, we have $\cov{\oK(f,\ux)}{\oK(f,\uy)} = \Psi_{\oux} \circ \cov{\Delta_{\ux}(f)}{\Delta_{\uy}(f)} \circ \Psi_{\ouy}^*$, where $\Psi_{\ouy}^*$ stands for the adjoint operator of $\Psi_{\ouy}$. A similar formula holds for $\tilde{f}$ and, since we use the operator norm subordinated to the Euclidean ones, this yields:
\begin{multline*}
\Norm*{\cov{\oK(f,\ux)}{\oK(f,\uy)}-\cov{\oK(\tilde{f},\ux)}{\oK(\tilde{f},\uy)}}\\
\begin{aligned}
&\leq \Norm*{\Psi_{\oux}}\Norm*{\cov{\Delta_{\ux}(f)}{\Delta_{\uy}(f)}-\cov{\Delta_{\ux}(\tilde{f})}{\Delta_{\uy}(\tilde{f})}}\Norm*{\Psi_{\ouy}^*}\\
&\leq  (q+1)^{d+2}\Norm*{\Psi_{\oux}}\Norm*{\Psi_{\ouy}} \max_{\substack{\norm{\alpha}\leq q\\ \norm{\beta}\leq q}}\ \max_{\substack{w \in \conv(\ux)\\ z \in \conv(\uy)}}\ \Norm*{\partial^{\alpha,\beta}r(w,z)-\partial^{\alpha,\beta}\tilde{r}(w,z)}.
\end{aligned}
\end{multline*}
By Lemma~\ref{lem: continuity Psi}, the map $\ux \mapsto \Psi_{\ux}$ is continuous, hence bounded by some constant $C_m\geq 0$ on the closed ball of center $0$ and radius $(q+1)^2\eta$ in $(\R^d)^{m+1}$. Assuming that $\Norm{\oux} \leq (q+1)^2\eta$ and $\Norm{\ouy}\leq (q+1)^2\eta$, and denoting by $C = (q+1)^{d+2} \max_{0 \leq m \leq q} C_m^2$, we obtain:
\begin{equation*}
\Norm*{\cov{\oK(f,\ux)}{\oK(f,\uy)}-\cov{\oK(\tilde{f},\ux)}{\oK(\tilde{f},\uy)}} \leq C \max_{\substack{\norm{\alpha}\leq q\\ \norm{\beta}\leq q}}\ \max_{\substack{w \in \conv(\ux)\\ z \in \conv(\uy)}}\ \Norm*{\partial^{\alpha,\beta}(r-\tilde{r})(w,z)}.
\end{equation*}

Now, let $A$ be a set of cardinality $q+1$ and $\cI \in \cP_A$. Let $I,J \in \cI$ and $\ux \in \diag_{\cI,\eta}$. Since $\oK(f,\ux_I)$ is invariant under permutation of $I$ (see Remark~\ref{rem: Kergin symmetric} and Definition~\ref{def: twisted Kergin interpolant}), we can choose an arbitrary ordering of $I$ (resp.~$J$) and write $\ux_I$ as $(x_0,\dots,x_{\norm{I}})$ (resp.~$\ux_J$ as $(y_0,\dots,y_{\norm{J}})$). By Lemma~\ref{lem: bound on oux I}, since $\ux \in \diag_{\cI,\eta}$, we have $\Norm{\obullet{\overgroup{\ux_I}}}\leq (q+1)^2\eta$ and $\Norm{\obullet{\overgroup{\ux_J}}}\leq (q+1)^2\eta$, so that the previous bound holds with $\ux$ replaced by $\ux_I$ and $\uy$ replaced by $\ux_J$. Hence the result.
\end{proof}

Let us now focus on the case where $U=\R^d$ and $f \in \cC^q(\R^d,\R^k)$ is a stationary field. The following is one of the key results of this section.

\begin{prop}[Uniform non-degeneracy of $\Sigma_\cI$ for stationary fields]
\label{prop: uniform non degeneracy Sigma I}
Let $A$ be a finite set of cardinality $q+1$. Let $f\in \cC^q(\R^d,\R^k)$ be a stationary centered Gaussian field satisfying \hypND{q} and \hypDCL{q}{\infty}. Then, there exists $\eta_0 >0$ such that, for any $\eta \in (0,\eta_0]$ and $\cI \in \cP_A$ the set $\brackets*{\Sigma_\cI(f,\ux) \mvert \ux \in \diag_{\cI,\eta}}$ is relatively compact in $ \sym^+\parentheses*{\prod_{I \in \cI}\R_{\norm{I}-1}[X]^k}$.
\end{prop}

The remainder of this section is mostly dedicated to the proof of Proposition~\ref{prop: uniform non degeneracy Sigma I}.
Given a Euclidean space $V$, a subset $\cO$ is relatively compact in $\sym^+(V)$ if and only if it is bounded in $\sym(V)$ and there exists $C>0$ such that: $\forall \Lambda \in \cO$, $\det(\Lambda)\geq C$. The next lemma takes care of the boundedness of $\brackets*{\Sigma_\cI(f,\ux) \mvert \ux \in \diag_{\cI,\eta}}$. Obtaining a uniform lower bound on determinants is more involved, and is dealt with below.

\begin{lem}[Boundedness of $\Sigma_\cI$ for stationary fields]
\label{lem: boundedness Sigma I}
Let $A$ be a set of cardinality $q+1$ and $\eta \geq 0$, let $f \in \cC^q(\R^d,\R^k)$ be a centered stationary Gaussian field, for all $\cI \in \cP_A$ we have:
\begin{equation*}
\sup_{\ux \in \diag_{\cI,\eta}} \Norm*{\Sigma_\cI(f,\ux)} <+\infty.
\end{equation*}
\end{lem}

\begin{proof}
Given $\cI \in \cP_A$, it is enough to prove that $\Sigma_I^J(f,\ux)$ is bounded on $\diag_{\cI,\eta}$ for all $I$ and $J \in \cI$, see Definition~\ref{def: Sigma I}. Let us fix $I$ and $J \in \cI$. By Lemma~\ref{lem: distance between covariances of twisted interpolants} applied on $U=\R^d$ with $\tilde{f}=0$, there exists $C \geq 0$ such that, for all $\ux \in \diag_{\cI,\eta}$,
\begin{equation*}
\Norm*{\strut \Sigma_I^J(f,\ux)}\leq C \max_{\norm{\alpha}\leq q,\norm{\beta}\leq q} \ \max_{w \in \conv(\ux_I), z \in \conv(\ux_J)} \ \Norm*{\strut \partial^{\alpha,\beta}r(w,z)},
\end{equation*}
where $r:\R^d \times \R^d \to \sym(\R^k)$ is the covariance kernel of $f$. Let $\alpha$ and $\beta \in \N^d$ be of length at most~$q$ and let $w,z \in \R^d$. For all $u$ and $v \in \R^k$ such that $\Norm{u}=1=\Norm{v}$ we have:
\begin{align*}
\prsc*{u}{\partial^{\alpha,\beta}r(w,z)v}&=\esp{\prsc*{\partial^\alpha f(w)}{u}\prsc*{\partial^\beta f(z)}{v}} \leq \esp{\prsc*{\partial^\alpha f(w)}{u}^2}^\frac{1}{2}\esp{\prsc*{\partial^\beta f(z)}{v}^2}^\frac{1}{2}\\
&\leq \prsc*{u}{\partial^{\alpha,\alpha}r(w,w)u\strut}^\frac{1}{2}\prsc*{v}{\partial^{\beta,\beta}r(z,z)v\strut}^\frac{1}{2}\leq \Norm*{\partial^{\alpha,\alpha}r(0,0)}^\frac{1}{2}\Norm*{\partial^{\beta,\beta}r(0,0)}^\frac{1}{2},
\end{align*}
where the last inequality uses the stationarity of $f$. Taking the maximum over $u$ and $v$ of norm $1$ yields that $\Norm*{\partial^{\alpha,\beta}r(w,z)}\leq \max_{\norm{\gamma}\leq q}\Norm*{\partial^{\gamma,\gamma}r(0,0)}$. Hence,
\begin{equation*}
\sup_{\ux \in \diag_{\cI,\eta}}\Norm{\Sigma_I^J(f,\ux)} \leq C \max_{\norm{\gamma}\leq q}\Norm*{\partial^{\gamma,\gamma}r(0,0)}.\qedhere
\end{equation*}
\end{proof}

An important step on the way to proving Proposition~\ref{prop: uniform non degeneracy Sigma I} is the following.

\begin{lem}[Uniform non-degeneracy for well-separated clusters]
\label{lem: uniform ND eta delta}
Let $A$ be a finite set of cardinality $q+1$. Let $f\in \cC^q(\R^d,\R^k)$ be a stationary centered Gaussian field satisfying \hypND{q} and \hypDCL{q}{\infty}. Let $\cI \in \cP_A$ and $\delta>0$, there exists $\eta \in (0,\delta]$ such that $\brackets*{\Sigma_\cI(f,\ux) \mvert \ux \in \diag_{\cI,\eta}\cap \diag_{\cI,\delta}}$ is relatively compact in $\sym^+\parentheses*{\prod_{I \in \cI}\R_{\norm{I}-1}[X]^k}$.
\end{lem}

\begin{proof}
For all $\eta\in (0,\delta]$, we know by Lemma~\ref{lem: boundedness Sigma I} that $\brackets*{\Sigma_\cI(f,\ux) \mvert \ux \in \diag_{\cI,\eta}\cap \diag_{\cI,\delta}}$ is bounded in $\sym\parentheses*{\prod_{I \in \cI}\R_{\norm{I}-1}[X]^k}$. Thus, we only need to prove that there exists $\eta \in (0,\delta]$ and $C >0$ such that: $\forall \ux \in \diag_{\cI,\eta}\cap \diag_{\cI,\delta}$, $\det\parentheses*{\Sigma_\cI(f,\ux)}\geq C$. If it were not the case, there would exists a sequence $(\ux^{(n)})_{n \in \N^*}$ such that $\det\parentheses*{\Sigma_\cI(f,\ux^{(n)})}\xrightarrow[n\to +\infty ]{}0$ and $\ux^{(n)} \in \diag_{\cI,\frac{1}{n}}\cap \diag_{\cI,\delta}$ for all $n \in \N^*$. Let us assume that such a sequence exists and derive a contradiction. This is done in three steps. First, we replace $\parentheses*{\ux^{(n)}}_{n \in \N}$ by a nice subsequence. Second, we prove that $\Sigma_\cI(f,\ux^{(n)})$ converges toward an operator admitting a block-diagonal structure with respect to some partition $\cJ \geq \cI$. And third, we prove that the diagonal blocks of the limit are positive-definite.

\paragraph*{Step 1: Extracting a nice subsequence.}
We prove by induction on $\card(A)$ that there exists an increasing $\varphi:\N^* \to \N^*$ such that, for all $a,b \in A$ we have either $\Norm{x_a^{\varphi(n)}-x_b^{\varphi(n)}} \xrightarrow[n \to +\infty]{}+\infty$ or $\parentheses{x_a^{\varphi(n)}-x_b^{\varphi(n)}}_{n \in \N^*}$ converges in $\R^d$. If $\card(A)=1$ this is true.

If $\card(A) >1$, without loss of generality we can assume that $A =\ssquarebrackets{0}{q}$. If $\parentheses{x_0^{(n)}-x_1^{(n)}}_{n \in \N^*}$ is bounded, we can extract a subsequence that converges. Otherwise, we can extract a subsequence that diverges to infinity as $n \to +\infty$. Iterating this procedure, we obtain an increasing $\varphi_0$ such that, for all $i \in \ssquarebrackets{1}{q}$ we have either $\Norm{x_0^{\varphi_0(n)}-x_i^{\varphi_0(n)}} \xrightarrow[n \to +\infty]{}+\infty$ or $\parentheses{x_0^{\varphi_0(n)}-x_i^{\varphi_0(n)}}_{n \in \N^*}$ converges in $\R^d$.

Let us denote by $B=\brackets{i \in \ssquarebrackets{1}{q}\mid \Norm{x_0^{\varphi_0(n)}-x_i^{\varphi_0(n)}} \xrightarrow[n \to +\infty]{}+\infty}$. If $B = \emptyset$, it is enough to define $\varphi=\varphi_0$. If $B \neq \emptyset$, we have $\card(B) < \card(A)$ and the induction hypothesis yields the existence of $\varphi_1$ such that, for all $a,b \in B$, either $\Norm{x_a^{\varphi_0\circ\varphi_1(n)}-x_b^{\varphi_0\circ\varphi_1(n)}} \xrightarrow[n \to +\infty]{}+\infty$ or $\parentheses{x_a^{\varphi_0\circ\varphi_1(n)}-x_b^{\varphi_0\circ\varphi_1(n)}}_{n \in \N^*}$ converges. Then, $\varphi = \varphi_0 \circ \varphi_1$ is the subsequence we are looking for.

By Lemma~\ref{lem: monotonicity diag I eta}, for all $n \in \N^*$ such that $\frac{1}{n}\leq \delta$, we have $\ux^{\varphi(n)} \in \diag_{\cI,\frac{1}{\varphi(n)}}\cap\diag_{\cI,\delta} \subset \diag_{\cI,\frac{1}{n}}$, because $\varphi(n) \geq n$. Thus, up to considering a subsequence, we can assume that: for all $a$ and $b \in A$ we have either $\Norm{x_a^{(n)}-x_b^{(n)}} \xrightarrow[n \to +\infty]{}+\infty$ or $\parentheses{x_a^{(n)}-x_b^{(n)}}_{n \in \N^*}$ converges.

\paragraph*{Step 2: Block structure of the limit operator.}
Let $\cJ \in \cP_A$ be defined by the property that $[a]_\cJ = [b]_\cJ$ if and only if $\parentheses{x_a^{(n)}-x_b^{(n)}}_{n \in \N^*}$ converges. Let $I \in \cI$ and let $i,j \in I$. For all $n \in \N^*$, we have $\ux^{(n)} \in \diag_{\cI,\frac{1}{n}}$. By Lemma~\ref{lem: clusters diameter}, this implies that $\Norm{x_i^{(n)}-x_j^{(n)}}\leq \frac{1}{n}\card(A) \xrightarrow[n \to +\infty]{}0$, hence $[i]_\cJ=[j]_\cJ$. This shows that, for all $a,b \in A$, if $[a]_\cI = [b]_\cI$ then $[a]_\cJ = [b]_\cJ$, that is $\cI \leq \cJ$.

Let $J \in \cJ$ and $j \in J$. By definition of $\cJ$, the following quantity converges in $\R^d$ as $n \to +\infty$:
\begin{equation*}
x_j^{(n)} - \flat(\ux_J^{(n)}) = x_j^{(n)} - \frac{1}{\norm{J}}\sum_{i \in J} x_i^{(n)} = \frac{1}{\norm{J}}\sum_{i \in J}x_j^{(n)}-x_i^{(n)}.
\end{equation*}
We denote by $x_j$ the limit of this sequence, which defines a vector $\ux_J=(x_j)_{j \in J}$ such that $\obullet{\overgroup{\ux_J^{(n)}}}\xrightarrow[n \to +\infty]{}\ux_J$. For all $i,j \in J$, we have
\begin{equation*}
x_i-x_j = \lim_{n \to +\infty}\parentheses*{x_i^{(n)}-\flat(\ux_J^{(n)})} - \parentheses*{x_j^{(n)}-\flat(\ux_J^{(n)})} = \lim_{n \to +\infty}x_i^{(n)}-x_j^{(n)}.
\end{equation*}
If $x_i=x_j$ then, for all $n$ large enough we have $\Norm{x_i^{(n)}-x_j^{(n)}}<\delta$. Since $\ux^{(n)} \in \diag_{\cI,\delta}$, this implies that $[i]_\cI = [j]_\cI$ by Lemma~\ref{lem: clusters diameter}. Conversely, we have already seen that if $[i]_\cI = [j]_\cI$ then $x_i^{(n)}-x_j^{(n)}\xrightarrow[n \to +\infty]{}0$, so that $x_i=x_j$. Thus $x_i=x_j$ if and only if $[i]_\cI=[j]_\cI$, that is $\ux_J \in \diag_{\cI_J}$. By Definition~\ref{def: Sigma I}, for all $n \in \N^*$ we have:
\begin{equation*}
\begin{pmatrix}
\Sigma_I^{I'}(f,\ux^{(n)})
\end{pmatrix}_{I,I' \in \cI_J} = \var{\oK(f,\ux^{(n)}_I)}_{I \in \cI_J} = \Sigma_{\cI_J}\parentheses*{f,\ux_J^{(n)}}.
\end{equation*}
Since $f$ is stationary, we have $\Sigma_{\cI_J}\parentheses{f,\ux_J^{(n)}}=\Sigma_{\cI_J}\parentheses{f,\obullet{\overgroup{\ux_J^{(n)}}}}\xrightarrow[n \to +\infty]{}\Sigma_{\cI_J}(f,\ux_J)$ by Remark~\ref{rem: Sigma I C0}.

Let $M \geq 0$ be such that, for all $J \in \cJ$, $j \in J$ and $n \in \N^*$, we have $\Norm{x_j^{(n)}-\flat(\ux_J^{(n)})}\leq M$. Such a constant exists, since we consider a finite number of converging sequences. Let $J,J' \in \cJ$ be such that $J \neq J'$, Let $i \in J$ and $j \in J'$, we have:
\begin{equation*}
\Norm*{\parentheses*{x_i^{(n)}-x_j^{(n)}}-\parentheses*{\flat(\ux_J^{(n)})-\flat(\ux_{J'}^{(n)})}}\leq \Norm*{x_i^{(n)}-\flat(\ux_J^{(n)})}+ \Norm*{x_j^{(n)}-\flat(\ux_{J'}^{(n)})}\leq 2M.
\end{equation*}
By definition of $\cJ$ we have $\Norm{x_i^{(n)}-x_j^{(n)}}\xrightarrow[n \to +\infty]{}+\infty$, hence $\Norm*{\flat(\ux_J^{(n)})-\flat(\ux_{J'}^{(n)})}\xrightarrow[n \to +\infty]{}+\infty$.

Let $I$ and $I' \in \cI$ be such that $I \subset J$ and $I' \subset J'$. Lemma~\ref{lem: distance between covariances of twisted interpolants} applied with $\tilde{f}=0$ yields the existence of $C\geq 0$ such that for all $n \in \N$,
\begin{equation*}
\Norm*{\Sigma_I^{I'}(f,\ux^{(n)})}\leq C \max_{\substack{\norm{\alpha}\leq q\\ \norm{\beta}\leq q}}\ \max_{\substack{w \in \conv(\ux_I^{(n)})\\ z \in \conv(\ux_{I'}^{(n)})}} \Norm*{\partial^{\alpha,\beta}r(w,z)}.
\end{equation*}
Let $w \in \conv(\ux_I^{(n)})$ and $z \in \conv(\ux_{I'}^{(n)})$, there exists $\us \in \simplex_I$ and $\ut \in \simplex_{I'}$ such that $w=\us\cdot\ux_I^{(n)}$ and  $z=\ut\cdot\ux_{I'}^{(n)}$. Then, we have
\begin{align*}
\Norm{w-z} &= \Norm*{\sum_{i \in I} s_i\parentheses*{x_i^{(n)}-\flat(\ux_J^{(n)})}-\sum_{j \in I'} t_j\parentheses*{x_j^{(n)}-\flat(\ux_{J'}^{(n)})} +\flat(\ux_J^{(n)})-\flat(\ux_{J'}^{(n)})}\\
&\geq \Norm*{\flat(\ux_J^{(n)})-\flat(\ux_{J'}^{(n)})} - \Norm*{\sum_{i \in I} s_i\parentheses*{x_i^{(n)}-\flat(\ux_J^{(n)})}}-\Norm*{\sum_{j \in I'} t_j\parentheses*{x_j^{(n)}-\flat(\ux_{J'}^{(n)})}}\\
&\geq \Norm*{\flat(\ux_J^{(n)})-\flat(\ux_{J'}^{(n)})} -2M.
\end{align*}

Recall that we assumed the existence of an even bounded function $g:\R^d \to [0,+\infty)$ going to~$0$ at infinity such that~\hypDCL{q}{\infty} holds. We deduce from this and the previous inequalities that
\begin{equation*}
\Norm*{\Sigma_I^{I'}(f,\ux^{(n)})}\leq C \sup_{\substack{w \in \conv(\ux_I^{(n)})\\ z \in \conv(\ux_{I'}^{(n)})}} g(z-w)\leq C \sup \brackets*{g(y) \mvert \Norm{y}\geq \Norm*{\flat(\ux_J^{(n)})-\flat(\ux_{J'}^{(n)})} -2M}.
\end{equation*}
Since $\Norm*{\flat(\ux_J^{(n)})-\flat(\ux_{J'}^{(n)})}\xrightarrow[n \to+\infty]{}+\infty$, the right-hand side goes to $0$ as $n \to +\infty$. Hence, $\Sigma_I^{I'}(f,\ux^{(n)})\xrightarrow[n \to +\infty]{}0$ for any $I,I' \in \cI$ that are included in different blocks of $\cJ$.

Finally, we have proved that, as $n \to +\infty$, the operator $\Sigma_\cI(f,\ux^{(n)}) = \parentheses*{\Sigma_I^{I'}(f,\ux^{(n)})\strut}_{I,I' \in \cI}$ converges towards an operator which is block-diagonal with respect to $\cJ \geq \cI$, with diagonal blocks $\parentheses*{\strut \Sigma_{\cI_J}(f,\ux_J)}_{J \in \cJ}$. Moreover, we have checked that $\ux_J \in \diag_{\cI_J}$ for all $J \in \cJ$. By continuity,
\begin{equation*}
\prod_{J \in \cJ} \det\parentheses*{\Sigma_{\cI_J}(f,\ux_J)} = \lim_{n \to +\infty} \det\parentheses*{\Sigma_\cI(f,\ux^{(n)})}=0.
\end{equation*}
Thus, there exists $J \in \cJ$ and $\ux_J \in \diag_{\cI_J}$ such that $\det\parentheses*{\Sigma_{\cI_J}(f,\ux_J)}=0$.

\paragraph*{Step 3: Non-degeneracy of the limit.}
Let $J \in \cJ$ and $\ux_J \in \diag_{\cI_J}$ be as above. Recalling Definition~\ref{def: diagonal inclusions}, there exists $\uy =(y_I)_{I \in \cI_J} \in (\R^d)^{\cI_J}\setminus \diag$ such that $\ux_J = \iota_{\cI_J}(\uy)$. For all $I \in \cI_J$, we have $\ux_I=(y_I,\dots,y_I)$, hence $\oK(f,\ux_I) = \sum_{\norm{\alpha}<\norm{I}}\frac{\partial^\alpha f(y_I)}{\alpha !}X^\alpha$, see Example~\ref{ex: twisted Kergin polynomial}.

We assumed that~\hypND{q} holds. Since $\norm{J}\leq q+1$, we also have~\hypND{\norm{J}-1}. Applying this non-degeneracy hypothesis with $l = \norm{\cI_J}\leq \norm{J}$ and $m_1,\dots,m_l$ the cardinalities of the blocks of~$\cI_J$, we obtain that the Gaussian vector $\parentheses*{\partial^{\alpha_I}f(y_I)}_{I \in \cI_J, \norm{\alpha_I}< \norm{I}}$ is non-degenerate. Hence $\parentheses*{\oK(f,\ux_I)}_{I \in \cI_J}$ is non-degenerate, and its variance operator $\Sigma_{\cI_J}(f,\ux_J)$ is positive-definite. This contradicts $\det\parentheses*{\Sigma_{\cI_J}(f,\ux_J)}=0$ and concludes the proof.
\end{proof}

\begin{proof}[Proof of Proposition~\ref{prop: uniform non degeneracy Sigma I}]
We will prove, by a backward induction on $\cP_A$, that the following property is satisfied for all $\cI \in \cP_A$: there exists $\eta_\cI >0$ such that, for all $\eta \in (0,\eta_I]$, the set $\brackets*{\Sigma_\cI(f,\ux) \mvert \ux \in \diag_{\cI,\eta}}$ is relatively compact in $\sym^+\parentheses*{\prod_{I \in \cI}\R_{\norm{I}-1}[X]^k}$. This will yield the result by letting $\eta_0 = \min_{\cI \in \cP_A}\eta_\cI$.

Let $\cI \in \cP_A$, and let us assume that the previous property holds for all $\cJ > \cI$. Note that this condition is satisfied if $\cI = \brackets{A}$. Let us denote by $\delta = \min \parentheses*{\brackets*{\eta_\cJ \mvert \cJ > \cI}\cup \brackets{1}} \in (0,1]$. By Lemma~\ref{lem: uniform ND eta delta}, there exists $\eta_\cI \in (0,\delta]$ such that $\brackets*{\Sigma_\cI(f,\ux) \mvert \ux \in \diag_{\cI,\eta_\cI} \cap \diag_{\cI,\delta}}$ is relatively compact in $\sym^+\parentheses*{\prod_{I \in \cI}\R_{\norm{I}-1}[X]^k}$. Let $\eta \in (0,\eta_\cI]$, since $\eta \leq \eta_\cI \leq \delta$ we have:
\begin{equation*}
\brackets*{\Sigma_\cI(f,\ux) \mvert \ux \in \diag_{\cI,\eta}} = \bigcup_{\cJ \geq \cI} \brackets*{\Sigma_\cI(f,\ux) \mvert \ux \in \diag_{\cI,\eta} \cap \diag_{\cJ,\delta}},
\end{equation*}
by Lemma~\ref{lem: monotonicity diag I eta}. Thus, it is enough to check that each of the sets appearing on the right-hand side is relatively compact in $\sym^+\parentheses*{\prod_{I \in \cI}\R_{\norm{I}-1}[X]^k}$.

In the case $\cJ = \cI$, by Lemma~\ref{lem: monotonicity diag I eta} we have $\diag_{\cI,\eta}\cap\diag_{\cI,\delta} \subset \diag_{\cI,\eta_\cI}$, so that
\begin{equation*}
\brackets*{\Sigma_\cI(f,\ux) \mvert \ux \in \diag_{\cI,\eta} \cap \diag_{\cI,\delta}} \subset \brackets*{\Sigma_\cI(f,\ux) \mvert \ux \in \diag_{\cI,\eta_\cI} \cap \diag_{\cI,\delta}}.
\end{equation*}
The right-hand side is relatively compact in $\sym^+\parentheses*{\prod_{I \in \cI}\R_{\norm{I}-1}[X]^k}$ by definition of $\eta_\cI$. Hence, so is the left-hand side.

In the case $\cJ > \cI$, since $\delta \leq \eta_\cJ$, there exists a compact set $\Gamma_1$ such that
\begin{equation*}
\brackets*{\Sigma_\cJ(f,\ux) \mvert \ux \in \diag_{\cJ,\delta}} \subset \Gamma_1 \subset \sym^+\parentheses*{\prod_{J \in \cJ}\R_{\norm{J}-1}[X]^k},
\end{equation*}
by the induction hypothesis. Besides, by Corollary~\ref{cor: uniform non-degeneracy Pi}, there exists a compact set $\Gamma_2$ such that
\begin{equation*}
\brackets*{\bigoplus_{J \in \cJ} \Pi_{\cI_J}(\cdot,\ux_J) \mvert \ux \in \diag_{\cI,\eta}\cap \diag_{\cJ,\delta}} \subset \Gamma_2 \subset \cL^\dagger\parentheses*{\prod_{J \in \cJ}\R_{\norm{J}-1}[X]^k,\prod_{J \in \cJ}\prod_{I \in \cI_J}\R_{\norm{I}-1}[X]^k}.
\end{equation*}
Since $(\Sigma,\Pi)\mapsto \Pi \circ \Sigma \circ \Pi^*$ is continuous, $\Gamma = \brackets*{\Pi \circ \Sigma \circ \Pi^* \mvert \Sigma \in \Gamma_1, \Pi \in \Gamma_2}$ is a compact subset of $\sym^+\parentheses*{\prod_{I \in \cI}\R_{\norm{I}-1}[X]^k}$.

Let us now consider $\ux \in \diag_{\cI,\eta}\cap \diag_{\cJ,\delta}$. Let $J \in \cJ$ and $I \in \cI_J$, by Lemma~\ref{lem: prop Pi BA}.\ref{item: Pi and Kergin} we have $\oK(f,\ux_I) = \Pi_I^{\brackets{J}}\parentheses*{\oK(f,\ux_J),\ux_J}$, hence $\parentheses*{\oK(f,\ux_I)}_{I \in \cI_J} = \Pi_{\cI_J}\parentheses*{\oK(f,\ux_J),\ux_J}$ by Definition~\ref{def: projection operators}. Since $\cJ \geq \cI$ we have $\cI = \bigsqcup_{J \in \cJ} \cI_J$. Thus, recalling Definition~\ref{def: block diagonal operators}, we have:
\begin{equation*}
\parentheses*{\oK(f,\ux_I)}_{I \in \cI} = \parentheses*{\Pi_{\cI_J}\parentheses*{\oK(f,\ux_J),\ux_J}}_{J \in \cJ} = \parentheses*{\bigoplus_{J \in \cJ}\Pi_{\cI_J}(\cdot,\ux_J)}\parentheses*{\oK(f,\ux_J)}_{J \in \cJ}.
\end{equation*}
It implies the following for variance operators:
\begin{equation*}
\Sigma_I(f,\ux) = \parentheses*{\bigoplus_{J \in \cJ}\Pi_{\cI_J}(\cdot,\ux_J)} \circ \Sigma_\cJ(f,\ux) \circ \parentheses*{\bigoplus_{J \in \cJ}\Pi_{\cI_J}(\cdot,\ux_J)}^* \in \Gamma.
\end{equation*}
Hence $\brackets*{\Sigma_\cI(f,\ux) \mvert \ux \in \diag_{\cI,\eta} \cap \diag_{\cJ,\delta}}\subset \Gamma$, which proves its relative compactness.

We have proved that for all $\eta \in (0,\eta_I]$, the set $\brackets*{\Sigma_\cI(f,\ux) \mvert \ux \in \diag_{\cI,\eta}}$ is relatively compact in $\sym^+\parentheses*{\prod_{I \in \cI}\R_{\norm{I}-1}[X]^k}$. This concludes the induction step, and the proof.
\end{proof}

To conclude this section we prove that, if $f$ is the limit of a family of Gaussian fields, then we can extend the result of Proposition~\ref{prop: uniform non degeneracy Sigma I} into a uniform non-degeneracy result for the whole family. Recall that we consider a convex open set $U \subset \R^d$ and a non-empty open $\Omega \subset U$ such that $\dist(\Omega,\R^d\setminus U)>0$. Let $\cR$ be an unbounded subset of $(0,+\infty)$. For all $R \in \cR$, let $f_R:RU \to \R^k$ be a centered Gaussian field. In the following, $f:\R^d \to \R^k$ is still a centered stationary Gaussian field. We first prove the uniform convergence of $\Sigma_\cI(f_R,\cdot)$ towards $\Sigma_\cI(f,\cdot)$ under the suitable assumptions.

\begin{prop}[Uniform convergence of $\Sigma_\cI$]
\label{prop: uniform convergence Sigma I}
Let $A$ be a set of cardinality $q+1$. Let us assume that $f$ and $(f_R)_{R \in \cR}$ satisfy Hypotheses~\hypReg{q}, \hypScL{q} and~\hypDC{q}{\infty}. Then, for all $\cI \in \cP_A$ and $\eta >0$, we have:
\begin{equation*}
\sup_{\ux \in \diag_{\cI,\eta} \cap (R\Omega)^A} \Norm*{\Sigma_\cI(f_R,\ux)- \Sigma_\cI(f,\ux)} \xrightarrow[R \to +\infty]{}0
\end{equation*}
\end{prop}

\begin{proof}
Let $\cI \in \cP_A$ and $\eta >0$. Recalling Definition~\ref{def: Sigma I}, we have $\Sigma_\cI(f,\ux) = \parentheses*{\Sigma_I^J(f,\ux)}_{I,J \in \cI}$ by blocks, and similarly $\Sigma_\cI(f_R,\ux) = \parentheses*{\Sigma_I^J(f_R,\ux)}_{I,J \in \cI}$. The estimate we want to prove does not depend on the norm on $\sym\parentheses*{\prod_{I \in \cI} \R_{\norm{I}-1}[X]^k}$. For the sake of the argument, say we use the operator norm associated with the Euclidean structure we used to define $\Sigma_\cI$. Then, it is enough to prove that, for all $I,J \in \cI$,
\begin{equation*}
\sup_{\ux \in \diag_{\cI,\eta} \cap (R\Omega)^A} \Norm*{\Sigma_I^J(f_R,\ux)- \Sigma_I^J(f,\ux)} \xrightarrow[R \to +\infty]{}0.
\end{equation*}

Recall that $r$ (resp.~$r_R$) stands for the covariance kernel of $f$ (resp.~$f_R$). Let $I,J \in \cI$ and $\epsilon>0$. By Lemma~\ref{lem: distance between covariances of twisted interpolants}, there exists $C> 0$ such that, for all $R \in \cR$, for all $\ux \in \diag_{\cI,\eta} \cap (RU)^A$,
\begin{equation}
\label{eq: uniform ND distance between var}
\Norm*{\Sigma_I^J(f_R,\ux)-\Sigma_I^J(f,\ux)} \leq C \max_{\norm{\alpha}\leq q, \norm{\beta}\leq q} \ \max_{w \in \conv(\ux_I), z \in \conv(\ux_J)} \ \Norm*{\strut \partial^{\alpha,\beta}r(w,z)-\partial^{\alpha,\beta}r_R(w,z)}.
\end{equation}
Let $g$ be the function appearing in~\hypDC{q}{\infty}. Since $g$ goes to $0$ at infinity, there exists $T \geq \eta$ such that $\sup_{\Norm{y}\geq T} g(y) \leq \epsilon$. Applying our hypothesis~\hypScL{q} to the closed ball $\B \subset \R^d$ of center~$0$ and radius $\card(A)T$, there exists $R_0 \in \cR$ such that: for all $R \in \cR$ such that $R \geq R_0$,
\begin{equation}
\label{eq: uniform ND scaling limit}
\sup_{\norm{\alpha} \leq q, \norm{\beta}\leq q}\ \sup_{y \in R\Omega}\ \sup_{w,z \in \B}\ \Norm*{\strut \partial^{\alpha,\beta} r_R(y+w,y+z) - \partial^{\alpha,\beta}r(w,z)} \leq \epsilon.
\end{equation}

Let $R\in \cR$ be such that $R\geq R_0$ and $\ux \in \diag_{\cI,\eta}\cap(R\Omega)^A$. Since $T \geq \eta$, by Lemma~\ref{lem: monotonicity diag I eta} there exists a unique $\cJ \geq \cI$ such that $\ux \in \diag_{\cI,\eta}\cap \diag_{\cJ,T}$. Recalling Definition~\ref{def: block coarser partition}, if $[I]_\cJ \neq [J]_\cJ$, then $\dist\parentheses*{\conv(\ux_I),\conv(\ux_J)} > T$ by Lemma~\ref{lem: clusters diameter}. Then, using~\hypDC{q}{\infty}, for all $w \in \conv(\ux_I)$ and $z \in \conv(\ux_J)$ we have
\begin{equation*}
\max_{\norm{\alpha} \leq q, \norm{\beta} \leq q}\Norm*{\partial^{\alpha,\beta}r_R(w,z)} \leq g(z-w) \leq \sup_{\Norm{y}\geq T}g(y) \leq \epsilon.
\end{equation*}
The same estimates holds for $r$. Indeed, as explained in Section~\ref{subsec: discussion of the main hypotheses}, if \hypScL{q} and~\hypDC{q}{\infty} are satisfied, then \hypDCL{q}{\infty} holds with the same parameter $\omega$ and the same function $g$. Hence, $\Norm{\Sigma_I^J(f_R,\ux)-\Sigma_I^J(f,\ux)} \leq 2C\epsilon$ by~\eqref{eq: uniform ND distance between var}. If $[I]_\cJ=[J]_\cJ$, let $i \in [I]_\cJ$. For all $j \in [I]_\cJ$ we have $\Norm{x_i-x_j}\leq \card(A)T$, by Lemma~\ref{lem: clusters diameter} once again. Hence $\conv(\ux_I) \cup \conv(\ux_J) \subset x_i+\B$ and
\begin{equation*}
\max_{\substack{\norm{\alpha}\leq q\\ \norm{\beta}\leq q}} \ \max_{\substack{w \in \conv(\ux_I)\\ z \in \conv(\ux_J)}} \ \Norm*{\strut \partial^{\alpha,\beta}(r_R-r)(w,z)} \leq \max_{\norm{\alpha}\leq q, \norm{\beta}\leq q} \ \max_{w, z \in \B} \ \Norm*{\strut \partial^{\alpha,\beta}(r_R-r)(x_i+w,x_i+z)} \leq \epsilon,
\end{equation*}
where we used Equation~\eqref{eq: uniform ND scaling limit}, the stationarity of $f$ and the fact that $x_i \in R\Omega$. Finally, by Equation~\eqref{eq: uniform ND distance between var}, we obtain that $\Norm{\Sigma_I^J(f_R,\ux)-\Sigma_I^J(f,\ux)} \leq C\epsilon$ in this case.

Thus, for all $R \in \cR \cap [R_0,+\infty)$, we have $\sup_{\ux \in \diag_{\cI,\eta} \cap (R\Omega)^A} \Norm*{\Sigma_I^J(f_R,\ux)- \Sigma_I^J(f,\ux)} \leq 2C\epsilon$,  which concludes the proof.
\end{proof}

\begin{cor}[Uniform non-degeneracy for families of fields]
\label{cor: uniform non degeneracy f R}
Let $A$ be a set of cardinality $q+1$. Let us assume that $f$ and $(f_R)_{R \in \cR}$ satisfy Hypotheses~\hypReg{q}, \hypND{q}, \hypScL{q} and~\hypDC{q}{\infty}. Then, there exists $\eta_0>0$ such that, for all $\eta \in (0,\eta_0]$, there exists $R_\eta \in \cR$ such that for all $\cI \in \cP_A$ the set
\begin{equation*}
\brackets*{\Sigma_\cI(f,\ux) \mvert \ux \in \diag_{\cI,\eta}} \cup \brackets*{\Sigma_\cI(f_R,\ux) \mvert R \in \cR \cap [R_\eta,+\infty) \ \text{and} \  \ux \in \diag_{\cI,\eta} \cap (R\Omega)^A}
\end{equation*}
is relatively compact in $\sym^+\parentheses*{\prod_{I \in \cI}\R_{\norm{I}-1}[X]^k}$.
\end{cor}

\begin{proof}
As explained in Section~\ref{subsec: discussion of the main hypotheses}, if \hypScL{q} and~\hypDC{q}{\infty} are satisfied then the correlation kernel $r$ of $f$ satisfies~\hypDC{q}{\infty}. Since we also assume~\hypReg{q} and \hypND{q}, the field $f$ satisfies the hypotheses of Proposition~\ref{prop: uniform non degeneracy Sigma I}. Let $\eta_0>0$ be given by Proposition~\ref{prop: uniform non degeneracy Sigma I} and let $\eta\in(0,\eta_0]$. It is enough to prove that, for all $\cI \in \cP_A$, there exists $R_\cI \in \cR$ such that
\begin{equation*}
\brackets*{\Sigma_\cI(f,\ux) \mvert \ux \in \diag_{\cI,\eta}} \cup \brackets*{\Sigma_\cI(f_R,\ux) \mvert R \in \cR \cap [R_\cI,+\infty) \ \text{and} \  \ux \in \diag_{\cI,\eta} \cap (R\Omega)^A}
\end{equation*}
is relatively compact in $\sym^+\parentheses*{\prod_{I \in \cI}\R_{\norm{I}-1}[X]^k}$. Then the result follows with $R_\eta = \max_{\cI \in \cP_A}R_\cI$.

Let $\cI \in \cP_A$. Since $\eta \in (0,\eta_0]$, the closure $\Gamma$ of $\brackets*{\Sigma_\cI(f,\ux) \mvert \ux \in \diag_{\cI,\eta}}$ is a compact subset of the open cone $\sym^+\parentheses*{\prod_{I \in \cI}\R_{\norm{I}-1}[X]^k}$. Hence, there exists $\epsilon >0$ such that
\begin{equation*}
\Gamma_\epsilon = \brackets*{\Lambda \in \sym\parentheses*{\prod_{I \in \cI}\R_{\norm{I}-1}[X]^k} \mvert \exists \Sigma \in \Gamma, \Norm{\Lambda-\Sigma} \leq \epsilon} \subset \sym^+\parentheses*{\prod_{I \in \cI}\R_{\norm{I}-1}[X]^k}.
\end{equation*}
By Proposition~\ref{prop: uniform convergence Sigma I}, there exists $R_\cI \in \cR$ such that: for all $R \in \cR \cap [R_\cI,+\infty)$ and $\ux \in \diag_{\cI,\eta}\cap (R\Omega)^A$, we have $\Norm{\Sigma_\cI(f_R,\ux)-\Sigma_\cI(f,\ux)}\leq \epsilon$. Since $\Sigma_\cI(f,\ux) \in \Gamma$ we have $\Sigma_\cI(f_R,\ux) \in \Gamma_\epsilon$. Hence
\begin{equation*}
\brackets*{\Sigma_\cI(f,\ux) \mvert \ux \in \diag_{\cI,\eta}} \cup \brackets*{\Sigma_\cI(f_R,\ux) \mvert R \in \cR \cap [R_\cI,+\infty) \ \text{and} \  \ux \in \diag_{\cI,\eta} \cap (R\Omega)^A} \subset \Gamma_\epsilon.
\end{equation*}
This proves the claimed relative compactness, since $\Gamma_\epsilon$ is compact as the continuous image of the product of $\Gamma$ with a closed ball.
\end{proof}

%%%%%%%%%%%%%%%%%%%%%%%%%%%%%%%%%%%%%%%%%%%%%%%%%%%%%%%%%%%%%%%%%%%%%%%%%%%%%%%%%%%%%%%%%%%%%%%%%%%%%%%
%%%%%%%%%%%%%%%%%%%%%%%%%%%%%%%%%%%%%%%%%%%%%%%%%%%%%%%%%%%%%%%%%%%%%%%%%%%%%%%%%%%%%%%%%%%%%%%%%%%%%%%

\section{Kac--Rice formulas revisited}
\label{sec: Kac-Rice formulas revisited}

As suggested by the title, this section is concerned with Kac--Rice formulas, which are a classical tool to study zero set of Gaussian fields, see for instance~\cite{AT2007,AAL2025,AW2009,Ste2022}. In Section~\ref{subsec: Bulinskaya Lemma and Kac--Rice formulas for factorial moments}, we recall the standard form of these formulas, that express the so-called factorial moments of the linear statistics associated with the zeros of a nice enough Gaussian field. In Section~\ref{subsec: evaluation maps and their kernels}, we introduce various maps on spaces of polynomials. Then, in Section~\ref{subsec: Kac--Rice densities as functions of variance operators}, we use these maps and the Kergin interpolant from Section~\ref{sec: polynomial interpolation and Gaussian fields} to derive alternative expressions of the Kac--Rice densities, which allow for a better understanding of their singularities along the diagonal. Section~\ref{subsec: Kac--Rice densities for cumulants} is dedicated to proving similar formulas adapted to cumulants and studying the associated densities.

%%%%%%%%%%%%%%%%%%%%%%%%%%%%%%%%%%%%%%%%%%%%%%%%%%%%%%%%%%%%%%%%%%%%%%%%%%%%%%%%%%%%%%%%%%%%%%%%%%%%%%%

\subsection{Bulinskaya Lemma and Kac--Rice formulas for factorial moments}
\label{subsec: Bulinskaya Lemma and Kac--Rice formulas for factorial moments}

In this section, we recall some classical facts about random submanifolds defined as zero sets of Gaussian fields. One of the main tools to study these random submanifolds are the Kac--Rice formulas, see Theorem~\ref{thm: Kac-Rice} below. Let us first recall the following definition.

\begin{dfn}[Jacobian]
\label{def: Jacobian}
Let $V$ and $W$ be Euclidean spaces such that $\dim(V) \geq \dim(W)$. Given $\Lambda \in \cL(V,W)$, we denote its \emph{Jacobian} by $\jac{\Lambda}=\sqrt{\det(\Lambda \Lambda^*)}$, where $\Lambda^*:W \to V$ stands for the adjoint operator of~$\Lambda$.
\end{dfn}

\begin{rem}
\label{rem: Jacobian}
We have $\jac{\Lambda}>0$ if and only if $\Lambda$ is surjective.
\end{rem}

Let us now consider an open subset $U \subset \R^d$ and a centered Gaussian field $f:U \to \R^k$, where $k \in \ssquarebrackets{1}{d}$. Under mild regularity and non-degeneracy assumptions, its zero set $Z$ is almost-surely a nice submanifold of $U$.

\begin{prop}[Bulinskaya Lemma]
\label{prop: Bulinskaya}
Let $f:U\to \R^k$ be a $\cC^2$ centered Gaussian field such that, for all $x \in U$, the Gaussian vector $\parentheses*{f(x),D_xf}$ is non-degenerate. Then, almost-surely,
\begin{equation*}
\brackets*{x \in U \mvert f(x)=0 \ \text{and} \ \jac{D_xf}=0} = \emptyset.
\end{equation*}
In particular, $Z=f^{-1}(0)$ is almost-surely a submanifold without boundary of codimension $k$ in $U$.
\end{prop}

\begin{proof}
This result is a special case of~\cite[Prop.~6.12]{AW2009}.
\end{proof}

Recalling Section~\ref{subsec: product spaces and diagonals}, we denote by $\nu$ the Riemannian volume measure on $Z$ induced by the Euclidean metric on $U \subset \R^d$, that is its $(d-k)$-dimensional Hausdorff measure. For any Borel-measurable function  $\phi:U \to \R$, we denote by
\begin{equation}
\label{eq: def linear statistics}
\prsc{\nu}{\phi} = \int_Z \phi(x) \dx\nu(x),
\end{equation}
the \emph{linear statistics} of $\nu$ associated with $\phi$. By Proposition~\ref{prop: Bulinskaya}, the measure $\nu$ is almost-surely a Radon measure on $U$, so that $\prsc{\nu}{\phi}$ is well-defined if $\phi$ is continuous with compact support. Actually, a by-product of the Kac--Rice formula for the expectation is that $\prsc{\nu}{\phi}$ is an almost-surely well-defined random variable as soon as $\phi$ is integrable on $U$, see~\cite[Rem.~6.18]{AL2025}. Given a non-empty finite set $A$, recall that the power measure $\nu^A$ and the factorial power measure $\nu^{[A]}$ were defined in Definition~\ref{def: power measure and factorial power measure}. These are almost-surely Radon measures on $U^A$ and their linear statistics are defined similarly to~\eqref{eq: def linear statistics}.

\begin{dfn}[Kac--Rice density]
\label{def: rho A}
Let $A$ be a non-empty finite set and $f\in \cC^1(U,\R^k)$ be a centered Gaussian field such that, for all $(x_a)_{a \in A} \notin \diag$, the Gaussian vector $\parentheses*{f(x_a)}_{a \in A}$ is non-degenerate. We define the \emph{Kac--Rice density} $\rho_A(f,\cdot):U^A \setminus \diag \to [0,+\infty)$ by:
\begin{equation*}
\forall \ux = (x_a)_{a \in A} \in U^A \setminus \diag, \qquad \rho_A(f,\ux) = \frac{\espcond{\prod_{a \in A}\jac{D_{x_a}f}}{\forall a \in A, f(x_a)=0}}{\det\parentheses*{2\pi \var{\parentheses*{f(x_a)}_{a \in A}}}^\frac{1}{2}},
\end{equation*}
where the numerator is the conditional expectation of $\prod_{a \in A}\jac{D_{x_a}f}$ given that $\parentheses*{f(x_a)}_{a \in A}=0$.
\end{dfn}

\begin{thm}[Kac--Rice formula]
\label{thm: Kac-Rice}
Let $A$ be a non-empty finite set and $f\in \cC^2(U,\R^k)$ be a centered Gaussian field such that: $\parentheses*{f(x),D_xf}$ is non-degenerate for all $x \in U$, and $\parentheses*{f(x_a)}_{a \in A}$ is non-degenerate for all $\ux \notin \diag$. Let $\Phi:U^A \to \R$ be a Borel-measurable function which is either non-negative or such that $\Phi \rho_A(f,\cdot) \in L^1(U^A)$, then
\begin{equation*}
\esp{\prsc{\nu^{[A]}}{\Phi}} = \int_{U^A} \Phi(\ux) \rho_A(f,\ux)\dx \ux.
\end{equation*}
\end{thm}

\begin{proof}
The field $f$ satisfies the hypotheses of Proposition~\ref{prop: Bulinskaya}, so that $\nu^{[A]}$ is an almost-surely well-defined Radon measure. Moreover, $\rho_A(f,\cdot)$ is well-defined outside of $\diag$.

Let us first consider the case $k=d$, so that $Z$ is almost-surely discrete. Let $\cB \subset U$ be a Borel subset and $\Phi = \one_{\cB^A}$ be the indicator function of $\cB^A \subset U^A$. Then $\esp{\prsc{\nu^{[A]}}{\Phi}}$ is the factorial moment of order $\norm{A}$ of the linear statistic $\prsc{\nu}{\one_\cB} = \card \parentheses*{Z \cap \cB}$, as explained in~\cite[p.~58]{AW2009}. In this case, the result is given by \cite[Thm.~6.2]{AW2009} if $\norm{A}=1$ and~\cite[Thm.~6.3]{AW2009} if $\norm{A}\geq 2$. The proofs of these results are still valid if $\Phi$ is the indicator function of any Borel subset of $U^A$.

If $k <d$, recall that $Z$ has almost-surely dimension $d-k>0$, so that $\nu^{[A]}=\nu^A$ by Lemma~\ref{lem: power vs factorial power}. If $\Phi$ is of the form $\one_{\cB^A}$, the result is given by \cite[Thm.~6.8]{AW2009} if $\norm{A}=1$ and~\cite[Thm.~6.9]{AW2009} if $\norm{A}\geq 2$. Once again, the proofs are still valid if $\Phi$ is the indicator function of any Borel subset of $U^A$.

Now that the result is proved for indicator functions of Borel sets, we recover simple functions by linearity, then non-negative functions by monotone convergence. When $\Phi\rho_A(f,\cdot)$ is integrable, since $\rho_A(f,\cdot)$ is non-negative, the result follows by applying the previous case to the non-negative and non-positive parts of $\Phi$.
\end{proof}

\begin{rem}
\label{rem: Kac-Rice}
It was proved in~\cite[Thm.~6.26]{AL2025} and~\cite[Thm.~1.13]{GS2024} that, if we assume in addition that $f$ is of class $\cC^{\norm{A}}$ and satisfies~\hypND{\norm{A}-1} over $U$, then $\rho_A(f,\cdot)$ is locally integrable on~$U^A$. In this case, if $\Phi \in L^\infty(U^A)$ has compact support, then $\Phi \rho_A(f,\cdot) \in L^1(U^A)$ and we can apply Theorem~\ref{thm: Kac-Rice}. We will prove in Section~\ref{subsec: validity of Kac-Rice} that, under the hypotheses of Theorem~\ref{thm: cumulants asymptotics for zero sets}, if $(\phi_a)_{a \in A}$ are in $L^1(U)\cap L^\infty(U)$ then $\phi_A^\otimes\, \rho_A(f,\cdot) \in L^1(U^A)$, where $\phi_A^\otimes$ is as in Definition~\ref{def: set-indexed products}.
\end{rem}

%%%%%%%%%%%%%%%%%%%%%%%%%%%%%%%%%%%%%%%%%%%%%%%%%%%%%%%%%%%%%%%%%%%%%%%%%%%%%%%%%%%%%%%%%%%%%%%%%%%%%%%

\subsection{Evaluation maps and their kernels}
\label{subsec: evaluation maps and their kernels}

The Kac--Rice densities are infamously singular along the diagonal, and a major difficulty in our work is to understand this singularity. The goal of this section is to introduce some evaluations maps, and their kernels, that we use to derive alternative expressions of the Kac--Rice densities and deal with their singularities.

In the following, we consider a non-empty finite set $A$. Recall that if $\ux=\parentheses{x_a}_{a \in A} \in (\R^d)^A$, we defined $\flat(\ux)$ as the barycenter of $\ux$ and $\oux=\parentheses*{x_a-\flat(\ux)}_{a \in A}$ in Definition~\ref{def: barycenter}.

\begin{dfn}[Centered evaluation map]
\label{def: ev x}
For all $\ux \in (\R^d)^A$, let $\ev_{\ux}:\R_{2\norm{A}-1}[X]^k \to (\R^k)^A$ be the linear map defined by $\ev_{\ux}:P \mapsto \parentheses*{\strut P(x_a-\flat(\ux))}_{a \in A}$.
\end{dfn}

\begin{lem}[Regularity of the evaluation]
\label{lem: regularity of ev}
The map $\ux \mapsto \ev_{\ux}$ is smooth from $(\R^d)^A \setminus \diag$ to $\cL^\dagger\parentheses*{\R_{2\norm{A}-1}[X]^k,(\R^k)^A}$.
\end{lem}

\begin{proof}
If $\ux \notin \diag$, then $\oux \notin \diag$, so that $\ev_{\ux}$ is surjective by Corollary~\ref{cor: surjectivity 1-jets}. Then, we observe that the coefficients of the matrix of $\ev_{\ux}$ in the canonical bases of $\R_{2\norm{A}-1}[X]^k$ and $(\R^k)^A$ are monomials in~$\oux$, hence polynomials in $\ux$. Thus $\ux \mapsto \ev_{\ux}$ is smooth.
\end{proof}

Given a finite-dimensional vector space $V$ and $n \in \ssquarebrackets{0}{\dim(V)}$, recall that the \emph{Grassmannian} $\gr{n}{V}$ is the set of codimension $n$ subspaces of $V$. This set has a canonical manifold structure that we use in the following. For more details on this matter, see~\cite[Chap.~5]{MS1974} or~\cite[Ex.~1.36]{Lee2013}.

\begin{dfn}[Evaluation kernel]
\label{def: G kernel ev}
For all $\ux \in (\R^d)^A\setminus \diag$, we denote by $\cG_A(\ux)=\ker(\ev_{\ux})$.
\end{dfn}

\begin{lem}[Regularity of $\cG_A$]
\label{lem: regularity GA}
The map $\cG_A:(\R^d)^A \setminus \diag \to \gr{k\norm{A}}{\R_{2\norm{A}-1}[X]^k}$ is smooth.
\end{lem}

\begin{proof}
In this proof, we use the canonical bases of $\R_{2\norm{A}-1}[X]^k$ and $(\R^k)^A$, and the inner products making them orthonormal. For all $\ux \in (\R^d)^A \setminus \diag$, the lines of the matrix of $\ev_{\ux}$ form a basis of $\cG_A(\ux)^\perp$. This basis depends smoothly on $\ux$ by Lemma~\ref{lem: regularity of ev}. Locally, the Gram--Schmidt procedure yields an orthonormal basis of $\cG_A(\ux)$ depending smoothly on $\ux$. We recover $\cG_A(\ux)$ by considering the span of this basis, which is also a smooth operation, cf.~\cite[Lem.~5.1 and Problem~5.A]{MS1974}.
\end{proof}

Let $V$ be a finite-dimensional space and $l \in \N$, we denote by $\fP_l(V)$ the space of homogeneous polynomial functions of degree $l$ on $V$. If $\Norm{\cdot}$ is a norm on $V$, then we define a subordinated norm on $\fP_l(V)$ by $\Norm{\fp} = \max \brackets*{\strut \norm{\fp(v)} \mvert v \in V, \Norm{v}\leq 1}$ for all $\fp\in \fP_l(V)$. For example, recalling Definition~\ref{def: Jacobian}, the map $\Jac^2:\Lambda \mapsto \det(\Lambda\Lambda^*)$ defines an element of $\fP_{2k}\parentheses*{\cL(\R^d,\R^k)}$.

Let $\ux \in (\R^d)^A \setminus \diag$, we denote by $\pi_{\ux}$ the orthogonal projection from $\R_{2\norm{A}-1}[X]^k$ onto $\cG_A(\ux)$. Given $a \in A$, the map $D_{x_a-\flat(\ux)}:P \mapsto D_{x_a-\flat(\ux)}P$ is linear from $\R_{2\norm{A}-1}[X]^k$ to $\cL(\R^d,\R^k)$, so that $\Jac^2 \circ D_{x_a-\flat(\ux)} \circ \pi_{\ux} \in \fP_{2k}\parentheses*{\R_{2\norm{A}-1}[X]^k}$. Considering these polynomials is one of the important ideas introduced in~\cite{GS2024}. It allows us to read the conditional Jacobian appearing in Definition~\ref{def: rho A} as the square root of some deterministic polynomial applied to the Kergin interpolant of the field. In this way, we can isolate the contribution of the random field~$f$ from that of the geometry of the tuple $\ux$, and better understand the behavior of this Jacobian.

\begin{lem}[Non-degeneracy of differentials]
\label{lem: non-degeneracy differentials}
Let $\ux \in (\R^d)^A \setminus \diag$, we have:
\begin{equation*}
\prod_{a \in A}\parentheses*{\Jac^2 \circ D_{x_a-\flat(\ux)} \circ \pi_{\ux}} \in \fP_{2k\norm{A}}\parentheses*{\R_{2\norm{A}-1}[X]^k} \setminus \brackets{0}.
\end{equation*}
\end{lem}

\begin{proof}
We already explained that, for all $a \in A$, the map $\Jac^2 \circ D_{x_a-\flat(\ux)} \circ \pi_{\ux}$ is a homogeneous polynomial of degree $2k$ on $\R_{2\norm{A}-1}[X]^k$. The product of these maps is therefore homogeneous of degree $2k\norm{A}$, and we need to check that it is non-zero.

Since $d \geq k$, there exists a surjective map $\Lambda_0 \in \cL(\R^d,\R^k)$. Since $\ux \notin \diag$, we have $\oux \notin \diag$. Then, by Corollary~\ref{cor: surjectivity 1-jets}, there exists $P \in \R_{2\norm{A}-1}[X]^k$ such that, for all $a \in A$, we have $P(x_a-\flat(\ux))=0$ and $D_{x_a-\flat(\ux)}P=\Lambda_0$. That is, $P \in \ker(\ev_{\ux})=\cG_A(\ux)$, so that $\pi_{\ux}(P)=P$, and we have:
\begin{equation*}
\prod_{a \in A}\parentheses*{\Jac^2 \circ D_{x_a-\flat(\ux)}\circ \pi_{\ux}}(P) = \prod_{a \in A} \jac{D_{x_a-\flat(\ux)}P}^2 = \jac{\Lambda_0}^{2\norm{A}} >0.\qedhere
\end{equation*}
\end{proof}

We conclude this section by introducing notation for the radial and spherical parts of the map studied in Lemma~\ref{lem: non-degeneracy differentials}. It will be used in Section~\ref{subsec: Kac--Rice densities as functions of variance operators} to define alternative expressions of the Kac--Rice densities.

\begin{dfn}[Polar decomposition of products of Jacobians]
\label{def: N theta}
For all $\ux \in (\R^d)^A \setminus \diag$, we denote by $N_A(\ux) = \Norm*{\prod_{a \in A}\Jac^2 \circ D_{x_a-\flat(\ux)}\circ \pi_{\ux}} \in (0,+\infty)$ and by $\theta_A(\ux) \in \S\fP_{2k\norm{A}}\parentheses*{\R_{2\norm{A}-1}[X]^k}$ the unit polynomial function defined by:
\begin{equation*}
\theta_A(\ux):P \mapsto \frac{1}{N_A(\ux)} \prod_{a \in A}\parentheses*{\Jac^2 \circ D_{x_a-\flat(\ux)}\circ \pi_{\ux}}(P).
\end{equation*}
For all $P \in \R_{2\norm{A}-1}[X]^k$ we denote by $\theta_A(\ux,P)=\theta_A(\ux)(P)$ for simplicity, so that if $P \in \cG_A(\ux)$ then $\prod_{a \in A}\jac{D_{x_a-\flat(\ux)}P} = \sqrt{N_A(\ux)}\sqrt{\theta_A(\ux,P)}$.
\end{dfn}

\begin{lem}[Regularity of $N_A$ and $\theta_A$]
\label{lem: regularity N theta}
The maps $\theta_A:(\R^d)^A \setminus \diag \to \S\fP_{2k\norm{A}}\parentheses*{\R_{2\norm{A}-1}[X]^k}$ and $N_A:(\R^d)^A \setminus \diag\to (0,+\infty)$ are smooth.
\end{lem}

\begin{proof}
It is enough to prove that $\ux \mapsto \prod_{a \in A} \Jac^2 \circ D_{x_a-\flat(\ux)}\circ \pi_{\ux}$ is smooth from $(\R^d)^A \setminus \diag$ to $\fP_{2k\norm{A}}\parentheses*{\R_{2\norm{A}-1}[X]^k}$, since we have already shown that it is never zero in Lemma~\ref{lem: non-degeneracy differentials}.

Given $\ux \notin \diag$, the orthogonal projection onto $\cG_A(\ux)=\ker(\ev_{\ux})$ is $\pi_{\ux} = \Id - \ev_{\ux}^*\parentheses{\ev_{\ux}\ev_{\ux}^*}^{-1}\ev_{\ux}$. Composing continuous linear maps, inverting them or taking their adjoint are smooth operations. Hence $\ux \mapsto \pi_{\ux}$ is smooth on $(\R^d)^A \setminus \diag$, by Lemma~\ref{lem: regularity of ev}. The map $\ux \mapsto \oux = \parentheses*{x_a-\flat(\ux)}_{a \in A}$ is also smooth.

Then, we observe that the function $\parentheses*{\uy,\Lambda,P} \mapsto \prod_{a \in A} (\Jac^2 \circ D_{y_a} \circ \Lambda)(P)$ is a polynomial on $(\R^d)^A \times \cL\parentheses*{\R_{2\norm{A}-1}[X]^k}\times \R_{2\norm{A}-1}[X]^k$, so that $\parentheses*{\uy,\Lambda} \mapsto \prod_{a \in A} (\Jac^2 \circ D_{y_a} \circ \Lambda)$ is smooth as a map with values in $\fP_{2k\norm{A}}\parentheses*{\R_{2\norm{A}-1}[X]^k}$. The conclusion follows by composing $\ux \mapsto \parentheses*{\oux,\pi_{\ux}}$ with the previous map.
\end{proof}

%%%%%%%%%%%%%%%%%%%%%%%%%%%%%%%%%%%%%%%%%%%%%%%%%%%%%%%%%%%%%%%%%%%%%%%%%%%%%%%%%%%%%%%%%%%%%%%%%%%%%%%

\subsection{Kac--Rice densities as functions of variance operators}
\label{subsec: Kac--Rice densities as functions of variance operators}

The goal of this section is to write the Kac--Rice density $\rho_A(f,\cdot)$ from Definition~\ref{def: rho A} in a way that allows to deal with its singularity along the diagonal. We write it as a product of two terms: the first one is a universal geometrical factor, depending only on $\ux$, that fully accounts for the singularity of $\rho_A(f,\cdot)$ along~$\diag$; the second one depends on the field $f$ but is not singular along~$\diag$.

Recalling Definitions~\ref{def: ev x}, \ref{def: G kernel ev} and~\ref{def: N theta}, we define the universal and field-dependent parts of the Kac--Rice density as follows.

\begin{dfn}[Universal part of the density]
\label{def: Upsilon}
Let $A$ be a non-empty finite set, we define
\begin{equation*}
\Upsilon_A: \ux \longmapsto \frac{\sqrt{N_A(\ux)}}{\jac{\ev_{\ux}}}
\end{equation*}
from $(\R^d)^A \setminus \diag$ to $(0,+\infty)$.
\end{dfn}

\begin{dfn}[Sets of variance operators and sets of parameters]
\label{def: M I}
Let $\cI$ be a partition of a finite set $A \neq \emptyset$, we denote by $\cM_\cI$ the following compact set:
\begin{equation*}
\cM_\cI = \prod_{I \in \cI} \gr{k\norm{I}}{\R_{2\norm{I}-1}[X]^k} \times \prod_{I \in \cI} \S\fP_{2k\norm{I}}\parentheses*{\R_{2\norm{I}-1}[X]^k}.
\end{equation*}
We will denote a generic element in $\cM_\cI$ by $\parentheses*{(\cG_I)_{I \in \cI}\strut,(\theta_I)_{I \in \cI}}$. We also denote by
\begin{align*}
&\cS_\cI = \sym\parentheses*{\prod_{I \in \cI} \R_{2\norm{I}-1}[X]^k} & &\text{and} & &\cS^+_\cI = \sym^+\parentheses*{\prod_{I \in \cI} \R_{2\norm{I}-1}[X]^k}.
\end{align*}
\end{dfn}

Let $\cI \in \cP_A$, if for all $I \in \cI$ we have $\cG_I \in \gr{k\norm{I}}{\R_{2\norm{I}-1}[X]^k}$, then $\prod_{I \in \cI}\cG_I$ is a subspace of codimension $k\norm{A}$ in $\prod_{I \in \cI}\R_{2\norm{I}-1}[X]^k$. In particular, we can apply operators on $\prod_{I \in \cI}\R_{2\norm{I}-1}[X]^k$ to tuples of the form $(P_I)_{I \in \cI} \in \prod_{I \in \cI}\cG_I$.

\begin{dfn}[Field-dependent part of the density]
\label{def: sigma I}
Let $A$ be a non-empty finite set and $\cI \in \cP_A$. We define a map $\sigma_\cI:\cM_\cI \times \cS^+_\cI \to [0,+\infty)$ by:
\begin{equation*}
\sigma_{\cI}:\parentheses*{(\cG_I)_{I \in \cI}\strut ,(\theta_I)_{I \in \cI},\Sigma} \longmapsto \int_{\prod_{I \in \cI} \cG_I} \parentheses*{\prod_{I \in \cI} \sqrt{\theta_I(P_I)}}\frac{\exp\parentheses*{-\frac{1}{2}\prsc{\Sigma^{-1}\uP}{\uP}}}{\sqrt{\det\parentheses*{2\pi\Sigma}}}\dx \uP,
\end{equation*}
where $\uP=(P_I)_{I \in \cI} \in \prod_{I \in \cI} \cG_I$ and $\dx\uP$ is the Lebesgue measure on $\prod_{I \in \cI} \cG_I$ induced by the inner product on $\prod_{I \in \cI}\R_{2\norm{I}-1}[X]^k$ introduced before Definition~\ref{def: Sigma I}.
\end{dfn}

An important idea is that, in the expression of $\rho_A(f,\ux)$ given by Definition~\ref{def: rho A}, we can replace~$f$ by its Kergin interpolant, introduced Section~\ref{sec: polynomial interpolation and Gaussian fields}. For this, we need to interpolate not only the values of $f$ on the $(x_a)_{a \in A}$ but also those of its differential. This suggest to consider the interpolant not on $\ux$ but on $\ux_{2A}=(\ux,\ux)$, where each point appears with double multiplicity, see~Section~\ref{subsec: product spaces and diagonals}.

Let $U \subset \R^d$ be a convex open set and $f \in \cC^{2\norm{A}-1}(U,\R^k)$ be a centered Gaussian field. Recalling Definition~\ref{def: barycenter}, one can check that $\flat\parentheses*{\ux_{2A}} = \flat(\ux)$ and $\obullet{\overgroup{\ux_{2A}}} =\parentheses{\oux,\oux} = \oux_{2A}$. In particular, $\oK(f,\ux_{2A}) = K\parentheses*{\flat(\ux)\cdot f,\oux,\oux}=\flat(\ux) \cdot K(f,\ux_{2A})$ for all $\ux \in U^A$, see Definition~\ref{def: twisted Kergin interpolant}. Then, for all $\cI \in \cP_A$, we have $\Sigma_{2\cI}(f,\ux_{2A})= \parentheses*{\Sigma_{2I}^{2J}(f,\ux_{2A})}_{I,J \in \cI} = \var{\parentheses*{\oK(f,\ux_{2I})}_{I \in \cI}}$, see Definition~\ref{def: Sigma I}.

\begin{lem}[Relation between notions of non-degeneracy]
\label{lem: relation between notions of non-degeneracy}
Let $f \in \cC^{2\norm{A}-1}(U,\R^k)$ be a centered Gaussian field, let $\cI \in \cP_A$ and $\ux \in U^A \setminus \diag$ be such that $\Sigma_{2\cI}(f,\ux_{2A})$ is positive-definite, then the Gaussian vector $\parentheses*{f(x_a)}_{a \in A}$ is non-degenerate, and $\parentheses*{f(x_a),D_{x_a}f}$ is non-degenerate for all $a \in A$.
\end{lem}

\begin{proof}
Our assumption on $\Sigma_{2\cI}(f,\ux_{2A})$ means that the Gaussian vector $\parentheses*{\oK(f,\ux_{2I})}_{I \in \cI}$ is non-degenerate. Since translations act by isomorphisms on each factor $\R_{2\norm{I}-1}[X]^k$, this means that $\parentheses*{\strut K(f,\ux_{2I})}_{I \in \cI}$ is non-degenerate. Then, applying Corollary~\ref{cor: surjectivity 1-jets} block-wise shows that the Gaussian vector $\parentheses*{\parentheses*{f(x_i),D_{x_i}f}_{i \in I}}_{I \in \cI} = \parentheses*{f(x_a),D_{x_a}f}_{a \in A}$ is non-degenerate. Hence the result.
\end{proof}

\begin{lem}[Kac--Rice formula revisited]
\label{lem: Kac-Rice revisited}
Let $f \in \cC^{2\norm{A}-1}(U,\R^k)$ be a centered Gaussian field. Let $\ux \in U^A \setminus \diag$ and $\cI \in \cP_A$ be such that $\Sigma_{2\cI}(f,\ux_{2A})$ is positive-definite, then
\begin{equation*}
\rho_A(f,\ux) = \parentheses*{\prod_{I \in \cI} \Upsilon_I(\ux_I)} \tilde{\sigma}_\cI(f,\ux),
\end{equation*}
where we denoted by $\tilde{\sigma}_\cI(f,\ux) := \sigma_\cI\parentheses*{(\cG_I(\ux_I))_{I \in \cI}\strut, (\theta_I(\ux_I))_{I \in \cI}, \Sigma_{2\cI}(f,\ux_{2A})}$.
\end{lem}

\begin{proof}
First, $\parentheses*{f(x_a)}_{a \in A}$ is non-degenerate by Lemma~\ref{lem: relation between notions of non-degeneracy}. Hence $\rho_A(f,\ux)$ is well-defined by Definition~\ref{def: rho A}. Starting from this definition, we proceed in two steps.

\paragraph*{Step 1: Replacing $f$ with its twisted Kergin interpolant.}
Let $I \in \cI$ and let us denote by $K_I = \oK(f,\ux_{2I}) \in \R_{2\norm{I}-1}[X]^k$. We have:
\begin{equation}
\label{eq: oK and evaluation}
K_I(X-\flat(\ux_I)) = \oK(f,\ux_{2I})(X-\flat(\ux_I)) = \parentheses*{\strut \flat(\ux_I)\cdot K(f,\ux_{2I})}(X-\flat(\ux_I))= K(f,\ux_{2I}).
\end{equation}
Recalling Definition~\ref{def: ev x}, we have:
\begin{equation*}
\ev_{\ux_I}(K_I) = \parentheses*{\strut K_I(x_i-\flat(\ux_I))}_{i \in I} = \parentheses*{\strut K(f,\ux_{2I})(x_i)}_{i \in I} = \parentheses*{\strut f(x_i)}_{i \in I},
\end{equation*}
using the interpolation properties of $K(f,\ux_{2I})$. Hence, we have $\parentheses*{f(x_a)}_{a \in A} = \parentheses*{\ev_{\ux_I}(K_I)}_{I \in \cI}$, and conditioning on $\parentheses*{f(x_a)}_{a \in A}=0$ is the same as conditioning on $K_I \in \cG_I(\ux_I)=\ker(\ev_{\ux_I})$ for all~$I \in \cI$.

Let $I \in \cI$ and $i \in I$, differentiating Equation~\eqref{eq: oK and evaluation} at $x_i$ and using Theorem~\ref{thm: Kergin interpolation} yields that
\begin{equation*}
D_{x_i-\flat(\ux_I)} K_I = D_{x_i} K(f,\ux_{2I})=D_{x_i}f.
\end{equation*}
Recalling Definition~\ref{def: N theta}, given that $K_I \in \cG_I(\ux_I)$, we obtain that:
\begin{equation*}
\prod_{i \in I}\jac{D_{x_i}f} = \prod_{i \in I} \jac{D_{x_i-\flat(\ux_I)} K_I} = \sqrt{N_I(\ux_I)}\sqrt{\theta_I(\ux_I,K_I)}.
\end{equation*}
Thus
\begin{equation*}
\rho_A(f,\ux) = \frac{\espcond{\prod_{I \in \cI} \sqrt{N_I(\ux_I)}\sqrt{\theta_I(\ux_I,K_I)}}{\forall I \in \cI, K_I\in \cG_I(\ux_I)}}{\det\parentheses*{2\pi \var{\parentheses*{\ev_{\ux_I}(K_I)}_{I \in \cI}}}^\frac{1}{2}}
\end{equation*}

For all $I \in \cI$, we have an orthogonal direct sum $\R_{2\norm{I}-1}[X]^k = \cG_I(\ux_I)^\perp \oplus \cG_I(\ux_I)$. Accordingly, we write $K_I = K_{I,0} + K_{I,1}$ where $K_{I,0} \in \cG_I(\ux_I)^\perp$ and $K_{I,1} \in \cG_I(\ux_I)$. Let us denote by $\Lambda_I$ the restriction of $\ev_{\ux_I}$ to $\cG_I(\ux_I)^\perp$, so that $\ev_{x_I}(K_I) = \Lambda_I(K_{I,0})$. For all $I \in \cI$, the linear map $\Lambda_I: \cG_I(\ux)^\perp \to (\R^k)^I$ is an isomorphism. Hence, using the notation introduced in Definition~\ref{def: block diagonal operators}, the block-diagonal operator $\bigoplus_{I \in \cI} \Lambda_I$ is an isomorphism and
\begin{equation*}
\var{\parentheses*{\ev_{\ux_I}(K_I)}_{I \in \cI}} = \var{\strut \parentheses*{\Lambda_I(K_{I,0})}_{I \in \cI}}= \parentheses*{\bigoplus_{I \in \cI} \Lambda_I} \var{\parentheses*{K_{I,0}}_{I \in \cI}} \parentheses*{\bigoplus_{I \in \cI} \Lambda_I}^*.
\end{equation*}
Taking the determinant on both sides yields:
\begin{equation*}
\det\parentheses*{2\pi \var{\parentheses*{\ev_{\ux_I}(K_I)}_{I \in \cI}}} = \det\parentheses*{2\pi \var{\parentheses*{K_{I,0}}_{I \in \cI}}}\prod_{I \in \cI} \det\parentheses*{\Lambda_I \Lambda_I^*}.
\end{equation*}
Observing that $\Lambda_I \Lambda_I^* = \ev_{\ux_I}\ev_{\ux_I}^*$, we get that $\det\parentheses*{\Lambda_I \Lambda_I^*}^\frac{1}{2} = \jac{\ev_{\ux_I}}$ for all $I \in \cI$. Moreover, because of the orthogonal projection in Definition~\ref{def: N theta}, we have $\theta_I(\ux_I,K_I) = \theta_I(\ux_I,K_{I,1})$ for all $I \in \cI$. Finally, recalling Definition~\ref{def: Upsilon}, we proved that:
\begin{equation*}
\rho_A(f,\ux) = \parentheses*{\prod_{I \in \cI} \Upsilon_I(\ux_I)} \frac{\espcond{\prod_{I \in \cI} \sqrt{\theta_I(\ux_I,K_{I,1})}}{\forall I \in \cI, K_{I,0}=0}}{\det\parentheses*{2\pi \var{\parentheses*{K_{I,0}}_{I \in \cI}}}^\frac{1}{2}}.
\end{equation*}

\paragraph*{Step 2: Dealing with the conditional expectation.}
Let $\cG = \prod_{I \in \cI} \cG_I(\ux_I)$, with our choice of inner product we have $\cG^\perp = \prod_{I \in \cI} \cG_I(\ux_I)^\perp$ and $\parentheses*{K_I}_{I \in \cI} = K_0+K_1$, where $K_0=\parentheses*{K_{I,0}}_{I \in \cI} \in \cG^\perp$ and $K_1=\parentheses*{K_{I,1}}_{I \in \cI} \in \cG$. We can use the orthogonal splitting $\prod_{I \in \cI} \R_{2\norm{I}-1}[X]^k = \cG^\perp \oplus \cG$ to write $\Sigma_{2\cI}(f,\ux_{2A}) = \var{\parentheses*{K_I}_{I \in \cI}}$ as a block-operator $\parentheses*{\begin{smallmatrix} \Sigma_0 & \Sigma_2^* \\ \Sigma_2 & \Sigma_1 \end{smallmatrix}}$, where $\Sigma_0 = \var{K_0}$, $\Sigma_1 = \var{K_1}$ and $\Sigma_2 = \cov{K_1}{K_0}$.

Since $(K_0,K_1)$ is a centered Gaussian with positive-definite variance $\Sigma_{2\cI}(f,\ux_{2A})$, the conditional distribution of $K_1$ given that $K_0=0$ is a centered Gaussian whose variance operator is $\tilde{\Sigma} = \Sigma_1 - \Sigma_2 \Sigma_0^{-1}\Sigma_2^*$, see~\cite[Prop.~1.2]{AW2009}. The operator $\tilde{\Sigma} \in \sym^+(\cG)$ is the Schur complement of $\Sigma_0$ in $\Sigma_{2\cI}(f,\ux_{2A})$. It is well-known that $\det\parentheses*{\Sigma_{2\cI}(f,\ux_{2A})} = \det(\Sigma_0)\det(\tilde{\Sigma})$ and that $\tilde{\Sigma}^{-1}$ is the bottom-right block of $\Sigma_{2\cI}(f,\ux_{2A})^{-1}$. In particular, for all $\uP \in \cG$, we have $\prsc{\tilde{\Sigma}^{-1}\uP}{\uP} = \prsc{\Sigma_{2\cI}(f,\ux_{2A})^{-1} \uP}{\uP}$. Thus,
\begin{multline*}
\frac{\espcond{\prod_{I \in \cI} \sqrt{\theta_I(\ux_I,K_{I,1})}}{\forall I \in \cI, K_{I,0}=0}}{\det\parentheses*{2\pi \var{\parentheses*{K_{I,0}}_{I \in \cI}}}^\frac{1}{2}}\\
\begin{aligned}
&= \frac{1}{\det\parentheses*{2\pi \Sigma_0}^\frac{1}{2}} \int_{\uP \in \cG} \prod_{I \in \cI} \sqrt{\theta_I(\ux_I,P_I)}\frac{\exp\parentheses*{-\frac{1}{2}\prsc{\tilde{\Sigma}^{-1}\uP}{\uP}}}{\det\parentheses*{2\pi\tilde{\Sigma}}^\frac{1}{2}}\dx \uP\\
&= \int_{\uP \in \cG} \prod_{I \in \cI} \sqrt{\theta_I(\ux_I,P_I)}\frac{\exp\parentheses*{-\frac{1}{2}\prsc{\Sigma_{2\cI}(f,\ux_{2A})^{-1}\uP}{\uP}}}{\det\parentheses*{\strut 2\pi \Sigma_{2\cI}(f,\ux_{2A})}^\frac{1}{2}}\dx \uP\\
&=\sigma_\cI\parentheses*{\parentheses*{\cG_I(\ux_I)}_{I \in \cI},\strut \parentheses*{\theta_I(\ux_I)}_{I \in \cI}, \Sigma_{2\cI}(f,\ux_{2A})},
\end{aligned}
\end{multline*}
which concludes the proof.
\end{proof}

Now that we explained how $\Upsilon_A$ and $\sigma_\cI$ are related to $\rho_A(f,\cdot)$, the remainder of this section is dedicated to proving some properties of these maps.

\begin{dfn}[Parametric smooth maps]
\label{def: C 0 infty}
Let $\cM$ be a topological space (that we think of as a parameter space) and $\Omega$ be an open subset of a finite-dimensional vector space $V$.
\begin{itemize}
\item Let $F:\cM \times \Omega \to \R$, for all $q \in \N$ and $(w,x)\in \cM \times \Omega$ we denote by $D^q_{(w,x)}F = D^q_x\parentheses*{F(w,\cdot)}$, if $F(w,\cdot):\Omega \to \R$ is $q$-differentiable at $x$. Similarly, assuming we are given a basis of $V$, for any multi-index $\alpha$ we denote by $\partial^\alpha F(w,x) = \partial^\alpha(F(w,\cdot))(x)$, if it is well-defined.

\item We denote by $\cCM$ the space of functions $F:\cM \times \Omega \to \R$ such that, for all $q \in \N$, the map $(w,x) \mapsto D_{(w,x)}^qF$ is well-defined and continuous on $\cM \times \Omega$. Equivalently, $F \in \cCM$ if and only if $\partial^\alpha F$ is well-defined and continuous on $\cM \times \Omega$ for all $\alpha$.
\end{itemize}
\end{dfn}

\begin{lem}[Regularity of $\sigma_\cI$]
\label{lem: regularity sigma I}
The map $\sigma_\cI$ from Definition~\ref{def: sigma I} belongs to $\cC^{0,\infty}\parentheses*{\cM_\cI \times \cS_\cI^+}$.
\end{lem}

\begin{proof}
It is enough to check the regularity of $\sigma_\cI$ locally, which can be done using local charts on the Grassmannians, cf.~\cite[Ex.~1.36]{Lee2013}. The proof is divided in two steps.

\paragraph*{Step 1: Reduction to a computation in local coordinates.}
Let $\cG_I^0 \in \gr{k\norm{I}}{\R_{2\norm{I}-1}[X]^k}$ for all $I \in \cI$. The map $\Lambda_I \mapsto \brackets*{P+\Lambda_I(P) \mvert P \in \cG_I^0}$ is a diffeomorphism from $\cL\parentheses*{\cG_I^0,(\cG_I^0)^\perp}$ onto $\brackets*{\cG \in \gr{k\norm{I}}{\R_{2\norm{I}-1}[X]^k} \mvert \cG \cap (\cG_I^0)^\perp = \brackets{0}}$, which defines local coordinates centered at~$\cG_I^0$. Reading $\sigma_\cI$ in these coordinates, we need to prove that it admits continuous partial derivatives at any order with respect to $\Sigma$.

For all $I \in \cI$, let $\Lambda_I\in \cL\parentheses*{\cG_I^0,(\cG_I^0)^\perp}$ and $\cG_I$ denote its graph. Then, $P \mapsto P +\Lambda_I(P)$ is an isomorphism from $\cG_I^0$ to $\cG_I$ whose Jacobian equals $\sqrt{\det\parentheses*{\Id +\Lambda_I^*\Lambda_I}}$. Starting from the integral expression in Definition~\ref{def: sigma I}, a change of variables yields that $\sigma_{\cI}\parentheses*{\strut \parentheses{\cG_I}_{I \in \cI},\parentheses*{\theta_I}_{I \in \cI},\Sigma}$ equals:
\begin{equation*}
\frac{\prod_{I \in \cI} \sqrt{\det\parentheses*{\Id +\Lambda_I^*\Lambda_I}}}{\sqrt{\det\parentheses*{2\pi\Sigma}}} \int_{\cG^0} \prod_{I \in \cI}\sqrt{\theta_I(P_I+\Lambda_I(P_I))}\exp\parentheses*{-\frac{\prsc{\strut \Sigma^{-1}\parentheses*{\uP+\Lambda(\uP)}}{\uP+\Lambda(\uP)}}{2}}\dx \uP,
\end{equation*}
where $\cG^0 =\prod_{I \in \cI} \cG_I^0$, we let $\uP = (P_I)_{I \in \cI} \in \cG^0$ and $\Lambda(\uP)=\parentheses*{\Lambda_I(P_I)}_{I \in \cI}$, and $\dx \uP$ stands for the Lebesgue measure on $\cG^0$. What we gained is that the domain of integration no longer depends on the arguments, and this expression depends explicitly on $\Lambda = \bigoplus_{I \in \cI}\Lambda_I$.

\paragraph*{Step 2: Regularity of $\sigma_{\cI}$ in the previous local coordinates.}
Since the maps $\Sigma \mapsto \Sigma^{-1}$ and $\parentheses*{\Lambda,\Sigma} \mapsto \det\parentheses*{2\pi\Sigma}^{-\frac{1}{2}} \prod_{I \in \cI} \det\parentheses*{\Id +\Lambda_I^*\Lambda_I}^{\frac{1}{2}}$ are smooth, we only need to prove that
\begin{equation}
\label{eq: regularity sigma I}
\parentheses*{\Lambda,\strut \theta, \Sigma} \longmapsto \int_{\cG^0} \prod_{I \in \cI}\sqrt{\theta_I(P_I+\Lambda_I(P_I))}\exp\parentheses*{-\frac{\prsc{\strut \Sigma \parentheses*{\uP+\Lambda(\uP)}}{\uP+\Lambda(\uP)}}{2}} \dx \uP
\end{equation}
admits continuous partial derivatives at any order with respect to $\Sigma$, where $\theta = \parentheses*{\theta_I}_{I \in \cI}$.

The integrand in~\eqref{eq: regularity sigma I} is continuous with respect to $\parentheses*{\Lambda,\theta, \Sigma,\uP}$, as well as its partial derivatives with respect to $\Sigma$. For all $I \in \cI$, we have $\Norm{P_I+\Lambda_I(P_I)} \leq (1+\Norm{\Lambda_I})\Norm{P_I}$ and, since $\theta_I$ has unit norm, $\norm{\theta_I\parentheses*{P_I+\Lambda_I(P_I)}} \leq (1+\Norm{\Lambda_I})\Norm{P_I}$. This is enough to obtain dominating functions with Gaussian decay with respect to $\Norm{\uP}$, uniformly with respect to $(\Lambda,\theta,\Sigma)$ in any compact subset of $\prod_{I \in \cI} \cL\parentheses*{\cG^0_I,(\cG^0_I)^\perp} \times \prod_{I \in \cI} \S\fP_{2k\norm{I}}\parentheses*{\R_{2\norm{I}-1}[X]^k} \times \cS_{\cI}^+$. Hence the result.
\end{proof}

\begin{lem}[Regularity of $\Upsilon$]
\label{lem: regularity Upsilon}
Let $A$ be non-empty and finite, the function $\Upsilon_A$ from Definition~\ref{def: Upsilon} is continuous on $(\R^d)^A \setminus \diag$ and locally integrable on $(\R^d)^A$. Moreover, we have $\tau \cdot \Upsilon_A = \Upsilon_A$ for all $\tau \in \R^d$.
\end{lem}

\begin{proof}
The continuity follows from Definition~\ref{def: Upsilon} and from Lemmas~\ref{lem: regularity of ev} and~\ref{lem: regularity N theta}. Let $\tau \in \R^d$ and $\ux \in (\R^A)^d \setminus \diag$. By Definition~\ref{def: ev x}, we have $\ev_{\tau\cdot \ux} =\ev_{\oux} = \ev_{\ux}$. Then, we get $\cG_A(\tau \cdot \ux) =\cG_A(\ux)$. And finally $N_A(\tau \cdot\ux) = N_A(\ux)$ by Definition~\ref{def: N theta}. Thus $\Upsilon_A(\tau\cdot \ux)=\Upsilon(\ux)$ for all $\ux \notin \diag$ and $\tau \in \R^d$, that is $\Upsilon_A$ is invariant under the action of diagonal translations. The difficulty is in proving the local integrability, which is done in two steps.

\paragraph*{Step 1: Reduction to lower bounding $\sigma_{\brackets{A}}$.}
Let us consider a standard Gaussian vector $f_0$ in $\R_{2\norm{A}-1}[X]^k$. On the one hand, $f_0$ defines a smooth centered Gaussian field from $\R^d$ to $\R^k$. On the other hand, if $\ux \in (\R^d)^A$, we have $K(f_0,\ux_{2A})=f_0$ by uniqueness in Theorem~\ref{thm: Kergin interpolation}. Hence the Gaussian vector $\oK(f_0,\ux_{2A}) = \flat(\ux)\cdot K(f_0,\ux_{2A}) = \flat(\ux)\cdot f_0$ is non-degenerate, since $P \mapsto \flat(\ux)\cdot P$ is an isomorphism of $\R_{2\norm{A}-1}[X]^k$. Thus, we have $\Sigma_{\brackets{2A}}(f_0,\ux_{2A}) = \var{\oK(f_0,\ux_{2A})} \in \cS^+_{\brackets{A}}$ for all $\ux \in (\R^d)^A$. Then, by Lemma~\ref{lem: Kac-Rice revisited}, we have $\rho_A(f_0,\cdot) = \Upsilon_A \tilde{\sigma}_{\brackets{A}}(f_0,\cdot)$ on $(\R^d)^A \setminus \diag$.

Let $\Gamma \subset (\R^d)^A$ be compact. We know from~\cite[Thm.~6.26]{AL2025} and~\cite[Thm.~1.13 and Lem.~2.9]{GS2024} that $\rho_A(f_0,\cdot)$ is locally integrable on $(\R^d)^A$. Indeed, in this case, the non-degeneracy of $f_0$ implies that of its multijets up to order $2\norm{A}$. In particular $\rho_A(f_0,\cdot)$ is integrable on $\Gamma$. In order to prove that $\Upsilon_A$ is integrable on $\Gamma$, it is thus enough to show that the non-negative function $\tilde{\sigma}_{\brackets{A}}(f_0,\cdot)$ is bounded from below on $\Gamma \setminus \diag$ by some constant $C>0$.

\paragraph*{Step 2: Existence of a lower bound.}
We prove the existence of $C>0$ by contradiction. If there exists no such bound, there is a sequence $\parentheses*{\ux^{(n)}}_{n \in \N}$ in $\Gamma \setminus \diag$ such that $\tilde{\sigma}_{\brackets{A}}\parentheses*{f_0,\ux^{(n)}} \xrightarrow[n \to +\infty]{}0$. By compactness of $\Gamma$, up to extracting a subsequence, we can assume that $\ux^{(n)}\xrightarrow[n \to +\infty]{}\ux \in \Gamma$. By Remark~\ref{rem: Sigma I C0}, we have $\Sigma_{\brackets{2A}}\parentheses*{f_0,\ux_{2A}^{(n)}} \xrightarrow[n \to +\infty]{} \Sigma_{\brackets{2A}}\parentheses*{f_0,\ux_{2A}} \in \cS^+_{\brackets{A}}$. By compactness of $\cM_{\brackets{A}}$, see Definition~\ref{def: M I}, we can also assume that $\parentheses*{\cG_A(\ux^{(n)}),\theta_A(\ux^{(n)})}\xrightarrow[n \to +\infty]{} \parentheses*{\strut \cG,\theta}\in \cM_{\brackets{A}}$. Finally by continuity of $\sigma_{\brackets{A}}$, see Lemma~\ref{lem: regularity sigma I}, we have:
\begin{align*}
0 &= \lim_{n \to +\infty} \tilde{\sigma}_{\brackets{A}}\parentheses*{f_0,\ux^{(n)}} = \lim_{n\to +\infty}\sigma_{\brackets{A}}\parentheses*{\cG_A(\ux^{(n)}),\theta_A(\ux^{(n)}), \Sigma_{\brackets{2A}}(f_0,\ux_{2A}^{(n)})}\\
&= \sigma_{\brackets{A}}\parentheses*{\cG,\theta,\Sigma_{\brackets{2A}}(f_0,\ux_{2A})} = \int_{\cG} \sqrt{\theta(P)}\frac{\exp\parentheses*{-\frac{1}{2}\prsc{\Sigma_{\brackets{2A}}(f_0,\ux_{2A})^{-1} P}{ P}}}{\sqrt{\det\parentheses*{2\pi\Sigma_{\brackets{2A}}(f_0,\ux_{2A})}}}\dx P.
\end{align*}
Hence $\theta_{\vert \cG}=0$. We will prove that this cannot happen.

For all $n \in \N$, we have $\Norm{\theta_A(\ux^{(n)})}=1$. Hence, there exists $P_n \in \R_{2\norm{A}-1}[X]^k$ such that $\Norm{P_n}\leq 1$ and $\theta_A(\ux^{(n)},P_n)=1$. Its orthogonal projection onto $\cG_A(\ux^{(n)})$ satisfies $\Norm{\pi_{\ux^{(n)}}(P_n)} \leq \Norm{P_n}\leq 1$ and $\theta_A(\ux^{(n)},\pi_{\ux^{(n)}}(P_n))=\theta_A(\ux^{(n)},P_n)=1$, see Definition~\ref{def: N theta}. Up to replacing $P_n$ by $\pi_{\ux^{(n)}}(P_n)$, we can assume $P_n \in \cG_A(\ux^{(n)})$. By compactness of the unit ball, we can also assume that $P_n \xrightarrow[n \to +\infty]{}P$.

As above, we work in a chart on $\gr{k\norm{A}}{\R_{2\norm{A}-1}[X]^k}$. Since $\cG_A(\ux^{(n)}) \xrightarrow[n \to +\infty]{}\cG$, for $n$ large enough there exists $\Lambda_n \in \cL(\cG,\cG^\perp)$ such that $\cG_A(\ux^{(n)})$ is the graph of $\Lambda_n$, and moreover $\Lambda_n \xrightarrow[n\to+\infty]{}0$. Let $Q_n \in \cG$ denote the orthogonal projection of $P_n$ onto $\cG$. Since $P_n \in \cG_A(\ux^{(n)})$, we have $P_n = Q_n + \Lambda_n(Q_n)$. Then,
\begin{equation*}
\Norm{P_n-Q_n}=\Norm{\Lambda_n(Q_n)}\leq \Norm{\Lambda_n}\Norm{Q_n}\leq \Norm{\Lambda_n}\Norm{P_n}\leq \Norm{\Lambda_n}\xrightarrow[n \to +\infty]{}0,
\end{equation*}
so that $Q_n \xrightarrow[n \to +\infty]{}P$, and thus $P \in \cG$. On the other hand,
\begin{align*}
\norm*{1-\theta(P)} &=\norm*{\theta_A(\ux^{(n)},P_n)-\theta(P)} \leq \norm*{\theta_A(\ux^{(n)},P_n)-\theta(P_n)}+\norm{\theta(P_n)-\theta(P)}\\
&\leq \Norm*{\theta_A(\ux^{(n)})-\theta} + \norm{\theta(P_n)-\theta(P)} \xrightarrow[n \to +\infty]{}0,
\end{align*}
by continuity of $\theta$. Thus $\theta(P)=1$ and $\theta_{\vert \cG}\neq 0$, which is the contradiction we sought.
\end{proof}

\begin{rem}
\label{rem: regularity rho A}
Let $A$ be a non-empty finite set and $\cI \in \cP_A$. Let $f \in \cC^{2\norm{A}-1}(U,\R^k)$ be a Gaussian field, where $U \subset \R^d$ is open and convex.
\begin{itemize}
\item Let $\uy \in U^A$ be such that $\Sigma_{2\cI}(f,\uy_{2A}) \in \cS_\cI^+$. By continuity of $\Sigma_{2\cI}(f,\cdot)$, there exists a compact neighborhood $\Gamma$ of $\uy$ such that $\brackets*{\Sigma_{2\cI}(f,\ux_{2A}) \mvert \ux \in \Gamma} \subset \cS_\cI^+$. In particular, the expression of $\rho_A(f,\cdot)$ derived in Lemma~\ref{lem: Kac-Rice revisited} is valid for all $\ux \in \Gamma \setminus \diag$. Then, we deduce from Remark~\ref{rem: Sigma I C0} and Lemmas~\ref{lem: regularity GA}, \ref{lem: regularity N theta}, \ref{lem: regularity sigma I} and~\ref{lem: regularity Upsilon} that $\rho_A(f,\cdot)$ is continuous on $\Gamma \setminus \diag$.

\item By continuity of $\Sigma_{2\cI}(f,\cdot)$ the set $\cM_\cI \times \brackets*{\Sigma_{2\cI}(f,\ux_{2A}) \mvert \ux \in \Gamma}$ is compact. Since $\sigma_\cI$ is continuous, it is bounded on this set, hence $\tilde{\sigma}_\cI(f,\cdot)$ is bounded on $\Gamma \setminus \diag$. This substantiate our claim that the singularity of $\rho_A(f,\cdot)$ is fully accounted for by the universal term $\ux \mapsto \prod_{I \in \cI} \Upsilon_I(\ux_I)$.
\end{itemize}
\end{rem}

To conclude this section we prove that, if $w \in \cM_\cI$ and $\Sigma \in \cS_\cI^+$ is block-diagonal, then $\sigma_\cI(w,\Sigma)$ is a product of terms of the same form. Let us introduce some notation before stating this result.

\begin{dfn}[Frobenius inner product]
\label{def: Frobenius inner product}
Let $V$ and $W$ be Euclidean spaces, the \emph{Frobenius inner product} on $\cL(V,W)$ is defined by $\prsc*{\Lambda_1}{\Lambda_2} = \Tr\parentheses*{\Lambda_1 \Lambda_2^*}$ for all $\Lambda_1, \Lambda_2 \in \cL(V,W)$.
\end{dfn}

Let $A \neq \emptyset$ be finite and $\cI \in \cP_A$. Recall that we equipped $\prod_{I \in \cI} \R_{2\norm{I}-1}[X]^k$ with an inner product such that the factors are orthogonal. Self-adjoint operators in $\cS_\cI$ have a natural block structure adapted to this product, and the Frobenius inner product on this space is such that $\prsc{\Sigma}{\Lambda} = \sum_{I,J \in \cI} \prsc{\Sigma_I^J}{\Lambda_I^J}$ for all $\Sigma = \parentheses*{\Sigma_I^J}_{I,J \in \cI}$ and $\Lambda=\parentheses*{\Lambda_I^J}_{I,J \in \cI} \in \cS_\cI$.

\begin{dfn}[Block-diagonal variance operators]
\label{def: block-diagonal variance}
Let $A$ be a finite set and $B \subset A$ be a proper subset, that is, such that $\emptyset \neq B \neq A$. Let $\cI \leq \brackets*{B,A\setminus B}$, we denote by:
\begin{align*}
\cS_{\cI,B} &= \brackets*{\Sigma = (\Sigma_I^J)_{I,J \in \cI} \in \cS_\cI \mvert \forall I \in \cI_B, \forall J \in \cI_{A\setminus B}, \Sigma_I^J=0} & &\text{and} & \cS_{\cI,B}^+ = \cS_{\cI,B} \cap \cS_\cI^+.
\end{align*}
For the Frobenius inner product, the orthogonal of $\cS_{\cI,B}$ is:
\begin{equation*}
\cS_{\cI,B}^\perp = \brackets*{\Sigma = (\Sigma_I^J)_{I,J \in \cI} \in \cS_\cI \mvert (\Sigma_I^J)_{I,J \in \cI_B}=0 \ \text{and} \ (\Sigma_I^J)_{I,J \in \cI_{A \setminus B}}=0}.
\end{equation*}
\end{dfn}

If $\cI \in \cP_A$ is such that $\cI \leq \brackets{B, A \setminus B}$, then  $\cI = \cI_{B} \sqcup \cI_{A \setminus B}$ so that
\begin{equation*}
\prod_{I \in \cI} \R_{2\norm{I}-1}[X]^k = \prod_{I \in \cI_B} \R_{2\norm{I}-1}[X]^k \times \prod_{I \in \cI_{A \setminus B}} \R_{2\norm{I}-1}[X]^k
\end{equation*}
and operators on this space have an associated $2 \times 2$ block structure. Then $\cS_{\cI,B}$ (resp.~$\cS_{\cI,B}^\perp$) is the subspace of block-diagonal (resp.~block-anti-diagonal) operators with respect to $\brackets{B,A \setminus B}$.

\begin{lem}[Exact clustering]
\label{lem: exact clustering}
Let $A$ be finite, let $B$ be a proper subset of $A$ and $\cI \in \cP_A$ such that $\cI \leq \brackets{B, A \setminus B}$. For all $\parentheses*{(\cG_I)_{I \in \cI},(\theta_I)_{I \in \cI}} \in \cM_\cI$ and $\Sigma=(\Sigma_I^J)_{I,J \in \cI} \in \cS_{\cI,B}^+$, we have:
\begin{equation*}
\sigma_\cI\parentheses*{\strut (\cG_I)_{I \in \cI},(\theta_I)_{I \in \cI},\Sigma} = \prod_{C \in \brackets{B,A\setminus B}} \sigma_{\cI_C}\parentheses*{\strut (\cG_I)_{I \in \cI_C},(\theta_I)_{I \in \cI_C},(\Sigma_I^J)_{I,J \in \cI_C}}.
\end{equation*}
\end{lem}

\begin{proof}
Let $\parentheses*{(\cG_I)_{I \in \cI},(\theta_I)_{I \in \cI}} \in \cM_\cI$ and $\Sigma=(\Sigma_I^J)_{I,J \in \cI} \in \cS_{\cI,B}^+$. For all $C \in \brackets{B,A\setminus B}$, let us denote by $\Sigma_C = \parentheses*{\Sigma_I^J}_{I,J \in \cI_C}$, so that $\Sigma = \parentheses*{\begin{smallmatrix} \Sigma_B & 0 \\ 0 & \Sigma_{A \setminus B}\end{smallmatrix}}$ by blocks, with respect to $\brackets{B,A\setminus B}$. Then $\det\parentheses*{2\pi \Sigma} = \det\parentheses*{2\pi \Sigma_B}\det\parentheses*{2\pi \Sigma_{A \setminus B}}$ and $\Sigma^{-1} = \parentheses*{\begin{smallmatrix} \Sigma_B^{-1} & 0 \\ 0 & \Sigma_{A \setminus B}^{-1}\end{smallmatrix}}$. For all $\uP \in \prod_{I \in \cI}\R_{2\norm{I}-1}[X]^k$, we have:
\begin{equation*}
\prsc*{\Sigma^{-1}\uP}{\uP} = \prsc*{\parentheses*{\begin{smallmatrix} \Sigma_B^{-1} & 0 \\ 0 & \Sigma_{A \setminus B}^{-1}\end{smallmatrix}}\parentheses*{\begin{smallmatrix} \uP_{\cI_B} \\ \uP_{\cI_{A \setminus B}}\end{smallmatrix}}}{\parentheses*{\begin{smallmatrix} \uP_{\cI_B} \\ \uP_{\cI_{A \setminus B}}\end{smallmatrix}}}=\sum_{C \in \brackets{B,A\setminus B}}\prsc{\Sigma_C^{-1} \uP_{\cI_C}}{\uP_{\cI_C}},
\end{equation*}
hence $e^{-\frac{1}{2}\prsc*{\Sigma^{-1}\uP}{\uP}} =\prod_{C \in \brackets{B,A \setminus B}}e^{-\frac{1}{2}\prsc*{\Sigma_C^{-1}\uP_{\cI_C}}{\uP_{\cI_C}}}$. Since $\cI = \cI_B \sqcup \cI_{A \setminus B}$, products indexed by $I \in \cI$ can be splitted into two products indexed by $\cI_B$ and $\cI_{A \setminus B}$ respectively. Then, starting from Definition~\ref{def: sigma I} and applying Fubini's Theorem, we get:
\begin{equation*}
\sigma_\cI\parentheses*{\strut (\cG_I)_{I \in \cI},(\theta_I)_{I \in \cI},\Sigma} = \prod_{C \in \brackets{B,A\setminus B}} \int_{\prod_{I \in \cI_C} \cG_I} \parentheses*{\prod_{I \in \cI_C}\sqrt{\theta_I(P_I)}}\frac{e^{-\frac{1}{2}\prsc{\Sigma_C^{-1}\uP_{\cI_C}}{\uP_{\cI_C}}}}{\sqrt{\det(2\pi \Sigma_C)}}\dx \uP_{\cI_C}.\qedhere
\end{equation*}
\end{proof}

\begin{lem}[Vanishing of the differential]
\label{lem: vanishing D sigma I}
Let $A$ be finite, let $B$ be a proper subset of $A$ and $\cI \in \cP_A$ be such that $\cI \leq \brackets{B, A \setminus B}$. For all $w \in \cM_\cI$, for all $\Sigma \in \cS_{\cI,B}^+$ and $\Lambda\in \cS_{\cI,B}^\perp$, we have $D_{(w,\Sigma)}\sigma_\cI\cdot \Lambda=0$.
\end{lem}

\begin{proof}
Recall that $\sigma_\cI \in \cC^{0,\infty}(\cM_\cI\times \cS_\cI^+)$ and that by differential we mean differential with respect to the second variable. Let $w=\parentheses*{(\cG_I)_{I \in \cI},(\theta_I)_{I \in \cI}} \in \cM_\cI$, then $\sigma_\cI(w,\cdot)=\varphi\psi$, where
\begin{align*}
&\varphi:\Sigma \mapsto \frac{1}{\sqrt{\det(2\pi\Sigma)}} & &\text{and} & &\psi:\Sigma \mapsto \int_{\prod_{I \in \cI} \cG_I} \parentheses*{\prod_{I \in \cI}\sqrt{\theta_I(P_I)}}\exp\parentheses*{-\frac{1}{2}\prsc{\Sigma^{-1}\uP}{\uP}}\dx \uP
\end{align*}
are smooth maps from $\cS_\cI^+$ to $[0,+\infty)$. Let $\Sigma= \parentheses*{\Sigma_I^J}_{I,J \in \cI} \in \cS_{\cI,B}^+$ and $\Lambda= \parentheses*{\Lambda_I^J}_{I,J \in \cI} \in \cS_{\cI,B}^\perp$. By Leibniz' Rule, in order to prove that $D_{(w,\Sigma)}\sigma_\cI \cdot \Lambda =0$ it is enough to prove that $D_\Sigma\varphi\cdot \Lambda = 0$ and $D_\Sigma\psi\cdot \Lambda = 0$. This is what we do in the following two steps.

\paragraph*{Step 1: Differential of $\varphi$.}
The differential of $\det$ at $2\pi\Sigma$ is $M \mapsto \det(2\pi\Sigma)\Tr\parentheses*{\Sigma^{-1}\frac{M}{2\pi}}$. By the Chain Rule, we obtain that $D_\Sigma\varphi\cdot \Lambda = -\frac{1}{2}\frac{\Tr(\Sigma^{-1}\Lambda)}{\sqrt{\det(2\pi\Sigma)}}$.

As above, let $\Sigma_C = \parentheses*{\Sigma_I^J}_{I,J \in \cI_C}$ for $C \in \brackets{B, A\setminus B}$. We have $\Sigma = \parentheses*{\begin{smallmatrix} \Sigma_B & 0 \\ 0 & \Sigma_{A \setminus B}\end{smallmatrix}}$ since $\Sigma \in \cS_{\cI,B}$. Let us also denote by $\Lambda_0 = \parentheses*{\Lambda_I^J}_{I\in \cI_{A \setminus B}, J \in \cI_B}$, so that $\Lambda = \parentheses*{\begin{smallmatrix} 0 & \Lambda_0^* \\ \Lambda_0 & 0\end{smallmatrix}}$ since $\Lambda \in \cS_{\cI,B}^\perp$. Then
\begin{equation*}
\Tr(\Sigma^{-1}\Lambda) = \Tr\parentheses*{\begin{pmatrix}\Sigma_B^{-1} & 0 \\ 0 & \Sigma_{A \setminus B}^{-1}\end{pmatrix}\begin{pmatrix} 0 & \Lambda_0^* \\ \Lambda_0 & 0\end{pmatrix}} = \Tr\begin{pmatrix} 0 & \Sigma_B^{-1} \Lambda_0^* \\ \Sigma_{A \setminus B}^{-1} \Lambda_0 & 0\end{pmatrix}=0,
\end{equation*}
and finally $D_\Sigma\varphi\cdot \Lambda =0$.

\paragraph*{Step 2: Differential of $\psi$.}
We will prove below that, for all $M \in \cS_\cI$,
\begin{equation}
\label{eq: differential psi}
D_\Sigma\psi\cdot M = \frac{1}{2}\int_{\prod_{I \in \cI} \cG_I} \parentheses*{\prod_{I \in \cI}\sqrt{\theta_I(P_I)}}\prsc*{\Sigma^{-1}M\Sigma^{-1}\uP}{\uP}\exp\parentheses*{-\frac{1}{2}\prsc{\Sigma^{-1}\uP}{\uP}}\dx \uP.
\end{equation}

With the same notation as above, we have $\Sigma^{-1}\Lambda\Sigma^{-1} = \parentheses*{\begin{smallmatrix} 0 & \Sigma_B^{-1} \Lambda_0^* \Sigma_{A \setminus B}^{-1} \\ \Sigma_{A \setminus B}^{-1} \Lambda_0\Sigma_B^{-1} & 0\end{smallmatrix}}$. Hence, for all $\uP \in \prod_{I \in \cI}\cG_I$, we have $\prsc*{\Sigma^{-1}\Lambda\Sigma^{-1}\uP}{\uP} = 2 \prsc{\Sigma_{A \setminus B}^{-1} \Lambda_0\Sigma_B^{-1} \uP_{\cI_B}}{\uP_{\cI_{A \setminus B}}}$. Assuming that~\eqref{eq: differential psi} holds for $M=\Lambda$, the integrand is then an odd function of $\uP_{\cI_B}=(P_I)_{I \in \cI_B}$ when $\uP_{\cI_{A \setminus B}}$ is fixed. Indeed, the $(\theta_I)_{I \in \cI_B}$ are homogeneous polynomial functions of even degree. Thus the integral of this function vanishes and $D_\Sigma\psi\cdot \Lambda=0$.

Let us now prove that~\eqref{eq: differential psi} holds. Let $\epsilon>0$ be small enough that $\Sigma+\B_\epsilon\subset \cS_\cI^+$, where $\B_\epsilon$ denotes the closed ball of center $0$ and radius $\epsilon$ in $\cS_\cI$. Recall that the differential of the smooth map $M \mapsto M^{-1}$ at $\Sigma$ is $M \mapsto -\Sigma^{-1}M\Sigma^{-1}$. By Taylor's Theorem, there exists $C>0$ such that, for all $M \in \B_\epsilon$, we have $\Norm{(\Sigma+M)^{-1}-\Sigma^{-1}}\leq C\Norm{M}$ and $\Norm{(\Sigma+M)^{-1}-\Sigma^{-1}+\Sigma^{-1}M\Sigma^{-1}}\leq C\Norm{M}^2$. Since $\Sigma$ is positive-definite, up to increasing $C$, we can assume that $\prsc{\Sigma^{-1}\uP}{\uP}\geq \frac{1}{C}\Norm{\uP}^2$ for all $\uP \in \prod_{I \in \cI}\cG_I$. Finally, up to decreasing $\epsilon$, we can assume that $\epsilon \leq \frac{1}{2C^2}$.

The second order Taylor expansion of $\exp$ at $0$ yields that, for all $\uP \in \prod_{I \in \cI}\cG_I$ and $M \in \B_\epsilon$:
\begin{multline*}
\norm*{\exp\parentheses*{-\frac{1}{2}\prsc*{((\Sigma+M)^{-1}-\Sigma^{-1})\uP}{\uP}}-1+\frac{1}{2}\prsc*{((\Sigma+M)^{-1}-\Sigma^{-1})\uP}{\uP}}\\
\begin{aligned}
&\leq \norm*{\prsc*{((\Sigma+M)^{-1}-\Sigma^{-1})\uP}{\uP}}^2 e^{\frac{1}{2}\norm*{\prsc*{((\Sigma+M)^{-1}-\Sigma^{-1})\uP}{\uP}}}\\
&\leq \Norm*{(\Sigma+M)^{-1}-\Sigma^{-1}}^2\Norm{\uP}^4e^{\frac{1}{2}\Norm*{(\Sigma+M)^{-1}-\Sigma^{-1}}\Norm{\uP}^2}\leq C^2 \Norm{M}^2\Norm{\uP}^4 e^{\frac{C\epsilon}{2}\Norm{\uP}^2}.
\end{aligned}
\end{multline*}
Thus,
\begin{multline*}
\norm*{e^{-\frac{1}{2}\prsc*{(\Sigma+M)^{-1}\uP}{\uP}}-e^{-\frac{1}{2}\prsc*{\Sigma^{-1}\uP}{\uP}}-\frac{1}{2}e^{-\frac{1}{2}\prsc*{\Sigma^{-1}\uP}{\uP}}\prsc*{\Sigma^{-1}M\Sigma^{-1}\uP}{\uP}}\\
\begin{aligned}
\leq& e^{-\frac{1}{2}\prsc*{\Sigma^{-1}\uP}{\uP}}\norm*{\exp\parentheses*{-\frac{1}{2}\prsc*{((\Sigma+M)^{-1}-\Sigma^{-1})\uP}{\uP}}-1+\frac{1}{2}\prsc*{((\Sigma+M)^{-1}-\Sigma^{-1})\uP}{\uP}}\\
&+e^{-\frac{1}{2}\prsc*{\Sigma^{-1}\uP}{\uP}}\norm*{\prsc*{((\Sigma+M)^{-1}-\Sigma^{-1}+\Sigma^{-1}M\Sigma^{-1})\uP}{\uP}}\\
\leq& C^2 \Norm{M}^2\Norm{\uP}^4 e^{-\frac{\Norm{\uP}^2}{2C}\parentheses*{1- C^2\epsilon}}+e^{-\frac{\Norm{\uP}^2}{2C}}C\Norm*{M}^2\Norm*{\uP}^2 \leq \Norm{M}^2 e^{-\frac{\Norm{\uP}^2}{4C}}\parentheses*{1+C\Norm{\uP}^2}^2.
\end{aligned}
\end{multline*}
Since $(\theta_I)_{I \in \cI}$ are polynomials in $\uP$, multiplying the previous terms by $\prod_{I \in \cI}\sqrt{\theta_I(P_I)}$ and integrating over $\prod_{I \in \cI} \cG_I$ yields:
\begin{equation*}
\norm*{\psi(\Sigma+M) - \psi(\Sigma) -\frac{1}{2} \int_{\prod_{I \in \cI} \cG_I} \parentheses*{\prod_{I \in \cI}\sqrt{\theta_I(P_I)}}\prsc*{\Sigma^{-1}M\Sigma^{-1}\uP}{\uP}e^{-\frac{1}{2}\prsc*{\Sigma^{-1}\uP}{\uP}}\dx \uP} =O(\Norm{M}^2)
\end{equation*}
as $M \to 0$. This proves that $D_\Sigma\psi$ is indeed given by Equation~\eqref{eq: differential psi}, and concludes the proof.
\end{proof}

%%%%%%%%%%%%%%%%%%%%%%%%%%%%%%%%%%%%%%%%%%%%%%%%%%%%%%%%%%%%%%%%%%%%%%%%%%%%%%%%%%%%%%%%%%%%%%%%%%%%%%%

\subsection{Kac--Rice densities for cumulants}
\label{subsec: Kac--Rice densities for cumulants}

In the previous section, we dealt with Kac--Rice densities for factorial moments. The goal of the present one is to introduce densities that play the same role for cumulants. In all this section, we consider a convex open set $U \subset \R^d$ and $k \in \ssquarebrackets{1}{d}$.

\begin{dfn}[Cumulant Kac--Rice density]
\label{def: F A}
Let $A$ be a non-empty finite set and $f\in \cC^1(U,\R^k)$ be a centered Gaussian field such that, for all $(x_a)_{a \in A} \notin \diag$, the Gaussian vector $\parentheses*{f(x_a)}_{a \in A}$ is non-degenerate. We define the \emph{cumulant Kac--Rice density} $\cF_A(f,\cdot):U^A \setminus \diag \to \R$ by:
\begin{equation*}
\forall \ux = (x_a)_{a \in A} \in U^A \setminus \diag, \qquad \cF_A(f,\ux) = \sum_{\cJ \in \cP_A} \mu_\cJ \prod_{J \in \cJ} \rho_J(f,\ux_J),
\end{equation*}
where the $(\rho_J)$ were defined in Definition~\ref{def: rho A}, and $\mu_\cJ= (-1)^{\norm{\cJ}-1}\parentheses*{\norm{\cJ}-1}!$ for all $\cJ \in \cP_A$. If $A = \ssquarebrackets{1}{p}$, we let $\cF_p=\cF_A$.
\end{dfn}

Like the Kac--Rice density $\rho_A(f,\cdot)$ from Definition~\ref{def: rho A}, the function $\cF_A(f,\cdot)$ is a priori singular along the diagonal. In order to work with this singularity, we will prove a result analogous to Lemma~\ref{lem: Kac-Rice revisited} for $\cF_A(f,\cdot)$. Recalling the notation for blocks with respect to a coarser partition introduced in Definition~\ref{def: block coarser partition}, we define the following.

\begin{dfn}[Parameter set of projections]
\label{def: L J}
Let $A$ be a non-empty finite set and let $\cI, \cJ \in \cP_A$ be such that $\cJ \leq \cI$, we denote by
\begin{equation*}
\cL_{\cJ,\cI}^\dagger = \prod_{J \in \cJ} \cL^\dagger\parentheses*{\R_{2\norm{[J]_\cI}-1}[X]^k,\R_{2\norm{J}-1}[X]^k}.
\end{equation*}
\end{dfn}

Let $\cI$ and $\cJ \in \cP_A$ be such that $\cJ \leq \cI$. Let $J \in \cJ$ and $I = [J]_\cI \in \cI$. Recalling Definition~\ref{def: projection operators} and Lemma~\ref{lem: prop Pi BA}.\ref{item: Pi linear surj}, we have $\Pi_{2J}^{2I}(\cdot,\ux_{2I}) \in \cL^\dagger\parentheses*{\R_{2\norm{I}-1}[X]^k,\R_{2\norm{J}-1}[X]^k}$ for all $\ux_I \in (\R^d)^I$. Typical elements of $\cL^\dagger_{\cJ,\cI}$ we are interested in are those of the form $\parentheses*{\Pi_{2J}^{2[J]_\cI}(\cdot,\ux_{2[J]_\cI})}_{J \in \cJ}$, where $\ux \in (\R^d)^A$.

Let $\cI$, $\cJ$ and $\cK \in \cP_A$ be such that $\cI \wedge \cK = \cJ$, see Definition~\ref{def: meet}. For all $K \in \cK$ and $\parentheses*{\Pi_J}_{J \in \cJ} \in \cL_{\cJ,\cI}^\dagger$, we denote by $\bigoplus_{J \in \cJ_K} \Pi_J \in \cL^\dagger\parentheses*{\prod_{I \in \cI} \R_{2\norm{I}-1}[X]^k,\prod_{J \in \cJ_K} \R_{2\norm{J}-1}[X]^k}$ the map defined by:
\begin{equation}
\label{eq: abuse of notation}
\bigoplus_{J \in \cJ_K} \Pi_J:\parentheses*{P_I}_{I \in \cI} \longmapsto \parentheses*{\Pi_J(P_{[J]_{\cI}})}_{J \in \cJ_K}.
\end{equation}
Note that there is a slight abuse of notation with respect to Definition~\ref{def: block diagonal operators}. Indeed, we have $\cJ_K = \brackets*{J \in \cJ \mvert J \subset K} = \brackets*{I \cap K \mvert I \in \cI} \setminus \brackets{\emptyset}$ so that, formally, the source space of $\bigoplus_{J \in \cJ_K}\Pi_J$ should be a product indexed by $\brackets*{[J]_\cI \mvert J \in \cJ_K}=\brackets*{I \in \cI\mvert I \cap K \neq \emptyset}$ instead of $\cI$. Here we first implicitly project onto $\prod_{\brackets*{I \in \cI \mvert I \cap K\neq \emptyset}} \R_{2\norm{I}-1}[X]^k$, then apply the block-diagonal operator with diagonal blocks $\parentheses*{\Pi_J}_{J \in \cJ_K}$, which is surjective since its diagonal blocks are.

\begin{dfn}[Field-dependent part of the cumulant density]
\label{def: F I J}
Let $A$ be a non-empty finite set and let $\cI, \cJ \in \cP_A$ be such that $\cJ \leq \cI$, we define a function $F_{\cI,\cJ}:\cM_\cJ \times \cL^\dagger_{\cJ,\cI} \times \cS_\cI^+ \to \R$ by
\begin{multline*}
F_{\cI,\cJ}:\parentheses*{\parentheses*{\cG_J}_{J \in \cJ},\parentheses*{\theta_J}_{J \in \cJ},\parentheses*{\Pi_J}_{J \in \cJ},\Sigma} \longmapsto\\
\sum_{\brackets*{\cK \in \cP_A\mvert \cI \wedge \cK= \cJ}} \mu_\cK \prod_{K \in \cK}\sigma_{\cJ_K}\parentheses*{\parentheses*{\cG_J}_{J \in \cJ_K},\parentheses*{\theta_J}_{J \in \cJ_K},\parentheses{\textstyle\bigoplus_{J \in \cJ_K}\Pi_J}\Sigma\parentheses{\bigoplus_{J \in \cJ_K}\Pi_J}^*},
\end{multline*}
where the maps $(\sigma_{\cJ_K})$ were defined in Definition~\ref{def: sigma I}.
\end{dfn}

The cumulant Kac--Rice density can be written in terms of the field-dependent parts we just introduced. This relation uses the functions introduced in Definitions~\ref{def: projection operators}, \ref{def: G kernel ev} and~\ref{def: N theta}.

\begin{lem}[Relation between $\cF_A$ and $F_{\cI,\cJ}$]
\label{lem: relation between cumulant densities}
Let $A$ be a non-empty finite set and $\cI \in \cP_A$. Let $f \in \cC^{2\norm{A}-1}(U,\R^k)$ be a centered Gaussian field, and  $\ux \in U^A \setminus \diag$ be such that $\Sigma_{2\cI}(f,\ux_{2A}) \in \cS_\cI^+$, then
\begin{equation*}
\cF_A(f,\ux) = \sum_{\substack{\cJ \in \cP_A\\ \cJ \leq \cI}} \tilde{F}_{\cI,\cJ}(f,\ux) \prod_{J \in \cJ} \Upsilon_J(\ux_J),
\end{equation*}
where $\tilde{F}_{\cI,\cJ}(f,\ux) = F_{\cI,\cJ}\parentheses*{\parentheses*{\cG_J(\ux_J)}_{J \in \cJ},\parentheses*{\theta_J(\ux_J)}_{J \in \cJ},\parentheses*{\Pi_{2J}^{2[J]_\cI}(\cdot,\ux_{2[J]_\cI})}_{J \in \cJ},\Sigma_{2\cI}(f,\ux_{2A})}$.
\end{lem}

\begin{proof}
By Lemma~\ref{lem: relation between notions of non-degeneracy}, since $\Sigma_{2\cI}(f,\ux_{2A}) \in \cS_\cI^+$ we have that $(f(x_a))_{a \in A}$ is non-degenerate, hence $\cF_A(f,\ux)$ is well-defined. We want to apply Lemma~\ref{lem: Kac-Rice revisited} to each factor in the definition of $\cF_A(f,\ux)$. For this, we first need to check the non-degeneracy of a family of variance operators.

Let $B \subset A$ be non-empty. For all $J \in \cI_B$, there exists a unique $I \in \cI$ such that $J = I \cap B$, and $\oK(f,\ux_{2J}) = \Pi_{2J}^{2I}\parentheses*{\oK(f,\ux_{2I}),\ux_{2I}}$ by Lemma~\ref{lem: prop Pi BA}.\ref{item: Pi and Kergin}. Hence,
\begin{equation}
\label{eq: projection variance Kac Rice cumulants}
\begin{aligned}
\Sigma_{2\cI_B}(f,\ux_{2B}) &= \var{\parentheses*{\oK(f,\ux_{2J}}_{J \in \cI_B}}= \var{\parentheses*{\Pi_{2(I \cap B)}^{2I}\parentheses*{\oK(f,\ux_{2I}),\ux_{2I}}}_{\brackets{I \in \cI \mid I \cap B \neq \emptyset}}}\\
&= \parentheses*{\bigoplus_{\brackets{I \in \cI \mid I \cap B \neq \emptyset}} \Pi_{2(I \cap B)}^{2I}(\cdot,\ux_{2I})} \Sigma_{2\cI}(f,\ux_{2A}) \parentheses*{\bigoplus_{\brackets{I \in \cI \mid I \cap B \neq \emptyset}} \Pi_{2(I \cap B)}^{2I}(\cdot,\ux_{2I})}^*,
\end{aligned}
\end{equation}
with the same abuse of notation as in~\eqref{eq: abuse of notation}. By Lemma~\ref{lem: prop Pi BA}.\ref{item: Pi linear surj}, for all $I \in \cI$ such that $I \cap B \neq \emptyset$ the map $\Pi_{2(I \cap B)}^{2I}(\cdot,\ux_{2I})$ is surjective. Hence, the direct sum of these maps is surjective and, since $\Sigma_{2\cI}(f,\ux_{2A}) \in \cS_\cI^+$, we have $\Sigma_{2\cI_B}(f,\ux_{2B}) \in \cS_{\cI_B}^+$ for all non-empty $B \subset A$.

Then, for all $\cK \in \cP_A$, we can use Lemma~\ref{lem: Kac-Rice revisited} for each block of $\cK$. This yields:
\begin{align*}
\cF_A(f,\ux) &= \sum_{\cK \in \cP_A} \mu_\cK \prod_{K \in \cK}\rho_K(f,\ux_K) = \sum_{\cK \in \cP_A} \mu_\cK \prod_{K \in \cK}\parentheses*{\prod_{J \in \cI_K} \Upsilon_J(\ux_J)}\tilde{\sigma}_{\cI_K}(f,\ux_K)\\
&= \sum_{\cK \in \cP_A} \mu_\cK \parentheses*{\prod_{J \in \cI \wedge \cK} \Upsilon_J(\ux_J)}\parentheses*{\prod_{K \in \cK} \tilde{\sigma}_{\cI_K}(f,\ux_K)}\\
&= \sum_{\substack{\cJ \in \cP_A\\ \cJ \leq \cI}} \parentheses*{\prod_{J \in \cJ} \Upsilon_J(\ux_J)} \sum_{\substack{\cK \in \cP_A\\ \cI \wedge \cK = \cJ}}\mu_\cK \prod_{K \in \cK} \tilde{\sigma}_{\cI_K}(f,\ux_K),
\end{align*}
where the third equality follows from $\cI \wedge \cK = \bigsqcup_{K \in \cK}\cI_K$. Then, if $\cI$, $\cJ$ and $\cK$ are such that $\cI \wedge \cK = \cJ$, for all $K \in \cK$ we have $\cI_K = \brackets*{I \cap K \mvert I \in \cI} \setminus \brackets{\emptyset} = \cJ_K$, and Equation~\eqref{eq: projection variance Kac Rice cumulants} yields:
\begin{equation*}
\Sigma_{2\cI_K}(f,\ux_{2K}) = \parentheses*{\bigoplus_{J \in \cJ_K} \Pi_{2J}^{2[J]_\cI}(\cdot,\ux_{2[J]_\cI})} \Sigma_{2\cI}(f,\ux_{2A}) \parentheses*{\bigoplus_{J \in \cJ_K} \Pi_{2J}^{2[J]_\cI}(\cdot,\ux_{2[J]_\cI})}^*.
\end{equation*}
Finally, recalling the definitions of $\tilde{F}_{\cI,\cJ}(f,\ux)$ and $\parentheses*{\tilde{\sigma}_{\cI_K}(f,\ux_K)}_{\emptyset \neq K \subset A}$ we recognize that
\begin{equation*}
\sum_{\brackets*{\cK \in \cP_A\mvert \cI \wedge \cK = \cJ}}\mu_\cK \prod_{K \in \cK} \tilde{\sigma}_{\cI_K}(f,\ux_K) = \tilde{F}_{\cI,\cJ}(f,\ux).\qedhere
\end{equation*}
\end{proof}

Let $U \subset \R^d$ be a convex open subset and $f:U \to \R^k$ be a centered Gaussian field. Let $\Omega \subset U$ be open, we let $\nu$ denote the Riemannian volume measure of $Z=f^{-1}(0)\cap \Omega$, and define its linear statistics by~\eqref{eq: def linear statistics}.

\begin{dfn}[Cumulants of linear statistics]
\label{def: cumulants of linear statistics}
Let $\uphi=(\phi_a)_{a \in A}$ be a finite family of Borel measurable functions on $\Omega$. We denote by $\kappa(\nu)(\uphi) = \kappa\parentheses*{\strut \parentheses{\prsc{\nu}{\phi_a}}_{a \in A}}$ whenever this cumulant is well-defined, see Definition~\ref{def: moments and cumulants}.
\end{dfn}

Let us express these cumulants in terms of the cumulant Kac--Rice densities. Recall that we defined $\cP_{A,n}$ in Definition~\ref{def: P A n}.

\begin{prop}[Kac--Rice formula for cumulants]
\label{prop: Kac-Rice cumulants}
Let $A\neq \emptyset$ be finite and $q =\max(2,2\norm{A}-1)$. Let $f \in \cC^q(U,\R^k)$ be a centered Gaussian field and $\eta >0$ be such that $\Sigma_{2\cI}(f,\ux_{2A}) \in \cS_\cI^+$ for all $\cI \in \cP_A$ and $\ux \in \diag_{\cI,\eta} \cap \Omega^A$. Let $\uphi=\parentheses*{\phi_a}_{a \in A}$ be Borel-measurable functions. We assume that for all non-empty $B \subset A$, for all $\cK \in \cP_{B,d-k}$, for all $\cI, \cJ \in \cP_\cK$ such that $\cJ \leq \cI$, the function
\begin{equation*}
\ux \longmapsto \parentheses*{\phi^\otimes_B\circ \iota_{\cK}(\ux)} \tilde{F}_{\cI,\cJ}(f,\ux)\prod_{J \in \cJ}  \Upsilon_J(\ux_J)
\end{equation*}
is integrable over $\Omega^{\cK} \cap \diag_{\cI,\eta}$. Then,
\begin{align*}
\kappa(\nu)(\uphi) &= \sum_{\cK \in \cP_{A,d-k}} \int_{\ux \in \Omega^\cK} \parentheses*{\phi^\otimes_A\circ \iota_{\cK}(\ux)} \cF_{\cK}(f,\ux) \dx \ux\\
&= \sum_{\cK \in \cP_{A,d-k}} \ \sum_{\substack{\cI,\cJ \in \cP_\cK\\ \cJ \leq \cI}}\ \int_{\ux \in \Omega^\cK \cap \diag_{\cI,\eta}} \parentheses*{\phi^\otimes_A\circ \iota_{\cK}(\ux)} \tilde{F}_{\cI,\cJ}(f,\ux)\prod_{J \in \cJ}  \Upsilon_J(\ux_J) \dx \ux.
\end{align*}
\end{prop}

\begin{proof}
There are two parts to the proof. The first one is to check that our assumptions ensure that we can use Theorem~\ref{thm: Kac-Rice} when we need to. The second one is computational.

\paragraph*{Step 1: Checking the hypotheses of Theorem~\ref{thm: Kac-Rice}.}
For all $\ux \in \Omega^A$, there exists $\cI \in \cP_A$ such that $\ux \in \diag_{\cI,\eta}$, see Definition~\ref{def: diag I eta}. Then, by Lemma~\ref{lem: relation between notions of non-degeneracy}, for all $\ux \in \Omega^A \setminus \diag$ we have that $\parentheses*{f(x_a)}_{a \in A}$ is non-degenerate. Moreover, $\parentheses*{f(x),D_xf}$ is non-degenerate for all $x \in \Omega$. Thus the non-degeneracy hypotheses in Theorem~\ref{thm: Kac-Rice} are satisfied. Let us focus on the integrability ones.

Let $B \subset A$ be non-empty and $\cK \in \cP_{B,d-k}$. Let $\ux \in \Omega^{\cK} \setminus \diag$. By Definition~\ref{def: F A}, for all non-empty $I \subset \cK$, we have:
\begin{equation*}
\cF_I(f,\ux_I) = \sum_{\cJ \in \cP_I} \mu_\cJ \prod_{J \in \cJ}\rho_J(f,\ux_J).
\end{equation*}
By Möbius inversion, see Proposition~\ref{prop: Mobius inversion}, we deduce that $\rho_{\cK}(f,\ux) = \sum_{\cI \in \cP_\cK} \prod_{I \in \cI} \cF_I(f,\ux_I)$. Using Lemma~\ref{lem: bijections partitions}.\ref{item: bijection coarser}, we can re-index this sum by the set of partitions of $B$ that are coarser than~$\cK$. That is, for all $\ux \in \Omega^{\cK} \setminus \diag$, we have $\rho_{\cK}(f,\ux) = \sum_{\cJ \geq \cK} \prod_{J \in \cJ} \cF_{\cK_J}(f,\ux_{\cK_J})$. Hence,
\begin{align*}
\phi^\otimes_B\parentheses*{\iota_{\cK}(\ux)} \rho_{\cK}(f,\ux) &= \parentheses*{\prod_{K \in \cK} \prod_{a \in K}\phi_a(x_K)} \sum_{\substack{\cJ \in \cP_B \\ \cJ \geq \cK}} \prod_{J \in \cJ} \cF_{\cK_J}(f,\ux_{\cK_J})\\
&= \sum_{\substack{\cJ \in \cP_B \\ \cJ \geq \cK}} \prod_{J \in \cJ} \parentheses*{\cF_{\cK_J}(f,\ux_{\cK_J}) \prod_{K \in \cK_J} \prod_{a \in K}\phi_a(x_K)}.
\end{align*}
By Fubini's Theorem, in order to prove that $\parentheses*{\phi^\otimes_B \circ \iota_\cK} \rho_\cK(f,\cdot) \in L^1(\Omega^\cK)$, it is enough to prove that $\parentheses*{\phi^\otimes_J \circ \iota_{\cK_J}} \cF_{\cK_J}(f,\cdot) \in L^1(\Omega^{\cK_J})$ for all $\cJ \geq \cK$ in $\cP_B$ and all $J \in \cJ$. Note that $\emptyset \neq J \subset A$ and $\cK_J \in \cP_{J,d-k}$. We want the former integrability condition to be satisfied for all non-empty $B \subset A$ and $\cK \in \cP_{B,d-k}$. By the previous discussion, it is enough to prove that: for all non-empty $B \subset A$, for all $\cK \in \cP_{B,d-k}$, we have $\parentheses*{\phi^\otimes_B \circ \iota_\cK} \cF_{\cK}(f,\cdot) \in L^1(\Omega^\cK)$.

Let $B \subset A$ be non-empty and $\cK \in \cP_{B,d-k}$. Let $\cI \in \cP_{\cK}$, by Lemma~\ref{lem: relation between cumulant densities} we have :
\begin{equation*}
\forall \ux \in \Omega^{\cK} \cap \diag_{\cI,\eta}, \qquad \parentheses*{\phi^\otimes_B \circ \iota_\cK(\ux)} \cF_{\cK}(f,\ux) = \sum_{\cJ \leq \cI} \parentheses*{\phi^\otimes_B \circ \iota_\cK(\ux)} \tilde{F}_{\cI,\cJ}(f,\ux)\prod_{J \in \cJ}  \Upsilon_J(\ux_J).
\end{equation*}
Our integrability assumptions ensure that each term on the right-hand side is integrable over $U^{\cK} \cap \diag_{\cI,\eta}$, hence so is the left-hand side. This holds for all $\cI \in \cP_\cK$, hence $\parentheses*{\phi^\otimes_B \circ \iota_\cK} \cF_{\cK}(f,\cdot)$ is integrable over $\Omega^\cK$ for all $\cK \in \cP_{B,d-k}$, where $\emptyset \neq B \subset A$. Finally, our integrability assumptions imply that $\parentheses*{\phi^\otimes_B \circ \iota_\cK} \rho_\cK(f,\cdot) \in L^1(\Omega^\cK)$ for all non-empty $B \subset A$ and $\cK \in \cP_{B,d-k}$.

\paragraph*{Step 2: Computation of the cumulants.}
By Definitions~\ref{def: moments and cumulants}, \ref{def: power measure and factorial power measure} and~\ref{def: cumulants of linear statistics} and Lemma~\ref{lem: power vs factorial power}, we have
\begin{align*}
\kappa(\nu)(\uphi) &= \sum_{\cJ \in \cP_A} \mu_\cJ \prod_{J \in \cJ}\esp{\prod_{j \in J}\prsc{\nu}{\phi_j}} = \sum_{\cJ \in \cP_A} \mu_\cJ \prod_{J \in \cJ} \esp{\prsc*{\nu^J}{\phi_J^\otimes}}\\
&= \sum_{\cJ \in \cP_A} \mu_\cJ \prod_{J \in \cJ} \sum_{\cK_J \in \cP_{J,d-k}} \esp{\prsc*{\nu^{[\cK_J]}}{\phi_J^\otimes \circ \iota_{\cK_J}}}.
\end{align*}
Let $\cJ \in \cP_A$ and $(\cK_J)_{J \in \cJ} \in \prod_{J \in \cJ} \cP_{J,d-k}$. If $d>k$, then $\cK_J = \bigwedge \cP_J$ for all $J \in \cJ$ and $\bigsqcup_{J \in \cJ} \cK_J = \bigwedge \cP_A$ is the only element of $\cP_{A,d-k}$. Otherwise $\cK = \bigsqcup_{J \in \cJ} \cK_J  \in \cP_A = \cP_{A,d-k}$, and each $\cK \leq \cJ$ in $\cP_A$ can be uniquely realized in this way, see Lemma~\ref{lem: bijections partitions}.\ref{item: bijection finer}. Re-indexing the previous sum, we obtain:
\begin{equation*}
\kappa(\nu)(\uphi) = \sum_{\cK \in \cP_{A,d-k}} \sum_{\cJ \geq \cK} \mu_\cJ \prod_{J \in \cJ} \esp{\prsc*{\nu^{[\cK_J]}}{\phi_J^\otimes \circ \iota_{\cK_J}}}.
\end{equation*}
We checked in the first step of the proof that we can apply the Kac--Rice formula, see Theorem~\ref{thm: Kac-Rice}, to compute each factor on the right-hand side. Hence,
\begin{align*}
\kappa(\nu)(\uphi) &= \sum_{\cK \in \cP_{A,d-k}} \sum_{\cJ \geq \cK} \mu_\cJ \prod_{J \in \cJ} \int_{\Omega^{\cK_J}} \parentheses*{\phi_J^\otimes \circ \iota_{\cK_J}(\ux_{\cK_J})} \rho_{\cK_J}(f,\ux_{\cK_J}) \dx \ux_{\cK_J}\\
&= \sum_{\cK \in \cP_{A,d-k}} \int_{\Omega^\cK} \parentheses*{\phi_A^\otimes \circ \iota_{\cK}(\ux)} \sum_{\cJ \geq \cK} \mu_\cJ \prod_{J \in \cJ} \rho_{\cK_J}(f,\ux_{\cK_J}) \dx \ux\\
&= \sum_{\cK \in \cP_{A,d-k}} \int_{\ux \in \Omega^\cK} \parentheses*{\phi^\otimes_A\circ \iota_{\cK}(\ux)} \cF_{\cK}(f,\ux) \dx \ux,
\end{align*}
where we used Lemma~\ref{lem: bijections partitions}.\ref{item: bijection coarser} to replace the sum over $\brackets{\cJ \in \cP_A \mid \cJ \geq \cK}$ by a sum over $\cP_\cK$ and recognized $\cF_\cK(f,\cdot)$, see Definition~\ref{def: F A}. This establishes the first expression of $\kappa(\nu)(\uphi)$.

In order to establish the second one, for each $\cK \in \cP_{A,d-k}$, we split the integral over $\Omega^\cK$ into a sum, indexed by $\cI \in \cP_\cK$, of integrals over $\Omega^\cK \cap \diag_{\cI,\eta}$, see Definition~\ref{def: diag I eta}. Then, the result follows from Lemma~\ref{lem: relation between cumulant densities}.
\end{proof}

The results we established for $\sigma_\cI$ at the end of Section~\ref{subsec: Kac--Rice densities as functions of variance operators} have natural analogues for the functions $F_{\cI,\cJ}$. We state and prove these results in the remainder on the present section.

\begin{lem}[Regularity of $F_{\cI,\cJ}$]
\label{lem: regularity F I J}
Let $A$ be a non-empty finite set and let $\cI, \cJ \in \cP_A$ be such that $\cJ \leq \cI$, then $F_{\cI,\cJ} \in \cC^{0,\infty}\parentheses*{(\cM_\cJ \times \cL_{\cJ,\cI}^\dagger) \times \cS_\cI^+}$.
\end{lem}

\begin{proof}
By Definition~\ref{def: F I J} it is enough to prove that, for all $\cK \in \cP_A$ such that $\cI \wedge \cK = \cJ$ and all $K \in \cK$, the following map belongs to $\cC^{0,\infty}\parentheses*{(\cM_\cJ \times \cL_{\cJ,\cI}^\dagger) \times \cS_\cI^+}$:
\begin{equation}
\label{eq: regularity F I J}
\parentheses*{\strut \parentheses*{\cG_J}_{J \in \cJ},\parentheses*{\theta_J}_{J \in \cJ},\parentheses*{\Pi_J}_{J \in \cJ},\Sigma} \mapsto \sigma_{\cJ_K}\parentheses*{\parentheses*{\cG_J}_{J \in \cJ_K},\parentheses*{\theta_J}_{J \in \cJ_K},\parentheses*{\textstyle\bigoplus_{J \in \cJ_K}\Pi_J}\Sigma\parentheses*{\textstyle\bigoplus_{J \in \cJ_K}\Pi_J}^*}.
\end{equation}

Let $\cK \in \cP_A$ such that $\cI \wedge \cK = \cJ$ and $K \in \cK$. We have $\cJ_K = \brackets*{J \in \cJ \mvert J \subset K} \subset \cJ$ since $\cJ \leq \cK$. Recalling Definition~\ref{def: M I}, the parameter space $\cM_{\cJ_K}$ is then a factor in  $\cM_\cJ$. The projection $\parentheses*{\strut \parentheses*{\cG_J}_{J \in \cJ},\parentheses*{\theta_J}_{J \in \cJ}} \mapsto \parentheses*{\strut \parentheses*{\cG_J}_{J \in \cJ_K},\parentheses*{\theta_J}_{J \in \cJ_K}}$ from $\cM_\cJ$ onto $\cM_{\cJ_K}$ is continuous. The map $\parentheses*{\parentheses*{\Pi_J}_{J \in \cJ},\Sigma} \mapsto \parentheses*{\bigoplus_{J \in \cJ_K}\Pi_J}\Sigma\parentheses*{\bigoplus_{J \in \cJ_K}\Pi_J}^*$ is smooth from $\cL_{\cJ,\cI}^\dagger \times \cS_\cI^+$ to $\cS_{\cJ_K}^+$. Then, by Lemma~\ref{lem: regularity sigma I}, the map~\eqref{eq: regularity F I J} indeed belongs to $\cC^{0,\infty}\parentheses*{(\cM_\cJ \times \cL_{\cJ,\cI}^\dagger) \times \cS_\cI^+}$.
\end{proof}

Recalling the notation introduced in Definition~\ref{def: block-diagonal variance}, we can now prove that the map $F_{\cI,\cJ}$ vanishes at second order along well-chosen subspaces.

\begin{lem}[Vanishing of $F_{\cI,\cJ}$]
\label{lem: vanishing F I J}
Let $A$ be a finite set and let $\cI, \cJ \in \cP_A$ be such that $\cJ \leq \cI$. Let $B$ be a proper subset of $A$ such that $\cI \leq \brackets*{B, A \setminus B}$. Then, for all $w \in \cM_\cJ \times \cL_{\cJ,\cI}^\dagger$, for all $\Sigma \in \cS_{\cI,B}^+$ and $\Lambda \in \cS_{\cI,B}^\perp$ we have $F_{\cI,\cJ}(w,\Sigma)=0$ and $D_{(w,\Sigma)}F_{\cI,\cJ}\cdot \Lambda =0$.
\end{lem}

\begin{proof}
The key idea is that projection operators of the form~\eqref{eq: abuse of notation} preserve the block structure of the operators $\Sigma$ and $\Lambda$, in a sense to be made precise below. Thus, we can deduce the result from Lemmas~\ref{lem: exact clustering} and~\ref{lem: vanishing D sigma I}. We deal first with the vanishing of the function, then with that of its transverse derivatives.

\paragraph*{Step 1: Vanishing of $F_{\cI,\cJ}$.}
For all $I \in \cI$ and $\emptyset \neq K \subset I$, let $\cG_K \in \gr{k\norm{K}}{\R_{2\norm{K}-1}[X]^k}$, $\theta_K \in \S\fP_{2k\norm{K}}\parentheses*{\R_{2\norm{K}-1}[X]^k}$ and $\Pi_K \in \cL^\dagger\parentheses*{\R_{2\norm{I}-1}[X]^k,\R_{2\norm{K}-1}[X]^k}$. For all $\cK \leq \cI$, we denote by $w_{\cK} = \parentheses*{\strut (\cG_K)_{K \in \cK},(\theta_K)_{K \in \cK},(\Pi_K)_{K \in \cK}} \in \cM_{\cK}\times \cL^\dagger_{\cK,\cI}$. In particular, $w_\cJ \in \cM_\cJ \times \cL^\dagger_{\cJ,\cI}$, and every element of $\cM_\cJ \times \cL^\dagger_{\cJ,\cI}$ can be realized in this way.

For all $K \subset A$, let us denote by $\epsilon_K = 1$ if $K \in \cJ$ and $\epsilon_K=0$ otherwise. Then, for all $\cK \in \cP_A$, we have $\prod_{K \in \cK}\epsilon_K =1$ if $\cK = \cJ$ and $\prod_{K \in \cK}\epsilon_K =0$ otherwise. Let $\Sigma \in \cS_{\cI,B}^+$, for every non-empty $K \subset A$, we denote by
\begin{equation*}
m_K =\sigma_{\cI_K}\parentheses*{\parentheses*{\cG_I}_{I \in \cI_K},\parentheses*{\theta_I}_{I \in \cI_K},\parentheses*{\textstyle\bigoplus_{I \in \cI_K}\Pi_I}\Sigma\parentheses*{\textstyle\bigoplus_{I \in \cI_K}\Pi_I}^*} \prod_{I \in \cI_K}\epsilon_I.
\end{equation*}
Then, recalling that $\bigsqcup_{K \in \cK} \cI_K =\cI \wedge \cK$ for all $\cK \in \cP_A$, we compute:
\begin{align*}
\sum_{\cK \in \cP_A} &\mu_\cK \prod_{K \in \cK} m_K= \sum_{\cK \in \cP_A} \mu_\cK \prod_{K \in \cK} \parentheses*{\sigma_{\cI_K}\parentheses*{\parentheses*{\cG_I}_{\cI_K},\parentheses*{\theta_I}_{\cI_K},\parentheses*{\textstyle\bigoplus_{I \in \cI_K}\Pi_I}\Sigma\parentheses*{\textstyle\bigoplus_{I \in \cI_K}\Pi_I}^*} \prod_{I \in \cI_K}\epsilon_I}\\
&= \sum_{\substack{\cJ' \in \cP_A\\ \cJ' \leq \cI}} \parentheses*{\prod_{J \in \cJ'}\epsilon_J} \sum_{\substack{\cK \in \cP_A\\ \cI \wedge \cK= \cJ'}} \mu_\cK \prod_{K \in \cK} \sigma_{\cI_K}\parentheses*{\parentheses*{\cG_I}_{\cI_K},\parentheses*{\theta_I}_{\cI_K},\parentheses*{\textstyle\bigoplus_{I \in \cI_K}\Pi_I}\Sigma\parentheses*{\textstyle\bigoplus_{I \in \cI_K}\Pi_I}^*}\\
&= \sum_{\substack{\cJ' \in \cP_A\\ \cJ' \leq \cI}} \parentheses*{\prod_{J \in \cJ'}\epsilon_J} F_{\cI,\cJ'}(w_{\cJ'},\Sigma) = F_{\cI,\cJ}(w_\cJ,\Sigma).
\end{align*}
Indeed, if $\cI \wedge \cK = \cJ'$ then $\cI_K = \cJ'_K$ for all $K \in \cK$. We will prove below that $m_K = m_{K\cap B}m_{K \setminus B}$ for all $K \subset A$, with the convention that $m_\emptyset=1$. Then,  it will follow from Lemma~\ref{lem: cancellation property} that $F_{\cI,\cJ}(w_\cJ,\Sigma)=\sum_{\cK \in \cP_A} \mu_\cK \prod_{K \in \cK} m_K=0$.

Let $K \subset A$. If $K \cap B =\emptyset$ or $K \setminus B=\emptyset$, then $m_K = m_{K\cap B}m_{K \setminus B}$ because $m_\emptyset=1$. In the following, we denote by $L = K \cap B$ and assume that $\emptyset \neq L \neq K$. For all $I \in \cI_K$, we denote by $\tilde{I}$ the unique block of $\cI$ such that $I = \tilde{I}\cap K$. Since $\cI \leq \brackets*{B,A\setminus B}$, we have either $\tilde{I} \subset B$ or $\tilde{I} \subset (A\setminus B)$. In the first case $I = \tilde{I} \cap K \subset L$, and in the second $I \subset K \setminus L$. Thus $\cI_K \leq \brackets*{L,K\setminus L}$ and $\cI_K = \cI_L \sqcup \cI_{K\setminus L}$. Moreover, $\cI_L=(\cI_K)_L = \brackets*{\tilde{I} \cap K \mvert \tilde{I} \in \cI_B}\setminus\brackets*{\emptyset}$ and $\cI_{K \setminus L}=(\cI_K)_{K\setminus L} = \brackets*{\tilde{I} \cap K \mvert \tilde{I} \in \cI_{A \setminus B}}\setminus\brackets*{\emptyset}$.

Let us denote by $\parentheses*{\Sigma_I^J}_{I,J \in \cI}$ the blocks of $\Sigma$. Then, recalling Equation~\eqref{eq: abuse of notation}, we have:
\begin{equation}
\label{eq: block projection Sigma I}
\parentheses*{\textstyle\bigoplus_{J \in \cI_K}\Pi_J}\Sigma\parentheses*{\textstyle\bigoplus_{J \in \cI_K}\Pi_J}^* = \parentheses*{\Pi_I \circ  \Sigma_{\tilde{I}}^{\tilde{J}} \circ \Pi_J^*}_{I,J \in \cI_K}.
\end{equation}
If $I \in \cI_L$ and $J \in \cI_{K \setminus L}$, then $\tilde{I} \in \cI_B$ and $\tilde{J}\in \cI_{A \setminus B}$, and we have $\Pi_I \circ  \Sigma_{\tilde{I}}^{\tilde{J}} \circ \Pi_J^*=0$ since $\Sigma \in \cS_{\cI,B}$, see Definition~\ref{def: block-diagonal variance}. Hence, $\parentheses*{\bigoplus_{J \in \cI_K}\Pi_J}\Sigma\parentheses*{\bigoplus_{J \in \cI_K}\Pi_J}^* \in \cS_{\cI_K,L}^+$. By Lemma~\ref{lem: exact clustering},
\begin{multline*}
m_K = \sigma_{\cI_L}\parentheses*{\parentheses*{\cG_I}_{I \in \cI_L},\parentheses*{\theta_I}_{I \in \cI_L},\parentheses*{\Pi_I \circ \Sigma_{\tilde{I}}^{\tilde{J}} \circ \Pi_J^*}_{I,J \in \cI_L}}\prod_{I \in \cI_L}\epsilon_I\\
\quad \times\sigma_{\cI_{K\setminus L}}\parentheses*{\parentheses*{\cG_I}_{I \in \cI_{K\setminus L}},\parentheses*{\theta_I}_{I \in \cI_{K\setminus L}},\parentheses*{\Pi_I \circ  \Sigma_{\tilde{I}}^{\tilde{J}} \circ \Pi_J^*}_{I,J \in \cI_{K \setminus L}}}\prod_{I \in \cI_{K \setminus L}}\epsilon_I.
\end{multline*}
Applying~\eqref{eq: block projection Sigma I} with $\cI_K$ replaced by $\cI_L$ and $\cI_{K\setminus L}$ respectively, we recognize this product to be $m_Lm_{K \setminus L} = m_{K \cap B}m_{K \setminus B}$. Thus, $m_K = m_{K \cap B}m_{K \setminus B}$ for all $K \subset A$ and, by Lemma~\ref{lem: cancellation property}, we have $F_{\cI,\cJ}(w_\cJ,\Sigma)=\sum_{\cK \in \cP_A} \mu_\cK \prod_{K \in \cK} m_K=0$.

\paragraph*{Step 2: Vanishing of $D F_{\cI,\cJ}$.}
Let $w=\parentheses*{\strut (\cG_J)_{J \in \cJ},(\theta_J)_{J \in \cJ},(\Pi_J)_{J \in \cJ}} \in \cM_\cJ\times \cL^\dagger_{\cJ,\cI}$ and $\Sigma \in \cS_\cI^+$. Let $\cK \in \cP_A$ such that $\cK\geq \cJ$ and $K \in \cK$, we denote by $w_K'=\parentheses*{\strut (\cG_J)_{J \in \cJ_K},(\theta_J)_{J \in \cJ_K}}$ and $\Sigma_K = \parentheses*{\bigoplus_{J \in \cJ_K}\Pi_J}\Sigma\parentheses*{\bigoplus_{J \in \cJ_K}\Pi_J}^*$. With this notation, Definition~\ref{def: F I J} becomes:
\begin{equation*}
F_{\cI,\cJ}(w,\Sigma) = \sum_{\brackets*{\cK \in \cP_A\mvert \cI \wedge \cK= \cJ}} \mu_\cK \prod_{K \in \cK}\sigma_{\cJ_K}(w'_K,\Sigma_K).
\end{equation*}
Since the $\Sigma \mapsto \Sigma_K$ is linear for all $K \in \cK \geq \cJ$, differentiating the previous expression yields that, for all $w \in \cM_{\cJ}\times \cL^\dagger_{\cJ,\cI}$, all $\Sigma \in \cS_\cI^+$ and all $\Lambda \in \cS_\cI$, we have:
\begin{equation}
\label{eq: differential F I J}
D_{(w,\Sigma)}F_{\cI,\cJ}\cdot \Lambda = \sum_{\brackets*{\cK \in \cP_A\mvert \cI \wedge \cK= \cJ}} \mu_\cK \sum_{K \in \cK} \parentheses*{\prod_{L \in \cK \setminus\brackets{K}}\sigma_{\cJ_L}\parentheses*{w'_L,\Sigma_L}}D_{(w'_K,\Sigma_K)}\sigma_{\cI_K}\cdot \Lambda_K.
\end{equation}

Let us fix $\cK  \in \cP_A$ such that $\cI \wedge\cK= \cJ$ and $K \in \cK$. For all $I \in \cJ_K = \cI_K$, let us denote by $\tilde{I}$ the only element of $\cI$ such that $I = \tilde{I}\cap K$. Writing $\Sigma=\parentheses*{\Sigma_I^J}_{I,J \in \cI}$ and $\Lambda = \parentheses*{\Lambda_I^J}_{I,J \in \cI}$ as block-operators we obtain:
\begin{align*}
\Sigma_K &= \parentheses*{\Pi_I \circ  \Sigma_{\tilde{I}}^{\tilde{J}} \circ \Pi_J^*}_{I,J \in \cI_K} & &\text{and} & \Lambda_K &= \parentheses*{\Pi_I \circ  \Lambda_{\tilde{I}}^{\tilde{J}} \circ \Pi_J^*}_{I,J \in \cI_K},
\end{align*}
as in Equation~\eqref{eq: block projection Sigma I}. Let us assume that $\Sigma \in \cS^+_{\cI,B}$ and $\Lambda \in \cS_{\cI,B}^\perp$. Denoting by $L = K \cap B$, we have as above $\cI_K = \cI_L \sqcup \cI_{K \setminus L}$. For all $I \in \cI_L$ and $J \in \cI_{K \setminus L}$, have $\tilde{I} \in \cI_B$ and $\tilde{J} \in \cI_{A \setminus B}$, and $\Pi_I \circ  \Sigma_{\tilde{I}}^{\tilde{J}} \circ \Pi_J^*=0$. Thus $\Sigma_K \in \cS^+_{\cI_K,L}$. For all $I,J \in \cI_L$ (resp.~$\cI_{K \setminus L}$), we have $\tilde{I},\tilde{J} \in \cI_B$ (resp.~$\cI_{A \setminus B}$), and $\Pi_I \circ  \Lambda_{\tilde{I}}^{\tilde{J}} \circ \Pi_J^*=0$. Thus $\Lambda_K \in \cS_{\cI_K,L}^\perp$. Then, by Lemma~\ref{lem: vanishing D sigma I}, we have $D_{(w'_K,\Sigma_K)}\sigma_{\cI_K}\cdot \Lambda_K=0$.

Finally, if $\Sigma \in \cS_{\cI,B}^+$ and $\Lambda \in \cS_{\cI,B}^\perp$, then each term on the right-hand side of~\eqref{eq: differential F I J} vanishes, and $D_{(w,\Sigma)}F_{\cI,\cJ}\cdot \Lambda =0$.
\end{proof}

Lemma~\ref{lem: vanishing F I J} shows that $F_{\cI,\cJ}$ vanishes at second order along the subspace $\cS_{\cI,B}$, which is defined by the vanishing of the off-diagonal blocks $(\Sigma_{I,J})_{I \in \cI_B, J \in \cI_{A \setminus B}}$ of $\Sigma \in \cS_{\cI}^+$. Using Taylor formulas, one can hope to deduce estimates of the type $F_{\cI,\cJ}(w,\Sigma) = O\parentheses*{\max_{I \in \cI_B, J \in \cI_{A \setminus B}} \Norm{\Sigma_I^J}^2}$ from this result. Actually, we can derive a better estimate by patching together estimates of the previous type for all $B \subset A$ such that $\cI \leq \brackets{B,A \setminus B}$. This is the object of the next section.

%%%%%%%%%%%%%%%%%%%%%%%%%%%%%%%%%%%%%%%%%%%%%%%%%%%%%%%%%%%%%%%%%%%%%%%%%%%%%%%%%%%%%%%%%%%%%%%%%%%%%%%
%%%%%%%%%%%%%%%%%%%%%%%%%%%%%%%%%%%%%%%%%%%%%%%%%%%%%%%%%%%%%%%%%%%%%%%%%%%%%%%%%%%%%%%%%%%%%%%%%%%%%%%

\section{Refined Hadamard lemma}
\label{sec: refined Hadamard lemma}

The purpose of this section to derive a refinement of the Hadamard lemma for smooth functions defined on an open subset of a Euclidean space and vanishing at second order along a well-chosen family of subspaces. The upshot is to deduce from this result an estimate for the functions~$F_{\cI,\cJ}$, see Definition~\ref{def: F I J} and Lemma~\ref{lem: vanishing F I J}.

In section~\ref{subsec: minimal multi-indices}, we introduce a notation of minimality for multi-indices indexed by a finite set. We use this notion of minimal multi-indices to state and prove a refined version of the Hadamard lemma for smooth function on a Euclidean space in Section~\ref{subsec: refined Hadamard lemma in RA}. Finally, in Section~\ref{subsec: refined Hadamard lemma for FIJ}, we deduce a key estimate for $F_{\cI,\cJ}$ from this refined Hadamard lemma.

%%%%%%%%%%%%%%%%%%%%%%%%%%%%%%%%%%%%%%%%%%%%%%%%%%%%%%%%%%%%%%%%%%%%%%%%%%%%%%%%%%%%%%%%%%%%%%%%%%%%%%%

\subsection{Minimal multi-indices}
\label{subsec: minimal multi-indices}

Recall that we introduced multi-indices indexed by $\ssquarebrackets{1}{d}$ in Definition~\ref{def: multi-indices}. In the following, we need to consider multi-indices indexed by a more general finite set. In this section, we first introduce the associated notation. Then, we define a set of minimal multi-indices and prove some of its properties. In all this section, $A$ is a non-empty finite set.

\begin{dfn}[Multi-indices indexed by $A$]
\label{def: multi-indices A}
The set of \emph{multi-indices} indexed by $A$ is $\N^A$. For all $\un =(n_a)_{a \in A} \in \N^A$, we denote by $\un !=\prod_{a \in A} n_a!$ and by $\norm{\un} = \sum_{a \in A} n_a$ its \emph{length}. If $B \subset A$, we let $\un_B=(n_b)_{b \in B} \in \N^B$. Finally, for all $a \in A$, we denote by $\one_a \in \N^A$ the multi-index having a~$1$ in position~$a$ and zeros elsewhere.
\end{dfn}

\begin{rem}
\label{rem: multi-indices B sub A}
If $B \subset A$, there is a canonical injection mapping $\um \in \N^B$ to the only $\un \in \N^A$ such that $\un_B=\um$ and $\un_{A \setminus B}=0$. Using this injection, we consider $\N^B$ as a subset of $\N^A$ in the following.
\end{rem}

Recall that the canonical order on $\N$ induces a partial order on $\N^A$, such that $\um \leq \un$ if and only if $m_a \leq n_a$ for all $a \in A$. Accordingly, we write $\um < \un$ if $\um \leq \un$ and $\um \neq \un$.

\begin{dfn}[Minimal multi-indices]
\label{def: minimal multi-indices}
Let $\cN \subset \N^A$, we say $\un \in \cN$ is \emph{minimal} in $\cN$ if for all $\um \in \cN$ such that $\um \leq \un$ we have $\um =\un$. We denote by $\min(\cN)$ the set of minimal elements of~$\cN$.
\end{dfn}

\begin{lem}[Existence of minimal elements]
\label{lem: well founded order}
Let $\cN \subset \N^A$ and $\un \in \cN$, there exists $\um \in \min(\cN)$ such that $\um \leq \un$.
\end{lem}

\begin{proof}
The set $\cN' = \brackets*{\um \in \cN \mvert \um \leq \un}$ is non-empty and finite. Hence it admits a minimal element, say $\um_0$. If $\um \in \cN$ is such that $\um \leq \um_0$, then $\um \in \cN'$, and hence $\um=\um_0$. Thus $\um_0 \in \min(\cN)$.
\end{proof}

In the following, we consider a family $\cB$ of non-empty subsets of $A$. That is $\cB \subset 2^A$ and $\emptyset \notin \cB$. Note that we do not assume the elements of $\cB$ to be disjoint as subsets of $A$. We can associate with $\cB$ the following set of minimal multi-indices.
\begin{equation}
\label{eq: def NB}
\cN_\cB = \brackets*{\un \in \N^A \mvert \forall B \in \cB, \norm{\un_B} \geq 2}.
\end{equation}

\begin{lem}[Description of $\min(\cN_\cB)$]
\label{lem: description of NB}
If $\cB \subset 2^A$ is such that $\emptyset \notin \cB$, then $\emptyset \neq \min(\cN_\cB) \subset \ssquarebrackets{0}{2}^A$. In particular $\min(\cN_\cB)$ is finite. Moreover, for all $\un \in \min(\cN_\cB)$ and $a \notin \bigcup_{B \in \cB} B$ we have $n_a=0$.
\end{lem}

\begin{proof}
Since $\emptyset \notin \cB$, we have $(2,\dots,2) \in \cN_\cB$. By Lemma~\ref{lem: well founded order}, there exists $\un \in \min(\cN_\cB)$ such that $\un \leq (2,\dots,2)$. In particular $\min(\cN_\cB) \neq \emptyset$.

Let $\un=(n_a)_{a \in A} \in \min(\cN_\cB)$. We define $\um=(m_a)_{a \in A} \in \ssquarebrackets{0}{2}^A$ by $m_a=\min(n_a,2)$ if $a \in \bigcup_{B \in \cB}B$ and $m_a=0$ otherwise. We have $\um \leq \un$. For all $B \in \cB$, if there exists $b \in B$ such that $n_b \geq 2$, then $\norm{\um_B}\geq m_b =2$; otherwise $\um_B=\un_B$, and $\norm{\um_B}=\norm{\un_B}\geq 2$. This proves that $\um \in \cN_\cB$, and hence $\um=\un$ by minimality of $\un$. The conclusion follows.
\end{proof}

\begin{ex}
\label{ex: NB}
Let us give some examples.

\begin{enumerate}
\item \label{item: NB 0} If $\cB = \emptyset$ then $\cN_\cB=\N^A$ and $\min(\cN_\cB)=\brackets{0}$.

\item \label{item: NB 1} If $\cB = \brackets{B}$, then $\min(\cN_\cB)=\brackets*{\un \in \N^B \mvert \norm{\un}=2} = \brackets*{\strut \one_b+ \one_c \mvert b,c \in B} \subset \N^A$.

\item \label{item: NB explicit} If $\cB = \brackets*{\strut \brackets{1,2}, \brackets{1,3}, \brackets{2,3}} \subset 2^{\ssquarebrackets{1}{3}}$ then $\min(\cN_\cB) = \brackets*{\strut (1,1,1), (2,2,0), (2,0,2), (0,2,2)}$.
\end{enumerate}
\end{ex}

%%%%%%%%%%%%%%%%%%%%%%%%%%%%%%%%%%%%%%%%%%%%%%%%%%%%%%%%%%%%%%%%%%%%%%%%%%%%%%%%%%%%%%%%%%%%%%%%%%%%%%%

\subsection{Refined Hadamard lemma in \texorpdfstring{$\R^A$}{}}
\label{subsec: refined Hadamard lemma in RA}

Let $A$ be a non-empty finite set and $\Omega \subset \R^A$ be an open subset. Let $\cM$ be a locally compact topological space. The goal of the present section is to prove a refined version of the Hadamard lemma which is adapted to functions in $\cC^{0,\infty}(\cM \times \Omega)$ vanishing at second order along several subspaces in $\Omega \subset \R^A$, see Definition~\ref{def: C 0 infty}. This is done in Proposition~\ref{prop: refined Hadamard lemma} below.

Consistently with the notation introduced in~Definitions~\ref{def: multi-indices} and~\ref{def: multi-indices A}, we let $\partial_a$ denote the partial derivative with respect to the coordinate indexed by $a \in A$, which acts on $\cC^{0,\infty}(\cM\times \Omega)$. For all $\ux =(x_a)_{a \in A} \in \R^A$ and $\un = (n_a)_{a \in A} \in \N^A$, we denote by $\ux^{\un} = \prod_{a \in A} x_a^{n_a}$. Given $B \subset A$ and $\um \in \N^B$, we have $\ux^{\um} = (\ux_B)^{\um}$ thanks to the inclusion $\N^B \subset \N^A$, see Remark~\ref{rem: multi-indices B sub A}. Similarly, $\partial^{\um}$ acts naturally on $\cCM$.

We start by proving a parametric version of the Hadamard lemma for functions of $\cCM$ vanishing at order $2$ along $\R^{A \setminus B}\times \brackets{0} \subset \R^A$, in the case where $\Omega$ is the product of two balls. In the following, we use the convention that $\R^\emptyset = \brackets{0}$.

\begin{lem}[Parametric Hadamard Lemma]
\label{lem: parametric Hadamard}
Let $B \subset A$ be non-empty. Let $\Omega = U \times U'$, where $U \subset \R^{A \setminus B}$ and $U' \subset \R^B$ are open balls, and $0 \in U'$. Let $F \in \cCM$ be such that: for all $(w,\ux) \in \cM \times \Omega$ such that $\ux_B=0$ and $\uh \in \brackets{0}\times \R^B$, we have $F(w,\ux)=0$ and $D_{(w,\ux)}F\cdot \uh=0$. Then, there exist functions $(F_{\um})_{\um \in \N^B, \norm{\um}=2}$ in $\cCM$ such that
\begin{equation*}
F:(w,\ux) \longmapsto \sum_{\um \in \N^B, \norm{\um}=2} \ux^{\um} F_{\um}(w,\ux).
\end{equation*}
\end{lem}

\begin{proof}
Let $w \in \cM$ and $\ux \in \Omega = U \times U'$. Since $U'$ is convex, we can write the first-order Taylor expansion with integral remainder of the function $\uy \mapsto F(w,\ux_{A \setminus B},\uy)$ between $0$ and $\ux_B$. Our hypotheses imply that its first-order Taylor polynomial at $0$ vanishes, hence
\begin{equation*}
F(w,\ux) = \sum_{\um \in \N^B, \norm{\um}=2} \ux^{\um} \underbrace{\frac{2}{\um !}\int_0^1 (1-t) \partial^{\um} F(w,\ux_{A \setminus B},t\ux_B) \dx t}_{=F_{\um}(w,\ux)}.
\end{equation*}

Let $\un\in \N^A$, the function $(w,\ux,t)\mapsto (1-t)t^{\norm{\un_B}} \partial^{\um+\un} F(w,\ux_{A \setminus B},t\ux_B)$ is continuous. Thus, for any compact subset $\Gamma \subset \cM \times \Omega$ it can be dominated by a constant on $\Gamma \times [0,1]$. This is enough to prove that the parametric integral $F_{\um}$ admits continuous partial derivatives at any order on the interior of $\Gamma$. Since $\cM$ is locally compact, so is $\cM \times \Omega$. Hence $F_{\um} \in \cCM$.
\end{proof}

As in Section~\ref{subsec: minimal multi-indices}, let $\cB \subset 2^A$ be a family of non-empty subsets of $A$. Recall that $\cN_\cB \subset \N^A$ is defined by Equation~\eqref{eq: def NB} and that $\min(\cN_\cB)$ is finite by Lemma~\ref{lem: description of NB}.

\begin{prop}[Refined Hadamard lemma]
\label{prop: refined Hadamard lemma}
Let $F \in \cCM$ be such that, for all $B \in \cB$: for all $(w,\ux) \in \cM \times \Omega$ such that $\ux_B=0$ and $\uh \in \brackets{0}\times\R^B$, we have $F(w,\ux)=0$ and $D_{(w,\ux)}F\cdot \uh=0$. Then there exist functions $\parentheses*{F_{\un}}_{\un \in \min(\cN_\cB)}$ in $\cCM$ such that
\begin{equation*}
F : (w,\ux) \longmapsto \sum_{\un \in \min(\cN_\cB)} \ux^{\un} F_{\un}(w,\ux).
\end{equation*}
\end{prop}

\begin{proof}
The proof is by induction on $\card(\cB)$. If $\cB = \emptyset$ then $\min(\cN_\cB)=\brackets{0}$, see Example~\ref{ex: NB}.\ref{item: NB 0}, and we get the result by letting $F_{0} = F$. Let us assume that $\card(\cB) \geq 1$ and let $\emptyset \neq B \in \cB$. The induction hypothesis applied to $\cB \setminus \brackets{B}$ yields functions $\parentheses*{G_{\un}}_{\un \in \min(\cN_{\cB \setminus \brackets{B}})}$ in $\cCM$ such that $F : (w,\ux) \mapsto \sum_{\un \in \min(\cN_{\cB \setminus \brackets{B}})} \ux^{\un} G_{\un}(w,\ux)$. In order to establish the induction step, we consider three cases.

\paragraph*{Case 1: $\Omega$ is a product of balls intersecting $\R^{A \setminus B}\times \brackets{0}$.}
We assume that $\Omega = U \times U'$, where $U \subset \R^{A \setminus B}$ and $U' \subset \R^B$ are open balls, and that $0 \in U'$.

For all $(w,\ux) \in \cM \times \Omega$, we have $(w,\ux_{A \setminus B},0) \in \cM \times \Omega$, hence
\begin{equation}
\label{eq: Hadamard order 0}
0 = F(w,\ux_{A \setminus B},0) = \sum_{\substack{\un \in \min(\cN_{\cB \setminus \brackets{B}})\\ \norm{\un_B}=0}} \ux^{\un} G_{\un}(w,\ux_{A \setminus B},0).
\end{equation}
Moreover, for all $b \in B$, we have $\partial_b F(w,\ux_{A \setminus B},0) = D_{(w,\ux_{A \setminus B},0)}F\cdot \one_b = 0$, so that:
\begin{equation*}
0 = x_b \partial_b F(w,\ux_{A \setminus B},0) = \sum_{\substack{\un \in \min(\cN_{\cB \setminus \brackets{B}})\\ \un_B=\one_b}} \ux^{\un} G_{\un}(w,\ux_{A \setminus B},0)+\sum_{\substack{\un \in \min(\cN_{\cB \setminus \brackets{B}})\\ \un_B=0}} \ux^{\un} x_b \partial_b G_{\un}(w,\ux_{A \setminus B},0).
\end{equation*}
Summing over $b \in B$, we obtain:
\begin{equation}
\label{eq: Hadamard order 1}
\sum_{\substack{\un \in  \min(\cN_{\cB \setminus \brackets{B}})\\ \norm{\un_B}=1}} \ux^{\un} G_{\un}(w,\ux_{A \setminus B},0)+ \sum_{\substack{\un \in \min(\cN_{\cB \setminus \brackets{B}})\\ \norm{\un_B}=0}} \sum_{b \in B} \ux^{\un} x_b \partial_b G_{\un}(w,\ux_{A \setminus B},0)=0.
\end{equation}

Let $\un \in \min(\cN_{\cB \setminus \brackets{B}})$, the map $(w,\ux) \mapsto G_{\un}(w,\ux)-G_{\un}(w,\ux_{A \setminus B},0)-\sum_{b \in B}x_b \partial_b G_{\un}(w,\ux_{A \setminus B},0)$ satisfies the hypotheses of Lemma~\ref{lem: parametric Hadamard}. Thus, recalling Example~\ref{ex: NB}.\ref{item: NB 1}, there exist functions $\parentheses*{H_{\um,\un}}_{\um \in \min(\cN_{\brackets{B}})}$ in $\cCM$ such that
\begin{equation*}
G_{\un}:(w,\ux) \mapsto G_{\un}(w,\ux_{A \setminus B},0)+\sum_{b \in B}x_b \partial_b G_{\un}(w,\ux_{A \setminus B},0)+\sum_{\um \in \min(\cN_{\brackets{B}})} \ux^{\um} H_{\um,\un}(w,\ux).
\end{equation*}
Then, using Equations~\eqref{eq: Hadamard order 0} and~\eqref{eq: Hadamard order 1}, for all $(w,\ux) \in \cM \times \Omega$ we have
\begin{multline*}
F(w,\ux) = \sum_{\substack{\un \in \min(\cN_{\cB \setminus \brackets{B}})\\ \norm{\un_B}\geq 2}} \ux^{\un}G_{\un}(w,\ux_{A \setminus B},0) +\sum_{\substack{\un \in \min(\cN_{\cB \setminus \brackets{B}})\\ \norm{\un_B}\geq 1}} \sum_{b \in B} \ux^{\un} x_b \partial_b G_{\un}(w,\ux_{A \setminus B},0)\\
+\sum_{\un \in \min(\cN_{\cB \setminus \brackets{B}})} \sum_{\um \in \min(\cN_{\brackets{B}})} \ux^{\um+\un} H_{\um,\un}(w,\ux).
\end{multline*}

Let $\un \in \min(\cN_{\cB \setminus \brackets{B}})$ be such that $\norm{\un_B}\geq 2$. We have $\un \in \cN_\cB$, hence there exists $\up_{\un} \in \min(\cN_\cB)$ such that $\up_{\un} \leq \un$, see Lemma~\ref{lem: well founded order}. Similarly, let $b \in B$ and $\un \in \min(\cN_{\cB \setminus \brackets{B}})$ be such that $\norm{\un_B}\geq 1$. We have $\un \in \cN_{\cB\setminus \brackets{B}}$, hence $\un+\one_b \in \cN_{\cB\setminus \brackets{B}}$. Besides $\norm{(\un+\one_b)_B}=\norm{\un_B}+1 \geq 2$, so that $\un+\one_b \in \cN_\cB$. By Lemma~\ref{lem: well founded order}, there exists $\uq_{\un,b} \in \min(\cN_\cB)$ such that $\uq_{\un,b} \leq \un + \one_b$. Finally, given $\um \in \min(\cN_{\brackets{B}})$ and $\un \in \min(\cN_{\cB \setminus \brackets{B}})$, we have $\um+\un \in \cN_{\cB}$. Hence, there exists $\ur_{\um,\un} \in \min(\cN_\cB)$ such that $\ur_{\um,\un} \leq \um +\un$. For all $\up \in \min(\cN_\cB)$, we define $F_{\up} \in \cCM$ by:
\begin{multline*}
F_{\up}:(w,\ux) \mapsto \sum_{\substack{\un \in \min(\cN_{\cB \setminus \brackets{B}})\\ \norm{\un_B}\geq 2;\ \up_{\un}=\up}}\ux^{\un-\up} G_{\un}(w,\ux_{A \setminus B},0) + \sum_{\substack{\un \in \min(\cN_{\cB \setminus \brackets{B}})\\ \norm{\un_B}\geq 1}}\sum_{\substack{b \in B \\ \uq_{\un,b}=\up}} \ux^{\un+\one_b-\up} \partial_b G_{\un}(w,\ux_{A \setminus B},0) \\
+ \sum_{\un \in \min(\cN_{\cB \setminus \brackets{B}})} \sum_{\substack{\um \in \min(\cN_{\brackets{B}})\\ \ur_{\um,\un}=\up}} \ux^{\um+\un-\up}H_{\um,\un}(w,\ux).
\end{multline*}
Then, $F:(w,\ux) \mapsto \sum_{\up \in \min(\cN_\cB)}\ux^{\up} F_{\up}(w,\ux)$ is the expansion we are looking for in this case.

\paragraph*{Case 2: $\Omega$ is a product of balls disjoint from $\R^{A \setminus B}\times \brackets{0}$.} 
Let us deal with a second special case. We assume that $\Omega = U \times U'$, where $U \subset \R^{A \setminus B}$ and $U' \subset \R^B$ are open balls, and $0 \notin U'$.

In this case, $\ux \mapsto \Norm{\ux_B}_2^2=\sum_{b \in B}x_b^2$ is a positive smooth function on $\Omega$. Recall that we described $\min(\cN_{\brackets{B}})$ in Example~\ref{ex: NB}.\ref{item: NB 1}. Let $\un \in \min(\cN_{\cB \setminus \brackets{B}})$ and $\um \in \min(\cN_{\brackets{B}})$. If there exists $b \in B$ such that $m_b=2$, we define $H_{\um,\un}:(w,\ux) \mapsto \frac{1}{\Norm{\ux_B}_2^2} G_{\un}(w,\ux)$. Otherwise we let $H_{\um,\un}=0$. We have $H_{\um,\un} \in \cCM$ for all $(\um,\un) \in \min(\cN_{\brackets{B}}) \times \min(\cN_{\cB \setminus \brackets{B}})$ and, for all $(w,\ux) \in \cM \times \Omega$,
\begin{equation*}
F(w,\ux)= \sum_{\un \in \min(\cN_{\cB \setminus \brackets{B}})} \ux^{\un} \parentheses*{\sum_{b \in B}x_b^2 \frac{G_{\un}(w,\ux)}{\Norm{\ux_B}_2^2}} = \sum_{\un \in \min(\cN_{\cB \setminus \brackets{B}});\, \um \in \min(\cN_{\brackets{B}})} \ux^{\um+\un} H_{\um,\un}(w,\ux).
\end{equation*}
As above, for all $\um \in\min(\cN_{\brackets{B}})$ and $\un \in \min(\cN_{\cB \setminus \brackets{B}})$, there exists $\ur_{\um,\un} \in \min(\cN_\cB)$ such that $\ur_{\um,\un} \leq \um +\un$. Then, $F:(w,\ux) \mapsto \sum_{\up \in \min(\cN_\cB)}\ux^{\up} F_{\up}(w,\ux)$ where, for all $\up \in \min(\cN_\cB)$, we defined:
\begin{equation*}
F_{\up}:(w,\ux) \longmapsto  \sum_{\un \in \min(\cN_{\cB \setminus \brackets{B}})}\ \sum_{\um \in \min(\cN_{\brackets{B}}); \,  \ur_{\um,\un}=\up} \ux^{\um+\un-\up}H_{\um,\un}(w,\ux).
\end{equation*}
This is the expansion of $F$ we are looking for in this second case.

\paragraph*{Case 3: General case.}
Let us now consider the case where $\Omega$ is an open subset of $\R^A$. In this case, $\Omega$ admits an open cover $\parentheses*{\Omega_i}_{i \in \aleph}$, where each $\Omega_i$ is of a type considered in one of the previous two special cases. Let $\parentheses*{\chi_i}_{i \in \aleph}$ be a smooth partition of unity subordinate to $\parentheses*{\Omega_i}_{i \in \aleph}$, cf.~\cite[Thm.~2.23]{Lee2013}. Given $i \in \aleph$, we can apply one the previous two cases to the restriction of $F$ to $\cM \times \Omega_i$. This yields functions $\parentheses*{F_{i,\un}}_{\un \in \min(\cN_\cB)}$ in $\cC^{0,\infty}\parentheses*{\cM\times\Omega_i}$ such that
\begin{equation*}
\forall (w,\ux) \in \cM \times \Omega_i, \qquad F(w,\ux) = \sum_{\un \in \min(\cN_\cB)} \ux^{\un} F_{i,\un}(w,\ux).
\end{equation*}

Let $\un \in \min(\cN_\cB)$, for all $i \in \aleph$ the function $(w,\ux) \mapsto\chi_i(\ux)F_{i,\un}(w,\ux)$ extends into a function of $\cCM$ vanishing outside $\cM \times \Omega_i$. Since each compact of $\Omega$ intersects the support of finitely many of the $(\chi_i)_{i\in \aleph}$, the map $F_{\un}:(w,\ux) \mapsto \sum_{i \in \aleph} \chi_i(\ux) F_{i,\un}(w,\ux)$ belongs to $\cCM$. Then,
\begin{equation*}
\forall (w,\ux) \in \cM \times \Omega, \qquad F(w,\ux) = \sum_{i \in \aleph} \chi_i(\ux) \parentheses*{\sum_{\un \in \min(\cN_\cB)} \ux^{\un} F_{i,\un}(w,\ux)} = \sum_{\un \in \min(\cN_\cB)} \ux^{\un} F_{\un}(w,\ux).
\end{equation*}
This establishes the induction step in the general case, and concludes the proof.
\end{proof}

%%%%%%%%%%%%%%%%%%%%%%%%%%%%%%%%%%%%%%%%%%%%%%%%%%%%%%%%%%%%%%%%%%%%%%%%%%%%%%%%%%%%%%%%%%%%%%%%%%%%%%%

\subsection{Refined Hadamard lemma for \texorpdfstring{$F_{\cI,\cJ}$}{}}
\label{subsec: refined Hadamard lemma for FIJ}

In this section, we state a version of Proposition~\ref{prop: refined Hadamard lemma} adapted to the maps $F_{\cI,\cJ}$ from Definition~\ref{def: F I J}. As a consequence, we derive an estimate for these maps, which will be fundamental in the proof of Theorem~\ref{thm: cumulants asymptotics for zero sets}. This is done in Proposition~\ref{prop: key estimate FIJ} below, which is a generalization of~\cite[Lem.~3.22]{Gas2023b}.

In all this section, let $A$ be a non-empty finite set, and let $\cI$ and $\cJ \in \cP_A$ be such that $\cJ \leq \cI$. Recall that $\cM_\cJ$, $\cS_\cI^+$ and $\cL_{\cJ,\cI}^\dagger$ were defined in Definitions~\ref{def: M I} and \ref{def: L J} respectively. Recall also that $F_{\cI,\cJ} \in \cC^{0,\infty}\parentheses*{\parentheses{\cM_\cJ \times \cL_{\cJ,\cI}^\dagger}\times \cS_\cI^+}$, see Lemma~\ref{lem: regularity F I J}.

For all $I \in \cI$, we defined the inner product on $\R_{2\norm{I}-1}[X]^k$ so that $\parentheses*{X^\alpha\otimes e_l}_{\norm{\alpha}<2\norm{I}, 1 \leq l \leq k}$ is an orthonormal basis, where $(e_1,\dots,e_k)$ is the standard basis of $\R^k$. This basis is indexed by:
\begin{equation}
\label{eq: def AI}
\cA_I = \brackets*{(I,\alpha,l) \in \brackets{I}\times \N^d \times \ssquarebrackets{1}{d} \mvert \norm{\alpha}<2\norm{I}}.
\end{equation}
Since we use the product Euclidean structure on $\prod_{I \in \cI} \R_{2\norm{I}-1}[X]^k$, we obtain an orthonormal basis of this space by juxtaposing the previous bases of each factor. The resulting basis is indexed by:
\begin{equation}
\label{eq: def A}
\cA = \bigsqcup_{I \in \cI} \cA_I.
\end{equation}

In the following, we identify a self-adjoint operator $\Lambda \in \cS_\cI = \sym\parentheses*{\prod_{I \in \cI} \R_{2\norm{I}-1}[X]^k}$ with its matrix $\parentheses*{\lambda_{ab}}_{a,b \in \cA}$ in the previous basis. The map $\Lambda \mapsto \parentheses*{(\lambda_{aa})_{a \in \cA}, (\sqrt{2}\lambda_{ab})_{\brackets{a,b}\in \cA^{\cwedge}}}$ is well-defined since $\parentheses*{\lambda_{ab}}_{a,b \in \cA}$ is symmetric. In particular, $\tilde{A}=\cA \sqcup \cA^{\cwedge}$ naturally indexes the coefficients on or above the diagonal in the matrix of $\Lambda$, the subset $\cA$ (resp.~$\cA^{\cwedge}$) indexing to diagonal (resp.~off-diagonal) coefficients. One can check that the previous map is an isometric isomorphism from $\cS_\cI$ equipped with the Frobenius inner product (see Definition~\ref{def: Frobenius inner product}) to~$\R^{\tilde{A}}$ with its standard Euclidean structure. Using these Euclidean coordinates, one can see $\cS_\cI^+$ as an open subset in $\R^{\tilde{A}}$.

Let $B$ be a proper subset of $A$ such that $\cI \leq \brackets{B,A\setminus B}$. Recall that $\cI = \cI_B \sqcup \cI_{A \setminus B}$. We denote by $\tilde{B} \subset \tilde{A}$ the set of indices associated with coefficients in the off-diagonal block above the diagonal in the $2\times 2$ block-decomposition induced by the partition $\brackets{B,A \setminus B}$. That is,
\begin{equation}
\label{eq: def B tilde}
\tilde{B} = \bigsqcup_{I \in \cI_B, J \in \cI_{A \setminus B}} \cA_I \times \cA_J \subset \cA^{\cwedge} \subset \tilde{A}, 
\end{equation}
where we used the canonical inclusion from Remark~\ref{rem: product as pairs}. Note that $\tilde{B} \neq \emptyset$. We also denote by
\begin{equation}
\label{eq: def B I tilde}
\tilde{\cB}_\cI = \brackets*{\tilde{B} \mvert \emptyset \varsubsetneq B \varsubsetneq A \ \text{such that} \ \cI \leq \brackets{B,A\setminus B}}.
\end{equation}
Using this notation, a special case of Proposition~\ref{prop: refined Hadamard lemma} is the following. Recall that $\min(\cN_{\tilde{\cB}_\cI})$ is a finite set of multi-indices indexed by $\tilde{A}$, see Equation~\eqref{eq: def NB} and Lemma~\ref{lem: description of NB}.

\begin{cor}[Refined Hadamard lemma for $F_{\cI,\cJ}$]
\label{cor: refined Hadamard lemma for FIJ}
Let $A$ be a non-empty finite set and let $\cI,\cJ \in \cP_A$ be such that $\cJ \leq \cI$. Then there exists a family $\parentheses*{F_{\cI,\cJ,\un}}_{\un \in \min(\cN_{\tilde{\cB}_\cI})}$ of functions in $\cC^{0,\infty}\parentheses*{\parentheses{\cM_\cJ \times \cL_{\cJ,\cI}^\dagger}\times \cS_\cI^+}$ such that, for all $w \in \cM_\cJ \times \cL_{\cJ,\cI}^\dagger$ and $\uSigma \in \cS_\cI^+\subset \R^{\tilde{A}}$ we have:
\begin{equation*}
F_{\cI,\cJ}(w,\uSigma) = \sum_{\un \in \min(\cN_{\tilde{\cB}_\cI})} \uSigma^{\un} F_{\cI,\cJ,\un}(w,\uSigma).
\end{equation*}
\end{cor}

\begin{proof}
As explained previously, we identify any self-adjoint operator $\Lambda \in \cS_\cI$ with the vector of its Euclidean coordinates $\uLambda=\parentheses*{(\lambda_{aa})_{a \in \cA}, (\sqrt{2}\lambda_{ab})_{\brackets{a,b}\in \cA^{\cwedge}}} \in \R^{\tilde{A}}$, where $(\lambda_{ab})_{a,b \in \cA}$ is the matrix of $\Lambda$ in the orthonormal basis discussed above. Under this identification, for any proper $B \subset A$ such that $\cI \leq \brackets{B,A\setminus B}$, the subspace $\cS_{\cI,B}$ from Definition~\ref{def: block-diagonal variance} corresponds to $\brackets{\uLambda\in \R^{\tilde{A}}\mid \uLambda_{\tilde{B}}=0}$, see Equation~\eqref{eq: def B tilde}. Similarly, $\cS_{\cI,B}^\perp$ corresponds to $\brackets{\uLambda\in \R^{\tilde{A}}\mid \uLambda_{\tilde{A}\setminus\tilde{B}}=0}$.

Thus, Lemma~\ref{lem: vanishing F I J} translates into the fact that for all $\tilde{B} \in \tilde{\cB}_\cI$ the following holds: for all $w \in \cM_\cJ \times \cL_{\cJ,\cI}^\dagger$, all $\uSigma \in \cS_\cI^+$ such that $\uSigma_{\tilde{B}}=0$ and all $\uLambda \in \R^{\tilde{A}}$ such that $\uLambda_{\tilde{A}\setminus \tilde{B}}=0$ we have $F_{\cI,\cJ}(w,\uSigma)=0$ and $D_{(w,\uSigma)}F_{\cI,\cJ}\cdot \uLambda=0$. The conclusion follows by Proposition~\ref{prop: refined Hadamard lemma}.
\end{proof}

Let $\cV$ be a non-empty finite set. For any proper $\cW \subset \cV$, recall that $\cW \times (\cV \setminus \cW)$ can be seen as a subset of $\cV^{\cwedge}$ thanks to the inclusion from Remark~\ref{rem: product as pairs}. We define:
\begin{equation}
\label{eq: def B V}
\cB_\cV = \brackets*{\cW \times (\cV \setminus \cW) \mvert \emptyset \varsubsetneq \cW \varsubsetneq \cV} \subset 2^{(\cV^{\cwedge})}.
\end{equation}
Note that $\cB_\cV=\emptyset$ if and only if $\card(\cV)=1$, and that $\emptyset \notin \cB_\cV$.

Let us now consider $\cI \in \cP_A$. If $B \subset A$ is a proper subset such that $\cI \leq \brackets{B,A\setminus B}$, then $\cI_B$ is a proper subset of $\cI$. Conversely, any proper subset of $\cI$ can be obtained in this way, compare Lemma~\ref{lem: bijections partitions}.\ref{item: bijection coarser} in the case of partitions with exactly two blocks. Thus, in this case, the previous definition can be rewritten as follows:
\begin{equation}
\label{eq: def B I}
\cB_\cI = \brackets*{\cI_B \times \cI_{A \setminus B}\mvert \emptyset \varsubsetneq B \varsubsetneq A \ \text{such that} \ \cI \leq \brackets{B,A\setminus B}} \subset 2^{(\cI^{\cwedge})}.
\end{equation}
A consequence of Corollary~\ref{cor: refined Hadamard lemma for FIJ} is the following key estimate.

\begin{prop}[Key estimate for $F_{\cI,\cJ}$]
\label{prop: key estimate FIJ}
Let $A$ be a non-empty finite set and $\cI,\cJ \in \cP_A$ be such that $\cJ \leq \cI \neq \brackets{A}$. For any compact $\Gamma \subset \cM_\cJ \times \cL_{\cJ,\cI}^\dagger \times \cS_\cI^+$, there exists $C_{\Gamma} \geq 0$ such that:
\begin{equation*}
\forall (w,\Sigma) \in \Gamma, \qquad \norm*{F_{\cI,\cJ}(w,\Sigma)} \leq C_{\Gamma} \sum_{\un \in \min(\cN_{\cB_\cI})} \prod_{\brackets{I,J} \in \cI^{\cwedge}} \Norm{\Sigma_I^J}^{n_{\brackets{I,J}}},
\end{equation*}
where $\Sigma = (\Sigma_I^J)_{I,J \in \cI}$ by blocks, and $\Norm{\Sigma_I^J}=\Norm{\Sigma_J^I}$ stands for the common operator norm of these linear maps, subordinated to the Euclidean ones.
\end{prop}

\begin{proof}
We start from the expansion we established in Corollary~\ref{cor: refined Hadamard lemma for FIJ}. By continuity of the functions $\parentheses*{F_{\cI,\cJ,\un}}_{\un \in \min(\cN_{\tilde{B}_\cI})}$ on the compact $\Gamma$, there exists $C \geq 0$ such that
\begin{equation*}
\forall (w,\Sigma) \in \Gamma, \qquad \norm*{F_{\cI,\cJ}(w,\Sigma)} \leq C \sum_{\un \in \min(\cN_{\tilde{\cB}_\cI})} \norm{\uSigma^{\un}},
\end{equation*}
where $(\Sigma_{ab})_{a,b \in \cA}$ is the matrix of $\Sigma$ in the orthonormal basis discussed at the beginning of this section, and $\uSigma = \parentheses*{(\Sigma_{aa})_{a \in \cA},(\sqrt{2}\Sigma_{ab})_{\brackets{a,b} \in \cA^{\cwedge}}} \in \R^{\tilde{A}}$.

Let $\un \in \min(\cN_{\tilde{\cB}_\cI})$. It is a multi-index indexed by
\begin{equation}
\label{eq: splitting A tilde}
\tilde{A} =\cA \sqcup \cA^{\cwedge} = \cA \sqcup \parentheses*{\bigsqcup_{I \in \cI} (\cA_I)^{\cwedge}} \sqcup \parentheses*{\bigsqcup_{\brackets{I,J} \in \cI^{\cwedge}} \cA_I\times \cA_J},
\end{equation}
see~\eqref{eq: def AI} and~\eqref{eq: def A}. By Equations~\eqref{eq: def B tilde} and~\eqref{eq: def B I tilde}, we have $\parentheses*{\cA \sqcup \bigsqcup_{I \in \cI} (\cA_I)^{\cwedge}} \cap \parentheses*{\bigcup_{\tilde{B} \in \tilde{\cB}_\cI} \tilde{B}}=\emptyset$. Hence, by Lemma~\ref{lem: description of NB}, we have $\un_{\cA}=0$, and $\un_{(\cA_I)^{\cwedge}}=0$ for all $I \in \cI$. Note that $\cA \sqcup \bigsqcup_{I \in \cI} (\cA_I)^{\cwedge}$ is the set indexing the coefficients in the diagonal blocks of the matrix of $\Sigma = (\Sigma_I^J)_{I,J \in \cI}$. We just proved that these coefficients do not appear in the product $\uSigma^{\un}$. Thus, given $\Sigma \in \cS_\cI^+$, we have:
\begin{equation*}
\norm{\uSigma^{\un}} = \prod_{\brackets{I,J} \in \cI^{\cwedge}} \prod_{(a,b) \in \cA_I \times \cA_J} \norm*{\sqrt{2}\Sigma_{ab}}^{n_{\brackets{a,b}}} \leq \sqrt{2}^{\norm{\un}} \prod_{\brackets{I,J} \in \cI^{\cwedge}} \Norm{\Sigma_I^J}^{\norm{\un_{\cA_I \times \cA_J}}}.
\end{equation*}
Indeed, if $I \neq J$ and $(a,b) \in \cA_I \times \cA_J$, then $\Sigma_{ab}$ is a coefficient in the matrix of the block $\Sigma_I^J$. With our choice of norms, we have $\norm{\Sigma_{ab}} \leq \Norm{\Sigma_I^J} = \Norm{\Sigma_J^I}$. In particular, there is no ambiguity coming from the fact that we consider unordered pairs.

Since $\cI \neq \brackets{A}$, we have $\cI^{\cwedge}\neq \emptyset$. Let us define $\um=(m_{\brackets{I,J}}) \in \N^{(\cI^{\cwedge})}$ by $m_{\brackets{I,J}}=\norm*{\un_{\cA_I \times \cA_J}}$ for all $\brackets{I,J} \in \cI^{\cwedge}$. Let $B$ be a proper subset of $A$ such that $\cI \leq \brackets{B,A \setminus B}$. We have:
\begin{equation}
\label{eq: key estimate relation multi-indices}
\norm{\um_{\cI_B \times \cI_{A \setminus B}}} = \sum_{(I,J) \in \cI_B \times \cI_{A \setminus B}} m_{\brackets{I,J}} = \sum_{(I,J) \in \cI_B \times \cI_{A \setminus B}} \sum_{(a,b) \in \cA_I \times \cA_J} n_{\brackets{a,b}} = \norm{\un_{\tilde{B}}}\geq 2,
\end{equation}
since $\un \in \cN_{\tilde{\cB}_\cI}$, see Equations~\eqref{eq: def NB}, \eqref{eq: def B tilde} and~\eqref{eq: def B I tilde}. Hence $\um \in \cN_{\cB_\cI}$, see Equation~\eqref{eq: def B I}.

Let us prove that $\um \in\min(\cN_{\cB_\cI})$. By Lemma~\ref{lem: well founded order}, there exists $\up \in \min(\cN_{\cB_\cI})$ such that $\up \leq \um$. We have $p_{\brackets{I,J}}\leq \norm{\un_{\cA_I \times \cA_J}}$ for all $\brackets{I,J} \in \cI^{\cwedge}$. Thanks to the splitting of $\tilde{A}$ given by~\eqref{eq: splitting A tilde}, there exists $\uq \in \N^{\tilde{A}}$ such that $\uq \leq \un$ and, for all $\brackets{I,J} \in \cI^{\cwedge}$, $\norm*{\uq_{\cA_I \times \cA_J}}=p_{\brackets{I,J}}$. The same computation as~\eqref{eq: key estimate relation multi-indices} shows that, for all $\tilde{B} \in \tilde{B}_\cI$, we have $\norm{\uq_{\tilde{B}}} = \norm{\up_{\cI_B \times \cI_{A \setminus B}}} \geq 2$. Hence, $\uq \in \cN_{\tilde{\cB}_\cI}$. By minimality of $\un$ we have $\uq=\un$, which implies $\um=\up \in\min(\cN_{\cB_\cI})$. Hence, for all $(w,\Sigma) \in \Gamma$, we have:
\begin{equation*}
\norm{\uSigma^{\un}} \leq \sqrt{2}^{\norm{\un}} \prod_{\brackets{I,J} \in \cI^{\cwedge}} \Norm{\Sigma_I^J}^{m_{\brackets{I,J}}} \leq \sqrt{2}^{\norm{\un}} \sum_{\up \in \min(\cN_{\cB_\cI})} \prod_{\brackets{I,J} \in \cI^{\cwedge}} \Norm{\Sigma_I^J}^{p_{\brackets{I,J}}}.
\end{equation*}

By Lemma~\ref{lem: description of NB}, the cardinality of $\min\parentheses{\cN_{\tilde{B}_\cI}}$ is at most $3 \norm{\tilde{A}}$, and moreover $\norm{\un} \leq 2 \norm{\tilde{A}}$ for all $\un \in \min(\cN_{\tilde{\cB}\cI})$. Letting $C_\Gamma = 3C\norm{\tilde{A}} 2^{\norm{\tilde{A}}}$, we conclude the proof by summing the previous inequalities over $\un \in \min\parentheses{\cN_{\tilde{B}_\cI}}$.
\end{proof}

%%%%%%%%%%%%%%%%%%%%%%%%%%%%%%%%%%%%%%%%%%%%%%%%%%%%%%%%%%%%%%%%%%%%%%%%%%%%%%%%%%%%%%%%%%%%%%%%%%%%%%%
%%%%%%%%%%%%%%%%%%%%%%%%%%%%%%%%%%%%%%%%%%%%%%%%%%%%%%%%%%%%%%%%%%%%%%%%%%%%%%%%%%%%%%%%%%%%%%%%%%%%%%%

\section{Edge-connected graphs}
\label{sec: edge-connected graphs}

The goal of this section is to recall some standard facts about graphs and to relate indexing sets of the form $\min(\cN_{\cB_\cI})$ to graphs. This is done in Section~\ref{subsec: graphs and edge-connectedness}. The graphs we consider satisfy a property reinforcing connectedness, called $2$-edge-connectedness. We prove in Section~\ref{subsec: ear decompositions and spanning trees} that $2$-edge-connectedness implies the existence of a nice family of spanning trees.

%%%%%%%%%%%%%%%%%%%%%%%%%%%%%%%%%%%%%%%%%%%%%%%%%%%%%%%%%%%%%%%%%%%%%%%%%%%%%%%%%%%%%%%%%%%%%%%%%%%%%%%

\subsection{Graphs and edge-connectedness}
\label{subsec: graphs and edge-connectedness}

We consider non-empty finite graphs that have no loops but may have multiple edges with the same endpoints. We are interested in keeping track of the labels of the vertices, but we do not care about the labels of the edges. This leads to the following definition, compare~\cite[Def.~1.1.2]{Wes2001}.

\begin{dfn}[Graph]
\label{def: graph}
A \emph{graph} $G$ is a couple $\parentheses*{\cV,\un}$ where $\cV$ is a non-empty finite set and $\un = \parentheses{n_{\brackets{u,v}}}_{\brackets{u,v} \in \cV^{\cwedge}}$ is a multi-index in $\N^{(\cV^{\cwedge})}$. Elements of $\cV$ are called the \emph{vertices} of $G$. For all $e = \brackets{u,v} \in \cV^{\cwedge}$, the integer $n_e$ is the number of \emph{edges} between $u$ and $v$. We denote by $\cE_G = \brackets*{e \in \cV^{\cwedge} \mvert n_e \geq 1}$ the set of pairs of vertices that are the \emph{endpoints} of at least one edge.
\end{dfn}

Since Definition~\ref{def: graph} is but one among many possible variations, we recall the following standard definitions for the reader's convenience.

\begin{dfn}[Path, cycle and spanning tree]
\label{def: path cycle tree}
Let $G=\parentheses*{\cV,\un}$ be a graph.
\begin{enumerate}
\item \label{item: def subgraph} We say that a graph $H=(\cW,\um)$ is a \emph{sub-graph} of $G$, and we write $H \subset G$, if $\cW \subset \cV$ and $\um \leq \un$. Note that $\um \in \N^{(\cV^{\cwedge})}$ since the first condition implies $\cW^{\cwedge} \subset \cV^{\cwedge}$, see Remark~\ref{rem: multi-indices B sub A}.

\item \label{item: def cycle} Let $l \geq 2$. A \emph{cycle} of length $l$ in $G$ is a sub graph of the type $\parentheses*{\brackets{v_1,\dots,v_l},\sum_{i=1}^l \one_{\brackets{v_{i-1},v_i}}}$, where $v_1,\dots,v_l \in \cV$ are pairwise distinct and we denoted $v_0=v_l$.

\item \label{item: def path} Let $l \in \N$ and $u,v \in \cV$. A \emph{path} of length $l$ between $u$ and $v$ in $G$ is a sub graph of the type $\parentheses*{\brackets{v_0,v_1,\dots,v_l},\sum_{i=1}^l \one_{\brackets{v_{i-1},v_i}}}$; where $v_0=u$, $v_l=v$ and $v_0,v_1,\dots,v_l \in \cV$ are pairwise distinct. The vertices $\parentheses{v_i}_{1 \leq i \leq l-1}$ are called the \emph{internal vertices} of the path.

\item \label{item: def tree} A connected graph without cycle is called a~\emph{tree}. In particular, a tree has no multiple edge.

\item \label{item: def spanning} We say that a sub graph $T=(\cW,\um) \subset G$ is a \emph{spanning tree} of $G$ if $T$ is a tree and $\cW=\cV$.
\end{enumerate}
\end{dfn}

We are interested in a class of graphs satisfying a property reinforcing connectedness, that we now introduce. This is a special case of the notion of $k$-edge-connectedness, see~\cite[Def.~4.1.7]{Wes2001}.

\begin{dfn}[$2$-edge-connectedness]
\label{def: 2-edge-connected}
We say that a graph $G=\parentheses*{\cV,\un}$ is \emph{$2$-edge-connected} if it is connected and remains so after removing any of its edges. That is, for all $\um < \un$ such that $\norm{\um}=\norm{\un}-1$, the graph $\parentheses*{\cV,\um}$ is connected. We say that $G$ is \emph{minimally $2$-edge-connected} if it is $2$-edge-connected and, for all $\um < \un$, the graph $(\cV,\um)$ is not $2$-edge-connected.
\end{dfn}

Let $\cV$ be a non-empty finite set. Recall that we defined $\cB_\cV$ in Equation~\eqref{eq: def B V} and $\cN_{\cB_\cV}$ in Equation~\eqref{eq: def NB}. With this in mind, we can characterize $2$-edge-connected graphs as follows.

\begin{lem}[Characterization of $2$-edge-connectedness]
\label{lem: characterization 2 edge connected}
A non-trivial graph $G = \parentheses*{\cV,\un}$ is $2$-edge-connected (resp.~minimally $2$-edge-connected) if and only if $\un \in \cN_{\cB_\cV}$ (resp.~$\un \in \min \parentheses*{\cN_{\cB_\cV}}$).
\end{lem}

\begin{proof}
Since we assume $G$ to be non-trivial, we have $\card(\cV)\geq 2$ and $\cV^{\cwedge}\neq \emptyset$. Let us first prove that $G=(\cV,\un)$ is connected if and only if, for all proper $\cW \subset \cV$, we have $\norm{\un_{\cW \times \parentheses*{\cV \setminus \cW}}}\geq 1$. Recall that $\cW \times \parentheses*{\cV \setminus \cW} \subset \cV^{\cwedge}$ by Remark~\ref{rem: product as pairs}.

\paragraph*{Step 1: Characterization of connectedness.}  Let us assume that $G$ is connected. Let $\cW \subset \cV$ be a proper subset. Let $w \in \cW$ and $v \in \cV \setminus \cW$. By connectedness, there exists a path from $w$ to $v$ in~$G$. Let $w=v_0,\dots,v_l=v$ denote the distinct vertices of this path, ordered as in Definition~\ref{def: path cycle tree}.\ref{item: def path}. There exists a smallest $j \in \ssquarebrackets{1}{l}$ such that $v_j \notin \cW$. Then, there is at least one edge of $G$ between $v_{j-1} \in \cW$ and $v_j \in \cV \setminus \cW$, and $\norm{\un_{\cW \times \parentheses*{\cV \setminus \cW}}} \geq n_{\brackets{v_{j-1},v_j}} \geq 1$. Thus, for all proper $\cW \subset \cV$, we have $\norm{\un_{\cW \times \parentheses*{\cV \setminus \cW}}}\geq 1$.

Conversely, let us assume that $\norm{\un_{\cW \times \parentheses*{\cV \setminus \cW}}}\geq 1$ for all $\emptyset \varsubsetneq \cW \varsubsetneq \cV$. Let $u \in \cV$, we denote by
\begin{equation*}
\cW = \brackets*{w \in \cV \mvert \text{there exists a path between} \ u \ \text{and} \ w \ \text{in} \ G},
\end{equation*}
the connected component of $u$. For all $w \in \cW$ and $v \notin \cW$ we have $n_{\brackets{v,w}}=0$. Since $u \in \cW$, if we had $\cW \neq \cV$, then $\cW$ would be a proper subset of $\cV$ such that $\un_{\cW \times \parentheses*{\cV \setminus \cW}}=0$, which would contradict our hypothesis. Hence $\cW = \cV$, and $G$ is connected.

\paragraph*{Step 2: Characterization of (minimal) $2$-edge-connectedness.} Assuming that $G$ is $2$-edge-connected, let $\cW$ be a proper subset of $\cV$. Since $G$ is connected, $\norm{\un_{\cW \times \parentheses*{\cV \setminus \cW}}}\geq 1$. Hence there exists $e \in \cW \times (\cV\setminus \cW)$ such that $n_e \geq 1$. Letting $\um=\un-\one_e$, our hypothesis ensures that $(\cV,\um)$ is also connected, hence $\norm{\um_{\cW \times \parentheses*{\cV \setminus \cW}}}\geq 1$ and $\norm{\un_{\cW \times \parentheses*{\cV \setminus \cW}}}=\norm{\um_{\cW \times \parentheses*{\cV \setminus \cW}}}+1 \geq 2$. This holds for all proper $\cW \subset \cV$, hence $\un \in \cN_{\cB_\cV}$.

Conversely, let us assume that $\un \in \cN_{\cB_\cV}$. Let $\um <\un$ be such that $\norm{\um}=\norm{\un}-1$. Then, there exists a pair $e \in \cV^{\cwedge}$ such that $\um = \un - \one_e$. For all proper subset $\cW \subset \cV$, we have $\norm{\um_{\cW \times \parentheses*{\cV \setminus \cW}}}\geq \norm{\un_{\cW \times \parentheses*{\cV \setminus \cW}}}-1 \geq 1$. Hence $(\cV,\um)$ is connected. This establishes the $2$-edge-connectedness of $G$, and concludes the proof of the first equivalence.

Let us now assume that $G$ is $2$-edge-connected graph, that is, $\un \in \cN_{\cB_\cV}$. If $G$ is minimally $2$-edge-connected, for all $\um < \un$, the graph $(\cV,\um)$ is not $2$-edge-connected, hence $\um \notin \cN_{\cB_\cV}$ by the characterization we just proved. Thus $\un \in \min(\cN_{\cB_\cV})$. Conversely, if $\un$ is minimal in $\cN_{\cB_\cV}$, then for all $\um <\un$ we have $\um \notin \cN_{\cB_\cV}$, and hence $(\cV,\um)$ is not $2$-edge-connected. Thus, $G$ is minimally $2$-edge-connected. This concludes the proof the second equivalence.
\end{proof}

\begin{ex}
\label{ex: 2-edge-connected}
Here are some useful examples.
\begin{itemize}
\item A tree is not $2$-edge-connected unless it is trivial.

\item Cycles are minimally $2$-edge-connected.

\end{itemize}
\end{ex}

\begin{dfn}[Set of minimally $2$-edge-connected graphs]
\label{def G V}
Let $\cV$ be a non-empty finite set, we denote by $\G_\cV$ the finite set of minimally $2$-edge-connected graphs with vertex set $\cV$. Since the set of vertices is explicit in this notation, we often denote $\un \in \G_\cV$ instead of $(\cV,\un)\in\G_\cV$. 
\end{dfn}

If $\card(\cV)=1$, then $\G_\cV$ contains only the trivial graph. Otherwise, $(\cV,\un) \mapsto \un$ is a bijection from $\G_\cV$ to $\min(\cN_{\cB_\cV})$ by Lemma~\ref{lem: characterization 2 edge connected}. By Lemma~\ref{lem: description of NB}, this shows that $\G_\cV$ is indeed finite. This also allows us to restate Proposition~\ref{prop: key estimate FIJ} as follows.

\begin{cor}[Key estimate with graphs]
\label{cor: key estimate graphs}
Let $A$ be a non-empty finite set and $\cI,\cJ \in \cP_A$ be such that $\cJ \leq \cI$. For any compact $\Gamma \subset \cM_\cJ \times \cL_{\cJ,\cI}^\dagger \times \cS_\cI^+$, there exists $C_{\Gamma} \geq 0$ such that:
\begin{equation*}
\forall (w,\Sigma) \in \Gamma, \qquad \norm*{F_{\cI,\cJ}(w,\Sigma)} \leq C_{\Gamma} \sum_{\un \in \G_\cI} \prod_{\brackets{I,J} \in \cI^{\cwedge}} \Norm{\Sigma_I^J}^{n_{\brackets{I,J}}},
\end{equation*}
where $\Sigma = (\Sigma_I^J)_{I,J \in \cI}$, and $\Norm{\Sigma_I^J}=\Norm{\Sigma_J^I}$ are operator norms subordinated to the Euclidean ones.
\end{cor}

\begin{proof}
If $\cI \neq \brackets{A}$ we have $\card(\cI)\geq 2$, and the result follows from Proposition~\ref{prop: key estimate FIJ} by re-indexing the sum, using the bijection between $\G_\cI$ and $\min(\cN_{\cB_\cI})$ we just described. If $\cI = \brackets{A}$, then $\norm{\G_\cI}=1$ and $\cI^{\cwedge}=\emptyset$. In this case, we just need to prove that $F_{\cI,\cJ}$ is bounded on the compact $\Gamma$, which follows from its continuity, see Lemma~\ref{lem: regularity F I J}.
\end{proof}

%%%%%%%%%%%%%%%%%%%%%%%%%%%%%%%%%%%%%%%%%%%%%%%%%%%%%%%%%%%%%%%%%%%%%%%%%%%%%%%%%%%%%%%%%%%%%%%%%%%%%%%

\subsection{Ear decompositions and spanning trees}
\label{subsec: ear decompositions and spanning trees}

An interesting feature of $2$-edge-connected graphs is that they admit a \emph{closed-ear decomposition}, cf.~\cite[Thm.~4.2.10]{Wes2001}. Such a decomposition allows us to prove that a $2$-edge-connected graph admits a nice family of spanning trees, which is the main result of this section, see Lemma~\ref{lem: existence of spanning trees}.

\begin{lem}[Closed-ear decomposition]
\label{lem: ear decomposition}
Let $G=(\cV,\un)$ be a non-trivial $2$-edge-connected graph. There exist $p \in \N$ and non-trivial sub-graphs $G_0=(\cV_0,\un^{(0)})$, \dots, $G_p=(\cV_p,\un^{(p)})$ of $G$ satisfying the following properties.
\begin{enumerate}
\item \label{item: ear G0} The graph $G_0$ is a cycle whose length, denoted by $l_0$, is maximal among cycles in $G$.

\item \label{item: ear Gq} For all $q \in \ssquarebrackets{1}{p}$, $G_q$ is either a cycle of length at most $l_0$ with exactly one vertex in $\bigcup_{i=0}^{q-1} \cV_i$ or a path of length less than $l_0$ between vertices of $\bigcup_{i=0}^{q-1} \cV_i$ with no internal vertex in $\bigcup_{i=0}^{q-1} \cV_i$.

\item \label{item: ear G} We have $\cV = \bigcup_{i=0}^p \cV_i$ and $\un = \sum_{i=0}^p \un^{(i)}$.
\end{enumerate}
\end{lem}

\begin{proof}
We build $(G_q)_{0 \leq q \leq p}$ by induction. Since $G$ is $2$-edge-connected and non-trivial, it is connected but it is not a tree, hence $G$ admits cycles. The length of these cycles is bounded by~$\norm{\un}$, hence there exists one whose length $l_0$ is maximal. We denote by $G_0=(\cV_0,\un^{(0)})$ any such cycle, so that Condition~\ref{item: ear G0} holds. In particular $\un^{(0)}\leq \un$.

Let $q \in \N$. We assume that we have built sub-graphs $G_0=(\cV_0,\un^{(0)})$, \dots, $G_q=(\cV_q,\un^{(q)})$ satisfying Conditions~\ref{item: ear G0} and~\ref{item: ear Gq}, and such that $\sum_{i=0}^q \un^{(i)} \leq \un$.

We first consider the case where $\sum_{i=0}^q \un^{(i)} = \un$. Let $v \in \cV$, since $G$ is connected and non-trivial, there exists $u \neq v$ such that $n_{\brackets{u,v}}\geq 1$. Hence, there exists $i \in \ssquarebrackets{0}{q}$ such that $n_{\brackets{u,v}}^{(i)}\geq 1$, which implies that $v \in \cV_i$. Thus, for all $v \in \cV$ we have $v \in \bigcup_{i=0}^q \cV_i$, and $\cV =\bigcup_{i=0}^q \cV_i$. In this case, we let $p=q$ and we are done.

In the case where $\sum_{i=0}^q \un^{(i)} < \un$, there exists $e=\brackets{u,v} \in \cV^{\cwedge}$ such that $\sum_{i=0}^q n_e^{(i)}<n_e$. If $v \notin \bigcup_{i=0}^q \cV_i$, by connectedness of $G$, there exists a path from $v$ to some $w \in \bigcup_{i=0}^q \cV_i$. Denoting by $v=v_0,\dots,v_l=w$ the distinct vertices of this path, ordered as in Definition~\ref{def: path cycle tree}.\ref{item: def path}, there is a smallest $j \in \ssquarebrackets{1}{l}$ such that $v_j\in\bigcup_{i=0}^q \cV_i$. Since $v_{j-1} \notin \bigcup_{i=0}^q \cV_i$, we have $\sum_{i=0}^q n_{\brackets{v_{j-1},v_j}}^{(i)}=0<n_{\brackets{v_{j-1},v_j}}$. Thus, up to replacing $e=\brackets{u,v}$ by $\brackets{v_{j-1},v_j}$, we can assume that $v \in \bigcup_{i=0}^q \cV_i$.

Let $\um = \un -\one_e\geq \sum_{i=0}^q \un^{(i)}$. By $2$-edge-connectedness of $G$, there is a path from $u$ to $v$ in the sub-graph $(\cV,\um)$. Denoting by $u=v_0,\dots,v_l=v$ the distinct vertices of this path, ordered as in Definition~\ref{def: path cycle tree}.\ref{item: def path}, we have $\one_{\brackets{v_l,v_0}} + \sum_{j=1}^l \one_{\brackets{v_{j-1},v_j}} \leq \one_{\brackets{u,v}}+\um \leq \un$. Hence, by adding an edge between $u$ and $v$, we obtain a cycle in $G$ of length~$l+1$. By maximality of $l_0$, we have $l+1 \leq l_0$.

Since $v=v_l \in \bigcup_{i=0}^q \cV_i$, there exists a smallest $j_0 \in \ssquarebrackets{0}{l}$ such that $v_{j_0}\in\bigcup_{i=0}^q \cV_i$. Recalling that $u=v_0$, we define $\cV_{q+1}=\brackets*{v,u,v_1,\dots,v_{j_0}}$ and $\un^{(q+1)}=\one_{\brackets{u,v}}+ \sum_{j=1}^{j_0} \one_{\brackets{v_{j-1},v_j}}$. If $j_0=l$, then $v_{j_0}=v$ and $G_{q+1}=\parentheses*{\cV_{q+1},\un^{(q+1)}}$ is a cycle of length $l+1\leq l_0$ such that $\cV_{q+1} \cap \parentheses*{\bigcup_{i=0}^q \cV_i} = \brackets{v}$. If $j_0 <l$ then $G_{q+1}$ is a non-trivial path of length $j_0+1\leq l< l_0$ from $v$ to $v_{j_0}$ such that $\cV_{q+1} \cap \parentheses*{\bigcup_{i=0}^q \cV_i} = \brackets{v,v_{j_0}}$. Thus, $G_{q+1}$ is of one of the two types described in Condition~\ref{item: ear Gq}.

For all $j \in \ssquarebrackets{1}{j_0}$, since $v_{j-1} \notin \bigcup_{i=0}^q \cV_i$, we have $\sum_{i=0}^q n^{(i)}_{\brackets{v_{j-1},v_j}}=0 < m_{\brackets{v_{j-1},v_j}}$. Hence,
\begin{equation*}
\sum_{i=0}^{q+1} \un^{(i)} = \sum_{i=0}^{q} \un^{(i)} + \sum_{j=1}^{j_0} \one_{\brackets{v_{j-1},v_j}}+\one_{\brackets{u,v}} \leq \um + \one_{\brackets{u,v}} \leq \un,
\end{equation*}
which concludes the induction step.

Each $G_q$ being non-trivial and connected, $\sum_{i=0}^q \norm{\un^{(i)}}$ increases with $q$. Since this quantity is bounded by $\norm{\un}$, the previous inductive construction process finishes after finitely many steps.
\end{proof}

\begin{lem}[Existence of spanning trees, cf.~{\cite[Lem.~3.17]{Gas2023b}}]
\label{lem: existence of spanning trees}
Let $G=\parentheses*{\cV,\un}$ be a $2$-edge-connected graph. It admits a family $\parentheses{T_v}_{v \in \cV}$ of spanning trees with the following property: for all $e \in \cE_G$, there exists $v \in \cV$ such that $\um^{(v)} \leq \un -\one_e$, where we denoted $T_v= \parentheses*{\cV,\um^{(v)}}$.
\end{lem}

\begin{proof}
If $G$ is trivial then it has no edge and $\cV=\brackets{v}$. Since $G$ is a tree, it is enough to let $T_v=G$. In the following, we assume that $G$ is non-trivial, in particular $\cV^{\cwedge}\neq \emptyset$.

Let $G_0=(\cV_0,\un^{(0)})$, \dots, $G_p=(\cV_p,\un^{(p)})$ be a decomposition of $G$ given by Lemma~\ref{lem: ear decomposition}. For all $i \in \ssquarebrackets{0}{p}$, we denote by $\cE_i = \brackets{e \in \cV^{\cwedge} \mid n^{(i)}_e\geq 1}$. We have $\card(\cE_i) \leq \norm{\un^{(i)}} \leq l_0 = \card(\cV_0)$, by Conditions~\ref{item: ear G0} and~\ref{item: ear Gq} in Lemma~\ref{lem: ear decomposition}. Hence there exists a surjection $\zeta_i:\cV_0 \to \cE_i$. Note that $\un^{(i)} - \one_{\zeta_i(v)} \in \N^{(\cV^{\cwedge})}$ for all $v \in \cV_0$, since $\zeta_i(v) \in \cE_i$.

Let $v \in \cV_0$, for all $q \in \ssquarebrackets{0}{p}$ we denote by $T_{q,v} = \parentheses*{\bigcup_{i=0}^q \cV_i, \sum_{i=0}^q \un^{(i)} - \one_{\zeta_i(v)}}$ and let $T_v=T_{p,v}$. Since $T_{0,v}$ is obtained by removing exactly one edge from the cycle $G_0$, it is a path and in particular a tree. Let $q \in \ssquarebrackets{0}{p-1}$, we assume that $T_{q,v}$ is a tree. If $G_{q+1}$ is a cycle, we obtain $T_{q+1,v}$ by adding to $T_{q,v}$ a path sharing exactly one vertex with it. If $G_{q+1}$ is a path, we obtain $T_{q+1,v}$ by adding to $T_{q,v}$ two paths of the previous kind, that have no common vertex. These operations preserve connectedness and do not create cycles, hence $T_{q+1,v}$ is again a tree. By induction, $T_v=T_{p,v}$ is a tree. It is actually a spanning tree since $\bigcup_{i=0}^p \cV_i=\cV$, by Condition~\ref{item: ear G} in Lemma~\ref{lem: ear decomposition}.

The previous construction shows that $G$ admits spanning trees. For all $v \in \cV \setminus \cV_0$, we denote by $T_v$ any spanning tree of $G$. For all $v \in \cV$, we let $T_v= (\cV,\um^{(v)})$.

Let $e \in \cE_G$, we have $n_e >0$. By Conditions~\ref{item: ear G} in Lemma~\ref{lem: ear decomposition}, there exists $i \in \ssquarebrackets{0}{p}$ such that $\un^{(i)}_e >0$, that is, $e \in \cE_i$. By surjectivity of $\zeta_i$, there exists $v \in \cV_0$ such that $\zeta_i(v)=e$. Then $\um^{(v)} = \sum_{i=0}^p \parentheses*{\un^{(i)} - \one_{\zeta_i(v)}} \leq \parentheses*{\sum_{i=0}^p \un^{(i)}} - \one_e = \un - \one_e$. This proves that $(T_v)_{v \in \cV}$ satisfies the claimed property.
\end{proof}

%%%%%%%%%%%%%%%%%%%%%%%%%%%%%%%%%%%%%%%%%%%%%%%%%%%%%%%%%%%%%%%%%%%%%%%%%%%%%%%%%%%%%%%%%%%%%%%%%%%%%%%
%%%%%%%%%%%%%%%%%%%%%%%%%%%%%%%%%%%%%%%%%%%%%%%%%%%%%%%%%%%%%%%%%%%%%%%%%%%%%%%%%%%%%%%%%%%%%%%%%%%%%%%

\section{Hölder--Brascamp--Lieb inequalities}
\label{sec: HBL inequalities}

This section is concerned with Hölder--Brascamp--Lieb inequalities, which are a family of inequalities generalizing those of Hölder and Young. In Section~\ref{subsec: interpolation of HBL inequalities}, we recall what these inequalities are and how to interpolate between them. In Section~\ref{subsec: HBL inequalities and edge-connected graphs}, we prove that specific Hölder--Brascamp--Lieb inequalities hold in the case of a product space, where the product is indexed by the vertices and edges of a $2$-edge-connected graph. Finally, in Section~\ref{subsec: integral estimates for the universal part of the Kac--Rice density}, we use these inequalities to establish integral estimates for the universal part $\Upsilon_A$ of the Kac--Rice density, see Definition~\ref{def: Upsilon}.

%%%%%%%%%%%%%%%%%%%%%%%%%%%%%%%%%%%%%%%%%%%%%%%%%%%%%%%%%%%%%%%%%%%%%%%%%%%%%%%%%%%%%%%%%%%%%%%%%%%%%%%

\subsection{Interpolation of Hölder--Brascamp--Lieb inequalities}
\label{subsec: interpolation of HBL inequalities}

In this section, we recall the definition of Hölder--Brascamp--Lieb inequalities and a theorem of Bennett, Carbery, Christ and Tao~\cite{BCCT2010} that allows us to interpolate between these inequalities.

Let $V$ be a Euclidean space, endowed with its Lebesgue measure $\dx x$, and $\varphi:V \to [-\infty,+\infty]$ be a Borel-measurable map. For all $s \geq 1$, we denote by $\Norm{\varphi}_s = \parentheses*{\int_V \norm{\varphi(x)}^s \dx x}^\frac{1}{s}$ the $L^s$-norm of~$\varphi$, which can be either a non-negative number or $+\infty$. Similarly, we denote by $\Norm{\varphi}_\infty \in [0,+\infty]$ the essential supremum of $\norm{\varphi}$.

Let $A$ be a non-empty finite set, let $V$ and $(V_a)_{a \in A}$ be Euclidean spaces. For all $a \in A$, let $\ell_a:V \to V_a$ be a linear surjection. We denote by $\ul=\parentheses*{\ell_a}_{a \in A}$ the collection of these surjections.

\begin{dfn}[Hölder--Brascamp--Lieb inequality]
\label{def: HBL inequality}
Let $\ul=(\ell_a)_{a \in A}$ be a family of surjections as above and $\up=(p_a)_{a \in A} \in [1,+\infty]^A$. We say that a \emph{Hölder--Brascamp--Lieb} (HBL) \emph{inequality} associated with $\parentheses*{\ul,\up}$ holds if there exists $C\in(0,+\infty)$ such that
\begin{equation}
\label{eq: HBL}
\int_{x \in V} \prod_{a \in A} \varphi_a(\ell_a(x)) \dx x \leq C \prod_{a \in A} \Norm{\varphi_a}_{p_a},
\end{equation}
for any family $(\varphi_a)_{a \in A}$ such that $\varphi_a :V_a \to [0,+\infty]$ is Borel-measurable for all $a \in A$. Note that both sides of Equation~\eqref{eq: HBL} can take the value $+\infty$.
\end{dfn}

Given $\ualpha=(\alpha_a)_{a \in A} \in [0,1]^A$, we denote by $\frac{1}{\ualpha}=\parentheses{\frac{1}{\alpha_a}}_{a \in A}\in [1,+\infty]^A$, with the usual convention that $\frac{1}{0}=+\infty$. A necessary and sufficient condition for a HBL inequality associated with $\parentheses*{\ul,\up}$ to hold was established in~\cite{BCCT2010}. It can be stated using the set defined by:
\begin{equation}
\label{eq: def convex HBL interpolation}
\fC_{\ul}=\brackets*{\ualpha \in [0,1]^A \mvert \text{a HBL inequality associated with}\ \parentheses*{\ul,\frac{1}{\ualpha}}\ \text{holds}}.
\end{equation}

\begin{thm}[Bennett--Carbery--Christ--Tao]
\label{thm: BCCT}
Let $\ul=(\ell_a)_{a \in A}$ be a family of surjections as above and $\ualpha=(\alpha_a)_{a \in A }\in [0,1]^A$. Then, $\ualpha\in \fC_{\ul}$ if and only if, for every subspace $W \subset V$, we have $\dim(W) \leq \sum_{a \in A} \alpha_a \dim(\ell_a(W))$, and equality holds when $W=V$.
\end{thm}

\begin{proof}
This is~\cite[Thm.~2.1]{BCCT2010} for $(\ell_a)_{a \in A}$ and $(p_a)_{a \in A}=\frac{1}{\ualpha}$.
\end{proof}

In general, it is difficult to use Theorem~\ref{thm: BCCT} to check that a given $\ualpha \in [0,1]^A$ belongs to~$\fC_{\ul}$. However, it implies the following qualitative result.

\begin{cor}[Interpolation of Hölder--Brascamp--Lieb inequalities]
\label{cor: interpolation of HBL inequalities}
The set $\fC_{\ul}$ is convex.
\end{cor}

\begin{proof}
Theorem~\ref{thm: BCCT} shows that $\fC_{\ul}$ is defined by a family of linear inequalities, hence it is convex.
\end{proof}

%%%%%%%%%%%%%%%%%%%%%%%%%%%%%%%%%%%%%%%%%%%%%%%%%%%%%%%%%%%%%%%%%%%%%%%%%%%%%%%%%%%%%%%%%%%%%%%%%%%%%%%

\subsection{Hölder--Brascamp--Lieb inequalities and edge-connected graphs}
\label{subsec: HBL inequalities and edge-connected graphs}

The goal of this section is to prove that HBL inequalities hold for some specific parameters, in the case where the family of surjections $\ul$ is related with a $2$-edge-connected graph.

Let $G=(\cV,\un)$ be a graph. For all $v \in \cV$, we denote by $\ell_v:\ux=(x_u)_{u \in \cV} \mapsto x_v$ from~$(\R^d)^\cV$ to~$\R^d$. Given $e=\brackets{v,w} \in \cE_G$, we define $\ell_e:(\R^d)^\cV \to \R^d$ either as $\ux \mapsto x_v-x_w$ or as $\ux \mapsto x_w-x_v$ arbitrarily. This defines a family $\ul = (\ell_a)_{a \in \cV \sqcup \cE_G}$ of linear surjections from $(\R^d)^\cV$ to~$\R^d$.

Let $\xi = (\xi_u)_{u \in \cV} \in (0,+\infty)^\cV$, we denote by $V_\xi = \brackets*{\ux=(x_u)_{u \in \cV} \in (\R^d)^\cV \mvert \sum_{u \in \cV} \xi_u x_u =0}$. Given $e=\brackets{v,w} \in \cE_G$, we let $\ell_e^{(\xi)}$ denote the restriction of $\ell_e$ to $V_\xi$. We let $\ul^{(\xi)}=\parentheses{\ell_e^{(\xi)}}_{e \in \cE_G}$, which defines a family of linear surjections from $V_\xi$ to $\R^d$. Indeed, given $e \in \cE_G$ and $z \in \R^d$, we have $z = \ell_e^{(\xi)}(\uy)$, where $\uy=(y_u)_{u \in \cV}$ is defined by $y_v = \frac{\xi_w}{\xi_v+\xi_w} z$, $y_w = -\frac{\xi_v}{\xi_v+\xi_w} z$ and $y_u=0$ for all $u \in \cV \setminus e$.

The purpose of the following is to find some relevant points in the convex sets $\fC_{\ul}$ and $\fC_{\ul^{(\xi)}}$, see~\eqref{eq: def convex HBL interpolation}. We first establish HBL inequalities as a consequence of the existence of a spanning tree of~$G$, see Definition~\ref{def: path cycle tree}.\ref{item: def spanning}. A similar result appears in the proof of~\cite[Lem.~4.6]{Gas2023b}. Note that, since $\un_{\cV^{\cwedge}\setminus \cE_G}=0$, for any sub-graph $\parentheses*{\cW,\um} \subset G$ we can consider $\um$ as an element of $\N^{\cE_G}$.

\begin{lem}[Spanning trees and HBL inequalities]
\label{lem: spanning tree and HBL}
Let $T=(\cV,\um)$ be a spanning tree of $G$ and $v \in \cV$. Then, we have $\one_v+\um \in \fC_{\ul}$. Moreover, for all $\xi \in (0,+\infty)^{\cV}$, we have $\um \in \fC_{\ul^{(\xi)}}$.
\end{lem}

\begin{proof}
We have $\cE_T= \brackets{e \in \cV^{\cwedge} \mid m_e\geq 1} \subset \cE_G$. Since $T$ is a tree, it has no multiple edge. Hence, $\um = \sum_{e \in \cE_T}\one_e \in \brackets{0,1}^{\cE_G}$, and $\one_v+\um = \sum_{a \in \brackets{v}\sqcup \cE_T} \one_a \in \brackets{0,1}^{\cV \sqcup \cE_G}$. We denote by $\up = (p_a)_{a \in \cV \sqcup \cE_G} = \frac{1}{\one_v+\um}$, that is, $p_a=1$ if $a \in \brackets{v}\sqcup \cE_T$ and $p_a =+\infty$ otherwise.

\paragraph*{Proof that $\one_v+\um \in \fC_{\ul}$.} We have to prove that a HBL inequality associated with $(\ul,\up)$ holds, see~\eqref{eq: HBL}. For all $a \in \cV \sqcup \cE_G$, let $\varphi_a:\R^d \to [0,+\infty]$ be a Borel-measurable function. We have:
\begin{equation*}
\int_{\ux \in (\R^d)^\cV} \prod_{a \in \cV \sqcup \cE_G} \varphi_a(\ell_a(\ux))\dx \ux \leq \parentheses*{\prod_{a \notin \brackets{v}\sqcup \cE_T} \Norm{\varphi_a}_\infty} \int_{\ux \in (\R^d)^\cV} \prod_{a \in \brackets{v}\sqcup \cE_T}\varphi_a(\ell_a(\ux))\dx \ux.
\end{equation*}
We claim that the linear map $\Lambda:\ux \mapsto \parentheses*{\ell_a(\ux)}_{a \in \brackets{v}\sqcup \cE_T}$ is an isomorphism from $(\R^d)^\cV$ to $(\R^d)^{\brackets{v}\sqcup \cE_T}$. Assuming this fact for now, we perform the linear change of variable $\uy=\Lambda(\ux)$, which yields:
\begin{equation*}
\int_{\ux \in (\R^d)^\cV} \prod_{a \in \brackets{v}\sqcup \cE_T}\varphi_a(\ell_a(\ux))\dx \ux = \int_{(\R^d)^{\brackets{v}\sqcup \cE_T}} \prod_{a \in \brackets{v}\sqcup \cE_T} \varphi_a(y_a)\frac{\dx \uy}{\Jac(\Lambda)}= \frac{1}{\Jac(\Lambda)}\prod_{a \in \brackets{v}\sqcup \cE_T}\Norm{\varphi_a}_1.
\end{equation*}
Thus, a HBL inequality associated with $(\ul,\up)$ holds with constant $\frac{1}{\jac{\Lambda}}$, and hence $\one_v+\um \in \fC_{\ul}$.

Let us now prove that $\Lambda:(\R^d)^\cV \to (\R^d)^{\brackets{v}\sqcup \cE_T}$ is an isomorphism. Since $T$ is a spanning tree, we have $1+\card(\cE_T) = 1 + \norm{\um} = \card(\cV)$, see~\cite[Thm.~2.1.4.B]{Wes2001}. Hence, the source and target spaces have the same dimension, and it is enough to prove that $\Lambda$ is injective.

Let $\ux=(x_u)_{u \in \cV} \in \ker(\Lambda)$. We prove that $x_u=0$ for all $u \in \cV$, by induction on the graph distance $\dist(u,v)$ from $u$ to $v$ in~$T$. If $d(u,v)=0$, then $u=v$ and $x_u=x_v=\ell_v(\ux)=0$. Let $u \in \cV$ be such that $\dist(u,v)=l\geq 1$. There exists a path from $u$ to $v$ of length $l$ in $T$. In particular, there exists $w \in \cV$ such that $\dist(u,w)=1$ and $\dist(w,v)=l-1$. Letting $e=\brackets{u,w}$, the first condition proves that $e \in \cE_T$, and hence $\Norm{x_u-x_w} = \Norm{\ell_e(\ux)}=0$. By the induction hypothesis, we have $x_w=0$, and hence $x_u=0$. This concludes the induction step. Since $T$ is a spanning tree of $G$, this proves that $\ux=0$. Thus, $\Lambda$ is injective, which concludes the proof that $\one_v+\um \in \fC_{\ul}$.

\paragraph*{Proof that $\um \in \fC_{\ul^{(\xi)}}$.} Let $\xi \in (0,+\infty)^\cV$. Considering $\um$ as an element of $\brackets{0,1}^{\cE_G}$, we observe that $\up_{\cE_G}= \frac{1}{\um}$. Thus, we have to prove that a HBL inequality associated with $\parentheses{\ul^{(\xi)},\up_{\cE_G}}$ holds. For any family $(\varphi_e)_{e \in \cE_G}$ of non-negative Borel-measurable functions on $\R^d$, we have:
\begin{equation*}
\int_{\ux \in V_\xi} \prod_{e \in \cE_G} \varphi_e(\ell_e(\ux))\dx \ux \leq \prod_{e \notin \cE_T} \Norm{\varphi_e}_\infty \int_{\ux \in V_\xi} \prod_{e \in \cE_T} \varphi_e(\ell_e(\ux))\dx \ux,
\end{equation*}
where $\dx \ux$ stands for the Lebesgue measure on the subspace $V_\xi \subset (\R^d)^\cV$. Let $\Lambda_\xi:V_\xi \to (\R^d)^{\cE_T}$ be the linear map $\ux \mapsto \parentheses*{\ell_e(\ux)}_{e \in \cE_T}$. As above, we have $\dim(V_\xi) = d(\card(\cV)-1)=d\card(\cE_T)$, because $T$ is a spanning tree. Assuming for now that $\Lambda_\xi$ is injective, it is an isomorphism. The change of variable $\uy=\Lambda_\xi(\ux)$ then yields:
\begin{equation*}
\int_{\ux \in V_\xi} \prod_{e \in \cE_T} \varphi_e(\ell_e(\ux))\dx \ux = \frac{1}{\jac{\Lambda_\xi}} \int_{\uy \in (\R^d)^{\cE_T}} \prod_{e \in \cE_T} \varphi_e(y_e)\dx \uy = \frac{1}{\jac{\Lambda_\xi}}\prod_{e \in \cE_T}\Norm{\varphi_e}_1.
\end{equation*}
Thus, a HBL inequality associated with $(\ul^{(\xi)},\up_{\cE_G})$ holds, and $\um \in \fC_{\ul^{(\xi)}}$.

We still need to prove that $\Lambda_{\xi}$ is injective. Let $\ux=(x_u)_{u \in \cV} \in V_\xi$ be such that $\Lambda_\xi(\ux)=0$. Let $v \in \cV$. As above, we prove by induction on $\dist(u,v)$ that $x_u=x_v$ for all $u \in \cV$. Then, since $\ux \in V_\xi$, we have $0= \sum_{u \in \cV}\xi_u x_u=x_v\parentheses*{\sum_{u \in \cV}\xi_u}$. Since the $(\xi_u)_{u \in \cV}$ are positive, this implies that $x_v=0$, hence $\ux=0$. Finally, $\Lambda_\xi$ is injective, as claimed. Hence the result.
\end{proof}

We can now prove the main result of this section, which states that HBL inequalities of the form~\eqref{eq: HBL} hold for some particular exponents if $G$ is $2$-edge-connected.

\begin{prop}[HBL inequalities for $2$-edge-connected graphs]
\label{prop: HBL graphs}
Let $G=(\cV,\un)$ be a $2$-edge-connected graph such that $\card(\cV)\geq 2$. Let $\ul=(\ell_a)_{a \in \cV \sqcup \cE_G}$ be the family of surjections defined above. There exist exponents $(p_e)_{e \in \cE_G} \in [2,+\infty]^{\cE_G}$, and positive constants $C$ and $C'$ such that, for any Borel-measurable functions $\parentheses*{\varphi_a}_{a \in \cV \sqcup \cE_G}$ on $\R^d$, we have:
\begin{equation}
\label{eq: HBL L 2}
\int_{\ux \in (\R^d)^\cV} \prod_{a \in \cV \sqcup \cE_G} \norm*{\varphi_a(\ell_a(\ux))}\dx \ux \leq C \parentheses*{\prod_{v \in \cV} \Norm{\varphi_v}_2} \parentheses*{\prod_{e \in \cE_G} \Norm{\varphi_e}_\frac{p_e}{n_e}},
\end{equation}
and, letting $q_e = \frac{\norm{\cV}}{\norm{\cV}-1}\frac{p_e}{2}$ for all $e \in \cE_G$,
\begin{equation}
\label{eq: HBL L p/p-1}
\int_{\ux \in (\R^d)^\cV} \prod_{a \in \cV \sqcup \cE_G} \norm*{\varphi_a(\ell_a(\ux))}\dx \ux \leq C' \parentheses*{\prod_{v \in \cV} \Norm{\varphi_v}_{\norm{\cV}}} \parentheses*{\prod_{e \in \cE_G} \Norm{\varphi_e}_\frac{q_e}{n_e}}.
\end{equation}
Moreover, for all $\xi \in (0,+\infty)^\cV$, there exists $C_\xi >0$ such that for any Borel-measurable functions $\parentheses*{\varphi_e}_{e \in \cE_G}$ on $\R^d$, we have:
\begin{equation}
\label{eq: HBL xi}
\int_{\ux \in V_\xi} \prod_{e \in \cE_G} \norm*{\varphi_e(\ell_e(\ux))}\dx \ux \leq C_\xi \prod_{e \in \cE_G} \Norm{\varphi_e}_\frac{q_e}{n_e}.
\end{equation}
\end{prop}

\begin{proof}
The strategy is to find relevant points in the set $\fC_{\ul}$ (resp.~$\fC_{\ul^{(\xi)}}$) defined by~\eqref{eq: def convex HBL interpolation}, which satisfy some additional constraints. For any non-negative Borel-measurable functions $(\varphi_a)_{a \in \cV \sqcup \cE_G}$, we have:
\begin{equation*}
\int_{(\R^d)^\cV} \prod_{a \in \cV \sqcup \cE_G} \varphi_a(\ell_a(\ux))\dx \ux \leq \parentheses*{\prod_{e \in \cE_G}\Norm{\varphi_e}_\infty} \int_{(\R^d)^\cV} \prod_{v \in \cV} \varphi_v(x_v)\dx \ux = \parentheses*{\prod_{v \in \cV}\Norm{\varphi_v}_1} \parentheses*{\prod_{e \in \cE_G}\Norm{\varphi_e}_\infty}.
\end{equation*}
This shows that $\ugamma=\sum_{v \in \cV}\one_v \in \fC_{\ul}$. Besides, since $G=(\cV,\un)$ is $2$-edge-connected, it admits a family $(T_v)_{v \in \cV}$ of spanning trees given by Lemma~\ref{lem: existence of spanning trees}. For all $v \in \cV$, letting $T_v=(\cV,\um^{(v)})$, we have $\one_v+\um^{(v)} \in \fC_{\ul}$ by Lemma~\ref{lem: spanning tree and HBL}.

Recall that $\fC_{\ul}$ is convex, see Corollary~\ref{cor: interpolation of HBL inequalities}. We define $\ualpha=(\alpha_a)_{a \in \cV \sqcup \cE_G}$ and $\ubeta=(\beta_a)_{a \in \cV \sqcup \cE_G} \in \fC_{\ul}$ by $\ubeta = \frac{1}{\norm{\cV}} \sum_{v \in \cV} (\one_v+\um^{(v)})$ and $\ualpha = \frac{\norm{\cV}}{2(\norm{\cV}-1)}\ubeta+\parentheses*{1-\frac{\norm{\cV}}{2(\norm{\cV}-1)}}\ugamma$. Indeed, since we assumed $\norm{\cV}\geq 2$, we have $\frac{\norm{\cV}}{2(\norm{\cV}-1)} \in [0,1]$.

For all $v\in \cV$, we have $\um^{(v)} \in \N^{\cE_G} \subset \N^{\cV \sqcup \cE_G}$, and $\um^{(v)}_\cV=0$. Hence $\ubeta_\cV = \frac{1}{\norm{\cV}}\sum_{v \in \cV} \one_v=\frac{1}{\norm{\cV}}\ugamma_\cV$. Then,
\begin{equation*}
\ualpha_\cV = \frac{\norm{\cV}}{2(\norm{\cV}-1)}\ubeta_\cV + \parentheses*{1-\frac{\norm{\cV}}{2(\norm{\cV}-1)}}\ugamma_\cV = \ugamma_\cV \parentheses*{1+\frac{1}{2(\norm{\cV}-1)}-\frac{\norm{\cV}}{2(\norm{\cV}-1)}} = \frac{1}{2}\ugamma_\cV.
\end{equation*}
That is, for all $v \in \cV$, we have $\alpha_v=\frac{1}{2}$ and $\beta_v=\frac{1}{\norm{\cV}}$.

Let $e \in \cE_G$, we have $n_e \geq 1$ by definition of $\cE_G$. For all $v \in \cV$, since $T_v$ is a tree it cannot have multiple edges, and hence $m^{(v)}_e \leq 1$. Moreover, since we obtained $(T_v)_{v \in \cV}$ from Lemma~\ref{lem: existence of spanning trees}, there exists $u \in \cV$ such that $m_e^{(u)}<n_e$. If $n_e=1$, then $m_e^{(u)}=0$, and
\begin{equation*}
\beta_e = \frac{1}{\norm{\cV}}\sum_{v \in \cV} m_e^{(v)} \leq \frac{1}{\norm{\cV}}\card(\cV \setminus \brackets{u}) = n_e\frac{\norm{\cV}-1}{\norm{\cV}}.
\end{equation*}
If $n_e \geq 2$, we have $\beta_e= \frac{1}{\norm{\cV}}\sum_{v \in \cV} m_e^{(v)} \leq 1 \leq 2\frac{\norm{\cV}-1}{\norm{\cV}} \leq n_e\frac{\norm{\cV}-1}{\norm{\cV}}$, since we assumed that $\norm{\cV} \geq 2$. Thus, for all $e \in \cE_G$, we have $\beta_e \leq n_e\frac{\norm{\cV}-1}{\norm{\cV}}$ and $\alpha_e = \frac{\norm{\cV}}{2(\norm{\cV}-1)}\beta_e \leq \frac{n_e}{2}$.

For all $e \in \cE_G$, we define $p_e = \frac{n_e}{\alpha_e} \in [2,+\infty]$ and $q_e=\frac{\norm{\cV}}{\norm{\cV}-1}\frac{p_e}{2}=\frac{n_e}{\beta_e}$. Recalling Definition~\ref{def: HBL inequality} and Equation~\eqref{eq: def convex HBL interpolation}, since $\alpha \in \fC_{\ul}$, a HBL inequality associated with $(\ul,\frac{1}{\ualpha})$ holds for some constant $C>0$. Since $\frac{1}{\ualpha}= \sum_{v \in \cV} 2 \one_v+\sum_{e \in \cE_G}\frac{p_e}{n_e}\one_e$, this concludes the proof of~\eqref{eq: HBL L 2}. Similarly, we have $\frac{1}{\ubeta}= \sum_{v \in \cV} \norm{\cV} \one_v+\sum_{e \in \cE_G}\frac{q_e}{n_e}\one_e$, and $\ubeta \in \fC_{\ul}$, which proves~\eqref{eq: HBL L p/p-1}.

Let $\xi \in (0,+\infty)^{\cV}$, by Lemma~\ref{lem: spanning tree and HBL} we have $\um^{(v)} \in\fC_{\ul^{(\xi)}}$ for all $v \in \cV$. Considering the projections of the previous multi-indices on $\N^{\cE_G}$, we have: $\ubeta_{\cE_G} = \frac{1}{\norm{\cV}}\sum_{v \in \cV} \um^{(v)} \in \fC_{\ul^{(\xi)}}$ by convexity. Thus a HBL inequality associated with $\ul^{(\xi)}$ and $\frac{1}{\ubeta_{\cE_G}} = \sum_{e \in \cE_G}\frac{q_e}{n_e}\one_e$ holds, which proves the existence of $C_\xi>0$ such that~\eqref{eq: HBL xi} holds.
\end{proof}

%%%%%%%%%%%%%%%%%%%%%%%%%%%%%%%%%%%%%%%%%%%%%%%%%%%%%%%%%%%%%%%%%%%%%%%%%%%%%%%%%%%%%%%%%%%%%%%%%%%%%%%

\subsection{Integral estimates for the universal part of the Kac--Rice density}
\label{subsec: integral estimates for the universal part of the Kac--Rice density}

This section is concerned with proving upper bounds for some integrals where the universal part of the Kac--Rice density appears. Recall that the thick diagonals $\diag_{\cI,\eta}$ were defined in Definition~\ref{def: diag I eta}.

\begin{lem}[Local mass of the singularity]
\label{lem: integral of Upsilon J}
Let $A$ be a non-empty finite set, $\cJ \in \cP_A$ and $\eta>0$. Then, the following positive constant is finite:
\begin{equation*}
C_{\cJ,\eta} = \int_{\ux \in \odiag}\ \prod_{J \in \cJ} \Upsilon_J(\ux_J) \dx \ux <+\infty,
\end{equation*}
where $\odiag = \brackets*{\ux \in \diag_{\brackets{A},\eta} \mvert \flat(\ux)=0}$, the functions $\Upsilon_J$ were introduced in Definition~\ref{def: Upsilon}, and $\dx \ux$ is the Lebesgue measure on the subspace $\flat^{-1}(0) = \brackets*{\ux \in (\R^d)^A \mvert \sum_{a \in A}x_a =0}$.
\end{lem}

\begin{proof}
By Lemma~\ref{lem: regularity Upsilon}, for all $J \in \cJ$, the function $\Upsilon_J$ is translation-invariant. Hence, letting $\B$ denote the closed ball of center $0$ and volume $1$ in $\R^d$, we have:
\begin{equation*}
C_{\cJ,\eta} = \int_{\tau \in \B} \int_{\ux \in \odiag} \prod_{J \in \cJ} \Upsilon_J(\tau \cdot \ux_J) \dx \ux\dx \tau = \norm{A}^{-\frac{d}{2}}\int_{\uy \in \Gamma} \prod_{J \in \cJ} \Upsilon_J(\uy_J) \dx \uy,
\end{equation*}
where $\Gamma = \brackets{\tau \cdot \ux \mid \ux \in \odiag \ \text{and}\ \tau \in \B} = \brackets{\uy \in \diag_{\brackets{A},\eta} \mid \flat(\uy) \in \B}$ and we performed the linear change of variable $\uy =\tau \cdot \ux$, whose Jacobian is $\norm{A}^{-\frac{d}{2}}$. Denoting by $\Gamma_J$ the projection of $\Gamma$ onto $(\R^d)^J$ for all $J \in \cJ$, we have $\Gamma \subset \prod_{J \in \cJ} \Gamma_J$, and hence:
\begin{equation*}
C_{\cJ,\eta}\leq \prod_{J \in \cJ} \int_{\uy_J \in \Gamma_J} \Upsilon_J(\uy_J) \dx \uy_J.
\end{equation*}
By Corollary~\ref{cor: compactness diag A eta 0}, the set $\odiag$ is compact. Therefore, $\Gamma$ is compact, and so are its projections $(\Gamma_J)_{J \in \cJ}$. By Lemma~\ref{lem: regularity Upsilon}, the functions $(\Upsilon_J)_{J \in \cJ}$ are locally integrable. Hence, $C_{\cJ,\eta}<+\infty$.
\end{proof}

\begin{dfn}[An integral transform]
\label{def: integral transform}
Let $A$ be a non-empty finite set, $\cJ \in \cP_A$ and $\eta>0$. Let $\uphi = (\phi_a)_{a \in A}$ be a family of Borel-measurable functions from $\R^d$ to $(-\infty,+\infty]$. We define a function $\cT_{\cJ,\eta}(\uphi)$ on $\R^d$ by the following formula, whenever it makes sense:
\begin{equation*}
\cT_{\cJ,\eta}(\uphi) : \tau \longmapsto \int_{\ux \in \odiag} \parentheses*{\prod_{a \in A}\phi_a(\tau+x_a)}\parentheses*{\prod_{J \in \cJ} \Upsilon_J(\ux_J)} \dx \ux,
\end{equation*}
where $\dx \ux$ is the Lebesgue measure on the Euclidean subspace $\flat^{-1}(0) \subset (\R^d)^A$.
\end{dfn}

For reasons that will be apparent later on, it is convenient for us to state the following estimates as the continuity of some multilinear maps between Lebesgue spaces.

\begin{lem}[Regularity of $\cT_{\cJ,\eta}$]
\label{lem: regularity T J eta}
Let $A \neq \emptyset$ be finite, $\cJ \in \cP_A$ and $\eta>0$. For all $s \in [1,+\infty]$, the map $\cT_{\cJ,\eta}:\parentheses*{L^{s\norm{A}}(\R^d)}^A \to L^s(\R^d)$ is multilinear continuous of norm at most~$C_{\cJ,\eta}$.
\end{lem}

\begin{proof}
Since $\cT_{\cJ,\eta}$ is multilinear, the proof of its continuity boils down to proving an integral estimate. Let $\uphi=(\phi_a)_{a \in A}$ be Borel-measurable functions on $\R^d$. For all $\tau \in \R^d$, we have:
\begin{equation*}
\int_{\ux \in \odiag} \parentheses*{\prod_{a \in A}\norm*{\phi_a(\tau+x_a)}}\parentheses*{\prod_{J \in \cJ} \Upsilon_J(\ux_J)} \dx \ux \leq C_{\cJ,\eta} \prod_{a \in A}\Norm{\phi_a}_\infty.
\end{equation*}
In particular, if $\phi_a \in L^\infty(\R^d)$ for all $a \in A$, then $\cT_{\cJ,\eta}(\uphi):\R^d \to \R$ is well-defined and bounded by $C_{\cJ,\eta} \prod_{a \in A}\Norm{\phi_a}_\infty$. This proves the result in the case $s=+\infty$.

Let us now assume that $s \in [1,+\infty)$. In the following, we denote by $\norm{\uphi}=(\norm{\phi_a})_{a \in A}$. Applying Jensen's inequality to the probability measure with density $\ux \mapsto \frac{1}{C_{\cJ,\eta}}\prod_{J \in \cJ} \Upsilon_J(\ux_J)$, we get:
\begin{equation*}
\parentheses*{\cT_{\cJ,\eta}(\norm{\uphi})(\tau)}^s \leq (C_{\cJ,\eta})^{s-1}\int_{\ux \in \odiag} \prod_{a \in A}\norm*{\phi_a(\tau+x_a)}^s\prod_{J \in \cJ} \Upsilon_J(\ux_J) \dx \ux
\end{equation*}
for all $\tau \in \R^d$. Integrating with respect to $\tau$ and applying Hölder's inequality, we obtain:
\begin{align*}
\int_{\tau \in \R^d}\parentheses*{\cT_{\cJ,\eta}(\norm{\uphi})(\tau)}^s \dx \tau &\leq (C_{\cJ,\eta})^{s-1} \int_{\ux \in \odiag} \Norm*{\prod_{a \in A} \phi_a(\cdot+x_a)}_s^s \prod_{J \in \cJ} \Upsilon_J(\ux_J) \dx \ux\\
&\leq (C_{\cJ,\eta})^{s-1} \int_{\ux \in \odiag} \parentheses*{\prod_{a \in A} \Norm*{\phi_a(\cdot+x_a)}_{s\norm{A}}^s} \prod_{J \in \cJ} \Upsilon_J(\ux_J) \dx \ux\\
&\leq (C_{\cJ,\eta})^s \prod_{a \in A} \Norm*{\phi_a}_{s\norm{A}}^s.
\end{align*}
Thus, if $\phi_a \in L^{s\norm{A}}(\R^d)$ for all $a \in A$, then $\cT_{\cJ,\eta}(\norm{\uphi})(\tau) <+\infty$ for almost every $\tau \in \R^d$. In this case, $\cT_{\cJ,\eta}(\uphi)$ is well-defined almost everywhere, and $\Norm{\cT_{\cJ,\eta}(\uphi)}_s \leq \Norm{\cT_{\cJ,\eta}(\norm{\uphi})}_s \leq C_{\cJ,\eta} \prod_{a \in A}\Norm{\phi_a}_{s\norm{A}}$. This proves the result if $s \in [1,+\infty)$.
\end{proof}

Let $A$ be a non-empty finite set and $\cI \in \cP_A$. As in Section~\ref{subsec: HBL inequalities and edge-connected graphs}, we define $\ell_I:(y_J)_{J \in \cI} \mapsto y_I$ from $(\R^d)^\cI$ to $\R^d$ for all $I \in \cI$. Similarly, for all $e = \brackets{I,J} \in \cI^{\cwedge}$ we define $\ell_e$ either as $\ell_I -\ell_J$ or as $\ell_J-\ell_I$ arbitrarily. Recalling Definition~\ref{def: barycenter}, we let $\flat_\cI:\ux \mapsto (\flat(\ux_I))_{I \in \cI}$ from $(\R^d)^A$ to $(\R^d)^\cI$.

\begin{dfn}[Integration functionals]
\label{def: H G J eta}
Let $A$ be a non-empty finite set, $\cI,\cJ \in \cP_A$ be such that $\cJ \leq \cI$ and $\eta >0$. Let $G=(\cI,\un)$ be a graph. For every families $\uphi=(\phi_a)_{a \in A}$ and $\ug=\parentheses*{g_e}_{e \in \cE_G}$ of Borel-measurable functions from $\R^d$ to $(-\infty,+\infty]$, we denote by
\begin{equation*}
\cH_{G,\cJ,\eta}(\uphi,\ug) = \int_{\ux \in \diag_{\cI,\eta}} \parentheses*{\prod_{a \in A}\phi_a(x_a)}\parentheses*{\prod_{J \in \cJ} \Upsilon_J(\ux_J)} \parentheses*{\prod_{e \in \cE_G} g_e\parentheses*{\strut\ell_e\parentheses*{\flat_\cI(\ux)}}^{n_e}} \dx \ux,
\end{equation*}
if this integral is well-defined.
\end{dfn}

Recalling Proposition~\ref{prop: Kac-Rice cumulants}, Corollary~\ref{cor: key estimate graphs} and Lemma~\ref{lem: distance between covariances of twisted interpolants}, quantities of the form $\cH_{G,\cJ,\eta}(\uphi,\ug)$ are typically what we need to control in order to prove Theorem~\ref{thm: cumulants asymptotics for zero sets}. The estimates we need are summarized in the following lemma.

\begin{lem}[Integral estimates for $\cH_{G,\cJ,\eta}$]
\label{lem: estimates H G J eta}
If $\card(\cI)\geq 2$ and the graph $G=(\cI,\un)$ is $2$-edge-connected, then the map $\ug \mapsto \cH_{G,\cJ,\eta}(\cdot,\ug)$ is well-defined and continuous in the following three cases, where the exponents $(p_e)_{e \in \cE_G}$ and $(q_e)_{e \in \cE_G}$ are given by Proposition~\ref{prop: HBL graphs}:
\begin{enumerate}
\item \label{item: estimate H infty} from $\prod_{e \in \cE_G} L^\infty(\R^d)$ to the space of the continuous multilinear forms on $\prod_{I \in \cI} \parentheses*{L^{\norm{I}}(\R^d)}^I$;

\item \label{item: estimate H p} from $\prod_{e \in \cE_G} L^{p_e}(\R^d)$ to the space of the continuous multilinear forms on $\prod_{I \in \cI} \parentheses*{L^{2\norm{I}}(\R^d)}^I$;

\item \label{item: estimate H q} from $\prod_{e \in \cE_G} L^{q_e}(\R^d)$ to the space of the continuous multilinear forms on $\prod_{I \in \cI} \parentheses*{L^{\norm{\cI}\norm{I}}(\R^d)}^I$.
\end{enumerate}
\end{lem}

\begin{proof}
For any family $\uvarphi=(\varphi_a)_{a \in A \sqcup \cE_G}$ of Borel-measurables functions on $\R^d$, we define:
\begin{equation*}
\tilde{\cH}_{G,\cJ,\eta}(\uvarphi) = \int_{\ux \in \diag_{\cI,\eta}} \parentheses*{\prod_{a \in A}\varphi_a(x_a)}\parentheses*{\prod_{J \in \cJ} \Upsilon_J(\ux_J)} \parentheses*{\prod_{e \in \cE_G} \varphi_e\parentheses*{\strut \ell_e\parentheses*{\flat_\cI(\ux)}}} \dx \ux,
\end{equation*}
if this integral makes sense, so that $\cH_{G,\cJ,\eta}\parentheses*{\strut(\phi_a)_{a \in A},(g_e)_{e \in \cE_G}} =\tilde{\cH}_{G,\cJ,\eta}\parentheses*{\strut (\phi_a)_{a \in A},(g_e^{n_e})_{e \in \cE_G}}$. The first step of the proof is to bound $\tilde{\cH}_{G,\cJ,\eta}(\uvarphi)$ by an integral quantity where the functions $(\cT_{\cJ_I,\eta})_{I \in \cI}$ from Definition~\ref{def: integral transform} appear. Then, the claimed continuities follow from Proposition~\ref{prop: HBL graphs} and Lemma~\ref{lem: regularity T J eta}.

\paragraph*{Step 1: Bounding $\tilde{\cH}_{G,\cJ,\eta}(\uvarphi)$.}
Let $\ux \in \diag_{\cI,\eta}$, for all $I \in \cI$, we have $\ux_I \in \diag_{\brackets{I},\eta}$ by Lemma~\ref{lem: projection diag I eta}. Hence, we have $\ux_I = \flat(\ux_I) + \obullet{\overgroup{\ux_I}}$, where $\flat(\ux_I)\in \R^d$ and $\obullet{\overgroup{\ux_I}} \in \obullet{\diag}_{\brackets{I},\eta}$ were defined in Definition~\ref{def: barycenter}. Thus $\ux \mapsto \parentheses*{\flat_\cI(\ux),(\obullet{\overgroup{\ux_I}})_{I \in \cI}}$ is an isomorphism from $(\R^d)^A$ to $(\R^d)^\cI \times \flat_\cI^{-1}(0)$ that maps $\diag_{\cI,\eta}$ to a subset of $(\R^d) ^\cI \times \prod_{I \in \cI}\obullet{\diag}_{\brackets{I},\eta}$. One can check that the inverse of its Jacobian is $\prod_{I \in \cI} \norm{I}^{\frac{d}{2}} \leq \norm{A}^{\frac{d\norm{A}}{2}}$.

Let us consider a family $\uvarphi=(\varphi_a)_{a \in A\sqcup \cE_G}$ of Borel-measurables functions from $\R^d$ to $(-\infty,+\infty]$, and denote $\norm{\uvarphi} = (\norm{\varphi_a})_{a \in A}$. Starting from the definition of $\tilde{\cH}_{G,\cJ,\eta}(\norm{\uvarphi})$, we perform the linear change of variable $\parentheses{\tau_I,\uz_I}_{I \in \cI}= \parentheses{\flat(\ux_I),\obullet{\overgroup{\ux_I}}}_{I \in \cI}$. Recalling that $\cJ \leq \cI$ and that the functions $(\Upsilon_J)_{J \in \cJ}$ are non-negative and translation-invariant (see Lemma~\ref{lem: regularity Upsilon}), we obtain:
\begin{multline}
\label{eq: bound H G J eta}
\tilde{\cH}_{G,\cJ,\eta}\parentheses*{\norm{\uvarphi}} =\int_{\ux \in \diag_{\cI,\eta}} \prod_{I \in \cI} \parentheses*{\prod_{i \in I}\norm{\varphi_i(x_i)}\prod_{J \in \cJ_I} \Upsilon_J(\ux_J)} \prod_{e\in\cE_G} \norm*{\varphi_e\parentheses*{\strut \ell_e\parentheses*{\flat_\cI(\ux)}}} \dx \ux\\
\leq \norm{A}^{\frac{d\norm{A}}{2}} \int_{\utau \in (\R^d)^\cI} \int_{\uz \in \prod_{I \in \cI}\obullet{\diag}_{\brackets{I},\eta}} \prod_{I \in \cI} \parentheses*{\prod_{i \in I}\norm{\varphi_i(\tau_I+z_i)}\prod_{J \in \cJ_I} \Upsilon_J(\uz_J)} \prod_{e\in\cE_G} \norm*{\varphi_e\parentheses*{\strut \ell_e(\utau)}} \dx \uz \dx \utau\\
\leq \norm{A}^{\frac{d\norm{A}}{2}} \int_{\utau \in (\R^d)^\cI} \prod_{I \in \cI} \cT_{\cJ_I,\eta}\parentheses{\norm{\uvarphi_I}}(\tau_I) \prod_{e\in\cE_G} \norm*{\varphi_e\parentheses*{\ell_e(\utau)}}\dx \utau.
\end{multline}

\paragraph*{Step 2: Deducing the continuity from HBL inequalities.} Let us first prove Item~\ref{item: estimate H p}. By Proposition~\ref{prop: HBL graphs} and Lemma~\ref{lem: regularity T J eta} with $s=2$, we deduce from~\eqref{eq: bound H G J eta} that:
\begin{align*}
\tilde{\cH}_{G,\cJ,\eta}\parentheses*{\norm{\uvarphi}} &\leq C\norm{A}^{\frac{d\norm{A}}{2}} \parentheses*{\prod_{I \in \cI} \Norm*{\cT_{\cJ_I,\eta}\parentheses{\norm{\uvarphi_I}}}_2}\parentheses*{\prod_{e \in \cE_G} \Norm{\varphi_e}_\frac{p_e}{n_e}}\\
&\leq C\norm{A}^{\frac{d\norm{A}}{2}}\parentheses*{\prod_{I \in \cI} C_{\cJ_I,\eta}} \parentheses*{\prod_{I \in \cI} \prod_{i \in \cI} \Norm{\varphi_i}_{2\norm{I}}}\parentheses*{\prod_{e \in \cE_G} \Norm{\varphi_e}_\frac{p_e}{n_e}}.
\end{align*}
If $\uvarphi \in \prod_{I \in \cI}\parentheses*{L^{2\norm{I}}(\R^d)}^I \times \prod_{e \in \cE_G}L^\frac{p_e}{n_e}(\R^d)$ then $\tilde{\cH}_{G,\cJ,\eta}\parentheses*{\norm{\uvarphi}} <+\infty$, and $\tilde{\cH}_{G,\cJ,\eta}\parentheses*{\uvarphi}$ is well-defined. Thus, $\tilde{\cH}_{G,\cJ,\eta}$ is a well-defined multilinear form on $\prod_{I \in \cI}\parentheses*{L^{2\norm{I}}(\R^d)}^I \times \prod_{e \in \cE_G}L^\frac{p_e}{n_e}(\R^d)$. For all $\uvarphi$ in this space, we have $\norm*{\tilde{\cH}_{G,\cJ,\eta}\parentheses*{\uvarphi}} \leq \tilde{\cH}_{G,\cJ,\eta}\parentheses*{\norm{\uvarphi}}$, and the previous inequality shows that $\tilde{\cH}_{G,\cJ,\eta}$ is continuous. Hence, $\uvarphi_{\cE_G} \mapsto \tilde{\cH}_{G,\cJ,\eta}\parentheses{\cdot,\uvarphi_{\cE_G}}$ is multilinear continuous from $\prod_{e \in \cE_G} L^\frac{p_e}{n_e}(\R^d)$ to the space of continuous multilinear forms on $\prod_{I \in \cI}\parentheses*{L^{2\norm{I}}(\R^d)}^I$.

For all $n \in \N^*$ and $p \in [1,+\infty]$, the map $(\varphi_1,\dots,\varphi_n) \mapsto \prod_{i=1}^{n}\varphi_i$ is multilinear continuous from $\parentheses*{L^p(\R^d)}^n$ to $L^\frac{p}{n}(\R^d)$, by Hölder's inequality. Hence, for all $e \in \cE_G$, the map $g_e \mapsto g_e^{n_e}$ is continuous from $L^{p_e}(\R^d)$ to $L^\frac{p_e}{n_e}(\R^d)$, as the restriction of one the previous to the deepest diagonal. Thus, the map $\ug \mapsto \cH_{G,\cJ,\eta}(\cdot,\ug)$ is well-defined and continuous between the space appearing in Item~\ref{item: estimate H p}, as the composition of $(g_e)_{e \in \cE_G} \mapsto (g_e^{n_e})_{e \in \cE_G}$ and $(\varphi_e)_{e \in \cE_G} \mapsto \cH_{G,\cJ,\eta}\parentheses*{\cdot,(\varphi_e)_{e \in \cE_G}}$.

The proof of Item~\ref{item: estimate H q} is similar. We just replace $(p_e)_{e \in \cE_G}$ by $(q_e)_{e \in \cE_G}$ in the previous argument, use~Equation~\eqref{eq: HBL L p/p-1} instead of~\eqref{eq: HBL L 2} and use, this time, the continuity of $T_{\cJ_I,\eta}$ from $\parentheses*{L^{\norm{\cI}\norm{I}}(\R^d)}^I$ to $L^{\norm{\cI}}(\R^d)$ for all $I \in \cI$.

Concerning Item~\ref{item: estimate H infty}, we start from Equation~\eqref{eq: bound H G J eta} and bound the functions $(\varphi_e)_{e \in \cE_G}$ by their sup-norm. Using Lemma~\ref{lem: regularity T J eta} with $s=1$ to bound $\Norm*{\cT_{\cJ_I,\eta}\parentheses{\norm{\uvarphi_I}}}_1$ for all $I \in \cI$, we obtain:
\begin{equation*}
\tilde{\cH}_{G,\cJ,\eta}\parentheses*{\norm{\uvarphi}} \leq \norm{A}^{\frac{d\norm{A}}{2}}\parentheses*{\prod_{I \in \cI} C_{\cJ_I,\eta}} \parentheses*{\prod_{I \in \cI} \prod_{i \in \cI} \Norm{\varphi_i}_{\norm{I}}}\parentheses*{\prod_{e \in \cE_G} \Norm{\varphi_e}_\infty}.
\end{equation*}
The conclusion follows by the same argument as in the other two cases.
\end{proof}

We need a similar estimate for an integral over $\obullet{\diag}_{\cI,\eta} := \diag_{\cI,\eta} \cap \flat^{-1}(0)$, with respect to the Lebesgue measure on Euclidean the subspace $\flat^{-1}(0)$. The integrand in Lemma~\ref{lem: integrable dominating function} will naturally appear as the dominating function in a dominated convergence later on.

\begin{lem}[Integrable dominating function]
\label{lem: integrable dominating function}
Let $A$ be a non-empty finite set, let $\cI,\cJ \in \cP_A$ be such that $\cJ \leq \cI$ and $\eta >0$. Let $G=(\cI,\un)$ be a $2$-edge-connected graph. For every $e \in \cE_G$, let $g_e \in L^{q_e}(\R^d)$ be non-negative, where $(q_e)_{e \in \cE_G}$ is given by Proposition~\ref{prop: HBL graphs}. Then,
\begin{equation*}
\int_{\ux \in \obullet{\diag}_{\cI,\eta}} \parentheses*{\prod_{J \in \cJ} \Upsilon_J(\ux_J)} \parentheses*{\prod_{e \in \cE_G} g_e\parentheses*{\ell_e\parentheses*{\flat_\cI(\ux)}}^{n_e}} \dx \ux <+\infty,
\end{equation*}
where $\dx \ux$ is the Lebesgue measure on $\flat^{-1}(0) = \brackets*{(x_a)_{a \in A} \in (\R^d)^A \mvert \sum_{a \in A} x_a =0}$.
\end{lem}

\begin{proof}
Let us denote by $\xi=\parentheses*{\norm{I}}_{I \in \cI} \in (0,+\infty)^{\cI}$ and $V_\xi = \brackets*{(\tau_I)_{I \in \cI} \in(\R^d)^\cI \mvert \sum_{I \in \cI} \norm{I}\tau_I=0}$. We have checked, in the first step of the proof of Lemma~\ref{lem: estimates H G J eta}, that $\ux \mapsto \parentheses*{\flat_\cI(\ux),(\obullet{\overgroup{\ux_I}})_{I \in \cI}}$ is an isomorphism from $(\R^d)^A$ to $(\R^d)^\cI \times \flat_{\cI}^{-1}(0)$ that maps $\diag_{\cI,\eta}$ onto a subset of $(\R^d)^\cI \times \prod_{I \in \cI} \obullet{\diag}_{\brackets{I},\eta}$. For all $\ux=(x_a)_{a \in A} \in (\R^d)^A$, we have $\norm{A}\flat(\ux) = \sum_{I \in \cI}\sum_{i \in I}x_i = \sum_{I \in \cI}\norm{I}\flat(\ux_I)$. Hence, the previous isomorphism restricts into an isomorphism from $\flat^{-1}(0)$ to $V_\xi \times \flat_\cI^{-1}(0)$, that maps $\obullet{\diag}_{\cI,\eta}$ onto a subset of $V_\xi \times \prod_{I \in \cI} \obullet{\diag}_{\brackets{I},\eta}$. Letting $C>0$ denote the inverse of its Jacobian, we perform the corresponding linear change of variable. Then, the integral we are interested in is bounded by:
\begin{multline*}
C\int_{\utau \in V_\xi} \parentheses*{\prod_{I \in \cI} \int_{\uz_I \in \obullet{\diag}_{\brackets{I},\eta}} \prod_{J \in \cJ_I} \Upsilon_J(\uz_J) \dx \uz_I} \parentheses*{\prod_{e \in \cE_G} g_e\parentheses*{\ell_e(\tau_I)}^{n_e}}\dx \utau\\
= C\parentheses*{\prod_{I \in \cI}C_{\cJ_I,\eta}}\int_{\utau \in V_\xi} \prod_{e \in \cE_G} g_e\parentheses*{\ell_e(\tau_I)}^{n_e}\dx \utau\leq C\parentheses*{\prod_{I \in` \cI}C_{\cJ_I,\eta}}C_\xi \prod_{e \in \cE_G} \Norm{g_e^{n_e}}_\frac{q_e}{n_e},
\end{multline*}
by Proposition~\ref{prop: HBL graphs}. Since $g_e \in L^{q_e}(\R^d)$ for all $e \in \cE_G$, the right-hand side is finite.
\end{proof}

%%%%%%%%%%%%%%%%%%%%%%%%%%%%%%%%%%%%%%%%%%%%%%%%%%%%%%%%%%%%%%%%%%%%%%%%%%%%%%%%%%%%%%%%%%%%%%%%%%%%%%%
%%%%%%%%%%%%%%%%%%%%%%%%%%%%%%%%%%%%%%%%%%%%%%%%%%%%%%%%%%%%%%%%%%%%%%%%%%%%%%%%%%%%%%%%%%%%%%%%%%%%%%%

\section{Proof of Theorem~\ref{thm: cumulants asymptotics for zero sets}: cumulants asymptotics for zero sets}
\label{sec: proof of thm cumulants asymptotics zero sets}

This section is concerned with the proof of the cumulants asymptotics of Theorem~\ref{thm: cumulants asymptotics for zero sets}. Let us recall the setting. The following notation will remain fixed in this whole section.

Let $d \in \N^*$ and $k \in \ssquarebrackets{1}{d}$, we denote by $f:\R^d \to \R^k$ a centered stationary Gaussian field and by~$r$ its covariance kernel. Let $\Omega \subset U \subset \R^d$ be open subsets such that $U$ is convex and $\dist(\Omega,\R^d\setminus U) >0$. Let $\cR \subset (0,+\infty)$ be unbounded. For all $R \in \cR$, we consider a centered Gaussian field $f_R:RU \to \R^k$ and denote by $r_R$ its covariance kernel. In order to simplify some statements, we let $f_\infty=f$ and use the convention that $\infty \Omega = \R^d= \infty U$. In the following, $R$ is always implicitly assumed to be an element of $\cR \sqcup \brackets{\infty}$.

Let $p \in \N^*$, we assume in all this section that the following hold: \hypReg{\max(2,2p-1)}, \hypScL{2p-1}, \hypND{2p-1} and \hypDC{2p-1}{s}, for some $s \in [1,+\infty]$. Let $g:\R^d \to [0,+\infty)$ and $\omega>0$ be those appearing in Hypothesis~\hypDC{2p-1}{s}, recall that the function $g_\omega$ is defined as $g_\omega:x \mapsto \sup_{\Norm{y}\leq \omega} g(x+y)$. Recall also that, under these assumptions, \hypDCL{2p-1}{s} also holds for the same $\omega$ and $g$, see Section~\ref{subsec: discussion of the main hypotheses}.

%%%%%%%%%%%%%%%%%%%%%%%%%%%%%%%%%%%%%%%%%%%%%%%%%%%%%%%%%%%%%%%%%%%%%%%%%%%%%%%%%%%%%%%%%%%%%%%%%%%%%%%

\subsection{Validity of the Kac--Rice formula under \texorpdfstring{\hypDC{2p-1}{\infty}}{}}
\label{subsec: validity of Kac-Rice}

In this section, we work under the assumption that~\hypDC{2p-1}{\infty} holds, that is $g$ is bounded and goes to $0$ at infinity. The parameter $\omega$ might seem irrelevant in this case, yet it will be useful later on to have it appear in the statements of this section. We will prove that the fields $(f_R)_{R \in \cR \sqcup \brackets{\infty}}$ are uniformly non-degenerate in the sense of Corollary~\ref{cor: uniform non degeneracy f R} and that the hypotheses of the Kac--Rice formula for cumulants are satisfied, see Proposition~\ref{prop: Kac-Rice cumulants}. As a consequence, we prove Theorem~\ref{thm: cumulants asymptotics for zero sets}.\ref{item: cumulants asymptotics infty}.

\begin{lem}[Existence of a good scale parameter]
\label{lem: existence eta}
There exists $\eta \in (0,\frac{\omega}{2p}]$ and $R_\eta \in \cR$ such that the following holds. For all non-empty $A \subset \ssquarebrackets{1}{p}$, for all $\cK \in \cP_A$, for all $\cJ \leq \cI$ in $\cP_\cK$, there exists a compact $\Gamma_{\cI,\cJ,\eta} \subset \cM_\cJ \times \cL^\dagger_{\cJ,\cI}\times \cS_\cI^+$ such that, for all $R \geq R_\eta$ and $\ux \in \diag_{\cI,\eta}\cap (R\Omega)^\cK$,
\begin{equation*}
\parentheses*{\parentheses*{\cG_J(\ux_J)}_{J \in \cJ},\parentheses*{\theta_J(\ux_J)}_{J \in \cJ},\parentheses*{\Pi_{2J}^{2[J]_\cI}(\cdot,\ux_{2[J]_\cI})}_{J \in \cJ},\Sigma_{2\cI}(f_R,\ux_{2\cK})} \in \Gamma_{\cI,\cJ,\eta}.
\end{equation*}
\end{lem}

\begin{proof}
For all non-empty $A \subset \ssquarebrackets{1}{p}$ and $\cK \in \cP_A$, we have $\norm{\cK}\leq p$. Under the hypotheses of this section, we can apply Corollary~\ref{cor: uniform non degeneracy f R} to the set $2\cK$. This yields the existence of $\eta_\cK >0$ such that, for all $\eta \in (0,\eta_\cK]$, there exists $R_{\eta,\cK} \in \cR$ such that, for all $\tilde{\cI} \in \cP_{2\cK}$, the set
\begin{equation*}
\brackets*{\Sigma_{\tilde{\cI}}(f_R,\ux) \mvert R \geq R_{\eta,\cK} \ \text{and} \  \ux \in \diag_{\tilde{\cI},\eta} \cap (R\Omega)^{2\cK}}
\end{equation*}
is relatively compact in $\sym^+\parentheses*{\prod_{I \in \tilde{\cI}}\R_{\norm{I}-1}[X]^k}$. If $\tilde{\cI} \in \cP_{2\cK}$ is of the form $2\cI$ with $\cI \in \cP_\cK$, then $\sym^+\parentheses*{\prod_{I \in \tilde{\cI}}\R_{\norm{I}-1}[X]^k} = \cS_\cI^+$, see Definition~\ref{def: M I}. Moreover, for all $\ux \in \diag_{\cI,\eta} \cap (R\Omega)^\cK$, we have $\ux_{2\cK} \in \diag_{\tilde{\cI},\eta} \cap (R\Omega)^{2\cK}$, see Remark~\ref{rem: diag I eta doubled points}. Hence, $\brackets*{\Sigma_{2\cI}(f_R,\ux_{2\cK}) \mvert R \geq R_{\eta,\cK} \ \text{and} \  \ux \in \diag_{\cI,\eta} \cap (R\Omega)^\cK}$ is relatively compact in $\cS_\cI^+$ for all $\cI \in \cP_\cK$. We define
\begin{align*}
& & \eta &= \min \brackets*{\strut \eta_\cK \mvert \emptyset \neq A \subset \ssquarebrackets{1}{p} \ \text{and} \ \cK \in \cP_A} \cup \brackets*{\frac{\omega}{2p}}>0\\
&\text{and} & R_\eta &= \max \brackets*{\strut R_{\eta,\cK} \mvert \emptyset \neq A \subset \ssquarebrackets{1}{p} \ \text{and} \ \cK \in \cP_A}.
\end{align*}
Then, for all non-empty $A \subset \ssquarebrackets{1}{p}$, all $\cK \in \cP_A$ and all $\cI \in \cP_\cK$, there exists a compact $\Gamma'_{\cI,\eta} \subset \cS_\cI^+$ such that $\brackets*{\Sigma_{2\cI}(f_R,\ux_{2\cK}) \mvert R \geq R_\eta \ \text{and} \  \ux \in \diag_{\cI,\eta} \cap (R\Omega)^\cK} \subset \Gamma'_{\cI,\eta}$.

Let $A \subset \ssquarebrackets{1}{p}$ be non-empty, $\cK \in \cP_A$ and $\cJ \leq \cI$ in $\cP_\cK$. Let $J \in \cJ$ and $I = [J]_\cI$. For all $\ux \in \diag_{\cI,\eta}$, we have $\ux_I \in \diag_{\brackets{I},\eta}$, by Lemma~\ref{lem: projection diag I eta}. Then, by Lemma~\ref{lem: prop Pi BA} and Corollary~\ref{cor: compactness diag A eta 0},
\begin{equation*}
\brackets*{\Pi_{2J}^{2[J]_\cI}(\cdot,\ux_{2[J]_\cI}) \mvert \ux \in \diag_{\cI,\eta}} \subset \brackets*{\Pi_{2J}^{2I}(\cdot,\ux_{2I}) \mvert \ux_I \in \obullet{\diag}_{\brackets{I},\eta}} = \Gamma''_{J,\eta},
\end{equation*}
and $\Gamma''_{J,\eta}$ is compact in $\cL^\dagger\parentheses*{\R_{2\norm{[J]_\cI}-1}[X]^k,\R_{2\norm{J}-1}[X]^k}$ as the continuous image of a compact. We define $\Gamma_{\cI,\cJ,\eta}= \cM_\cJ \times \prod_{J \in \cJ} \Gamma''_{J,\eta} \times \Gamma'_{\cI,\eta}$, which is a compact subset of $\cM_\cJ \times \cL^\dagger_{\cJ,\cI} \times \cS_\cI^+$ (see Definitions~\ref{def: M I} and~\ref{def: L J}) with the claimed property.
\end{proof}

A first consequence of Lemma~\ref{lem: existence eta} is the following.

\begin{lem}[Uniform domination of $\tilde{F}_{\cI,\cJ}$]
\label{lem: uniform domination FIJ}
Let $\eta \in (0,\frac{\omega}{2p}]$ and $R_\eta \in \cR$ be given by Lemma~\ref{lem: existence eta}. There exists $C_\eta>0$ such that, for all non-empty $A \subset \ssquarebrackets{1}{p}$, all $\cK \in \cP_A$, all $\cJ \leq \cI$ in $\cP_\cK$, all $R \geq R_\eta$ and all $\ux \in \diag_{\cI,\eta} \cap (R\Omega)^\cK$ we have:
\begin{equation*}
\norm*{\tilde{F}_{\cI,\cJ}(f_R,\ux)} \leq C_\eta \sum_{\un \in \G_\cI} \prod_{e \in \cI^{\cwedge}} g_\omega\parentheses*{\strut\ell_e(\flat_\cI(\ux))}^{n_e},
\end{equation*}
where $\tilde{F}_{\cI,\cJ}$ is the map appearing in Lemma~\ref{lem: relation between cumulant densities}, we defined $\G_\cI$ in Definition~\ref{def G V} and, for all $e = \brackets{I,J} \in \cI^{\cwedge}$, we have $g_\omega\parentheses*{\strut\ell_e(\flat_\cI(\ux))} = g_\omega\parentheses*{\flat(\ux_I)-\flat(\ux_J)}=g_\omega\parentheses*{\flat(\ux_J)-\flat(\ux_I)}$ since $g$ is even.
\end{lem}

\begin{proof}
For all non-empty $A \subset \ssquarebrackets{1}{p}$, all $\cK \in \cP_A$ and all $\cJ \leq \cI$ in $\cP_\cK$, let $\Gamma_{\cI,\cJ,\eta}$ be a compact set given by Lemma~\ref{lem: existence eta}. Let $C_{\cI,\cJ,\eta} \geq 0$ be the constant associated with $\Gamma_{\cI,\cJ,\eta}$ by Corollary~\ref{cor: key estimate graphs}. We define $C_1 = \max \brackets*{C_{\cI,\cJ,\eta}\mvert \emptyset \neq A \subset \ssquarebrackets{1}{p}; \cK \in \cP_A; \cJ \leq \cI \ \text{in} \ \cP_\cK}$. Recalling the relation between $F_{\cI,\cJ}$ and $\tilde{F}_{\cI,\cJ}$ stated in Lemma~\ref{lem: relation between cumulant densities}, for all $R \geq R_\eta$ and $\ux \in \diag_{\cI,\eta}\cap (R\Omega)^\cK$, we have:
\begin{equation*}
\norm*{\tilde{F}_{\cI,\cJ}\parentheses*{f_R,\ux}} \leq C_1 \sum_{\un \in \G_\cI}\ \prod_{\brackets{I,J} \in \cI^{\cwedge}} \Norm*{\Sigma_{2I}^{2J}\parentheses*{f_R,\ux_{2\cK}}}^{n_{\brackets{I,J}}}.
\end{equation*}

By Lemma~\ref{lem: distance between covariances of twisted interpolants}, there exists $C_2 \geq 1$, depending only on $p$ and $\eta$, such that for all $I,J \in \cI$ and $\ux \in \diag_{\cI,\eta} \cap (R\Omega)^\cK$, we have:
\begin{equation*}
\Norm*{\Sigma_{2I}^{2J}\parentheses*{f_R,\ux_{2\cK}}} \leq C_2 \max_{\substack{\norm{\alpha}< 2p\\ \norm{\beta}<2p}} \ \max_{\substack{w \in \conv(\ux_I)\\ z \in \conv(\ux_J)}} \ \Norm*{\strut \partial^{\alpha,\beta}r_R(w,z)} \leq C_2 \max_{\substack{w \in \conv(\ux_I)\\ z \in \conv(\ux_J)}} g(z-w),
\end{equation*}
where the second inequality is given by~\hypDC{2p-1}{\infty} and~\hypDCL{2p-1}{\infty}. If $\ux \in \diag_{\cI,\eta}$ then, by Lemma~\ref{lem: bound on oux I}, for all $i \in I$ we have $\Norm{x_i-\flat(\ux_I)} \leq p\eta$. Hence, $\Norm*{w - \flat(\ux_I)} \leq p\eta$ for all $w \in \conv(\ux_I)$, and similarly $\Norm{z-\flat(\ux_J)}\leq p \eta$ for all $z \in \conv(\ux_J)$. Thus,
\begin{equation*}
\Norm*{(z-w)-(\flat(\ux_J)-\flat(\ux_J))} \leq \Norm{w-\flat(\ux_I)} + \Norm{z-\flat(\ux_J)} \leq 2p\eta \leq \omega.
\end{equation*}
Finally, for all $I,J \in \cI$, we have $\max_{(w,z) \in \conv(\ux_I)\times \conv(\ux_J)} g(z-w) \leq g_\omega\parentheses*{\flat(\ux_I)-\flat(\ux_J)}$. Recalling Lemmas~\ref{lem: description of NB} and~\ref{lem: characterization 2 edge connected}, we have $\norm{\un}\leq 2p$ for all $\un \in \G_\cI$. The result follows, with $C_\eta= C_1 C_2^{2p}$.
\end{proof}

Under the hypotheses of this section, Lemma~\ref{lem: uniform domination FIJ} implies that we can indeed use the Kac--Rice formula for cumulants, see Proposition~\ref{prop: Kac-Rice cumulants}.

\begin{lem}[Validity of the Kac--Rice formula for cumulants]
\label{lem: validity of Kac Rice cumulants}
Let $\eta \in (0,\frac{\omega}{2p}]$ and $R_\eta \in \cR$ be given by Lemma~\ref{lem: existence eta}. Let $R \geq R_\eta$ and $\phi_1,\dots,\phi_p \in L^1(R\Omega)\cap L^\infty(R\Omega)$. For all non-empty $A \subset \ssquarebrackets{1}{p}$, all $\cK\in \cP_A$ and all $\cI, \cJ \in \cP_\cK$ such that $\cJ \leq \cI$, the following function is integrable over $(R\Omega)^{\cK} \cap \diag_{\cI,\eta}$:
\begin{equation*}
\ux \longmapsto \parentheses*{\phi^\otimes_A \circ \iota_{\cK}(\ux)} \tilde{F}_{\cI,\cJ}(f_R,\ux)\prod_{J \in \cJ}  \Upsilon_J(\ux_J).
\end{equation*}
\end{lem}

\begin{proof}
Let $\emptyset \neq A \subset \ssquarebrackets{1}{p}$ and $\cK \in \cP_A$. For all $K \in \cK$, the function $\varphi_K= \prod_{i \in K}\norm{\phi_i}$ is non-negative, and $\varphi_K \in L^1(R\Omega)\cap L^\infty(R\Omega)$ by Hölder's inequality. Extending this function by $0$ outside of $R\Omega$, we obtain $\varphi_K \in L^1(\R^d)\cap L^\infty(\R^d)$. Moreover, for all $\ux =(x_K)_{K \in \cK} \in (R\Omega)^\cK$, we have $\norm*{\phi_A^\otimes \circ \iota_\cK(\ux)} = \prod_{K \in \cK} \prod_{i \in K}\norm{\phi_i(x_K)} = \varphi^\otimes_\cK(\ux)$.

Let $\cI, \cJ \in \cP_\cK$ be such that $\cJ \leq \cI$. By Lemma~\ref{lem: uniform domination FIJ}, for all $\ux \in (R\Omega)^\cK \cap \diag_{\cI,\eta}$ we have:
\begin{equation}
\label{eq: validity of Kac Rice cumulants}
\norm*{\parentheses*{\phi^\otimes_A \circ \iota_{\cK}(\ux)} \tilde{F}_{\cI,\cJ}(f_R,\ux)\prod_{J \in \cJ}  \Upsilon_J(\ux_J)} \leq C_\eta \sum_{\un \in \G_\cI} \varphi^\otimes_\cK(\ux)\parentheses*{\prod_{J \in \cJ}  \Upsilon_J(\ux_J)}\parentheses*{\prod_{e \in \cI^{\cwedge}}g_\omega\parentheses*{\strut\ell_e(\flat_\cI(\ux))}^{n_e}}.
\end{equation}
We need to check that each term on the right-hand side is integrable over $(R\Omega)^\cK \cap \diag_{\cI,\eta}$, or equivalently over $\diag_{\cI,\eta}$.

Let us first assume that $\card(\cI) \geq 2$. Let $G=(\cI,\un) \in \G_\cI$. Recalling Definition~\ref{def: H G J eta}, we deduce from Lemma~\ref{lem: estimates H G J eta}.\ref{item: estimate H infty} that
\begin{multline*}
\int_{\ux \in \diag_{\cI,\eta}}\varphi^\otimes_\cK(\ux) \parentheses*{\prod_{J \in \cJ}  \Upsilon_J(\ux_J)}\parentheses*{\prod_{e \in \cI^{\cwedge}} g_\omega\parentheses*{\strut\ell_e(\flat_\cI(\ux))}^{n_e}} \dx \ux = \cH_{G,\cJ,\eta}(\uvarphi,\ug) \\ \leq \Norm*{\cH_{G,\cJ,\eta}(\cdot,\ug)} \prod_{I \in \cI}\prod_{K \in I} \Norm{\varphi_K}_{\norm{I}} <+\infty,
\end{multline*}
where $\uvarphi=(\varphi_K)_{K \in \cK}$, $\ug=(g_\omega,\dots,g_\omega) \in L^\infty(\R^d)^{\cE_G}$ and $\Norm*{\cH_{G,\cJ,\eta}(\cdot,\ug)}$ is the operator norm of $\cH_{G,\cJ,\eta}(\cdot,\ug)$ as a multilinear map on $\prod_{I \in \cI} L^{\norm{I}}(\R^d)^I$. Note that $n_e =0$ for all $e \in \cI^{\cwedge}\setminus \cE_G$. If $\card(\cI)=1$, then $\cI=\brackets{\cK}$ and $\cI^{\cwedge}=\emptyset$. In this case, we have
\begin{equation*}
\int_{\ux \in \diag_{\brackets{\cK},\eta}}\parentheses*{\prod_{K \in \cK}\varphi_K(x_K)} \parentheses*{\prod_{J \in \cJ}  \Upsilon_J(\ux_J)} \dx \ux \leq p^{\frac{dp}{2}} \int_{y \in \R^d} \cT_{\cJ,\eta}(\uvarphi)(y) \dx y \leq p^{\frac{dp}{2}} \Norm*{\cT_{\cJ,\eta}(\uvarphi)}_1
\end{equation*}
by the same kind of computation as~\eqref{eq: bound H G J eta}. The right-hand side is finite by Lemma~\ref{lem: regularity T J eta}. Thus, in both cases, each term on the right-hand side of~\eqref{eq: validity of Kac Rice cumulants} is integrable over $(R\Omega)^\cK \cap \diag_{\cI,\eta}$. Hence the result.
\end{proof}

Let $\eta \in (0,\frac{\omega}{2p}]$ and $R_\eta \in \cR$ be given by Lemma~\ref{lem: existence eta}. By Lemma~\ref{lem: relation between notions of non-degeneracy}, for all $R \geq R_\eta$ and $x \in R\Omega$, the centered Gaussian vector $\parentheses*{f_R(x),D_xf_R}$ is non-degenerate. Then, by Proposition~\ref{prop: Bulinskaya}, for all $R \geq R_\eta$, the set $Z_R = f_R^{-1}(0)\cap (R\Omega)$ is almost-surely a codimension-$k$ submanifold. In particular, the random measure $\nu_R$ on $\Omega$ is almost-surely well-defined by Equation~\eqref{eq: def linear statistics nuR}. Thanks to Lemma~\ref{lem: validity of Kac Rice cumulants}, we can express the cumulants $\kappa(\nu_R)(\uphi)$ by the Kac--Rice formula.

\begin{cor}
\label{cor: validity of Kac Rice cumulants}
Let $\eta \in (0,\frac{\omega}{2p}]$ and $R_\eta \in \cR$ be given by Lemma~\ref{lem: existence eta}. Let $\uphi=(\phi_i)_{1 \leq i \leq p}$ where $\phi_i \in L^1(\Omega)\cap L^\infty(\Omega)$ for all $i \in \ssquarebrackets{1}{p}$. For all $R \in \cR \cap [R_\eta,+\infty]$, we have:
\begin{equation*}
\kappa(\nu_R)(\uphi) = \frac{1}{R^{pd}}\sum_{\cK \in \cP_{p,d-k}} \sum_{\substack{\cI,\cJ \in \cP_\cK\\ \cJ \leq \cI}}\ \int_{(R\Omega)^\cK \cap \diag_{\cI,\eta}} \parentheses*{\prod_{K \in \cK}\prod_{i \in K}\phi_i\parentheses*{\frac{x_K}{R}}\!} \tilde{F}_{\cI,\cJ}(f_R,\ux)\prod_{J \in \cJ}  \Upsilon_J(\ux_J) \dx \ux.
\end{equation*}
In particular, this proves Theorem~\ref{thm: cumulants asymptotics for zero sets}.\ref{item: cumulants asymptotics infty}.
\end{cor}

\begin{proof}
Let $R \in \cR \cap [R_\eta,+\infty]$. As explained above, the definition on $\eta$ and $R_\eta$ ensure that $\nu_R$ and $Z_R$ are well-defined Radon measures on $\Omega$ and $R\Omega$ respectively. Since these measures are related by~\eqref{eq: def linear statistics nuR}, we have:
\begin{equation*}
\kappa(\nu_R)(\uphi) = \kappa\parentheses*{\parentheses*{\prsc{\nu_R}{\phi_i}}_{1 \leq i \leq p}} = \kappa\parentheses*{\parentheses*{\frac{\prsc*{Z_R}{\phi_i\parentheses*{\frac{\cdot}{R}}}}{R^d}}_{1 \leq i \leq p}}=\frac{\kappa(Z_R)\parentheses*{\phi_1\parentheses*{\frac{\cdot}{R}},\dots,\phi_p\parentheses*{\frac{\cdot}{R}}}}{R^{pd}},
\end{equation*}
By Lemma~\ref{lem: existence eta}, for all $\cI \in \cP_p$ and $\ux \in \diag_{\cI,\eta}\cap (R\Omega)^p$, we have $\Sigma_{2\cI}(f_R,\ux,\ux) \in \cS_\cI^+$. For all $i \in \ssquarebrackets{1}{p}$, we have $\phi_i\parentheses*{\frac{\cdot}{R}} \in L^1(R\Omega) \cap L^\infty(R\Omega)$, so that the integrability hypotheses in Proposition~\ref{prop: Kac-Rice cumulants} are satisfied, by Lemma~\ref{lem: validity of Kac Rice cumulants}.

We obtain the claimed expression of $\kappa(\nu_R)(\uphi)$ by applying Proposition~\ref{prop: Kac-Rice cumulants} to $f_R$ over $R\Omega$ with the test-functions $\parentheses*{\phi_i\parentheses*{\frac{\cdot}{R}}}_{1 \leq i \leq p}$. By Lemma~\ref{lem: validity of Kac Rice cumulants}, each integral on the right-hand side is well-defined and finite, which proves Theorem~\ref{thm: cumulants asymptotics for zero sets}.\ref{item: cumulants asymptotics infty}.
\end{proof}

%%%%%%%%%%%%%%%%%%%%%%%%%%%%%%%%%%%%%%%%%%%%%%%%%%%%%%%%%%%%%%%%%%%%%%%%%%%%%%%%%%%%%%%%%%%%%%%%%%%%%%%

\subsection{Cumulants asymptotics under \texorpdfstring{\hypDC{2p-1}{\frac{p}{p-1}}}{}}
\label{subsec: cumulants asymptotics under DC}

In this section, we work under the assumption that $p\geq 2$ and \hypDC{2p-1}{\frac{p}{p-1}} holds. In particular, everything we proved in Section~\ref{subsec: validity of Kac-Rice} is still valid. Our goal is to prove Theorem~\ref{thm: cumulants asymptotics for zero sets}.\ref{item: cumulants asymptotics p}, i.e., derive the exact asymptotics of the cumulants of order $p$ for the linear statistics of $\nu_R$, or equivalently $Z_R$. It is enough to derive the asymptotics of each term in the expression of the cumulants given by Corollary~\ref{cor: validity of Kac Rice cumulants}.

In the following, given $\varphi \in L^1(\Omega) \cap L^\infty(\Omega)$, we consider $\varphi$ as a function in $L^1(\R^d)\cap L^\infty(\R^d)$ by extending it by $0$ outside of $\Omega$. We will use repeatedly the fact that, for all $R>0$ and $s \in [1,+\infty)$, we have:
\begin{equation}
\label{eq: Lp norms and scaling}
\Norm*{\varphi\parentheses*{\frac{\cdot}{R}}}_s = \parentheses*{\int_{\R^d} \norm*{\varphi\parentheses*{\frac{x}{R}}}^s \dx x}^\frac{1}{s} = \parentheses*{R^d \int_{\R^d} \norm*{\varphi\parentheses*{y}}^s \dx y}^\frac{1}{s} = R^\frac{d}{s} \Norm{\varphi}_s.
\end{equation}

\begin{lem}[Exact asymptotics]
\label{lem: exact asymptotics}
Let $\eta \in (0,\frac{\omega}{2p}]$ and $R_\eta \in \cR$ be given by Lemma~\ref{lem: existence eta}. Let $\cK \in \cP_p$ and $\cI,\cJ \in \cP_\cK$ be such that $\cJ \leq \cI$. Let $\uphi = \parentheses{\phi_i}_{1 \leq i \leq p}$ be functions in $L^1(\Omega)\cap L^\infty(\Omega)$, we have:
\begin{multline}
\label{eq: exact asymptotics}
\frac{1}{R^d}\int_{\ux \in \diag_{\cI,\eta} \cap (R\Omega)^\cK} \parentheses*{\prod_{K \in \cK} \prod_{i\in K}\phi_i\parentheses*{\frac{x_K}{R}}} \tilde{F}_{\cI,\cJ}\parentheses{f_R,\ux} \prod_{J \in \cJ}\Upsilon_J(\ux_J)\dx \ux\\
\xrightarrow[R \to +\infty]{} \norm{\cK}^\frac{d}{2} \parentheses*{\int_{\Omega} \prod_{i=1}^p\phi_i(y)\dx y} \int_{\uz \in  \obullet{\diag}_{\cI,\eta}} \tilde{F}_{\cI,\cJ}\parentheses*{f,\uz}\prod_{J \in \cJ}\Upsilon_J(\uz_J) \dx \uz,
\end{multline}
where $\obullet{\diag}_{\cI,\eta} = \brackets*{\ux \in \diag_{\cI,\eta} \mvert \flat(\ux)=0} \subset (\R^d)^\cK$ and $\dx \uz$ is the Lebesgue measure on $\flat^{-1}(0)$.
\end{lem}

\begin{proof}
There are two steps to the proof. The first one is to replace $f_R$ by the limit field $f$ in the integral on the left-hand side of~\eqref{eq: exact asymptotics} and prove that the error goes to $0$ as $R \to +\infty$. The second one is to prove the claimed convergence once we replaced $f_R$ by $f$.

\paragraph*{Step 1: Replacing $f_R$ by $f$.} By definition of $\eta$ and $R_\eta$, see Lemma~\ref{lem: existence eta}, for all $R \geq R_\eta$ and $\ux \in \diag_{\cI,\eta} \cap (R\Omega)^\cK$, we have
\begin{equation*}
\parentheses*{\parentheses*{\cG_J(\ux_J)}_{J \in \cJ},\parentheses*{\theta_J(\ux_J)}_{J \in \cJ},\parentheses*{\Pi_{2J}^{2[J]_\cI}(\cdot,\ux_{2[J]_\cI})}_{J \in \cJ},\Sigma_{2\cI}(f_R,\ux_{2\cK})} \in \Gamma_{\cI,\cJ,\eta}.
\end{equation*}
Recall that, if $\ux \in \diag_{\cI,\eta} \cap (R\Omega)^\cK$, then $\ux_{2\cK} \in \diag_{2\cI,\eta} \cap (R\Omega)^{2\cK}$, see Remark~\ref{rem: diag I eta doubled points}. Applying Proposition~\ref{prop: uniform convergence Sigma I} with $A = 2\cK$ yields that
\begin{equation*}
\sup_{\ux \in \diag_{\cI,\eta} \cap (R\Omega)^\cK} \Norm*{\Sigma_{2\cI}(f_R,\ux_{2\cK})-\Sigma_{2\cI}(f,\ux_{2\cK})} \xrightarrow[R \to +\infty]{}0.
\end{equation*}
By Lemma~\ref{lem: regularity F I J}, the function $F_{\cI,\cJ}$ is continuous, and hence uniformly continuous on the compact $\Gamma_{\cI,\cJ,\eta}$. Thus, recalling the definition of $\tilde{F}_{\cI,\cJ}$ in Lemma~\ref{lem: relation between cumulant densities}, we have:
\begin{equation*}
\sup_{\ux \in \diag_{\cI,\eta} \cap (R\Omega)^\cK} \norm*{\tilde{F}_{\cI,\cJ}(f_R,\ux)-\tilde{F}_{\cI,\cJ}(f,\ux)} \xrightarrow[R \to +\infty]{}0.
\end{equation*}
For all $R \geq R_\eta$ and $\uz \in \obullet{\diag}_{\cI,\eta}$ we let $\epsilon_R(\uz)=0$ if $\brackets*{\tau \in \R^d \mvert \tau \cdot \uz \in (R\Omega)^\cK}= \emptyset$ and
\begin{equation*}
\epsilon_R(\uz) = \sup \brackets*{\norm*{\tilde{F}_{\cI,\cJ}(f_R,\tau \cdot \uz)-\tilde{F}_{\cI,\cJ}(f,\tau \cdot \uz)} \mvert \tau \in \R^d \ \text{such that} \ \tau\cdot \uz \in (R\Omega)^\cK}
\end{equation*}
otherwise. Then, $\epsilon_R(\uz) \leq \sup_{\ux \in \diag_{\cI,\eta} \cap (R\Omega)^\cK} \norm*{\tilde{F}_{\cI,\cJ}(f_R,\ux)-\tilde{F}_{\cI,\cJ}(f,\ux)} \xrightarrow[R \to +\infty]{}0$ for all $\uz \in \obullet{\diag}_{\cI,\eta}$.

Recalling that we extended each $\phi_i$ by $0$ outside of $\Omega$, the change of variable $(\tau,\uz) = (\flat(\ux),\oux)$ and the translation-invariance of the $(\Upsilon_J)_{J \in \cJ}$, see Lemma~\ref{lem: regularity Upsilon}, yield:
\begin{multline}
\label{eq: exact asymptotics replace tilde f}
\frac{1}{R^d}\int_{\ux \in (R\Omega)^\cK \cap \diag_{\cI,\eta}} \norm*{\prod_{K \in \cK} \prod_{i\in K}\phi_i\parentheses*{\frac{x_K}{R}}} \norm*{\tilde{F}_{\cI,\cJ}\parentheses*{f_R,\ux}-\tilde{F}_{\cI,\cJ}\parentheses*{f,\ux}} \prod_{J \in \cJ}\Upsilon_J(\ux_J) \dx \ux\\
\begin{aligned}
&\leq \frac{1}{R^d}\int_{\diag_{\cI,\eta}} \norm*{\prod_{K \in \cK} \prod_{i\in K}\phi_i\parentheses*{\frac{x_K}{R}}} \epsilon_R(\oux) \prod_{J \in \cJ}\Upsilon_J(\oux_J) \dx \ux\\
&\leq \norm{\cK}^\frac{d}{2} \int_{\uz \in \obullet{\diag}_{\cI,\eta}} \parentheses*{\frac{1}{R^d}\int_{\tau \in \R^d} \prod_{K \in \cK} \prod_{i\in K}\norm*{\phi_i\parentheses*{\frac{\tau + z_K}{R}}}\dx \tau} \epsilon_R(\uz)\prod_{J \in \cJ}\Upsilon_J(\uz_J) \dx \uz.
\end{aligned}
\end{multline}
For all $\uz \in \obullet{\diag}_{\cI,\eta}$, by Hölder's inequality and Equation~\eqref{eq: Lp norms and scaling}, we have:
\begin{equation*}
\frac{1}{R^d}\int_{\tau \in \R^d} \prod_{K \in \cK} \prod_{i\in K}\norm*{\phi_i\parentheses*{\frac{\tau + z_K}{R}}}\dx \tau \leq  \prod_{K \in \cK} \prod_{i\in K} \frac{1}{R^\frac{d}{p}} \Norm*{\phi_i\parentheses*{\frac{\cdot + z_K}{R}}}_p = \prod_{i=1}^p \Norm{\phi_i}_p <+\infty,
\end{equation*}
recalling that $\ssquarebrackets{1}{p} = \bigsqcup_{K \in \cK}K$. Thus, considering the integral over $\uz \in \obullet{\diag}_{\cI,\eta}$ on the last line of~\eqref{eq: exact asymptotics replace tilde f}, the integrand is bounded by $\parentheses*{\prod_{i=1}^p \Norm{\phi_i}_p} \epsilon_R(\uz) \prod_{J \in \cJ}\Upsilon_J(\uz_J) \xrightarrow[R \to +\infty]{}0$.

We need a dominating function in order to apply Lebesgue's dominated convergence Theorem. By Lemma~\ref{lem: uniform domination FIJ}, for all $R \geq R_\eta$ and $\uz \in \obullet{\diag}_{\cI,\eta}$ we have:
\begin{equation}
\label{eq: exact asymptotics domination}
\epsilon_R(\uz) \leq \sup_{\brackets*{\tau \in \R^d \mvert \tau\cdot \uz \in (R\Omega)^\cK}}\parentheses*{\norm*{\tilde{F}_{\cI,\cJ}(f_R,\tau\cdot \uz)}+\norm*{\tilde{F}_{\cI,\cJ}(f,\tau\cdot \uz)}} \leq 2C_\eta \sum_{\un \in \G_\cI} \prod_{e \in \cI^{\cwedge}} g_\omega\parentheses*{\strut\ell_e(\flat_\cI(\uz))}^{n_e}.
\end{equation}
Recall that $\G_\cI$ is finite by Lemmas~\ref{lem: description of NB} and~\ref{lem: characterization 2 edge connected}. Let $G = (\cI,\un) \in \G_\cI$, and let $(q_e)_{e \in \cE_G}$ be given by Proposition~\ref{prop: HBL graphs}. For all $e \in \cI^{\cwedge}$, we have either $n_e=0$, or $e \in \cE_G$ and
\begin{equation*}
q_e = \frac{\norm{\cI}}{\norm{\cI}-1}\frac{p_e}{2} \geq 1+\frac{1}{\norm{\cI}-1} \geq 1+\frac{1}{p-1}=\frac{p}{p-1},
\end{equation*}
so that $g_\omega \in L^\frac{p}{p-1}(\R^d) \cap L^\infty(\R^d) \subset L^{q_e}(\R^d)$. Thus, multiplying the right-hand side of~\eqref{eq: exact asymptotics domination} by $\prod_{J \in \cJ}\Upsilon_J(\uz_J)$, we obtain a function which is independent of $R$, and integrable over $\obullet{\diag}_{\cI,\eta}$ by Lemma~\ref{lem: integrable dominating function}. Then, by Lebesgue's Theorem, the last term in~\eqref{eq: exact asymptotics replace tilde f} goes to $0$ as $R\to +\infty$, and hence so does the first. This proves that, in order to prove~\eqref{eq: exact asymptotics}, it is enough to prove that this asymptotics holds when $f_R$ is replaced by $f$ on the left-hand side.

\paragraph*{Step 2: Convergence for the limit field $f$.} Since the limit field $f$ is stationary, the covariance operator $\Sigma_{2\cI}(f,\cdot)$ is invariant by diagonal translations, see Remark~\ref{rem: Sigma I C0}. By Lemma~\ref{lem: prop Pi BA}.\ref{item: Pi translation inv} and Section~\ref{subsec: evaluation maps and their kernels}, the maps $\parentheses*{\cG_J}_{J \in \cJ}$, $\parentheses*{\theta_J}_{J \in \cJ}$, and $\parentheses*{\Pi_{2J}^{2[J]_\cI}}_{J \in \cJ}$ all are translation-invariant. Hence $\tilde{F}_{\cI,\cJ}(f,\cdot)$ is also translation-invariant. By the same change of variable as in the first step, we obtain:
\begin{equation}
\label{eq: convergence limit field}
\begin{aligned}
\frac{1}{R^d}\int_{\ux \in (R\Omega)^\cK \cap \diag_{\cI,\eta}} &\parentheses*{\prod_{K \in \cK} \prod_{i\in K}\phi_i\parentheses*{\frac{x_K}{R}}} \tilde{F}_{\cI,\cJ}\parentheses*{f,\ux} \prod_{J \in \cJ}\Upsilon_J(\ux_J) \dx \ux\\
=\norm{\cK}^\frac{d}{2}& \int_{\uz \in \obullet{\diag}_{\cI,\eta}} \parentheses*{\frac{1}{R^d}\int_{\tau \in \R^d} \prod_{K \in \cK} \prod_{i\in K}\phi_i\parentheses*{\frac{\tau + z_K}{R}}\dx \tau} \tilde{F}_{\cI,\cJ}\parentheses*{f,\uz}\prod_{J \in \cJ}\Upsilon_J(\uz_J) \dx \uz\\
=\norm{\cK}^\frac{d}{2}& \int_{\uz \in \obullet{\diag}_{\cI,\eta}} \parentheses*{\int_{y \in \R^d} \prod_{K \in \cK} \prod_{i\in K}\phi_i\parentheses*{y+\frac{z_K}{R}}\dx y} \tilde{F}_{\cI,\cJ}\parentheses*{f,\uz} \prod_{J \in \cJ}\Upsilon_J(\uz_J) \dx \uz.
\end{aligned}
\end{equation}

Let us fix $\uz=(z_K)_{K \in \cK} \in \obullet{\diag}_{\cI,\eta}$. Let $K \in \cK$ and $i \in K$. By continuity of the translations we have $\phi_i\parentheses*{\cdot+\frac{z_K}{R}} \xrightarrow[R \to +\infty]{}\phi_i$ in $L^p(\R^d)$. By continuity of the product from $\parentheses*{L^p(\R^d)}^p$ to $L^1(\R^d)$, we obtain $\prod_{K \in \cK}\prod_{i \in K}\phi_i\parentheses*{\cdot+\frac{z_K}{R}} \xrightarrow[R \to +\infty]{} \prod_{i=1}^p\phi_i$ in $L^1(\R^d)$. Thus, in the last line of~\eqref{eq: convergence limit field}, the integrand converges pointwise to $\parentheses*{\int_{y \in \R^d}\prod_{i=1}^p\phi_i(y) \dx y}\tilde{F}_{\cI,\cJ}(f,\uz)\prod_{J \in \cJ}\Upsilon_J(\uz_J)$. Recalling that the integral over $y \in \R^d$ is bounded by $\prod_{i=1}^p\Norm{\phi_i}_p$, uniformly in $\uz \in \obullet{\diag}_{\cI,\eta}$ and $R \geq R_\eta$, we can use the same dominating function as in the first step, see Lemmas~\ref{lem: uniform domination FIJ} and~\ref{lem: integrable dominating function}, and the conclusion follows by another dominated convergence.
\end{proof}

\begin{proof}[Proof of Theorem~\ref{thm: cumulants asymptotics for zero sets}.\ref{item: cumulants asymptotics p}]
We start by dealing with the case $p\geq 2$ under the hypothesis that \hypDC{2p-1}{\frac{p}{p-1}} holds. Then, we deal with the case $p=1$, for the sake of completeness.

\paragraph*{Case $p \geq 2$.}
Let $\phi_1,\dots,\phi_p \in L^1(\Omega) \cap L^\infty(\Omega)$. Let $\eta>0$ and $R_\eta \in \cR$ be given by Lemma~\ref{lem: existence eta}. For all $R \geq R_\eta$, we use the expression of $\kappa(\nu_R)(\uphi)$ derived in Corollary~\ref{cor: validity of Kac Rice cumulants}. Multiplying by $R^{d(p-1)}$ and using Lemma~\ref{lem: exact asymptotics} to compute the asymptotics of each term in this expression, we obtain:
\begin{equation*}
R^{d(p-1)}\kappa(\nu_R)(\phi_1,\dots,\phi_p) \xrightarrow[R \to +\infty]{} \gamma_p(f) \int_{x \in \Omega} \prod_{i=1}^p \phi_i(x) \dx x,
\end{equation*}
where
\begin{equation*}
\gamma_p(f) = \sum_{\cK \in \cP_{p,d-k}}\norm{\cK}^\frac{d}{2} \sum_{\cI\in \cP_\cK} \int_{\uz \in  \obullet{\diag}_{\cI,\eta}} \sum_{\cJ \leq \cI} \tilde{F}_{\cI,\cJ}\parentheses*{f,\uz}\prod_{J \in \cJ}\Upsilon_J(\uz_J) \dx \uz.
\end{equation*}
Recalling Definition~\ref{def: F A} and Lemma~\ref{lem: relation between cumulant densities}, we can simplify the expression of this constant:
\begin{equation*}
\gamma_p(f) = \sum_{\cK \in \cP_{p,d-k}}\norm{\cK}^\frac{d}{2} \sum_{\cI\in \cP_\cK} \int_{\uz \in  \obullet{\diag}_{\cI,\eta}} \cF_{\cK}(f,\uz) \dx \uz = \sum_{\cK \in \cP_{p,d-k}}\norm{\cK}^\frac{d}{2} \int_{\uz \in  \flat^{-1}(0)} \cF_{\cK}(f,\uz) \dx \uz.
\end{equation*}
Observing that $\cF_\cK$ actually only depends on $\norm{\cK}$ and recalling Definition~\ref{def: P A n}, we finally obtain:
\begin{equation}
\label{eq: expression gamma p f}
\gamma_p(f) = \begin{cases} \displaystyle p^\frac{d}{2} \int_{\brackets*{\uz \in (\R^d)^p \mvert \flat(\uz)=0}} \cF_p(f,\uz)\dx \uz, & \text{if} \ k <d,\\
\displaystyle \sum_{q=1}^p \card\parentheses*{\brackets*{\cI \in \cP_p \mvert \norm{\cI}=q}}\, q^\frac{d}{2} \int_{\brackets*{\uz \in (\R^d)^q \mvert \flat(\uz)=0}} \cF_q(f,\uz)\dx \uz, & \text{if} \ k =d. \end{cases}
\end{equation}

\paragraph*{Case $p=1$.} In this case, we work under the only assumption that \hypReg{2}, \hypND{1} and \hypScL{1} hold. In particular, we cannot use the results from Section~\ref{subsec: validity of Kac-Rice}. We denote by $\cI = \brackets{\brackets{1}}$ the only partition of $\brackets{1}$ for simplicity.

Let $R \in \cR \sqcup \brackets{\infty}$, for all $x \in RU$, we have $\obullet{(x,x)}=0$ and, by Definitions~\ref{def: divided differences}, \ref{def: isomorphisms Delta Psi} and~\ref{def: Sigma I} and Lemma~\ref{lem: oK in terms of Delta and Psi} we have:
\begin{equation}
\label{eq: variance case p=1}
\Sigma_{2\cI}(f_R,x,x) = \var{\oK(f_R,x,x)} = \var{\Psi_0\circ \Delta_{(x,x)}(f_R)} = \Psi_0\var{f_R(x),D_xf_R} \Psi_0^*.
\end{equation}
Since $f_\infty=f$ is stationary, the function $x \mapsto\Sigma_{2\cI}(f,x,x)$ is constant. By \hypND{1}, the Gaussian vector $(f(0),D_0f)$ is non-degenerate. Since $\Psi_0$ is an isomorphism, we have $\Sigma_{2\cI}(f,0,0) \in \cS_\cI^+$. Then, we deduce from \hypScL{1} that $\sup_{x \in R\Omega} \Norm*{\Sigma_{2\cI}(f_R,x,x) - \Sigma_{2\cI}(f,0,0)} \xrightarrow[R \to +\infty]{}0$. Hence,~there exists $R_0 \in \cR$ and a compact $\Gamma_0 \subset \cS_\cI^+$ such that $\Sigma_{2\cI}(f_R,x,x) \in \Gamma_0$ for all $R \geq R_0$ and $x \in R\Omega$.

The maps $\cG_{\brackets{1}}$, $\theta_{\brackets{1}}$ and $\Upsilon_{\brackets{1}}$ are translation-invariant, see Section~\ref{subsec: evaluation maps and their kernels} and Lemma~\ref{lem: regularity Upsilon}, hence constant. The map $\sigma_\cI$ is uniformly continuous on $\brackets*{\parentheses*{\cG_{\brackets{1}}(0),\theta_{\brackets{1}}(0)}}\times \Gamma_0$ by Lemma~\ref{lem: regularity sigma I}. Hence, letting $\gamma_1(f)=\rho_{\brackets{1}}(f,0)$, the previous uniform convergence and Lemma~\ref{lem: Kac-Rice revisited} imply that
\begin{equation}
\label{eq: uniform CV rho 1}
\sup_{x \in R\Omega} \norm*{\rho_{\brackets{1}}(f_R,x)- \gamma_1(f)} \xrightarrow[R \to +\infty]{}0.
\end{equation}

Let $\phi \in L^1(\Omega)$ and $R \geq R_0$. By Equation~\eqref{eq: variance case p=1}, for all $x \in R\Omega$, the Gaussian vector $\parentheses*{f_R(x),D_xf_R}$ is non-degenerate. Moreover, $\rho_{\brackets{1}}(f_R,\cdot)$ is bounded over $R\Omega$ by~\eqref{eq: uniform CV rho 1}, hence $\phi\parentheses*{\frac{\cdot}{R}}\rho_{\brackets{1}}(f_R,\cdot) \in L^1(R\Omega)$. Using the Kac--Rice formula for the expectation, see Theorem~\ref{thm: Kac-Rice}, and the uniform convergence derived in Equation~\eqref{eq: uniform CV rho 1}, we have:
\begin{align*}
\kappa(\nu_R)(\phi) &= \esp{\prsc*{\nu_R}{\phi}} = \esp{\prsc*{\frac{Z_R}{R^d}}{\phi\parentheses*{\frac{\cdot}{R}}}} = \frac{1}{R^d}\int_{R\Omega} \phi\parentheses*{\frac{x}{R}}\rho_{\brackets{1}}(f_R,x) \dx x\\
&= \int_{\Omega}\phi(x)\rho_{\brackets{1}}(f_R,Rx)\dx x \xrightarrow[R \to +\infty]{} \gamma_1(f) \int_\Omega \phi(x) \dx x.
\end{align*}
Finally, since $\parentheses*{f(0),D_0f}$ is non-degenerate, we have $\gamma_1(f) = \rho_{\brackets{1}}(f,0) >0$, see Definition~\ref{def: rho A}.
\end{proof}

%%%%%%%%%%%%%%%%%%%%%%%%%%%%%%%%%%%%%%%%%%%%%%%%%%%%%%%%%%%%%%%%%%%%%%%%%%%%%%%%%%%%%%%%%%%%%%%%%%%%%%%

\subsection{Asymptotic upper bound for cumulants under \texorpdfstring{\hypDC{2p-1}{2}}{}}
\label{subsec: asymptotic upper bound for cumulants}

The goal of this section is to prove Theorem~\ref{thm: cumulants asymptotics for zero sets}.\ref{item: cumulants asymptotics 2}. We work under the assumption that only \hypDC{2p-1}{2} holds. In particular, the results from Section~\ref{subsec: validity of Kac-Rice} are still valid, but not those of Section~\ref{subsec: cumulants asymptotics under DC}.

Let $\eta >0$ and $R_\eta \in \cR$ be given by Lemma~\ref{lem: existence eta}. Let $\cK \in \cP_p$ and $\cI,\cJ \in \cP_\cK$ be such that $\cJ \leq \cI$. Finally, let $\uphi = \parentheses{\phi_i}_{1 \leq i \leq p}$ be functions in $L^1(\Omega)\cap L^\infty(\Omega)$, that we extend by $0$ outside of $\Omega$.

\begin{lem}[Continuity with respect to the decay function]
\label{lem: continuity decay function}
Let us assume that $\card(\cI) \geq 2$. Let $G=(\cI,\un) \in \G_\cI$ and let $(p_e)_{e \in \cE_G}$ be the associated exponents given by Lemma~\ref{lem: estimates H G J eta}. For all $R \geq R_\eta$, for all $\ug=(g_e)_{e \in \cE_G}$ and $\tilde{\ug}=(\tilde{g}_e)_{e \in \cE_G} \in \prod_{e \in \cE_G}L^{p_e}(\R^d)$, we have:
\begin{multline*}
\norm*{\int_{\diag_{\cI,\eta}} \norm*{\prod_{K \in \cK}\prod_{i \in K}\phi_i\parentheses*{\frac{x_K}{R}}}\parentheses*{\prod_{e \in \cE_G} g_e\parentheses*{\ell_e\parentheses*{\strut\flat_\cI(\ux)}}^{n_e}-\prod_{e \in \cE_G} \tilde{g}_e\parentheses*{\ell_e\parentheses*{\strut\flat_\cI(\ux)}}^{n_e}} \prod_{J \in \cJ}\Upsilon_J(\ux_J)\dx \ux}\\
\leq R^\frac{d\norm{\cI}}{2} \Norm*{\strut \cH_{G,\cJ,\eta}(\cdot,\ug) - \cH_{G,\cJ,\eta}(\cdot,\tilde{\ug})}_{\mathrm{op},2} \prod_{I \in \cI} \prod_{K \in I}\Norm*{\prod_{i \in K}\phi_i}_{2\norm{I}},
\end{multline*}
where $\Norm{\cdot}_{\mathrm{op},2}$ is the operator norm for multilinear forms on $\prod_{I \in \cI}\parentheses*{L^{2\norm{I}}(\R^d)}^I$.
\end{lem}

\begin{proof}
Let $R \geq R_\eta$ and $\ug,\tilde{\ug} \in \prod_{e \in \cE_G} L^{p_e}(\R^d)$. We let $\varphi_K = \prod_{i \in K}\norm{\phi_i\parentheses*{\frac{\cdot}{R}}} \in L^1(\R^d) \cap L^\infty(\R^d)$ for all $K \in \cK$ and $\uvarphi=(\varphi_K)_{K \in \cK}$. Recalling Definition~\ref{def: H G J eta}, and using Lemma~\ref{lem: estimates H G J eta}.\ref{item: estimate H p}, the integral we want to control is
\begin{align*}
\norm*{\cH_{G,\cJ,\eta}\parentheses*{\uvarphi,\ug}-\cH_{G,\cJ,\eta}\parentheses*{\uvarphi,\tilde{\ug}}} &= \norm*{\parentheses*{\cH_{G,\cJ,\eta}\parentheses*{\cdot,\ug}-\cH_{G,\cJ,\eta}\parentheses*{\cdot,\tilde{\ug}}}(\uvarphi)}\\
&\leq \Norm*{\parentheses*{\cH_{G,\cJ,\eta}\parentheses*{\cdot,\ug}-\cH_{G,\cJ,\eta}\parentheses*{\cdot,\tilde{\ug}}}}_{\mathrm{op},2}\prod_{I \in \cI} \prod_{K \in I}\Norm{\varphi_K}_{2\norm{I}}.
\end{align*}
For all $I \in \cI$ and $K \in I$, we have $\Norm{\varphi_K}_{2\norm{I}} = R^\frac{d}{2\norm{I}} \Norm*{\prod_{i\in K}\phi_i}_{2\norm{I}}$, see Equation~\eqref{eq: Lp norms and scaling}. Then, the conclusion follows from $\sum_{I \in \cI} \sum_{K \in I} \frac{1}{2\norm{I}} = \sum_{I \in \cI} \frac{1}{2} = \frac{\norm{\cI}}{2}$.
\end{proof}

\begin{cor}[Asymptotic upper bound]
\label{cor: asymptotic upper bound L2}
If $\card(\cI)\geq 3$ then, for all $G=(\cI,\un) \in \G_\cI$, we~have:
\begin{equation*}
\frac{1}{R^\frac{d\norm{\cI}}{2}}\int_{\diag_{\cI,\eta}} \norm*{\prod_{K \in \cK}\prod_{i \in K}\phi_i\parentheses*{\frac{x_K}{R}}}\parentheses*{\prod_{J \in \cJ}\Upsilon_J(\ux_J)}\parentheses*{\prod_{e \in \cE_G} g_w\parentheses*{\ell_e\parentheses*{\strut\flat_\cI(\ux)}}^{n_e}} \dx \ux \xrightarrow[R \to +\infty]{}0.
\end{equation*}
\end{cor}

\begin{proof}
Let $G=(\cI,\un) \in \G_\cI$ and let $(p_e)_{e \in \cE_G}$ be the exponents from Proposition~\ref{prop: HBL graphs}. For all $e \in \cE_G$, we have $p_e \geq 2$, hence $g_\omega \in L^2(\R^d)\cap L^\infty(\R^d) \subset L^{p_e}(\R^d)$. Let $\ug = (g_\omega,\dots,g_\omega) \in \prod_{e \in \cE_G}L^{p_e}(\R^d)$. Let $\tilde{\ug}=(\tilde{g}_e)_{e \in \cE_G} \in \prod_{e \in \cE_G} L^{q_e}(\R^d) \cap L^{p_e}(\R^d)$, where $q_e = \frac{\norm{\cI}}{\norm{\cI}-1}\frac{p_e}{2}$ for all $e \in \cE_G$. For all $K \in \cK$ let $\varphi_K = \prod_{i \in K}\norm*{\phi_i\parentheses*{\frac{\cdot}{R}}}$. Applying Lemma~\ref{lem: estimates H G J eta}.\ref{item: estimate H q} to $\tilde{\ug}$ and proceeding as above, we obtain:
\begin{multline*}
\norm*{\int_{\diag_{\cI,\eta}} \norm*{\prod_{K \in \cK}\prod_{i \in K}\phi_i\parentheses*{\frac{x_K}{R}}}\parentheses*{\prod_{J \in \cJ}\Upsilon_J(\ux_J)}\parentheses*{\prod_{e \in \cE_G} \tilde{g}_e\parentheses*{\ell_e\parentheses*{\strut\flat_\cI(\ux)}}^{n_e}} \dx \ux} = \Norm*{\cH_{G,\cJ,\eta}(\uvarphi,\tilde{\ug})}\\
\leq \Norm*{\cH_{G,\cJ,\eta}(\cdot,\tilde{\ug})}_{\mathrm{op},\cI} \prod_{I \in \cI} \prod_{K \in I} \Norm*{\varphi_K}_{\norm{\cI}\norm{I}} \leq R^d \Norm*{\cH_{G,\cJ,\eta}(\cdot,\tilde{\ug})}_{\mathrm{op},\cI} \prod_{I \in \cI} \prod_{K \in I} \Norm*{\prod_{i \in K}\phi_i}_{\norm{\cI}\norm{I}}
\end{multline*}
by~\eqref{eq: Lp norms and scaling}, since $\sum_{I \in \cI} \sum_{i \in I} \frac{1}{\norm{\cI}\norm{I}}=1$. Here $\Norm*{\cdot}_{\mathrm{op},\cI}$ is the operator norm for multilinear forms on $\prod_{I \in \cI} L^{\norm{\cI}\norm{I}}(\R^d)^I$. Then, by Lemma~\ref{lem: continuity decay function}, we have:
\begin{multline}
\label{eq: asymptotic upper bound}
\frac{1}{R^\frac{d\norm{\cI}}{2}}\norm*{\int_{\diag_{\cI,\eta}} \norm*{\prod_{K \in \cK}\prod_{a \in K}\phi_a\parentheses*{\frac{x_K}{R}}}\parentheses*{\prod_{J \in \cJ}\Upsilon_J(\ux_J)}\parentheses*{\prod_{e \in \cE_G} g_w\parentheses*{\ell_e\parentheses*{\strut\flat_\cI(\ux)}}^{n_e}} \dx \ux}\\
\leq C\parentheses*{\Norm*{\strut \cH_{G,\cJ,\eta}(\cdot,\ug) - \cH_{G,\cJ,\eta}(\cdot,\tilde{\ug})}_{\mathrm{op},2}+ R^{d(1-\frac{\norm{\cI}}{2})} \Norm*{\cH_{G,\cJ,\eta}(\cdot,\tilde{\ug})}_{\mathrm{op},\cI}},
\end{multline}
where $C\geq 0$ is some constant depending on $\uphi$.

For all $e \in \cE_G$, the subspace $L^{q_e}(\R^d) \cap L^{p_e}(\R^d)$ is dense in $L^{p_e}(\R^d)$. Indeed, if $p_e=+\infty$ then $q_e=+\infty$; and if $p_e<+\infty$ then $q_e<+\infty$, and $L^{q_e}(\R^d) \cap L^{p_e}(\R^d)$ contains the space of continuous functions with compact support, which is dense in $L^{p_e}(\R^d)$ in this case. Hence, $\prod_{e \in \cE_G} L^{q_e}(\R^d) \cap L^{p_e}(\R^d)$ is dense in $\prod_{e \in \cE_G}L^{p_e}(\R^d)$.

Let $\epsilon>0$. The density we just established and Lemma~\ref{lem: estimates H G J eta}.\ref{item: estimate H p} prove that there exists a function $\tilde{\ug} \in \prod_{e \in \cE_G} L^{q_e}(\R^d) \cap L^{p_e}(\R^d)$ such that $\Norm*{\strut \cH_{G,\cJ,\eta}(\cdot,\ug) - \cH_{G,\cJ,\eta}(\cdot,\tilde{\ug})}_{\textrm{op},2} \leq \epsilon$. Then, since $\card(\cI)\geq 3$, for all $R \in \cR$ large enough the right-hand side of~\eqref{eq: asymptotic upper bound} is less than $2C\epsilon$, which concludes the proof.
\end{proof}

\begin{proof}[Proof of Theorem~\ref{thm: cumulants asymptotics for zero sets}.\ref{item: cumulants asymptotics 2}]
Let $p \geq 3$. Let $\phi_1,\dots,\phi_p \in L^1(\Omega) \cap L^\infty(\Omega)$, we extend them by $0$ outside of $\Omega$. Let $\eta>0$ and $R_\eta \in \cR$ be given by Lemma~\ref{lem: existence eta}. For all $R \geq R_\eta$, we can express $\kappa(\nu_R)(\uphi)$ by Corollary~\ref{cor: validity of Kac Rice cumulants}, where $\uphi=(\phi_i)_{1 \leq i \leq p}$. Then, we are reduced to upper bounding a finite sum of terms of the form:
\begin{equation*}
\frac{1}{R^{pd}} \int_{\diag_{\cI,\eta}} \norm*{\prod_{K \in \cK}\prod_{i \in K}\phi_i\parentheses*{\frac{x_K}{R}}} \norm*{\tilde{F}_{\cI,\cJ}(f_R,\ux)} \prod_{J \in \cJ}  \Upsilon_J(\ux_J) \dx \ux,
\end{equation*}
where $\cK \in \cP_p$ and $\cJ \leq \cI$ in $\cP_\cK$. By Lemma~\ref{lem: uniform domination FIJ}, it is actually enough to upper bound a finite number of terms of the following form, where $\un \in \G_\cI$:
\begin{equation}
\label{eq: cumulant asymptotics L2}
\frac{1}{R^{pd}} \int_{\diag_{\cI,\eta}} \norm*{\prod_{K \in \cK}\prod_{i \in K}\phi_i\parentheses*{\frac{x_K}{R}}} \parentheses*{\prod_{J \in \cJ}  \Upsilon_J(\ux_J)} \parentheses*{\prod_{e \in \cI^{\cwedge}} g_\omega\parentheses*{\strut\ell_e(\flat_\cI(\ux))}^{n_e}} \dx \ux.
\end{equation}

If $\card(\cI) \geq 3$, for all $\un \in \G_\cI$, the term in \eqref{eq: cumulant asymptotics L2} is $o\parentheses{R^{d\parentheses{\frac{\norm{\cI}}{2}-p}}}$, by Corollary~\ref{cor: asymptotic upper bound L2}. Since $\norm{\cI} \leq \norm{\cK}\leq p$, these terms are indeed $o\parentheses{R^{-\frac{pd}{2}}}$.

If $\card(\cI)=2$, let $G=(\cI,\un) \in \G_\cI$ and let $(p_e)_{e \in \cE_G}$ be the associated exponents given by Proposition~\ref{prop: HBL graphs}. As in the proof of Corollary~\ref{cor: asymptotic upper bound L2}, we have $\ug = (g_\omega,\dots,g_\omega) \in \prod_{e \in \cE_G}L^{p_e}(\R^d)$ since we assumed that \hypDC{2p-1}{2} holds. Applying Lemma~\ref{lem: continuity decay function} with $\tilde{\ug}=0$, we obtain that the term in~\eqref{eq: cumulant asymptotics L2} is $O\parentheses{R^{d\parentheses{\frac{\norm{\cI}}{2}-p}}}$. Since $\norm{\cI}=2$ and $p \geq 3$, we have $\frac{\norm{\cI}}{2}-p< -\frac{p}{2}$. Thus, terms of the form~\eqref{eq: cumulant asymptotics L2} with $\card(\cI)=2$ are also $o\parentheses{R^{-\frac{pd}{2}}}$.

If $\card(\cI)=1$, then $\cI= \brackets{\cK}$ and $\cI^{\cwedge}=\emptyset$. In this case, the integral in~\eqref{eq: cumulant asymptotics L2} can be dealt with by the same change of variable as in~\eqref{eq: bound H G J eta}, recall Definition~\ref{def: integral transform}. Letting $\varphi_K =\prod_{i \in K}\norm*{\phi_i\parentheses*{\frac{\cdot}{R}}}$ for all $K \in \cK$ and $\uvarphi=\parentheses*{\varphi_K}_{K \in \cK}$, we obtain:
\begin{equation*}
\int_{\diag_{\brackets{\cK},\eta}} \parentheses*{\prod_{K \in \cK}\varphi_K(x_K)} \parentheses*{\prod_{J \in \cJ}  \Upsilon_J(\ux_J)} \dx \ux = \norm{\cK}^{\frac{d}{2}}\int_{\tau \in \R^d} \cT_{\cJ,\eta}(\uvarphi)(\tau) \dx \tau = \norm{\cK}^{\frac{d}{2}}\Norm*{\cT_{\cJ,\eta}(\uvarphi)}_1.
\end{equation*}
By Lemma~\ref{lem: regularity T J eta} and Equation~\eqref{eq: Lp norms and scaling}, we have
\begin{equation*}
\Norm*{\cT_{\cJ,\eta}(\uvarphi)}_1\leq C_{\cJ,\eta}\prod_{K \in \cK}\Norm*{\varphi_K}_{\norm{\cK}} \leq C_{\cJ,\eta}\prod_{K \in \cK}R^\frac{d}{\norm{\cK}} \Norm*{\prod_{i \in K}\phi_i}_{\norm{\cK}}=O(R^d).
\end{equation*}
Then, terms of the form~\eqref{eq: cumulant asymptotics L2} with $\card(\cI)=1$ are $O\parentheses{R^{d(1-p)}}$. Since $p \geq 3$, these terms are $o\parentheses{R^{-\frac{pd}{2}}}$, which concludes the proof.
\end{proof}

%%%%%%%%%%%%%%%%%%%%%%%%%%%%%%%%%%%%%%%%%%%%%%%%%%%%%%%%%%%%%%%%%%%%%%%%%%%%%%%%%%%%%%%%%%%%%%%%%%%%%%%
%%%%%%%%%%%%%%%%%%%%%%%%%%%%%%%%%%%%%%%%%%%%%%%%%%%%%%%%%%%%%%%%%%%%%%%%%%%%%%%%%%%%%%%%%%%%%%%%%%%%%%%

\section{The case of critical points}
\label{sec: the case of critical points}

As explained in the introduction, Theorem~\ref{thm: cumulants asymptotics for critical points} is not a consequence of Theorem~\ref{thm: cumulants asymptotics for zero sets} because gradient fields never satisfy Hypothesis \hypND{q} for $q \geq 1$. The purpose of this section is to explain how to adapt the contents of Sections~\ref{sec: partitions cumulants and diagonals} to~\ref{sec: proof of thm cumulants asymptotics zero sets} in order to prove Theorem~\ref{thm: cumulants asymptotics for critical points}. We deal with modifications concerning Kergin interpolation in Section~\ref{subsec: Kergin interpolation for gradient fields} and with those concerning Kac--Rice formulas in Section~\ref{subsec: Kac--Rice formulas for gradient fields}. The few other adaptations we need are dealt with in Section~\ref{subsec: proof of thm cumulants asymptotics for critical points}.

%%%%%%%%%%%%%%%%%%%%%%%%%%%%%%%%%%%%%%%%%%%%%%%%%%%%%%%%%%%%%%%%%%%%%%%%%%%%%%%%%%%%%%%%%%%%%%%%%%%%%%%

\subsection{Kergin interpolation for gradient fields}
\label{subsec: Kergin interpolation for gradient fields}

The goal of this section is to explain what should be adapted in Section~\ref{sec: polynomial interpolation and Gaussian fields} in order to deal with the case of gradient fields. The key point is to remark that the Kergin interpolant of a gradient is itself the gradient of some polynomial, see~\cite{GS2024}.

Let $q \in \N$, recall that $\R_q[X]$ is the space of polynomials of degree at most $q$ in $d$ variables $X=(X_1,\dots,X_d)$. If $Q \in \R_{q+1}[X]$, then $\nabla Q \in \R_q[X]^d$ is a polynomial map from $\R^d$ to itself.

\begin{dfn}[Polynomial gradients]
\label{def: polynomial gradients}
We denote by $\nabla \R_{q+1}[X] = \brackets*{\nabla Q \mvert\strut Q \in \R_{q+1}[X]}$.
\end{dfn}

\begin{lem}[Translation-invariance of {$\nabla \R_{q+1}[X]$}]
\label{lem: polynomial gradients}
For all $q \in \N$, the set $\nabla \R_{q+1}[X]$ is a vector subspace of $\R_q[X]^d$, which is invariant under the action of $\R^d$ by translation.
\end{lem}

\begin{proof}
Let $P \in \nabla \R_{q+1}[X]$ and $\tau \in \R^d$, let $Q \in \R_{q+1}[X]$ be such that $\nabla Q =P$, we have:
\begin{equation*}
\tau \cdot P = \nabla Q(X+\tau) = \nabla (\tau \cdot Q) \in \nabla \R_{q+1}[X].\qedhere
\end{equation*}
\end{proof}

Let $U \subset \R^d$ be a convex open subset. The Poincaré Lemma implies the following result.

\begin{lem}[Kergin interpolant of gradients]
\label{lem: Kergin interpolant of gradients}
Let $A$ be a non-empty finite set and $h \in \cC^{\norm{A}}(U,\R)$. For all $\ux \in U^A$, we have $K(\nabla h,\ux) \in \nabla \R_{\norm{A}}[X]$ and $\oK(\nabla h,\ux) \in \nabla \R_{\norm{A}}[X]$.
\end{lem}

\begin{proof}
Let $\ux \in U^A$, the fact that $K(\nabla h,\ux) \in \nabla \R_{\norm{A}}[X]$ is proved in~\cite[Lem.~2.5]{GS2024}. Then, we deduce that $\oK(\nabla h,\ux) = \flat(\ux)\cdot K(\nabla h,\ux) \in \nabla \R_{\norm{A}}[X]$ by translation-invariance, see Lemma~\ref{lem: polynomial gradients}.
\end{proof}

Thanks to Lemma~\ref{lem: Kergin interpolant of gradients}, if $k=d$ and $f :U \to\R^k$ is the gradient of $h:U \to \R$, we can replace spaces of the form $\R_q[X]^k$ by the subspaces $\nabla \R_{q+1}[X]$ in Sections~\ref{sec: polynomial interpolation and Gaussian fields} to~\ref{sec: proof of thm cumulants asymptotics zero sets}. Regularity results and estimates remain true by restriction. We need to check that surjectivity and non-degeneracy results also remain true after restricting to these subspaces of polynomial gradients. The following two results are the counterparts for gradients of Lemma~\ref{lem: surjectivity blockwise Kergin} and Corollary~\ref{cor: surjectivity 1-jets} respectively.

\begin{lem}[Surjectivity of block-wise interpolation for gradients]
\label{lem: surjectivity blockwise gradients}
Let $A$ be a non-empty finite set, let $\ux \in (\R^d)^A$ and $\cI \in \cQ_0(\ux)$. The linear map $P \mapsto \parentheses*{K(P,\ux_I)}_{I \in \cI}$ is surjective from $\nabla\R_{\norm{A}}[X]$ to $\prod_{I \in \cI} \nabla\R_{\norm{I}}[X]$.
\end{lem}

\begin{proof}
We use the same partition of unity $(\chi_I)_{I \in \cI}$ as in the proof of Lemma~\ref{lem: surjectivity blockwise Kergin}. For all $I \in \cI$, let $P_I \in \nabla\R_{\norm{I}}[X]$ and $Q_I \in \R_{\norm{I}}[X]$ be such that $P_I = \nabla Q_I$. Let $h = \sum_{I \in \cI}\chi_I Q_I$ and $P=K(\nabla h,\ux)$. By Lemma~\ref{lem: Kergin interpolant of gradients}, we have $P \in \nabla\R_{\norm{A}}[X]$. For all $I \in \cI$, we have:
\begin{equation*}
K(P,\ux_I) = K(K(\nabla h,\ux),\ux_I) = K(\nabla h,\ux_I) = K(\nabla Q_I,\ux_I) = K(P_I,\ux_I) = P_I,
\end{equation*}
since $h =Q_I$ on a neighborhood of $\conv(\ux_I)$ and $K(\cdot,\ux_I)$ is the identity on $\nabla \R_{\norm{I}}[X] \subset \R_{\norm{I}-1}[X]^d$, see Theorem~\ref{thm: Kergin interpolation}. Hence the result.
\end{proof}

Let $h:U\to \R$ be a $\cC^2$ function. By definition, we have $Dh:x \mapsto \prsc{\nabla h(x)}{\cdot}$, where $\prsc{\cdot}{\cdot}$ stands for the canonical inner product on $\R^d$. Then, for all $x \in U$ and all $u,v \in \R^d$, we have
\begin{equation}
\label{eq: differential of gradient}
D^2_xh(u,v) = \prsc{D_x(\nabla h)\cdot u}{v}.
\end{equation}
That is, $D_x(\nabla h)\in \sym(\R^d)$ is the self-adjoint operator corresponding to the symmetric bilinear form $D_x^2h$ through the Euclidean structure.

\begin{cor}[$1$-jets surjectivity for gradients]
\label{cor: surjectivity 1-jets gradients}
Let $A$ be a non-empty finite set and $\ux \in (\R^d)^A \setminus \diag$, the linear map $P \mapsto \parentheses*{\strut \parentheses*{P(x_a),D_{x_a}P}}_{a \in A}$ is surjective from $\nabla\R_{2\norm{A}}[X]$ to $\parentheses*{\R^d \times \sym(\R^d)}^A$.
\end{cor}

\begin{proof}
The proof is similar to that of Corollary~\ref{cor: surjectivity 1-jets}. Since $\ux \notin \diag$, Lemma~\ref{lem: surjectivity blockwise gradients} shows that the map $P \mapsto \parentheses*{\strut K(P,x_a,x_a)}_{a \in A}$ is surjective from $\nabla\R_{2\norm{A}}[X]$ to $(\nabla\R_2[X])^A$. For all $a \in A$, the polynomials $P$ and $K(P,x_a,x_a)$ share the same value and the same differential at $x_a$. Thus, to conclude it is enough to prove that, for all $y \in \R^d$, the map $\vartheta_y:P \mapsto (P(y),D_yP)$ is an isomorphism from $\nabla \R_2[X]$ to $\R^d \times \sym(\R^d)$.

Let $y \in \R^d$ and $Q \in \R_2[X]$, we have $\vartheta_y(\nabla Q)=\parentheses*{\nabla Q(y),D_y(\nabla Q)} \in \R^d \times \sym(\R^d)$, see~\eqref{eq: differential of gradient}. Hence $\vartheta_y\parentheses*{\nabla \R_2[X]}\subset \R^d \times \sym(\R^d)$. For all $P \in \nabla \R_2[X]$, we have $P = P(y) + D_yP(X-y)$, since $P$ is a polynomial map of degree $1$. Hence $\vartheta_y$ is injective. The kernel of $\nabla:\R_2[X] \to \nabla\R_2[X]$ is $\R_0[X] \simeq \R$, so that $\dim(\nabla \R_2[X])=\dim(\R_2[X])-1 = d + \frac{d(d+1)}{2}$. Hence $\vartheta_y$ is actually an isomorphism by a dimension argument.
\end{proof}

Let $A$ be a non-empty finite set and $\cI \in \cP_A$. Let $h \in \cC^{\norm{A}}(U,\R)$ be a centered Gaussian field and let $f \in \cC^{\norm{A}-1}(U,\R^d)$ be the centered Gaussian field $f = \nabla h$. For all $\ux \in U^A$, the random vector $\parentheses*{\oK(f,\ux_I)}_{I \in \cI} \in \prod_{I \in \cI} \R_{\norm{I}-1}[X]^d$ belongs deterministically to the subspace $\prod_{I \in \cI} \nabla \R_{\norm{I}}[X]$, see Lemma~\ref{lem: Kergin interpolant of gradients}. Thus, in the case where $f$ is a gradient field, we can consider $\parentheses*{\oK(f,\ux_I)}_{I \in \cI}$ as a centered Gaussian vector in $\prod_{I \in \cI} \nabla \R_{\norm{I}}[X]$, and its variance operator $\Sigma_\cI(f,\ux)=\Sigma_\cI(\nabla h,\ux)$ as an element of $\sym\parentheses*{\prod_{I \in \cI} \nabla \R_{\norm{I}}[X]}$, compare Definition~\ref{def: Sigma I}. With this in mind, we can state the following counterpart of Lemma~\ref{lem: uniform ND eta delta} for gradient fields.

\begin{lem}[Uniform non-degeneracy for gradient fields]
\label{lem: uniform ND eta delta gradient}
Let $A$ be a finite set of cardinality $q+1$. Let $h\in \cC^{q+1}(\R^d,\R)$ be a stationary centered Gaussian field and let $f = \nabla h$. We assume that $f$ satisfies \hypNND{q} and \hypDCL{q}{\infty}. Let $\cI \in \cP_A$ and $\delta>0$, there exists $\eta \in (0,\delta]$ such that $\brackets*{\Sigma_\cI(f,\ux) \mvert \ux \in \diag_{\cI,\eta}\cap \diag_{\cI,\delta}}$ is relatively compact in $\sym^+\parentheses*{\prod_{I \in \cI}\nabla\R_{\norm{I}}[X]}$.
\end{lem}

\begin{proof}
The proof is similar to that of Lemma~\ref{lem: uniform ND eta delta}, except for Step 3 which needs to be adapted in order to use \hypNND{q} instead of \hypND{q}.

First, we observe that, for all $p \in \N^*$, the map $\nabla:\R_p[X] \to \nabla \R_p[X]$ is surjective and its kernel is $\R_0[X]$. Hence, it restricts into an isomorphism from $\Span\brackets*{X^\beta \mvert 0 < \norm{\beta}\leq p}$ to $\nabla \R_p[X]$. In particular, $\parentheses*{\nabla(X^\beta)}_{0 < \norm{\beta}\leq p}$ is a basis of $\nabla \R_p[X]$.

Then, let $(e_1,\dots,e_d)$ denote the standard basis of $\R^d$. With the same notation as in the proof of Lemma~\ref{lem: uniform ND eta delta}, we modify Step 3 as follows. For all $I \in \cI_J$, we have $\ux_I=(y_I,\dots,y_I) \in (\R^d)^I$. Hence, recalling that $f=\nabla h$, we have:
\begin{align*}
\oK(f,\ux_I) &= \sum_{\norm{\alpha}<\norm{I}} \frac{\partial^\alpha f(y_I)}{\alpha !}X^\alpha = \sum_{\substack{1 \leq i \leq d\\ \norm{\alpha}<\norm{I}}} \frac{\partial^\alpha \partial_i h(y_I)}{\alpha !}X^\alpha \otimes e_i = \sum_{\substack{1 \leq i \leq d\\ \norm{\beta}\leq \norm{I} ; \beta_i>0}} \frac{\partial^\beta h(y_I)}{\beta !} \beta_i X^{\beta-\one_i} \otimes e_i\\
& = \sum_{\substack{1 \leq i \leq d\\ 0 <\norm{\beta}\leq \norm{I}}} \frac{\partial^\beta h(y_I)}{\beta !} \beta_i X^{\beta-\one_i} \otimes e_i = \sum_{0 <\norm{\beta}\leq \norm{I}} \frac{\partial^\beta h(y_I)}{\beta !} \nabla(X^\beta).
\end{align*}
Since $\parentheses*{\nabla(X^\beta)}_{0<\norm{\beta}\leq \norm{I}}$ is a basis of $\nabla\R_{\norm{I}}[X]$ for all $I \in \cI_J$, we have that $\parentheses*{\oK(f,\ux_I)}_{I \in \cI_J}$ is non-degenerate in $\prod_{I \in \cI_J} \nabla\R_{\norm{I}}[X]$ if and only if the centered Gaussian vector $\parentheses*{\partial^\beta h(y_I)}_{I \in \cI_J; 0<\norm{\beta}\leq \norm{I}}$ is non-degenerate. Since $J \subset A$, we have $\norm{J}\leq q+1$, and the last condition is implied by \hypNND{q}, because $\uy = (y_I)_{I \in \cI_J} \notin \diag$.

Finally, we obtain that $\Sigma_{\cI_J}(f,\ux_J) \in \sym^+\parentheses*{\prod_{I \in \cI_J}\nabla\R_{\norm{I}}[X]}$, and we can conclude by contradiction as in the proof of Lemma~\ref{lem: uniform ND eta delta}.
\end{proof}

Except for the three results we just discussed, the content of Section~\ref{sec: polynomial interpolation and Gaussian fields} adapts directly to the case of gradient fields by replacing spaces of the form $\R_q[X]^k$ by the corresponding $\nabla \R_{q+1}[X]$; replacing the use of Hypothesis \hypND{q} by that of Hypothesis \hypNND{q}; and bearing in mind that the field $f$ (resp.~$f_R$) is the gradient of a field $h$ (resp.~$h_R$). 

%%%%%%%%%%%%%%%%%%%%%%%%%%%%%%%%%%%%%%%%%%%%%%%%%%%%%%%%%%%%%%%%%%%%%%%%%%%%%%%%%%%%%%%%%%%%%%%%%%%%%%%

\subsection{Kac--Rice formulas for gradient fields}
\label{subsec: Kac--Rice formulas for gradient fields}

In this section, we discuss the changes we need to make in order to adapt the content of Section~\ref{sec: Kac-Rice formulas revisited} to the case of gradient fields.

The first point is inherited from the changes discussed in Section~\ref{subsec: Kergin interpolation for gradient fields}: we replace spaces of the form $\R_{2\norm{A}-1}[X]^k$, where $A$ is a non-empty finite set, by $\nabla \R_{2\norm{A}}[X]$. For example, for all $\ux \in (\R^d)^A$, we now have $\ev_{\ux} \in \cL^\dagger\parentheses*{\nabla\R_{2\norm{A}}[X],(\R^d)^A}$ and $\cG_A(\ux) \in \gr{d\norm{A}}{\nabla\R_{2\norm{A}}[X]}$, cf.~Lemmas~\ref{lem: regularity of ev} and~\ref{lem: regularity GA}. These changes affect similarly Definitions~\ref{def: N theta}, \ref{def: Upsilon}, \ref{def: M I}, \ref{def: sigma I}, \ref{def: L J} and~\ref{def: F I J}. In particular, if $f$ is the gradient of some field $h$ and $\cI \in \cP_A$, we now have:
\begin{equation*}
\Sigma_{2\cI}(f,\ux) =\Sigma_{2\cI}(\nabla h,\ux) \in \cS_\cI := \sym\parentheses*{\prod_{I \in \cI} \nabla\R_{2\norm{I}}[X]}.
\end{equation*}

Then, we need to relax the hypotheses of the Bulinskaya Lemma (Proposition~\ref{prop: Bulinskaya}) and the Kac--Rice formula (Theorem~\ref{thm: Kac-Rice}). Indeed, if the centered Gaussian field $f:U \to \R^d$ is the gradient of some field $h:U \to \R$ then, for all $x \in U$, the Gaussian vector $\parentheses*{f(x),D_xf}\in\R^d \times \sym(\R^d)$ is degenerate as a random vector in $\R^d \times \cL(\R^d)$. It turns out that, for fields from $\R^d$ to itself, and in particular for gradient fields, we can replace Proposition~\ref{prop: Bulinskaya} by the following.

\begin{prop}[Bulinskaya Lemma, maximal codimension case]
\label{prop: Bulinskaya gradient}
Let $U \subset \R^d$ be open and $f\in\cC^2(U,\R^d)$ be a centered Gaussian field such that, for all $x \in U$, the Gaussian vector $f(x)$ is non-degenerate. Then, almost-surely,
\begin{equation*}
\brackets*{x \in U \mvert f(x)=0 \ \text{and} \ \jac{D_xf}=0} = \emptyset.
\end{equation*}
In particular, $Z=f^{-1}(0)$ is almost-surely a closed discrete subset of $U$.
\end{prop}

\begin{proof}
This is a particular case of \cite[Prop.~6.5]{AW2009}.
\end{proof}

The gradient fields we consider satisfy these relaxed hypotheses, so we can replace the use of Proposition~\ref{prop: Bulinskaya} by that of Proposition~\ref{prop: Bulinskaya gradient} in their case. Similarly, we can relax the hypotheses of Theorem~\ref{thm: Kac-Rice} in the case where $f:U \to \R^d$, by removing the assumption that $\parentheses*{f(x),D_xf}$ be non-degenerate for all $x \in U$. This is because this non-degeneracy hypothesis is only there to ensure that we can use the Bulinskaya Lemma (see~\cite[Thm.~6.2 and~6.3]{AW2009}), which is already ensured in this case by the assumption that $\parentheses*{f(x_a)}_{a \in A}$ be non-degenerate for all $\ux \notin \diag$.

\begin{rem}
\label{rem: rho A gradient}
Let $A$ be a non-empty finite set and $f:U \to \R^d$ be a centered Gaussian field. The expressions of the Kac--Rice densities $\rho_A(f,\cdot)$ and $\cF_A(f,\cdot)$, see Definitions~\ref{def: rho A} and~\ref{def: F A}, do not depend on whether $f$ is a gradient field or not. As a consequence, the same is true of the constants $\gamma_p(f)$ in Theorems~\ref{thm: cumulants asymptotics for zero sets} and~\ref{thm: cumulants asymptotics for critical points}, see~\eqref{eq: expression gamma p f}. On the contrary, the expressions of $\Upsilon_A$, $\sigma_\cI$ and $\cF_{\cI,\cJ}$, see Definitions~\ref{def: Upsilon}, \ref{def: sigma I} and~\ref{def: F I J}, depend on whether we consider gradient fields or not.
\end{rem}

We conclude this section by pointing out three other points where the content of Section~\ref{sec: Kac-Rice formulas revisited} needs to be mildly modified in order to deal with gradient fields. Everything else adapts seamlessly.
\begin{itemize}
\item In the proof of Lemma~\ref{lem: non-degeneracy differentials}, we need to choose $\Lambda_0 \in \sym^+(\R^d)$, for example $\Lambda_0 = \Id$. Then, using Corollary~\ref{cor: surjectivity 1-jets gradients} instead of Corollary~\ref{cor: surjectivity 1-jets}, we obtain the existence of $P \in \nabla \R_{2\norm{A}}[X]$ such that $P(x_a-\flat(\ux))=0$ and $D_{x_a -\flat(\ux)}P=\Id$ for all $a \in A$.

\item The conclusion of Lemma~\ref{lem: relation between notions of non-degeneracy} needs to be weakened: for gradient fields we can only conclude that $\parentheses*{f(x_a)}_{a \in A}$ is non-degenerate. This is not an issue, since Lemma~\ref{lem: relation between notions of non-degeneracy} is only used to check the non-degeneracy hypotheses in the Bulinskaya Lemma and the Kac--Rice formula and we just relaxed these hypotheses for gradient fields. The proof of this weaker version is similar to that of Lemma~\ref{lem: relation between notions of non-degeneracy}, using Corollary~\ref{cor: surjectivity 1-jets gradients} instead of Corollary~\ref{cor: surjectivity 1-jets}.

\item In Step 1 of the proof of Lemma~\ref{lem: regularity Upsilon}, we have to change the definition of $f_0$. Let $h_0$ be a standard Gaussian vector in $\R_{2\norm{A}}[X]$. We let $f_0 = \nabla h_0$, so that $f_0$ is both a gradient field on $\R^d$ and a centered Gaussian vector in $\nabla \R_{2\norm{A}}[X]$, which is non-degenerate since $\nabla$ is surjective and $h_0$ is non-degenerate. Then, we deduce the local integrability of $\rho_A(f_0,\cdot)$ from~\cite[Thm.~1.13 and Lem.~2.10]{GS2024}, compare also~\cite[Thm.~7.7]{AL2025}. The remainder of the proof is similar to that of Lemma~\ref{lem: regularity Upsilon}.
\end{itemize}

%%%%%%%%%%%%%%%%%%%%%%%%%%%%%%%%%%%%%%%%%%%%%%%%%%%%%%%%%%%%%%%%%%%%%%%%%%%%%%%%%%%%%%%%%%%%%%%%%%%%%%%

\subsection{Proof of Theorem~\ref{thm: cumulants asymptotics for critical points}: cumulants asymptotics for critical points}
\label{subsec: proof of thm cumulants asymptotics for critical points}

In this section, we discuss the changes to implement in order to turn the proof of Theorem~\ref{thm: cumulants asymptotics for zero sets} into a proof of Theorem~\ref{thm: cumulants asymptotics for critical points}. We already discussed above the changes to implement in Section~\ref{sec: polynomial interpolation and Gaussian fields} and~\ref{sec: Kac-Rice formulas revisited}. The contents of Sections~\ref{sec: partitions cumulants and diagonals}, \ref{sec: edge-connected graphs} and~\ref{sec: HBL inequalities} are completely unaffected by whether the Gaussian fields we consider are gradient fields or not. However, we need to modify slightly Sections~\ref{sec: refined Hadamard lemma} and~\ref{sec: proof of thm cumulants asymptotics zero sets}.

At the beginning of Section~\ref{subsec: refined Hadamard lemma for FIJ}, due to the changes already discussed in Section~\ref{subsec: Kergin interpolation for gradient fields}, for all $I \in \cI$ the space $\R_{2\norm{I}-1}[X]^k$ is replaced by $\nabla \R_{2\norm{I}}[X]$ in the case of gradient fields. Since the linear map $\nabla:\R_{2\norm{I}}[X] \to \nabla \R_{2\norm{I}}[X]$ is surjective with kernel $\R_0[X]$, a basis of $\nabla \R_{2\norm{I}}[X]$ is $\parentheses*{\nabla(X^\beta)}_{0 < \norm{\beta} \leq 2\norm{I}}$. We choose to work with the inner product on $\nabla \R_{2\norm{I}}[X]$ making the previous basis orthonormal, this choice being of no consequence to the rest of the argument. We also change the definition of $\cA_I$, see Equation~\eqref{eq: def AI}, into
\begin{equation*}
\cA_I = \brackets*{(I,\beta) \in \brackets{I}\times \N^d \mvert 0 < \norm{\beta} \leq 2\norm{I}}.
\end{equation*}
Then, proceeding as in Section~\ref{subsec: refined Hadamard lemma for FIJ}, we obtain the exact analogue of Proposition~\ref{prop: key estimate FIJ} in the case of gradient fields. Recall that the definition of $\cF_{\cI,\cJ}$ depends on whether we consider gradient fields or not, see Remark~\ref{rem: rho A gradient}. No other modification is needed in Section~\ref{sec: refined Hadamard lemma}.

In Section~\ref{sec: proof of thm cumulants asymptotics zero sets}, we mostly have to implement changes inherited from those of the previous sections: replace $\R_{2\norm{A}-1}[X]^k$ by $\nabla \R_{2\norm{A}}[X]$ where needed, replace the use of Proposition~\ref{prop: Bulinskaya} by that of Proposition~\ref{prop: Bulinskaya gradient}, replace Hypothesis~\hypND{2p-1} by \hypNND{2p-1}, etc. Apart from these, we only need to take care of the Case $p=1$ in the proof of Theorem~\ref{thm: cumulants asymptotics for critical points}.\ref{item: cumulants asymptotics p crit}. The proof is similar to the same case in the proof of Theorem~\ref{thm: cumulants asymptotics for zero sets}.\ref{item: cumulants asymptotics p}, except for the fact that $\parentheses*{f(0),D_0f}$ is non-degenerate as a Gaussian vector in $\R^d\times \sym(\R^d)$, as a consequence of Hypothesis~\hypNND{1}, but not as a Gaussian vector in $\R^d \times \cL(\R^d)$. In the case where $f$ and $(f_R)_{R \in \cR}$ are gradient fields, this is enough for the argument to go through. Finally, we have $\gamma_2(f)>0$ by~\cite[thm.~1.7]{Gas2025}.

%%%%%%%%%%%%%%%%%%%%%%%%%%%%%%%%%%%%%%%%%%%%%%%%%%%%%%%%%%%%%%%%%%%%%%%%%%%%%%%%%%%%%%%%%%%%%%%%%%%%%%%
%%%%%%%%%%%%%%%%%%%%%%%%%%%%%%%%%%%%%%%%%%%%%%%%%%%%%%%%%%%%%%%%%%%%%%%%%%%%%%%%%%%%%%%%%%%%%%%%%%%%%%%

\section{Proofs of limit theorems and other corollaries}
\label{sec: proofs limit theorems and other corollaries}

In this section, we deduce Corollaries~\ref{cor: moment asymptotics}, \ref{cor: concentration} and~\ref{cor: hole probability} and Theorems~\ref{thm: law of large numbers} and~\ref{thm: central limit theorem} from our cumulants asymptotics.

Let us briefly recall our framework. We consider an unbounded set $\cR \subset (0,+\infty)$, as well as a convex open set $U \subset \R^d$ and a non-empty open subset $\Omega \subset U$ at positive distance from $\R^d \setminus U$. Let $k \in \ssquarebrackets{1}{d}$ and $f:\R^d \to \R^k$ be a centered stationary Gaussian field. For all $R \in \cR$, let $f_R:RU \to \R^k$ be a centered Gaussian field. We denote by $Z_R$ the random measure of integration over $f_R^{-1}(0)$ with respect to its $(d-k)$-dimensional volume measure, which is almost-surely well-defined under our assumptions. Finally, we let $\nu_R$ be the random measure on $U$ defined by~\eqref{eq: def linear statistics nuR}.

In the setting of Theorem~\ref{thm: cumulants asymptotics for zero sets}, the limit field $f$ satisfies a non-degeneracy hypothesis of the type \hypND{q}, among other assumptions. In the setting of Theorem~\ref{thm: cumulants asymptotics for critical points}, we assume that $f$ and the $(f_R)_{R \in \cR}$ are gradient fields. That is, $k=d$ and there exist centered Gaussian fields $h$ and $(h_R)_{R \in \cR}$ such that $f=\nabla h$ and $f_R = \nabla h_R$ for all $R \in \cR$. In this case, the limit field $f$ satisfies, among other things, a non-degeneracy hypothesis of the type \hypNND{q}. Since both settings are formally identical, we can deal with both of them at once in the following.

%%%%%%%%%%%%%%%%%%%%%%%%%%%%%%%%%%%%%%%%%%%%%%%%%%%%%%%%%%%%%%%%%%%%%%%%%%%%%%%%%%%%%%%%%%%%%%%%%%%%%%%

\subsection{Moment asymptotics, concentration and hole probability}
\label{subsec: moment asymptotics concentration and hole probability}

This section deals with the proofs of Corollaries~\ref{cor: moment asymptotics}, \ref{cor: concentration} and~\ref{cor: hole probability}. We start by deducing the asymptotics of the joint moments of the linear statistics of $\nu_R$ from that of their joint cumulants.

\begin{proof}[Proof of Corollary~\ref{cor: moment asymptotics}.\ref{item: moment asymptotics infty}]
Let $p \in \N^*$, we assume that \hypReg{\max(2,2p-1)}, \hypScL{2p-1} and \hypDC{2p-1}{\infty} hold. Moreover, we assume that either \hypNND{2p-1} or \hypND{2p-1} holds, depending on whether we consider gradient fields or not.

Let $\phi_1,\dots,\phi_p \in L^1(\Omega) \cap L^\infty(\Omega)$. Let $R \in \cR$, for all $i \in \ssquarebrackets{1}{p}$ we have $\norm*{\prsc{\nu_R}{\phi_i}} \leq \prsc*{\nu_R}{\norm*{\phi_i}}$. Thus, it is enough to prove that $\esp{\prod_{i=1}^p \prsc{\nu_R}{\norm*{\phi_i}}}<+\infty$ for all $R \in \cR$ large enough. By Definition~\ref{def: moments and cumulants} and Corollary~\ref{cor: moments in terms of cumulants}, we have:
\begin{equation*}
\esp{\prod_{i=1}^p \prsc{\nu_R}{\norm*{\phi_i}}} = \sum_{\cJ \in \cP_p} \prod_{J \in \cJ} \kappa\parentheses*{\parentheses*{\prsc{\nu_R}{\norm*{\phi_j}}}_{j \in J}} = \sum_{\cJ \in \cP_p} \prod_{J \in \cJ} \kappa(\nu_R)\parentheses*{\parentheses*{\norm*{\phi_j}}_{j \in J}}
\end{equation*}
Under our assumptions, either by Theorem~\ref{thm: cumulants asymptotics for zero sets}.\ref{item: cumulants asymptotics infty} or Theorem~\ref{thm: cumulants asymptotics for critical points}.\ref{item: cumulants asymptotics infty crit}, for all non-empty $J \subset \ssquarebrackets{1}{p}$ we have that $\kappa(\nu_R)\parentheses*{\parentheses*{\norm*{\phi_j}}_{j \in J}}$ is well-defined and finite for $R$ large enough, not depending on $J$ nor on $(\phi_1,\dots,\phi_p)$. Hence the result.
\end{proof}

\begin{proof}[Proof of Corollary~\ref{cor: moment asymptotics}.\ref{item: moment asymptotics 2}]
Let $p \in \N^*$, we work under the same assumptions as above, except that Hypothesis \hypDC{2p-1}{\infty} is replaced by the stronger Hypothesis~\hypDC{2p-1}{2}.

Let $\phi_1,\dots,\phi_p \in L^1(\Omega) \cap L^\infty(\Omega)$. By Corollary~\ref{cor: moments in terms of cumulants}, we have:
\begin{equation}
\label{eq: proof moments}
\esp{\prod_{i=1}^p \prsc{\nu_R}{\phi_i}} = \sum_{\cJ \in \cP_p} \prod_{J \in \cJ} \kappa(\nu_R)\parentheses*{\parentheses*{\phi_j}_{j \in J}}.
\end{equation}
Under our hypotheses, by Theorems~\ref{thm: cumulants asymptotics for zero sets}.\ref{item: cumulants asymptotics 2} and~\ref{thm: cumulants asymptotics for zero sets}.\ref{item: cumulants asymptotics p} (resp. Theorems~\ref{thm: cumulants asymptotics for critical points}.\ref{item: cumulants asymptotics 2 crit} and~\ref{thm: cumulants asymptotics for critical points}.\ref{item: cumulants asymptotics p crit} in the case of gradient fields), for all $J \subset \ssquarebrackets{1}{p}$ such that $\norm{J}\geq 2$, we have $\kappa(\nu_R)\parentheses*{\parentheses*{\phi_j}_{j \in J}} \xrightarrow[R \to +\infty]{} 0$. The terms associated with singletons being convergent as $R \to +\infty$, all the products on the right-hand side of~\eqref{eq: proof moments} vanish in the limit, except maybe for the one indexed by $\brackets*{\strut \brackets{i}\mvert 1 \leq i \leq p} = \bigwedge \cP_p$. Thus, we have:
\begin{equation*}
\esp{\prod_{i=1}^p \prsc{\nu_R}{\phi_i}} = \prod_{i=1}^p \kappa(\nu_R)(\phi_i) +o(1) \xrightarrow[R \to +\infty]{} \gamma_1(f)^p \prod_{i=1}^p \int_{x \in \Omega} \phi_i(x) \dx x,
\end{equation*}
which proves~\eqref{eq: moment asymptotics}.

Let us now deal with central moments. Let $J \subset \ssquarebrackets{1}{p}$, developing the joint cumulant by multi-linearity and using Lemma~\ref{lem: cancellation property} to cancel the terms where an expectation appears, we have:
\begin{align*}
R^{\norm{J}\frac{d}{2}}\kappa\parentheses*{\parentheses*{\prsc{\nu_R}{\phi_j}\strut-\esp{\prsc{\nu_R}{\phi_j}}}_{j \in J}} &= \begin{cases} R^{\norm{J}\frac{d}{2}}\kappa(\nu_R)\parentheses*{\parentheses*{\phi_j}_{j \in J}}, & \text{if} \ \norm{J}\geq 2;\\ 0, & \text{if} \ \norm{J}=1;\end{cases}\\ 
&\xrightarrow[R \to +\infty]{} \begin{cases}
\gamma_2(f) \displaystyle\int_{x \in \Omega} \prod_{j \in J} \phi_j(x) \dx x, & \text{if} \ \norm{J}=2;\\
0, & \text{otherwise};
\end{cases}
\end{align*}
by Theorems~\ref{thm: cumulants asymptotics for zero sets}.\ref{item: cumulants asymptotics 2} and~\ref{thm: cumulants asymptotics for zero sets}.\ref{item: cumulants asymptotics p} (resp. Theorems~\ref{thm: cumulants asymptotics for critical points}.\ref{item: cumulants asymptotics 2 crit} and~\ref{thm: cumulants asymptotics for critical points}.\ref{item: cumulants asymptotics p crit} in the case of gradient fields). Then, proceeding as in the case of non-central moments, by Corollary~\ref{cor: moments in terms of cumulants} we have:
\begin{align*}
R^\frac{pd}{2} \esp{\prod_{i=1}^p \parentheses*{\prsc{\nu_R}{\phi_i}-\esp{\prsc{\nu_R}{\phi_i}\strut}}} &= \sum_{\cJ \in \cP_p} \prod_{J \in \cJ} R^{\norm{J}\frac{d}{2}}\kappa\parentheses*{\parentheses*{\prsc{\nu_R}{\phi_j}\strut-\esp{\prsc{\nu_R}{\phi_j}}}_{j \in J}}\\
&\xrightarrow[R \to +\infty]{} \sum_{\substack{\cJ \in \cP_p\\ \forall J \in \cJ, \norm{J}= 2}} \prod_{J \in \cJ} \parentheses*{\gamma_2(f) \int_{x \in \Omega} \prod_{j \in J}\phi_j(x) \dx x},
\end{align*}
which is a reformulation of~\eqref{eq: central moment asymptotics}.
\end{proof}

Classically, the asymptotics we just obtained for central moments imply that the linear statistics of $\nu_R$ concentrate around their mean as $R \to +\infty$. We now prove that they also concentrate around the asymptotics of their mean.

\begin{proof}[Proof of Corollary~\ref{cor: concentration}]
Let $\phi \in L^1(\Omega) \cap L^\infty(\Omega)$ and $\epsilon>0$. We let $p \in \N^*$ and assume that \hypReg{4p-1}, \hypScL{4p-1}, \hypDC{4p-1}{2} and \hypND{4p-1} (resp.~\hypNND{4p-1} in the case of gradient fields) hold.

Either by Theorem~\ref{thm: cumulants asymptotics for zero sets}.\ref{item: cumulants asymptotics p} or Theorem~\ref{thm: cumulants asymptotics for critical points}.\ref{item: cumulants asymptotics p crit}, for all $R \in \cR$ large enough, we have:
\begin{equation*}
\norm*{\esp{\prsc{\nu_R}{\phi}}-\gamma_1(f) \int_{x \in \Omega}\phi(x) \dx x} =\norm*{\kappa(\nu_R)(\phi)-\gamma_1(f) \int_{x \in \Omega}\phi(x) \dx x} < \frac{\epsilon}{2}.
\end{equation*}
Thus, $\norm*{\prsc{\nu_R}{\phi}-\gamma_1(f) \int_{x \in \Omega}\phi(x) \dx x} \geq \epsilon$ implies that $\norm*{\prsc{\nu_R}{\phi}-\esp{\prsc{\nu_R}{\phi}}\strut} \geq \frac{\epsilon}{2}$. Using Markov's inequality and the estimates for the central moment of order $2p$ given by Corollary~\ref{cor: moment asymptotics}.\ref{item: moment asymptotics 2}, we obtain:
\begin{multline*}
\P\parentheses*{\norm*{\prsc{\nu_R}{\phi}-\gamma_1(f) \int_{x \in \Omega}\phi(x) \dx x} \geq \epsilon} \leq \P\parentheses*{\norm*{\prsc{\nu_R}{\phi}-\esp{\prsc{\nu_R}{\phi}}} \geq \frac{\epsilon}{2}}\\
\leq \parentheses*{\frac{2}{\epsilon}}^{2p} \esp{\parentheses*{\prsc{\nu_R}{\phi}-\esp{\prsc{\nu_R}{\phi}\strut}}^{2p}} = O\parentheses*{\frac{1}{R^{pd}}},
\end{multline*}
which concludes the proof.
\end{proof}

The concentration result of Corollary~\ref{cor: concentration} translates into an estimate for the hole probability of $\frac{1}{R}Z_R$, where $Z_R$ stands for the vanishing locus of $f_R$.

\begin{proof}[Proof of Corollary~\ref{cor: hole probability}]
Let $p \in \N^*$, we assume that \hypReg{4p-1}, \hypScL{4p-1}, \hypDC{4p-1}{2} and \hypND{4p-1} (resp.~\hypNND{4p-1} in the case of gradient fields) hold.

Let $\cO \subset \Omega$ be a non-empty open set. Let $\cB \subset \cO$ be a closed ball of positive radius, so that its indicator function $\one_\cB$ belongs to $L^1(\Omega) \cap L^\infty(\Omega)$ and is compactly supported in $\cO$. For all $R \in \cR$, if $\frac{1}{R}Z_R \cap \cO = \emptyset$, then $Z_R \cap R\cB=\emptyset$, and $\prsc{\nu_R}{\one_\cB}=0$ by Equation~\eqref{eq: relation nuR volume}. In this case, $\norm*{\prsc{\nu_R}{\one_\cB}-\gamma_1(f)\int_{x \in \Omega}\one_\cB(x)\dx x} = \gamma_1(f) \int_{x \in \cB}\dx x$. Hence, applying Corollary~\ref{cor: concentration} with $\epsilon = \gamma_1(f) \int_{x \in \cB}\dx x>0$, we obtain:
\begin{equation*}
\P\parentheses*{\frac{1}{R}Z_R \cap \cO=\emptyset } \leq \P\parentheses*{\norm*{\prsc{\nu_R}{\one_\cB}-\gamma_1(f)\int_{x \in \Omega}\one_\cB(x)\dx x} \geq \epsilon} = O\parentheses*{R^{-pd}}.\qedhere
\end{equation*}
\end{proof}

%%%%%%%%%%%%%%%%%%%%%%%%%%%%%%%%%%%%%%%%%%%%%%%%%%%%%%%%%%%%%%%%%%%%%%%%%%%%%%%%%%%%%%%%%%%%%%%%%%%%%%%

\subsection{Proof of Theorems~\ref{thm: law of large numbers stationary case} and~\ref{thm: law of large numbers}: Law of Large Numbers}
\label{subsec: proof of thm law of large numbers}

We start by proving Theorem~\ref{thm: law of large numbers}. Let $p \in \N^*$, in the following, we assume that \hypReg{4p-1}, \hypScL{4p-1}, \hypDC{4p-1}{2} and \hypND{4p-1} (resp.~\hypNND{4p-1} in the case of gradient fields) hold. We also fix a sequence $(R_n)_{n \in \N^*}$ of elements of $\cR$ such that $\sum_{n \geq 1} R_n^{-pd}<+\infty$. Let us first prove the Law of Large Numbers for the linear statistics of $\nu_R$.

\begin{proof}[Proof of Theorem~\ref{thm: law of large numbers}: the case of linear statistics]
Let $\phi \in L^1(\Omega) \cap L^\infty(\Omega)$. Under the assumptions of Theorem~\ref{thm: law of large numbers}, we can apply Corollary~\ref{cor: moment asymptotics} to derive the asymptotics of the central moment of order $2p$ of $\prsc{\nu_R}{\phi}$.

Since $\sum_{n \geq 1}R_n^{-pd}<+\infty$, we have $R_n \xrightarrow[n \to +\infty]{}+\infty$. Hence, by Corollary~\ref{cor: moment asymptotics}.\ref{item: moment asymptotics infty}, there exists $n_0 \in \N^*$ such that, for all $n \geq n_0$, the random variable $\prsc{\nu_{R_n}}{\phi}$ has a finite moment of order~$2p$. Then,
\begin{equation*}
\esp{\sum_{n \geq n_0} \parentheses*{\prsc{\nu_{R_n}}{\phi}-\esp{\prsc{\nu_{R_n}}{\phi}\strut}}^{2p}} = \sum_{n \geq n_0} \esp{\parentheses*{\prsc{\nu_{R_n}}{\phi}-\esp{\prsc{\nu_{R_n}}{\phi}\strut}}^{2p}} <+\infty,
\end{equation*}
since $\esp{\parentheses*{\prsc{\nu_{R_n}}{\phi}-\esp{\prsc{\nu_{R_n}}{\phi}\strut}}^{2p}} = O(R_n^{-pd})$ by~Corollary~\ref{cor: moment asymptotics}.\ref{item: moment asymptotics 2}. Thus, almost-surely, we have
\begin{equation*}
\sum_{n \geq n_0} \parentheses*{\prsc{\nu_{R_n}}{\phi}-\esp{\prsc{\nu_{R_n}}{\phi}}\strut}^{2p} <+\infty,
\end{equation*}
which implies $\norm*{\prsc{\nu_{R_n}}{\phi}-\esp{\prsc{\nu_{R_n}}{\phi}\strut}} \xrightarrow[n \to +\infty]{}0$. We have $\esp{\prsc{\nu_{R_n}}{\phi}} \xrightarrow[n \to +\infty]{} \gamma_1(f) \displaystyle\int_{\Omega}\phi(x) \dx x$ by~\eqref{eq: moment asymptotics}. Hence, $\prsc{\nu_{R_n}}{\phi} \xrightarrow[n \to +\infty]{\text{a.s.}} \gamma_1(f) \displaystyle\int_{\Omega}\phi(x) \dx x$.
\end{proof}

We can now prove a functional Law of Large Numbers, that is, for the random measures $\nu_R$ themselves.

\begin{proof}[Proof of Theorem~\ref{thm: law of large numbers}: the case of random measures]
We want to show that $\nu_{R_n}$ almost-surely converges to $\gamma_1(f) \dx x$ in the sense of the weak-$*$ convergence for Radon measures on $\Omega$; i.e., that almost-surely, for all $\phi \in \cC^0_c(\Omega)$, we have $\prsc{\nu_{R_n}}{\phi} \xrightarrow[n \to +\infty]{} \gamma_1(f) \int_{\Omega}\phi \dx x$. This follows from the first part of Theorem~\ref{thm: law of large numbers} and an approximation argument.

Recall that the usual topology of $\cC^0_c(\Omega)$ is characterized by the fact that $\phi_j \xrightarrow[j \to +\infty]{}\phi$ if and only if there exists a compact $K \subset \Omega$ containing the supports of all the $(\phi_j)_{j \in \N}$ and $\Norm{\phi_j - \phi}_\infty \xrightarrow[j \to +\infty]{}0$. Recall also that $\cC^0_c(\R)$ endowed with this topology is separable. Let $(\phi_j)_{j \in \N^*}$ be a dense sequence in this space. Let $\phi_0 = \one_\Omega$ denote the indicator function of $\Omega$.

Let $(K_i)_{i \in \N}$ be an increasing sequence of compact subsets of $\Omega$ such that $\bigcup_{i \geq 0} K_i = \Omega$. For all $i \in \N$, there exists a non-negative $\chi_i \in \cC^0_c(\Omega)$ such that $\chi_i$ is constant at $1$ on $K_i$. Then, for all $(i,j) \in \N^2$, we have $\chi_i\phi_j \in \cC_c^0(\Omega) \subset L^1(\Omega) \cap L^\infty(\Omega)$. By the first part of Theorem~\ref{thm: law of large numbers} and countability of $\N^2$, we have almost-surely:
\begin{equation}
\label{eq: as CV}
\forall (i,j) \in \N^2, \qquad \prsc{\nu_{R_n}}{\chi_i\phi_j} \xrightarrow[n \to +\infty]{} \gamma_1(f) \int_{\Omega} \chi_i\phi_j \dx x.
\end{equation}

Let us fix a realization of the sequence $(\nu_{R_n})_{n \in \N^*}$ such that \eqref{eq: as CV} holds, and work with this deterministic sequence of measures. Let $\phi \in \cC^0_c(\Omega)$, for all $(i,j) \in \N^2$, we have:
\begin{multline*}
\norm*{\prsc{\nu_{R_n}}{\phi} - \gamma_1(f) \int_\Omega \phi \dx x} \\
\leq \prsc{\nu_{R_n}}{\norm*{\phi-\chi_i\phi_j}}+\norm*{\prsc{\nu_{R_n}}{\chi_i\phi_j} - \gamma_1(f) \int_\Omega \chi_i\phi_j \dx x}+\gamma_1(f) \int_\Omega \norm*{\phi - \chi_i\phi_j}\dx x.
\end{multline*}

Let $i \in \N$ be such that the support of $\phi$ is contained in $K_i$. Then $\phi = \chi_i\phi$, so that we have $\norm{\phi - \chi_i\phi_j} = \chi_i \norm{\phi-\phi_j}\leq \chi_i\phi_0 \Norm{\phi-\phi_j}_\infty$ for all $j \in \N$. Thus, for all $j \in \N$,
\begin{multline*}
\norm*{\prsc{\nu_{R_n}}{\phi} - \gamma_1(f) \int_\Omega \phi \dx x} \\
\leq \Norm{\phi-\phi_j}_\infty \parentheses*{\prsc{\nu_{R_n}}{\chi_i\phi_0}+\gamma_1(f)\int_\Omega \chi_i\phi_0\dx x}+\norm*{\prsc{\nu_{R_n}}{\chi_i\phi_j} - \gamma_1(f) \int_\Omega \chi_i\phi_j \dx x}.
\end{multline*}
Since we assume that \eqref{eq: as CV} holds, the sequence $\parentheses*{\prsc{\nu_{R_n}}{\chi_i\phi_0}}_{n \in \N^*}$ is convergent, hence bounded. Thus, there exists $C \geq 0$ such that, for all $j \in \N$,
\begin{equation*}
\norm*{\prsc{\nu_{R_n}}{\phi} - \gamma_1(f) \int_\Omega \phi \dx x} \leq C \Norm{\phi-\phi_j}_\infty +\norm*{\prsc{\nu_{R_n}}{\chi_i\phi_j} - \gamma_1(f) \int_\Omega \chi_i\phi_j \dx x}.
\end{equation*}

Let $\epsilon>0$, there exists $j \in \N^*$ such that $\Norm{\phi-\phi_j}_\infty \leq \epsilon$. Equation~\eqref{eq: as CV} for $(i,j)$ implies that the second term on the right-hand side goes to $0$ as $n \to +\infty$. Hence, for all $n$ large enough, we have
\begin{equation*}
\norm*{\prsc{\nu_{R_n}}{\phi} - \gamma_1(f) \int_\Omega \phi \dx x} \leq (C+1)\epsilon.
\end{equation*}
This proves that $\prsc{\nu_{R_n}}{\phi} \xrightarrow[n \to +\infty]{}\gamma_1(f) \int_\Omega \phi \dx x$ for all $\phi \in \cC_c^0(\Omega)$, under the assumption that \eqref{eq: as CV} holds. This condition being satisfied almost-surely, this concludes the proof.
\end{proof}

We conclude this section with the proof of Theorem~\ref{thm: law of large numbers stationary case}. It consists in putting together a special case of Theorem~\ref{thm: law of large numbers} with a monotonicity argument.

\begin{proof}[Proof of Theorem~\ref{thm: law of large numbers stationary case}]
Let us consider a centered stationary Gaussian field $f:\R^d \to \R^k$. We let $\cR=(0,+\infty)$, $U=\R^d =\Omega$ and $f_R=f$ for all $R >0$. Then, as explained in Section~\ref{subsec: discussion of the main hypotheses}, the sequence $(f_R)_{R \in \cR}$ fits in our general framework. Moreover, under the hypotheses of Theorem~\ref{thm: law of large numbers stationary case}, it satisfies \hypReg{3}, \hypScL{3}, \hypDC{3}{2} and either \hypNND{3} or \hypND{3}, depending on whether $f$ is gradient field or not.

Let $\cB \subset \R^d$ be a bounded Borel subset which is star-shaped with respect to $0$. We denote by $\one_\cB \in L^1(\R^d) \cap L^\infty(\R^d)$ its indicator function. Recalling Equation~\eqref{eq: relation nuR volume}, we apply Theorem~\ref{thm: law of large numbers} with $\phi=\one_\cB$, $p=1$ and $R_n=n^2$ for all $n \in \N^*$. This yields that, almost-surely,
\begin{equation}
\label{eq: as CV volume}
\frac{\Vol_{d-k}(Z \cap n^2\cB)}{n^{2d}} = \prsc{\nu_{n^2}}{\one_\cB} \xrightarrow[n \to +\infty]{} \gamma_1(f) \int_{\R^d} \one_\cB(x) \dx x = \gamma_1(f) \Vol_d(\cB).
\end{equation}

For all $R \geq 1$, let us denote by $n_R = \floor{\sqrt{R}}$, so that $n_R^2 \leq R < (n_R+1)^2$. Since $\cB$ is star-shaped with respect to $0$, we have $n_R^2 \cB \subset R \cB \subset (n_R+1)^2\cB$, and hence:
\begin{equation*}
\frac{\Vol_{d-k}\parentheses*{Z \cap n_R^2\cB}}{R^d} \leq \frac{\Vol_{d-k}\parentheses*{Z \cap R\cB}}{R^d} \leq \frac{\Vol_{d-k}\parentheses*{Z \cap (n_R+1)^2\cB}}{R^d}.
\end{equation*}
Besides, $\frac{1}{(n_R+1)^2} \leq \frac{1}{R} \leq \frac{1}{n_R^2}$, so that
\begin{equation*}
\parentheses*{1-\frac{1}{n_R+1}}^{2d}\frac{\Vol_{d-k}\parentheses*{Z \cap n_R^2\cB}}{(n_R)^{2d}} \leq \frac{\Vol_{d-k}\parentheses*{Z \cap R\cB}}{R^d} \leq \frac{\Vol_{d-k}\parentheses*{Z \cap (n_R+1)^2\cB}}{(n_R+1)^{2d}}\parentheses*{1+\frac{1}{n_R}}^{2d}.
\end{equation*}
Since $n_R \xrightarrow[R \to +\infty]{}+\infty$, if~\eqref{eq: as CV volume} holds, we obtain that $\frac{\Vol_{d-k}\parentheses*{Z \cap R\cB}}{R^d}\xrightarrow[R \to +\infty]{}\gamma_1(f) \Vol_d(\cB)$. This concludes the proof, since we know that~\eqref{eq: as CV volume} holds almost-surely.
\end{proof}

%%%%%%%%%%%%%%%%%%%%%%%%%%%%%%%%%%%%%%%%%%%%%%%%%%%%%%%%%%%%%%%%%%%%%%%%%%%%%%%%%%%%%%%%%%%%%%%%%%%%%%%

\subsection{Proof of Theorem~\ref{thm: central limit theorem}: Central Limit Theorem}
\label{subsec: proof of central limit theorem}

In this section, we deduce the Central Limit Theorem~\ref{thm: central limit theorem} from our cumulants asymptotics, under the assumptions that \hypReg{q}, \hypScL{q}, \hypDC{q}{2} and \hypND{q} (resp.~\hypNND{q} in the case of gradient fields) hold for all $q \in \N$. We start by proving the CLT for linear statistics.

\begin{proof}[Proof of Theorem~\ref{thm: central limit theorem}: the case of linear statistics]
We start by the proving the result for a single linear statistics, that is, for $n=1$. Let $\phi \in L^1(\Omega) \cap L^\infty(\Omega)$. For all $R \in \cR$ large enough, we have:
\begin{equation*}
\kappa_1\parentheses*{R^\frac{d}{2}\parentheses*{\prsc{\nu_R}{\phi}-\esp{\prsc{\nu_R}{\phi}\strut}}} = 0.
\end{equation*}

Let $p \geq 2$ be an integer. Under the hypotheses of Theorem~\ref{thm: central limit theorem}, for all $R \in \cR$ large enough, the $p$-th cumulant of $\prsc{\nu_R}{\phi}$ is well-defined, see Theorem~\ref{thm: cumulants asymptotics for zero sets}.\ref{item: cumulants asymptotics infty} (resp.~Theorem~\ref{thm: cumulants asymptotics for critical points}.\ref{item: cumulants asymptotics infty crit} in the case of gradient fields). Using the multi-linearity of joint cumulants, and the cancellation property of Lemma~\ref{lem: cancellation property} to remove the terms where an expectation appears, we have:
\begin{equation*}
\kappa_p\parentheses*{R^\frac{d}{2}\parentheses*{\prsc{\nu_R}{\phi}-\esp{\prsc{\nu_R}{\phi}\strut}}} = R^\frac{pd}{2}\kappa_p\parentheses*{\prsc{\nu_R}{\phi}}.
\end{equation*}
Then, by Theorem~\ref{thm: cumulants asymptotics for zero sets}.\ref{item: cumulants asymptotics 2} and~\ref{thm: cumulants asymptotics for zero sets}.\ref{item: cumulants asymptotics p} (resp.~Theorem~\ref{thm: cumulants asymptotics for critical points}.\ref{item: cumulants asymptotics 2 crit} and~\ref{thm: cumulants asymptotics for critical points}.\ref{item: cumulants asymptotics p crit}), we have:
\begin{equation*}
\kappa_p\parentheses*{R^\frac{d}{2}\parentheses*{\prsc{\nu_R}{\phi}-\esp{\prsc{\nu_R}{\phi}\strut}}} \xrightarrow[R \to +\infty]{} \begin{cases}\gamma_2(f)\displaystyle\int_{x \in \Omega} \phi(x)^2 \dx x, & \text{if} \ p =2;\\ 0, & \text{otherwise.} \end{cases}
\end{equation*}
Hence, we have: $R^\frac{d}{2}\parentheses*{\prsc{\nu_R}{\phi}-\esp{\prsc{\nu_R}{\phi}\strut}} \xrightarrow[R \to +\infty]{\text{law}}\gauss{\gamma_2(f)\Norm{\phi}_2^2}$ by Proposition~\ref{prop: method of cumulants}.

Now, let $n \in \N^*$ let $\phi_1,\dots,\phi_n \in L^1(\Omega) \cap L^\infty(\Omega)$ and $\Lambda=\gamma_2(f)\begin{pmatrix}
\int_{\Omega}\phi_i\phi_j\dx x
\end{pmatrix}_{1 \leq i,j\leq n}$. For all $R \in \cR$ large enough, let $X_R = R^\frac{d}{2}\parentheses*{\prsc*{\nu_R}{\phi_i}-\esp{\prsc*{\nu_R}{\phi_i}\strut}}_{1 \leq i \leq n} \in \R^n$. We want to prove that $X_R$ converges in distribution towards the centered Gaussian of variance $\Lambda$. By Lévy's Theorem, it is enough to show that $\esp{e^{i\prsc{z}{X_R}}}\xrightarrow[R \to +\infty]{}e^{-\frac{1}{2}\prsc{z}{\Lambda z}}$ for all $z \in \R^n$, where $\prsc{\cdot}{\cdot}$ stands for the standard Euclidean inner product on $\R^n$.

Let $z=(z_1,\dots,z_n) \in \R^n$ and $\phi = \sum_{i=1}^n z_i\phi_i\in L^1(\Omega) \cap L^\infty(\Omega)$. By the previous case, we have
\begin{align*}
\prsc{z}{X_R} &= R^\frac{d}{2}\sum_{i=1}^n z_i\parentheses*{\prsc*{\nu_R}{\phi_i}-\esp{\prsc*{\nu_R}{\phi_i}\strut}}\\
&= R^\frac{d}{2}\parentheses*{\prsc{\nu_R}{\phi}-\esp{\prsc{\nu_R}{\phi}\strut}} \xrightarrow[R \to +\infty]{\text{law}} \gauss{\gamma_2(f)\Norm{\phi}_2^2}.
\end{align*}
In particular, $\esp{e^{i\prsc{z}{X_R}}}\xrightarrow[R \to +\infty]{}e^{-\frac{1}{2}\gamma_2(f)\Norm{\phi}_2^2}$. The conclusion follows, since
\begin{equation*}
\gamma_2(f)\Norm*{\phi}_2^2 = \gamma_2(f) \int_{x \in \Omega} \parentheses*{\sum_{i=1}^n z_i\phi_i(x)}^2\dx x = \sum_{1 \leq i,j \leq n}z_iz_j \gamma_2(f) \int_{x \in \Omega} \phi_i(x)\phi_j(x) \dx x = \prsc{z}{\Lambda z}.\qedhere
\end{equation*}
\end{proof}

The remainder of this section is dedicated to proving the functional version of the Central Limit Theorem~\ref{thm: central limit theorem}. First, we need to recall a few facts about generalized random fields. Note that we use the term ``generalized function'' instead of ``distribution'' in order to avoid any possible confusion with the concept of distribution of a random variable.

\begin{dfn}[Test-functions and generalized functions]
\label{def: D and D'}
Recall that $\Omega \subset \R^d$ is open.
\begin{itemize}
\item A \emph{test-function} on $\Omega$ is a $\cC^\infty$ function $\phi:\Omega\to\R$ with compact support. The space $\cD(\Omega)$ of test-functions is endowed with its standard topology, which we do not need to describe for our purpose. We refer to~\cite[Chap.~I.2]{Sch1966} for more details on this matter.

\item A \emph{generalized function} on $\Omega$ is a linear form $T$ on $\cD(\Omega)$, which is continuous in the sense~that: for any compact $K \subset \Omega$, there exist $C \geq 0$ and $n \in \N$ such that $\norm*{T(\phi)}\leq C \sum_{\norm{\alpha}\leq n} \Norm*{\partial^\alpha \phi}_\infty$ for all $\phi \in \cD(\Omega)$ supported in~$K$. We denote by $\cD'(\Omega)$ the space of generalized functions.

\item The canonical pairing between $\cD'(\Omega)$ and $\cD(\Omega)$ is denoted by $\prsc{\cdot}{\cdot}$, so that we denote $\prsc{T}{\phi}$ instead of $T(\phi)$ for all $T \in \cD'(\Omega)$ and $\phi \in \cD(\Omega)$.

\item A \emph{generalized random field} on $\Omega$ is a random variable taking values in $\cD'(\Omega)$, endowed with the Borel $\sigma$-field induced by its usual topology, see~\cite[Sect.~III]{Fer1967} for details.
\end{itemize}
\end{dfn}

Radon measures on $\Omega$ are examples of generalized functions, see~\cite[Chap.~I.2]{Sch1966}. As discussed in Section~\ref{subsec: validity of Kac-Rice}, under the hypotheses of Theorem~\ref{thm: central limit theorem}, for all $R \in \cR$ large enough, $\nu_R$ is an almost-surely well-defined Radon measure on $\Omega$. The characterization of~\cite[p.~61]{Fer1967} shows that these $\nu_R$ define generalized random fields in the sense of Definition~\ref{def: D and D'}.

\begin{lem}(Mean measure)
\label{lem: E nuR distrib}
Under the assumptions of Theorem~\ref{thm: central limit theorem}, for all $R \in \cR$ large enough, the linear form $\esp{\nu_R}:\phi \mapsto \esp{\prsc{\nu_R}{\phi}}$ on $\cD(\Omega)$ defines a generalized function.
\end{lem}

\begin{proof}
Let $R_1 \in \cR$ be given by Theorem~\ref{thm: cumulants asymptotics for zero sets}.\ref{item: cumulants asymptotics infty} (resp.~Theorem~\ref{thm: cumulants asymptotics for critical points}.\ref{item: cumulants asymptotics infty crit} in the case of gradient fields). Then, $\esp{\nu_R}$ is a well-defined linear form on $\cD(\Omega)$ for all $R \in \cR$ such that $R \geq R_1$. 

Let us fix such a $R\geq R_1$. Let $K \subset \Omega$ be compact and $\one_K$ denote its indicator function. For all $\phi \in \cD(\Omega)$ supported in $K$, we have $\norm{\phi}\leq \Norm{\phi}_\infty \one_K$, so that
\begin{equation*}
\norm*{\prsc{\esp{\nu_R}}{\phi}} = \norm*{\esp{\prsc{\nu_R}{\phi}}}\leq \esp{\prsc{\nu_R}{\norm{\phi}}} \leq \Norm{\phi}_\infty \esp{\prsc{\nu_R}{\one_K}}.
\end{equation*}
Thus, $\esp{\nu_R} \in \cD'(\Omega)$.
\end{proof}

A generalized random field on $\Omega$ can be described by its characteristic functional, which play the same role as the characteristic function for random variables.

\begin{dfn}[Characteristic functional]
\label{dfn: characteristic functional}
Let $T \in \cD'(\Omega)$ be a generalized random field. Its \emph{characteristic functional} is the map $\Theta_T:\phi \mapsto \esp{e^{i\prsc{T}{\phi}}}$ from $\cD(\Omega)$ to $\C$.
\end{dfn}

\begin{prop}[Standard Gaussian White Noise]
\label{prop: white noise}
There exists a generalized random field~$W$ such that $\Theta_W:\phi \mapsto e^{-\frac{1}{2}\Norm{\phi}_2^2}$, or equivalently such that $\prsc{W}{\phi} \sim \gauss{\Norm{\phi}_2^2}$ for all $\phi \in \cD(\Omega)$. The field $W$, called the \emph{standard Gaussian White Noise} on $\Omega$, is unique in distribution.
\end{prop}

\begin{proof}
Since convergence in $\cD(\Omega)$ implies convergence in the $L^2$ sense, the map $\phi \mapsto e^{-\frac{1}{2}\Norm{\phi}_2^2}$ is continuous on $\cD(\Omega)$. By~\cite[Sect.~II.5.4]{Fer1967}, this map is of positive type. Hence, by the Bochner--Minlos Theorem~\cite[Thm.~II.5.3.(a)]{Fer1967}, there exists a generalized random field $W$ on $\Omega$ such that $\Theta_W:\phi \mapsto e^{-\frac{1}{2}\Norm{\phi}_2^2}$.

Let $W$ be such a generalized random field. Let $n \in \N^*$ and $\phi_1,\dots,\phi_n \in \cD(\Omega)$. For all $z =(z_1,\dots,z_n) \in \R^n$, we have:
\begin{equation*}
\esp{e^{i \sum_{j=1}^n z_j \prsc{W}{\phi_j}}} = \Theta_W\parentheses*{\sum_{j=1}^n z_j \phi_j} = \exp\parentheses*{-\frac{1}{2}\sum_{1\leq i,j \leq n}z_iz_j \int_{\Omega} \phi_i\phi_j \dx x}.
\end{equation*}
Thus, $\parentheses*{\prsc{W}{\phi_i}}_{1 \leq i\leq n}$ is a centered Gaussian vector with covariance matrix $\begin{pmatrix}
\int_\Omega \phi_i\phi_j\dx x
\end{pmatrix}_{1 \leq i,j \leq n}$. By~\cite[Prop.~II.4.2.(b)]{Fer1967}, the distribution of $W$ is uniquely determined by the distributions of its marginals $\parentheses*{\prsc{W}{\phi_i}}_{1 \leq i\leq n}$, for all $n \in \N^*$ and $\phi_1,\dots,\phi_n \in \cD(\Omega)$. Hence its uniqueness.

We have just seen that, if $\Theta_W:\phi \mapsto e^{-\frac{1}{2}\Norm{\phi}_2^2}$, then $\prsc{W}{\phi}\sim \gauss{\Norm{\phi}_2^2}$ for all $\phi \in \cD(\Omega)$. Conversely, for all $\phi \in \cD(\Omega)$, if $\prsc{W}{\phi}\sim \gauss{\Norm{\phi}_2^2}$, then $\Theta_W(\phi)=\esp{e^{i\prsc{W}{\phi}}} = e^{-\frac{1}{2}\Norm{\phi}_2^2}$.
\end{proof}

Now that we established the existence of the standard Gaussian White Noise and characterized its distribution, our functional CLT follows from the CLT for linear statistics and the Lévy--Fernique Theorem, which extends Lévy's Continuity Theorem to the setting of generalized random fields.

\begin{proof}[Proof of Theorem~\ref{thm: central limit theorem}: the functional case]
Under the hypotheses of Theorem~\ref{thm: central limit theorem}, by the preceding discussion and Lemma~\ref{lem: E nuR distrib}, for all $R \in \cR$ large enough, $T_R = R^\frac{d}{2}\parentheses*{\nu_R-\esp{\nu_R}}$ is a well-defined generalized random field on $\Omega$.

Let $\phi \in \cD(\Omega) \subset L^1(\Omega) \cap L^\infty(\Omega)$. By Theorem~\ref{thm: central limit theorem} in the case of linear statistics, we have:
\begin{equation*}
\prsc{T_R}{\phi} = R^\frac{d}{2}\parentheses*{\prsc{\nu_R}{\phi}-\esp{\prsc{\nu_R}{\phi}\strut}} \xrightarrow[R \to +\infty]{\text{law}} \gauss{\gamma_2(f)\Norm{\phi}_2^2} = \sqrt{\gamma_2(f)}\prsc{W}{\phi},
\end{equation*}
hence $\Theta_{T_R}(\phi) = \esp{e^{i\prsc{T_R}{\phi}}} \xrightarrow[R \to +\infty]{} \esp{e^{i\sqrt{\gamma_2(f)}\prsc{W}{\phi}}} = \Theta_{\sqrt{\gamma_2(f)}W}(\phi)$. By the Lévy--Fernique Theorem~\cite[Thm.~III.6.5]{Fer1967}, this pointwise convergence of $\Theta_{T_R}$ implies that $T_R \xrightarrow[R \to +\infty]{}\sqrt{\gamma_2(f)}W$ in distribution in $\cD'(\Omega)$.
\end{proof}

%%%%%%%%%%%%%%%%%%%%%%%%%%%%%%%%%%%%%%%%%%%%%%%%%%%%%%%%%%%%%%%%%%%%%%%%%%%%%%%%%%%%%%%%%%%%%%%%%%%%%%%
%%%%%%%%%%%%%%%%%%%%%%%%%%%%%%%%%%%%%%%%%%%%%%%%%%%%%%%%%%%%%%%%%%%%%%%%%%%%%%%%%%%%%%%%%%%%%%%%%%%%%%%

\bibliographystyle{amsplain}
\bibliography{CumulantsAsymptotics}

%%%%%%%%%%%%%%%%%%%%%%%%%%%%%%%%%%%%%%%%%%%%%%%%%%%%%%%%%%%%%%%%%%%%%%%%%%%%%%%%%%%%%%%%%%%%%%%%%%%%%%%
%%%%%%%%%%%%%%%%%%%%%%%%%%%%%%%%%%%%%%%%%%%%%%%%%%%%%%%%%%%%%%%%%%%%%%%%%%%%%%%%%%%%%%%%%%%%%%%%%%%%%%%

\end{document}